\documentclass[letterpaper,11pt]{article}
\usepackage{theorem,amsmath,amssymb,amscd}
\usepackage[all]{xy}
\usepackage{verbatim}
\usepackage[hypertex, hyperfootnotes=false]{hyperref}

\theoremstyle{change}
\allowdisplaybreaks
\newcommand{\R}{\mathbb{R}}
\newcommand{\C}{\mathbb{C}}
\newcommand{\Z}{\mathbb{Z}}
\newcommand{\Q}{\mathbb{Q}}
\newcommand{\q}{\mathcal{Q}}
\renewcommand{\S}{\mathcal{S}}
\newcommand{\A}{\mathbb{A}}

\newcommand{\bs}{\backslash}

\newcommand{\SH}{\mathfrak h}
\newcommand{\HH}{\mathbb{H}}
\newcommand{\St}{\mathrm{St}}
\newcommand{\GL}{\mathrm{GL}}
\newcommand{\PGL}{\mathrm{PGL}}
\newcommand{\GU}{\mathrm{GU}}
\newcommand{\U}{\mathrm{U}}
\newcommand{\SL}{\mathrm{SL}}
\newcommand{\SO}{\mathrm{SO}}
\newcommand{\SU}{\mathrm{SU}}
\renewcommand{\O}{\mathrm{O}}
\newcommand{\Sp}{\mathrm{Sp}}
\newcommand{\GSp}{\mathrm{GSp}}
\newcommand{\PGSp}{\mathrm{PGSp}}
\newcommand{\Res}{\mathrm{Res}}
\newcommand{\SSp}{\mathrm{Sp}}
\newcommand{\OF}{\mathfrak o}
\newcommand{\tr}{{\rm tr}}
\newcommand{\tl}{\widetilde l}
\renewcommand{\i}{\iota}

\newcommand{\G}{{G'_3}}
\newcommand{\p}{\mathfrak p}
\newcommand{\Aut}{\mathrm {Aut}}
\renewcommand{\P}{\mathfrak P}

\newcommand{\AI}{\mathcal{AI}}
\newcommand{\mat}[4]{{\setlength{\arraycolsep}{0.5mm}\left[
\begin{array}{cc}#1&#2\\#3&#4\end{array}\right]}}
\newcommand{\qed}{\hspace*{\fill}\rule{1ex}{1ex}}
\def\qdots{\mathinner{\mkern1mu\raise0pt\vbox{\kern7pt\hbox{.}}\mkern2mu
\raise3.4pt\hbox{.}\mkern2mu\raise7pt\hbox{.}\mkern1mu}}

\newcommand{\tR}{\widetilde R}
\newcommand{\Mat}{\text{Mat}}
\renewcommand{\St}{{\rm St}}
\renewcommand{\Re}{{\mathrm{Re}}}

\newenvironment{proof}{\vspace{1ex}\noindent{\bf Proof.}\hspace{0.5em}}
	{\hfill\qed\vspace{2ex}}

\newtheorem{thm}{Theorem}[subsection]
\newtheorem{theorem}[thm]{Theorem.}

\newtheorem{lemma}[thm]{Lemma.}

\newtheorem{corollary}[thm]{Corollary.}

\newtheorem{proposition}[thm]{Proposition.}
\newtheorem{rem}[thm]{Remark.}

\begin{document}

%

\begin{center}
{\Large Transfer of Siegel cusp forms of degree $2$}

\vspace{2ex}
Ameya Pitale\footnote{{\tt apitale@math.ou.edu}\qquad $^2${\tt abhishek.saha@bris.ac.uk}\qquad $^3${\tt rschmidt@math.ou.edu}\\[0.5ex] \phantom{xxx}MSC: 11F70, 11F46, 11F67}, Abhishek Saha$^2$, Ralf Schmidt$^3$


\vspace{3ex}
\begin{minipage}{80ex}
 \small
 {\sc Abstract.} Let $\pi$ be the automorphic representation of $\GSp_4(\A)$ generated by a full level cuspidal Siegel eigenform that is not a Saito-Kurokawa lift, and $\tau$ be an arbitrary cuspidal, automorphic representation of $\GL_2(\A)$. Using Furusawa's integral representation for $\GSp_4\times\GL_2$ combined with a pullback formula involving the unitary group $\GU(3,3)$, we prove that the $L$-functions $L(s,\pi\times\tau)$ are ``nice''. The converse theorem of Cogdell and Piatetski-Shapiro then implies that such representations $\pi$ have a functorial lifting to a cuspidal representation of $\GL_4(\A)$. Combined with the exterior-square lifting of Kim, this also leads to a functorial lifting of $\pi$ to a cuspidal representation of $\GL_5(\A)$. As an application, we obtain analytic properties of various $L$-functions related to full level Siegel cusp forms. We also obtain special value results for $\GSp_4\times\GL_1$ and $\GSp_4\times\GL_2$.
\end{minipage}

\end{center}

\tableofcontents

\section*{Introduction}
\addcontentsline{toc}{section}{Introduction}

   We will start by giving some general background on Siegel modular forms and automorphic representations and then go on to explain the contents of this work.

\vspace{3ex}
{\bf Siegel modular forms}

\vspace{2ex}
Classical elliptic modular forms, which are holomorphic functions on the complex upper half plane with certain transformation properties, can be generalized in various directions. One such generalization is the theory of \emph{Siegel modular forms}, which includes the elliptic case as the \emph{degree one} case. General references for Siegel modular forms are \cite{Freitag} and \cite{Klingen}. Just as in the elliptic case, Siegel modular forms come with a \emph{weight} and a \emph{level}. Parts of the theory generalize to the Siegel case in a straightforward way. For example, the space of Siegel modular forms of fixed weight and level is finite-dimensional. Siegel modular forms have Fourier expansions similar to that of elliptic modular forms.  Also, there is a theory of Hecke operators, hence a notion of eigenform, and each Siegel modular form is a linear combination of eigenforms. The most interesting Siegel modular forms are the \emph{cusp forms}, characterized by the vanishing of certain Fourier coefficients.

\vspace{3ex}
Beyond the elliptic case, much work has concentrated on Siegel modular forms of \emph{degree two}. For example, Igusa wrote two famous papers \cite{Igusa1962} and \cite{Igusa1964} in the 1960's, where he determined the structure of the ring of Siegel modular forms of degree two with respect to the full modular group $\SSp_4(\Z)$ (in analogy to the statement that the ring of modular forms with respect to $\SL_2(\Z)$ is generated by the algebraically independent Eisenstein series of weight $4$ and weight $6$). Another milestone came about a decade later, when Andrianov \cite{An1974} associated a degree-$4$ Euler product $L(s,F,{\rm spin})$, now known as the \emph{spin $L$-function}, to a Siegel modular form $F$ with respect to $\SSp_4(\Z)$ (assumed to be an eigenform for all Hecke operators) and proved its basic analytic properties: meromorphic continuation, functional equation, and control over the possible poles.

\vspace{3ex}
A few years after this, Saito and Kurokawa \cite{Ku1978} independently discovered the existence of degree two Siegel modular forms that violated the naive generalization of the Ramanujan conjecture: the statement that the roots of the Hecke polynomials in the denominator of the spin $L$-function of cusp forms have absolute value $1$. In a series of papers \cite{Maass1}, \cite{Maass2}, \cite{Maass3}, \cite{An1979}, \cite{Za1981} Maass, Andrianov and Zagier showed that such Siegel modular forms are precisely those that ``come from'' elliptic modular forms.\footnote{There exist  cuspidal automorphic representations of $\Sp_4(\A)$ (and of $\GSp_4(\A)$) which violate the
Ramanujan conjecture but are \emph{not} lifts from elliptic modular forms, such as the examples constructed by Howe and Piatetski-Shapiro \cite{howe-ps}. Such representations, however, cannot arise from holomorphic Siegel modular forms (of any level) with weight $k>2$; see Corollary 4.5 of \cite{PS2}.} More precisely, there is a construction, now known as the \emph{Saito-Kurokawa lifting}, which associates to an elliptic eigenform $f$ for $\SL_2(\Z)$ a Siegel eigenform $F$ for $\SSp_4(\Z)$, such that the spin $L$-function of $F$ is the product of the Hecke $L$-function of $f$ times two zeta factors. More precisely, with appropriate normalization and disregarding archimedean factors,
\begin{equation}\label{SKLfunctioneq}
 L(s,F,{\rm spin})=L(s,f)\zeta(s+1/2)\zeta(s-1/2).
\end{equation}
The zeta factors produce poles, and in fact Evdokimov \cite{Ev1980} and Oda \cite{Oda1981} have shown that the existence of a pole in the spin $L$-function characterizes Saito-Kurokawa liftings. The book \cite{eichzag} by Eichler and Zagier gives a streamlined account of the construction of Saito-Kurokawa lifts via the theory of Jacobi forms. If a holomorphic eigenform for $\SSp_4(\Z)$ is not Saito-Kurokawa, then in fact it satisfies the Ramanujan conjecture. This was proved by Weissauer \cite{weissram}.

\vspace{3ex}
Many questions that have been answered in the elliptic case remain open for higher degree Siegel modular forms. For example, there is as of yet no good theory of old- and newforms resembling the classical theory of Atkin and Lehner. Questions of level aside, even the case of Siegel modular forms with respect to the full modular group has many challenges remaining. For instance, it is not yet known if a Siegel modular form for $\SSp_4(\Z)$, assumed to be an eigenform for all Hecke operators, is determined by its Hecke eigenvalues (a statement known as \emph{multiplicity one}). This has to do with the difficulty of relating Hecke eigenvalues and Fourier coefficients (unlike in the elliptic case, where the Hecke eigenvalues \emph{are} the Fourier coefficients). Indeed, the Fourier coefficients of Siegel cusp forms of degree 2 are mysterious arithmetic objects, which are conjecturally related to central $L$-values of twisted spin $L$-functions; see~\cite{furusawa-shalika-fund} for a good discussion. The precise form of this relationship is known as \emph{B\"ocherer's conjecture}, and one of us has observed \cite{Sa2012} that a version of this conjecture implies multiplicity one.

\vspace{3ex}
{\bf Automorphic representations}

\vspace{2ex}
It is well known that the theory of elliptic modular forms embeds into the more general theory of automorphic forms on the group $\GL_2(\A)$. Here, $\A$ denotes the ring of adeles of the number field $\Q$. The details of this process are explained in \cite{Bu} and \cite{Ge1975}, and are roughly as follows. Given an elliptic cusp form $f$ of weight $k$ and level $N$ that is an eigenfunction of the Hecke operators $T(p)$ for all but finitely many primes $p$, one may associate to $f$ a complex-valued function $\Phi_f$ on $\GL_2(\A)$ satisfying certain invariance properties. In particular, $\Phi_f$ is left-invariant under the group of rational points $\GL_2(\Q)$, right invariant under a compact-open subgroup of the finite adeles depending on $N$, and transforms according to the character $e^{2\pi ikx}$ of the group $\SO(2)\cong\R/\Z$ at the archimedean place. Let $V$ be the space of functions spanned by right translates of $\Phi_f$. Then $V$ carries a representation of the group $\GL_2(\A)$. We denote this representation by $\pi_f$ and call it the \emph{automorphic representation attached to $f$}. Using the hypothesis that $f$ is an eigenform, one can prove that $\pi_f$ is irreducible. Like any irreducible representation of $\GL_2(\A)$, it factors as a restricted tensor product $\otimes_v\pi_v$, where $\pi_v$ is an irreducible, admissible representation of the local group $\GL_2(\Q_v)$. The product extends over all places $v$ of $\Q$, and for $v=\infty$ we understand that $\Q_v=\R$. The original modular form $f$, or rather its adelic version $\Phi_f$, can be recovered as a special vector in the representation $\pi$. In fact, if $f$ is a newform, then $\Phi_f$ is a pure tensor $\otimes_v \phi_v$, where each $\phi_p$ is a \emph{local newform} in $\pi_p$ for finite $p$, and $\phi_\infty$ is a lowest weight vector in $\pi_\infty$, a discrete series representation.

\vspace{3ex}
This procedure generalizes to Siegel modular forms $F$ of degree $n$ for the full modular group. The group $\GL_2$ is to be replaced by the symplectic similitude group $\GSp_{2n}$; see \cite{ASch} for details. The resulting representation $\pi_\infty$ of $\GSp_{2n}(\R)$ is a holomorphic discrete series representation with scalar minimal $K$-type determined by the weight of $F$. Unfortunately, for Siegel modular forms with level, the corresponding procedure is not quite as nice, due to the lack of a good theory of \emph{local newforms} for representations of the group $\GSp_{2n}(\Q_p)$, and our lack of knowledge of global multiplicity one. However, see~\cite{sahayosh} for a treatment of  adelization for Siegel cusp forms of arbitrary level that suffices for many applications.

\vspace{3ex}
Once a modular form is realized as a special vector in an automorphic representation, can the considerable machinery available for such representations be used to gain new insights into the classical theory? Sometimes, the answer is yes. For example,  a method of Langlands, formulated for automorphic representations, was used in \cite{ASch} to prove that the spin $L$-functions of Siegel modular forms of degree three have meromorphic continuation to the entire complex plane. As an example in the elliptic modular forms case, one might use the Langlands-Shahidi method to deduce analytic properties of several \emph{symmetric power $L$-functions} $L(s,f,{\rm sym}^n)$ attached to an elliptic cuspidal eigenform $f$.

\vspace{3ex}
There is however a serious limiting factor to the applicability of automorphic methods to Siegel modular forms of higher degree. Namely, if $F$ is an eigen-cuspform of degree $n>1$, assumed to be of full level for simplicity, then the associated automorphic representation $\pi_F$ of $\GSp_{2n}(\A)$ is \emph{non-generic} (meaning it has no Whittaker model). The obstruction comes from the archimedean place: If $\pi_F=\otimes_v\pi_v$, then the archimedean component $\pi_\infty$, which is a holomorphic discrete series representation, is not generic. If $\pi_F$ were generic, one could, for example, apply the Langlands-Shahidi method, and immediately obtain the analytic properties of a series of $L$-functions. Also, at least for the degree two case, questions of multiplicity one could be answered immediately; see \cite{JiSo2007}.

\vspace{3ex}
{\bf Functoriality}

\vspace{2ex}
Langlands' principle of functoriality, a central conjecture in the theory of automorphic representations, describes the relationships between automorphic objects living on two different algebraic groups. More precisely, let $G$ and $H$ be reductive, algebraic groups, which for simplicity we assume to be defined over $\Q$ and split. Attached to $G$ and $H$ are \emph{dual groups} $\hat G$ and $\hat H$, which are complex reductive Lie groups whose root systems are dual to those of $G$ and $H$, respectively. Then, according to the principle, every homomorphism of Lie groups $\hat G\rightarrow\hat H$ should give rise to a ``lifting'', or ``transfer'', of automorphic representations of $G(\A)$ to automorphic representations of $H(\A)$. For example, in \cite{GeJa1978}, Gelbart and Jacquet proved the existence of the \emph{symmetric square lifting} for $G=\GL_2$ to $H=\GL_3$. Here, $\hat G=\GL_2(\C)$ and $\hat H=\GL_3(\C)$, and $\hat G\rightarrow\hat H$ is an irreducible three-dimensional representation of $\hat G$. A more recent example, and one that we will use in this work, is Kim's exterior square lifting \cite{kim} from $G=\GL_4$ to $H=\GL_6$. Here, $\hat G=\GL_4(\C)$, $\hat H=\GL_6(\C)$, and $\hat G\rightarrow\hat H$ is the irreducible six-dimensional representation of $\hat G$ given as the exterior square of the four-dimensional standard representation.

\vspace{3ex}
What exactly qualifies as a ``lifting'' is often formulated in terms of the $L$-functions attached to automorphic representations. Let $G$ be as above, and let $\pi=\otimes_v\pi_v$ be an automorphic representation of $G(\A)$. As additional ingredient we need a finite-dimensional complex representation $\rho$ of $\hat G$. Attached to this data is an Euler product
$$
 L(s,\pi,\rho)=\prod_vL(s,\pi_v,\rho),
$$
where $s$ is a complex parameter. We ignore for a moment the fact that in many situations the local factors $L(s,\pi_v,\rho)$ may not be known; at least for almost all primes the factors \emph{are} known, and are of the form $Q(p^{-s})^{-1}$, where $Q(X)$ is a polynomial of degree equal to the dimension of $\rho$ and satisfying $Q(0)=1$. It is also known that the product converges in some right half plane. If $\rho$ is a natural ``standard'' representation, it is often omitted from the notation. Now if $H$ is a second group, and if $\varphi:\hat G\rightarrow\hat H$ is a homomorphism of Lie groups, then the associated ``lifting'' maps automorphic representations $\pi=\otimes\pi_v$ of $G(\A)$ to automorphic representations $\Pi=\otimes\Pi_v$ of $H(\A)$ in such a way that
\begin{equation}\label{functorialityLeq}
 L(s,\Pi,\rho)=L(s,\pi,\rho\circ\varphi)
\end{equation}
for all finite-dimensional representations $\rho$ of $\hat H$. Sometimes one can only prove that the Euler products coincide for almost all primes, in which case one may speak of a \emph{weak lifting}. For example, in \cite{kim}, Kim proved the equality of the relevant Euler products for the exterior square lifting at all primes outside $2$ and $3$. Later, Henniart \cite{He2009} showed the equality for the remaining primes, proving that Kim's lifting is in fact \emph{strong}.

\vspace{3ex}It seems worthwhile to emphasize here that each instance of lifting discovered so far has had numerous applications to number theory. Functoriality is a magic wand that forces additional constraints on the automorphic representations being lifted and allows us to prove various desirable local and global properties for them. To give just one example, Kim's \emph{symmetric fourth lifting} \cite{kim} from $\GL_2$ to $\GL_5$ has provided the best known bound towards the Ramanujan conjecture for cuspidal automorphic representations of $\GL_2$. By using these bounds for the case of $\GL_2$ over a totally real field, Cogdell, Piatetski-Shapiro and Sarnak \cite{CPSS} were able to confirm the last remaining
case of Hilbert's eleventh problem.

\vspace{3ex}
We mention that the Saito-Kurokawa lifting also fits into the framework of Langlands functoriality. Recall that this lifting maps elliptic modular forms to Siegel modular forms of degree $2$, so one would expect the relevant groups to be $G=\GL_2$ and $H=\GSp_4$, or rather, since all representations involved have trivial central character, the projective groups $G=\PGL_2$ and $H=\PGSp_4$. But in fact, one should really take $G=\PGL_2\times\PGL_2$; see \cite{La1979} and \cite{Sc2005}. The associated dual groups are $\hat G=\SL_2(\C)\times\SL_2(\C)$ and $\hat H=\SSp_4(\C)$, and the homomorphism of dual groups is given by
\begin{equation}\label{SKfunceq}
 \mat{a}{b}{c}{d},\mat{a'}{b'}{c'}{d'}\longmapsto\begin{bmatrix}a&&b\\&a'&&b'\\c&&d\\&c'&&d'\end{bmatrix}.
\end{equation}
The first factor $\PGL_2$ carries the representation $\pi_f$ associated to an elliptic eigenform $f$. The second, ``hidden'' factor $\PGL_2$ carries an \emph{anomalous} representation $\pi_{\rm an}$, which is a certain constituent of a parabolically induced representation. It is clear from \eqref{SKfunceq} that the lifting $\Pi$ of the pair $(\pi_f,\pi_{\rm an})$ 
has the property $L(s,\Pi)=L(s,\pi)L(s,\pi_{\rm an})$. Looking only at finite places, this identity is precisely the equality \eqref{SKLfunctioneq}. Hence, it is the presence of the anomalous representation $\pi_{\rm an}$ that accounts for the pole in the $L$-function, and the violation of the Ramanujan conjecture. Inside $\Pi$ one can find (the adelization of) the Siegel modular form $F$ that is the Saito-Kurokawa lifting of $f$.

\vspace{3ex}
{\bf The transfer of Siegel modular forms to $\GL_4$ and $\GL_5$}

\vspace{2ex}
Again we consider Siegel modular forms of degree $2$, restricting our attention to cusp forms and full level. If $F$ is an eigenform, we can attach to it a cuspidal, automorphic representation $\pi_F=\otimes_v \pi_v$ of $\GSp_4(\A)$. Now, the dual group of $\GSp_4$ is $\GSp_4(\C)$, which sits inside $\GL_4(\C)$. Interpreting the latter as the dual group of $\GL_4$, we see that the principle of functoriality predicts the existence of a lifting from $\GSp_4$ to $\GL_4$. In particular, we should be able to lift our modular form $F$ (or rather $\pi_F$) to an automorphic representation $\Pi$ of $\GL_4(\A)$. If $F$ is of Saito-Kurokawa type, it is obvious, but not very exciting, how to construct the lifting; the result will be a representation globally induced from the $(2,2)$-parabolic of $\GL_4(\A)$. In particular, the lifting is not cuspidal. It is much more intricate to lift non-Saito-Kurokawa forms; this is, in fact, the main theme of this work.

\vspace{3ex}
 Let us move our focus away from $\pi_F$ and consider for a moment all cuspidal representations of $\GSp_4(\A)$. What is the current status of the lifting from $\GSp_4$ to $\GL_4$ predicted by Langlands functoriality? In \cite{AS} Asgari and Shahidi have achieved the lifting for all \emph{generic} cuspidal, automorphic representations of $\GSp_4(\A)$. The reason for the restriction to generic representations lies in their use of the Langlands-Shahidi method. As emphasized already, Siegel cusp forms correspond to non-generic representations; so this method cannot be used to lift them.

\vspace{3ex}
Another commonly used tool to prove functoriality is the trace formula. Trace formula methods have the potential to prove the existence of liftings for \emph{all} automorphic representations. This method has been much developed by Arthur, but for specific situations is still subject to various versions of the \emph{fundamental lemma}. At the time of this writing it is unclear to the authors whether all the necessary ingredients for the lifting from $\GSp_4$ to $\GL_4$ are in place. Certainly, a complete proof is not yet published.

\vspace{3ex}
\emph{In this work we use the Converse Theorem to prove that full-level Siegel cusp forms of degree two can be lifted to $\GL_4$.} To the best of our knowledge, the Converse Theorem has not been used before to prove functorial transfer for a non-generic representation on a quasi-split group. Given $F$ and $\pi_F$ as above, it is easy enough to predict what the lifting $\Pi=\otimes\Pi_v$ to $\GL_4$ should be. In fact, we can \emph{define} $\Pi_v$, which is an irreducible, admissible representation of $\GL_4(\Q_v)$, for all places $v$ in such a way that the required lifting condition \eqref{functorialityLeq} is automatically satisfied. The only question is then: Is the representation $\Pi$ of $\GL_4(\A)$ thus defined automorphic? This is the kind of question the Converse Theorem is designed to answer. According to the version of the Converse Theorem given in \cite{CPS}, the answer is affirmative if the twisted (Rankin-Selberg) $L$-functions $L(s,\Pi\times\tau)$ are ``nice'' for \emph{all} automorphic representations $\tau$ of $\GL_2(\A)$, or alternatively, for all \emph{cuspidal} automorphic representations of $\GL_2(\A)$ and $\GL_1(\A)$. The $\GL_1$ twists are not a serious problem, so our main task will be to prove niceness for the $\GL_2$ twists. Recall that ``nice'' means the $L$-functions can be analytically continued to entire functions, satisfy a functional equation, and are bounded in vertical strips.

\vspace{3ex}
Once we establish the transfer of $\pi_F$ to $\GL_4$, we will go one step further and lift $\pi_F$ to the group $\GL_5$ as well. Recall that $\pi_F$ is really a representation of the projective group $\PGSp_4(\A)$, the dual group of which is $\SSp_4(\C)$. The lifting to $\GL_4$ comes from the inclusion $\SSp_4(\C)\rightarrow\GL_4(\C)$, or, in other words, the natural representation $\rho_4$ of $\SSp_4(\C)$ on $\C^4$. The ``next'' irreducible representation of $\SSp_4(\C)$ is five-dimensional, and we denote it by $\rho_5$. Interpreting $\rho_5$ as a homomorphism of dual groups $\SSp_4(\C)\rightarrow\GL_5(\C)$, the principle of functoriality predicts the existence of a lifting from $\PGSp_4$ to $\GL_5$. Using Kim's exterior square lifting \cite{kim}, one can in fact show that the transfer of $\pi_F$ to $\GL_5$ exists. To summarize, we will prove the following lifting theorem.

\vspace{2ex}
{\bf Theorem A:} \emph{Let $F$ be a cuspidal Siegel modular form of degree $2$ with respect to $\SSp_4(\Z)$. Assume that $F$ is an eigenform for all Hecke operators, and not of Saito-Kurokawa type. Let $\pi_F$ be the associated cuspidal, automorphic representation of $\GSp_4(\A)$. Then $\pi_F$ admits a strong lifting to an automorphic representation $\Pi_4$ of $\GL_4(\A)$, and a strong lifting to an automorphic representation $\Pi_5$ of $\GL_5(\A)$. Both $\Pi_4$ and $\Pi_5$ are cuspidal.}

\vspace{3ex}
For more precise statements of these results, see Theorem \ref{liftingtheorem} and Theorem \ref{liftingGL5theorem}.

\vspace{3ex}
{\bf Bessel models}

\vspace{2ex}
We have yet to explain how to prove ``niceness'' for the $L$-functions relevant for the Converse Theorem. Before doing so, let us digress and explain the important notion of \emph{Bessel model} for representations of $\GSp_4$. These models can serve as a substitute for the often missing Whittaker models. We start by explaining the local, non-archimedean notion of Bessel model. Thus, let $F$ be a $p$-adic field. We fix a non-trivial character $\psi$ of $F$. Recall that the \emph{Siegel parabolic subgroup} of $\GSp_4$ is the standard maximal parabolic subgroup whose radical $U$ is abelian. Let $S$ be a non-degenerate, symmetric matrix with coefficients in $F$. Then $S$ defines a character $\theta$ of $U(F)$ via
$$
 \theta(\mat{1}{X}{}{1})=\psi({\rm tr}(SX)).
$$
Let $T$ be the identity component of the subgroup of the Levi component of the Siegel parabolic fixing $\theta$. Hence, the elements $t$ of $T(F)$ satisfy $\theta(tut^{-1})=\theta(u)$ for all $u\in U(F)$. The group $T$ turns out to be abelian. In fact, it is a two-dimensional torus which is split exactly if $-\det(S)$ is a square in $F^\times$. The semidirect product $R=TU$ is called the \emph{Bessel subgroup} of $\GSp_4$ with respect to $\theta$. Every character $\Lambda$ of $T(F)$ gives rise to a character $\Lambda\otimes\theta$ of $R(F)$ via $(\Lambda\otimes\theta)(tu)=\Lambda(t)\theta(u)$.

\vspace{3ex}
Now, let $(\pi,V)$ be an irreducible, admissible representation of $\GSp_4(F)$. Let $\theta$ and $\Lambda$ be as above. We consider functionals $\beta:\:V\rightarrow\C$ with the property
$$
 \beta(\pi(tu)v)=\Lambda(t)\theta(u)\beta(v)
$$
for $t\in T(F)$, $u\in U(F)$ and $v\in V$. A non-zero such functional is called a \emph{$(\Lambda,\theta)$-Bessel functional} for $\pi$. It is known that, for fixed $\theta$ and $\Lambda$, there can be at most one such functional up to multiples; see \cite{NoPi1973}. It is also known that, unless $\pi$ is one-dimensional, there always exists a Bessel functional for \emph{some} choice of $\theta$ and $\Lambda$. If $\theta$ and $\Lambda$ are such that a $(\Lambda,\theta)$-Bessel functional for $\pi$ exists, then $\pi$ can be realized as a subspace of the space of functions $B:\:\GSp_4(F)\rightarrow\C$ with the transformation property
\begin{equation}\label{besseltranspropeq}
 B(tuh)=\Lambda(t)\theta(u)B(h)\qquad\text{for all }t\in T(F),\;u\in U(F),\;h\in\GSp_4(F),
\end{equation}
with the action of $\GSp_4(F)$ on this space given by right translation. Sugano \cite{Su} has determined the explicit formula in the above realization for the spherical vector in an unramified representation~$\pi$.

\vspace{3ex}
Similar definitions can be made, and similar statements hold, in the archimedean context. See \cite{PS} for explicit formulas for Bessel models for a class of lowest weight representations of $\GSp_4(\R)$. All we will need in this work are formulas for holomorphic discrete series representations with scalar minimal $K$-type. These have already been determined in \cite{Su}.

\vspace{3ex}
Next, consider global Bessel models. Given $S$ as above but with entries in $\Q$, we obtain a character $\theta$ of $U(\A)$ via a fixed non-trivial character $\psi$ of $\Q\backslash\A$. The resulting torus $T$ can be adelized, and is then isomorphic to the group of ideles $\A_L^\times$ of a quadratic extension $L$ of $\Q$. We assume that $-\det(S)$ is not a square in $\Q^\times$, so that $L$ is a field (and not isomorphic to $\Q\oplus\Q$). Let $\pi=\otimes\pi_v$ be a cuspidal, automorphic representation of $\GSp_4(\A)$, and let $V$ be the space of automorphic forms realizing $\pi$. Assume that a Hecke character $\Lambda$ of $\A_L^\times$ is chosen such that the restriction of $\Lambda$ to $\A^\times$ coincides with the central character of $\pi$. For each $\phi\in V$ consider the corresponding Bessel function
\begin{equation}\label{Bphieq2}
 B_\phi(g)=\int\limits_{Z_H(\A)R(\Q)\backslash R(\A)}
 (\Lambda\otimes\theta)(r)^{-1}\phi(rg)\,dr,
\end{equation}
where $Z_H$ is the center of $H:=\GSp_4$. If these integrals are non-zero, then we obtain a model of $\pi$ consisting of functions on $\GSp_4(\A)$ with a left transformation property analogous to \eqref{besseltranspropeq}. In this case, we say that $\pi$ has a \emph{global Bessel model} of type $(S,\Lambda,\psi)$. It implies that the corresponding \emph{local} Bessel models exist for all places $v$ of $\Q$. The local uniqueness of Bessel models implies global uniqueness. However, if we are given $\pi=\otimes\pi_v$ and some triple $(S,\Lambda,\psi)$ such that $\pi_v$ has a local Bessel model of type $(S_v,\Lambda_v,\psi_v)$ for each place $v$, then it does \emph{not} necessarily imply that $\pi$ has a global Bessel model of type $(S,\Lambda,\psi)$. Indeed, conjecturally, a global Bessel model exists in the above case if and only if a certain central $L$-value is non-vanishing; see~\cite{prasadbighash} for a discussion of this point.

\vspace{3ex}
So, what can be said about the \emph{existence} of global Bessel models for those automorphic representations coming from Siegel modular forms? Assume that $\pi_F=\otimes\pi_v$ is attached to a full level Siegel cusp form, as above. Then $\pi_p$, for each prime $p$, is a spherical representation. Such representations admit Bessel models for which the character $\Lambda_p$ is unramified. We would like our global Bessel data to be as unramified as possible. So the question arises, can we find a global Bessel model for which the Hecke character $\Lambda$ is unramified \emph{everywhere}? Not only that, we would like $L$ to be an \emph{imaginary} quadratic extension, and would like the archimedean component of $\Lambda$, which is a character of $\C^\times$, to be trivial. The existence of such a $\Lambda$ turns out to be related to the non-vanishing of certain Fourier coefficients of $F$; see Lemma \ref{piFBesselFourierlemma} for a precise statement. Using analytic methods and half-integral weight modular forms, the second author has recently proved \cite{squarefree} that the required non-vanishing condition is always satisfied. Hence, a particularly nice global Bessel model always exists for $\pi_F$. This removes assumption (0.1) of \cite{Fu}, and will make our results hold unconditionally for all cuspidal Siegel eigenforms of full level.

\vspace{3ex}
{\bf Furusawa's integrals}

\vspace{2ex}
We now return from the world of Bessel models to the problem of  proving ``niceness'' for the $L$-functions relevant for the Converse Theorem. Recall that, in order to apply the Converse Theorem for $\GL_4$, the essential task is to control the Rankin-Selberg $L$-functions $L(s,\Pi\times\tau)$, where $\Pi$ is the predicted transfer to $\GL_4(\A)$, and $\tau$ is an arbitrary cuspidal, automorphic representation of $\GL_2(\A)$. By the very definition of $\Pi$, this $L$-function equals $L(s,\pi\times\tau)$, where $\pi=\pi_F$ is the cuspidal representation of $\GSp_4(\A)$ attached to $F$. For this type of Rankin-Selberg product, Furusawa \cite{Fu} has pioneered an \emph{integral representation}, which we now explain. This integral representation involves an Eisenstein series on a unitary similitude group $\GU(2,2)$. Unitary groups are defined with respect to a quadratic extension $L$. Here, \emph{the appropriate quadratic field extension $L/\Q$ is the one coming from a global Bessel model for $\pi$}. Hence, given the cusp form $F$, we first find a particularly good Bessel model for $\pi$, with the Hecke character $\Lambda=\otimes\Lambda_v$ unramified everywhere, as explained above. The quadratic extension is then $L=\Q(\sqrt{-\det(S)})$. For the precise definition of $\GU(2,2)$ see \eqref{Gjdefinition}. Note that this group contains $\GSp_4$, which we henceforth abbreviate by $H$.

\vspace{3ex}
Let us next explain the Eisenstein series $E(h,s;f)$ appearing in Furusawa's integral representation. Eisenstein series come from sections in global parabolically induced representations. The relevant parabolic $P$ of $\GU(2,2)$ is the \emph{Klingen parabolic subgroup}, i.e., the maximal parabolic subgroup with non-abelian radical. There is a natural map from $\A_L^\times\times\A_L^\times\times\GL_2(\A)$ to the adelized Levi component of $P$. Therefore, via suitably chosen Hecke characters $\chi_0$ and $\chi$ of $\A_L^\times$, the $\GL_2(\A)$ cuspidal representation $\tau$ can be extented to a representation of the Levi component of $P(\A)$. Parabolic induction to all of $\GU(2,2)(\A)$ then yields representations $I(s,\chi,\chi_0,\tau)$, where $s$ is a complex parameter (see Sect.\ \ref{parabolicinductionsec} for details). The Eisenstein series is constructed from an analytic section $f$ in this family of induced representations via a familiar summation process; see \eqref{Edefeq}. By general theory, $E(h,s;f)$ is convergent in some right half plane, has meromorphic continuation to all of $\C$ and satisfies a functional equation.

\vspace{3ex}
With the Eisenstein series in place, Furusawa considers integrals of the form
\begin{equation}\label{globalintegraleq2}
 Z(s,f,\phi)=\int\limits_{H(\Q)Z_H(\A)\backslash H(\A)}E(h,s;f)\phi(h)\,dh.
\end{equation}
Here, as before, $Z_H$ denotes the center of $H=\GSp_4$. The function $\phi$ is a vector in the space of automorphic forms realizing $\pi$. Furusawa's ``basic identity'' (here equation (\ref{basicidentityeq})) shows that, if all the data are factorizable, then the integral $Z(s,f,\phi)$ is Eulerian, i.e., it factors into a product of local zeta integrals. More precisely, assume that $f=\otimes f_v$, with $f_v$ an analytic section in the local induced representation $I(s,\chi_v,\chi_{0,v},\tau_v)$. Assume also that $\phi=\otimes\phi_v$, a pure tensor in $\pi=\otimes\pi_v$. Then the local zeta integrals are of the form
\begin{equation}\label{localZseq2}
 Z(s,W_{f_v},B_{\phi_v})=\int\limits_{R(\Q_v)\backslash H(\Q_v)}W_{f_v}(\eta h,s)B_{\phi_v}(h)\,dh.
\end{equation}
Here, $\eta$ is a certain element in $\GU(2,2)(\Q_v)$ defined in \eqref{etadefeq}. The function $W_{f_v}$ is a Whittaker-type function depending on $f_v$, and $B_{\phi_v}$ is the vector corresponding to $\phi_v$ in the local $(\Lambda_v,\theta_v)$-Bessel model of $\pi_v$. We see how important it is that all the local Bessel models exist. We also see the usefulness of explicit formulas for $B_{\phi_v}$ in order to be able to evaluate the integrals \eqref{localZseq2}.

\vspace{3ex}
Furusawa has calculated the local integrals \eqref{localZseq2} in the non-archimedean case when all the local data is unramified. The result is
\begin{equation}\label{unramifiedfactorseq}
 Z(s,W_{f_p},B_{\phi_p})=\frac{L(3s+\frac 12, \tilde\pi_p \times \tilde\tau_p)}
  {L(6s+1,\chi_p|_{\Q_p^\times})L(3s+1,\tau_p \times \AI(\Lambda_p) \times \chi_p|_{\Q_p^\times})},
\end{equation}
where $\tilde\pi$ and $\tilde\tau$ are the contragredient representations, and where $\chi_p$ and $\Lambda_p$ are the local components of the Hecke characters $\chi$ and $\Lambda$ mentioned above. The symbol $\mathcal{AI}$ denotes automorphic induction; thus, the second $L$-factor in the denominator is a factor for $\GL_2\times\GL_2$. By taking the product of~\eqref{unramifiedfactorseq} over all unramified places, it follows that the quantity $Z(s,f,\phi)$ given by the integral (\ref{globalintegraleq2}) is essentially equal to the global $L$-function $L(3s +\frac12,\widetilde{\pi}\times\widetilde{\tau})$ divided by some well-understood global $L$-functions for $\GL_1$ and $\GL_2 \times \GL_2$ (here, ``essentially" means that we ignore a finite number of local factors corresponding to the ramified places). Consequently, if we can control the local factors at these bad (ramified) places, the integral (\ref{globalintegraleq2}) can be used to study $L(s,\pi\times\tau)$. In the end, $L(s,\pi\times\tau)$ will inherit analytic properties, like meromorphic continuation and functional equation, from the Eisenstein series appearing in \eqref{globalintegraleq2}. This is the essence of the method of integral representations.

\vspace{3ex}
{\bf The art of choosing distinguished vectors in local representations}

\vspace{2ex}

Recall that the identity \eqref{unramifiedfactorseq} for the local zeta integrals holds only if \emph{all} the local ingredients are unramified, including $\pi_p$, $\tau_p$, $\Lambda_p$ and $\chi_p$. The representations $\pi_p$ are always unramified since the modular form $F$ has full level. The character $\Lambda_p$ is also unramified by our choice of Bessel model. But, in order to apply the Converse Theorem, we need to be able to twist by arbitrary $\GL_2$ representations, meaning that $\tau_p$ could be any irreducible, admissible, infinite-dimensional representation of $\GL_2(\Q_p)$.

\vspace{3ex}
There is a natural choice for the function $\phi$ appearing in \eqref{globalintegraleq2}, namely, the adelization of the modular form $F$. The local vectors $\phi_p$ are then unramified at each finite place, and a lowest weight vector at the archimedean place. By the discussion at the end of the previous subsection, the quantity $Z(s,f,\phi)$  is equal to a ratio of global $L$-functions up to a finite number of factors coming from the bad places. We need to be able to explicitly evaluate these bad factors, and in particular, make sure that the local zeta integrals are all non-zero. This is where the correct choice of local data entering the zeta integrals becomes very important. In short, for each place $v$ where $\tau_v$ is not an unramified principal series, we have to make a choice of local section $f_v$ defining the Eisenstein series, and a wrong choice of $f_v$ may lead to integrals that are not computable, or worse, that are zero. There are in fact two important requirements that $f_v$ must satisfy:
\begin{enumerate}
 \item $f_v$ must be such that the local zeta integral $Z(s,W_{f_v},B_{\phi_v})$ is non-zero and explicitly computable.
 \item $f_v$ should be uniquely characterized by \emph{right transformation properties}.
\end{enumerate}
This second requirement is important in view of the calculation of \emph{local intertwining operators}, which are essential for obtaining the functional equation of $L(s,\pi\times\tau)$. The local intertwining operators map each $f_v$ to a vector in a similar parabolically induced representation via an explicit integral. This integral involves left transformations of $f_v$, and hence preserves all right transformation properties. If $f_v$ is indeed characterized by its properties on the right, we know a priori that the result of applying an intertwining operator is the function analogous to $f_v$. By uniqueness, the intertwining operator can then be calculated by evaluating at a single point.

\vspace{3ex}
For a finite prime $p$, it turns out that the local induced representations admit a \emph{local newform theory}. This will be the topic of Sect.~\ref{distvecnonarchsec}. In particular, there is a distinguished vector in $I(s,\chi_p,\chi_{0,p},\tau_p)$, unique up to multiples and characterized by being invariant under a certain congruence subgroup. Suitably normalized, this vector is a good and natural choice for $f_p$.

\vspace{3ex}
The choice of $f_v$ for $v=\infty$ is rather intricate and is the topic of Sect.~\ref{distvecarchsec}. It comes down to finding a suitable function on $\GU(2,2)(\C)$ with certain transformation properties on the left and on the right. Moreover, one has to assure that this function is $K$-finite, where $K$ is the maximal compact subgroup of $\GU(2,2)(\C)$. We will cook up an appropriate function as a certain polynomial in matrix coefficients; see Proposition \ref{W0l1l2distvecprop}.

\vspace{3ex}
Having defined all the local sections in this way, it is then possible to calculate the local zeta integrals at all places. The result is a formula similar to \eqref{unramifiedfactorseq}, namely
\begin{equation}\label{ramifiedfactorseq}
 Z(s,W_{f_v},B_{\phi_v})=\frac{L(3s+\frac 12, \tilde\pi_v \times \tilde\tau_v)}
  {L(6s+1,\chi_v|_{\Q_v^\times})L(3s+1,\tau_v \times \AI(\Lambda_v) \times \chi_v|_{\Q_v^\times})}Y_v(s),
\end{equation}
with an explicitly given correction factor $Y_v(s)$. The details are given in Theorem \ref{nonarchlocalzetatheorem} for the non-archimedean case and Corollary \ref{archlocalzetatheoremcor} for the archimedean case.

\vspace{3ex}
We would like to remark that the kind of careful selection of distinguished local vectors as described above is quite typical when one wants to precisely understand automorphic representations at highly ramified places. We will have to play a similar game again later, when we prove the pullback formula.

\vspace{3ex}

{\bf The global integral representation}

\vspace{2ex}
Let us summarize what we have so far. We started with a cuspidal Siegel eigenform $F$ for $\SSp_4(\Z)$. Its adelization $\phi$ generates an irreducible, cuspidal, automorphic representation $\pi=\pi_F$ of $\GSp_4(\A)$. Using a non-vanishing theorem for the Fourier coefficients of $F$, we can find a particularly nice global Bessel model for $\pi$. The involved quadratic extension $L/\Q$ gives rise to a unitary group $\GU(2,2)$. For $\tau=\otimes\tau_p$ an arbitrary cuspidal, automorphic representation of $\GL_2(\A)$ and some auxiliary characters $\chi$ and $\chi_0$, we consider the representation $I(s,\chi,\chi_0,\tau)$ induced from the Klingen parabolic subgroup of $\GU(2,2)$. It is possible to choose sections $f_v$ in the local representations $I(s,\chi_v,\chi_{0,v},\tau_v)$, for \emph{all} places $v$, so that the identity \eqref{ramifiedfactorseq} holds with an explicit factor $Y_v(s)$. Via Furusawa's ``basic identity'', the product of all the local zeta integrals equals the integral $Z(s,f,\phi)$ in \eqref{globalintegraleq2}. Hence, we obtain the \emph{global integral representation}
\begin{equation}\label{globalintegralrepresentationtheoremeq1a}
  Z(s,f,\phi)=\frac{L(3s+\frac 12, \tilde\pi \times \tilde\tau)}
  {L(6s+1,\chi|_{\A^\times})L(3s+1,\tau \times \AI(\Lambda) \times \chi|_{\A^\times})}Y(s),
\end{equation}
with an explicitly known function $Y(s)$. At this stage we obtain our first result about $L(s,\pi\times\tau)$, namely, that this $L$-function has meromorphic continuation to all of $\C$. This is because the same is true for the Eisenstein series appearing in $Z(s,f,\phi)$, and for the other functions in \eqref{globalintegralrepresentationtheoremeq1a} as well.

\vspace{3ex}
The integral representation \eqref{globalintegralrepresentationtheoremeq1a} may also be used to prove the expected functional equation satisfied by $L(s,\pi\times\tau)$. Since the functional equations for the other $L$-functions in \eqref{globalintegralrepresentationtheoremeq1a} are known, all one needs is the functional equation of the Eisenstein series $E(h,s;f)$. This in turn comes down to a calculation of local intertwining operators, which we carry out in Sects.~\ref{inter-non-arch} (non-archimedean case) and \ref{inter-arch} (archimedean case). As already mentioned, the characterization of our local sections by right transformation properties means that the intertwining operators need to be evaluated only at one specific point. We caution however that this evaluation is \emph{very difficult}, and our description in Sects.~\ref{inter-non-arch} and \ref{inter-arch} is essentially an overview that hides the actual length of the calculations involved. The determination of the functional equation, given the results of the intertwining operator calculations, is carried out in Sect.~\ref{functleqsec}. The result is exactly as it should be:

\vspace{3ex}
{\bf Theorem B:} \emph{The $L$-function $L(s,\pi\times\tau)$ has meromorphic continuation to all of $\C$ and satisfies the functional equation
 \begin{equation}\label{functionalequationtheoremeq2a}
  L(s,\pi\times\tau)=\varepsilon(s,\pi\times\tau)L(1-s,\tilde\pi\times\tilde\tau),
 \end{equation}
where $\varepsilon(s,\pi\times\tau)$ is the global $\varepsilon$-factor attached to the representation $\pi\times\tau$ via the local Langlands correspondence at every place.}

\vspace{3ex}
In Theorem \ref{functionalequationtheorem} this result is actually obtained under a mild hypothesis on the ramification of $\tau$, which however will be removed later in Theorem \ref{Lrhonsigmaranalyticpropertiestheorem}.

\vspace{3ex}
{\bf The pullback formula}

\vspace{2ex}
As mentioned earlier, we need to prove that the $L$-functions $L(s,\pi\times\tau)$ are nice, i.e., they can be analytically continued to entire functions, satisfy a functional equation, and are bounded in vertical strips. So far, using the global integral representation \eqref{globalintegralrepresentationtheoremeq1a}, we have proved meromorphic continuation and functional equation for $L(s,\pi\times\tau)$. It turns out that boundedness in vertical strips follows from a general theorem of Gelbart and Lapid~\cite{GL} once entireness is known. So it all boils down to showing that $L(s,\pi\times\tau)$ has no poles anywhere in the complex plane. Unfortunately, the global integral representation \eqref{globalintegralrepresentationtheoremeq1a} cannot be directly used to control the poles of this $L$-function. The reason is that the analytic properties of the Klingen-type Eisenstein series $E(h,s;f)$ are not understood to the required extent.

\vspace{3ex}
To control the poles, we will prove a \emph{pullback formula} and express our Eisenstein series $E(h,s;f)$ as an integral of a $\GL_2$ automorphic form against a restricted Eisenstein series on a larger group $\U(3,3)$. Let us briefly describe the general philosophy behind pullback formulas. Let $G_1$, $G_2$ and $G_3$ be semisimple groups such that there is an embedding $G_1 \times G_2 \rightarrow G_3$. Suppose that we want to understand a complicated Eisenstein series $E(g_2, s; f)$ on $G_2$ for which the inducing data $f$ essentially comes from an automorphic representation $\sigma$ on $G_1$. Then, one can often find a simpler (degenerate) Eisenstein series $E(g,s; \Upsilon)$ on the larger group $G_3$ such that there is a precise formula of the form
\begin{equation}\label{pullbackintro}
 \int\limits_{G_1(\Q) \bs G_1(\A)} E((g_1, g_2), s; \Upsilon) \Psi(g_1) dg_1= T(s)  E(g_2, s; f)
\end{equation}
 where $\Psi$ is a suitable vector in the space of $\sigma$ and $T(s)$ is an explicitly determined correction factor.

\vspace{3ex}
Pullback formulas have a long history. Garrett \cite{Ga1983} used pullback formulas for Eisenstein series on symplectic groups to study the triple product $L$-function, as well as to establish the algebraicity of certain
ratios of inner products of Siegel modular forms. Pullback formulas for Eisenstein series on unitary groups were first proved in a classical setting by Shimura~\cite{shibook1}. Unfortunately, Shimura only considers certain special types of Eisenstein series in his work, which do not include ours except in the very specific case when the local data is unramified everywhere.

\vspace{3ex}
In Theorem~\ref{theorem-global-pullback} we prove a pullback formula in the form~\eqref{pullbackintro}  when $G_i = \U(i,i)$ (for $i=1,2,3$), $\sigma$ is essentially the representation $\chi_0 \times \tau$ and $E(g_2,s;f)$ is (the restriction from $\GU(2,2)$ to $\U(2,2)$ of) the Eisenstein series involved in~\eqref{globalintegraleq2}. This results in a \emph{second global integral representation} for $L(s,\pi\times\tau)$ involving $E((g_1, g_2), s; \Upsilon)$; see Theorem \ref{theoremsecondintegralrep}. Since $E((g_1, g_2), s; \Upsilon)$ is a degenerate Siegel type Eisenstein series on $\U(3,3)$, its analytic properties are better understood. Indeed, by the work of Tan \cite{Tan}, we deduce that $L(s,\pi\times\tau)$ has at most one possible pole  in ${\rm Re}(s)\geq1/2$, namely at the point $s=1$ (Proposition \ref{atmostonepoleprop}). The proof of holomorphy at this point requires additional arguments.

\vspace{3ex}
We have not yet discussed how one goes about proving a formula like~\eqref{pullbackintro}. There are two main ingredients involved. The first ingredient is combinatorial and involves the computation of a certain double coset space. In our case, this has already been done by Shimura~\cite{shibook1}; see the proof of Theorem~\ref{theorem-global-pullback}. The second ingredient is local and involves a careful choice of vectors in local representations. Indeed, the double coset computation reduces the task of proving the pullback formula to making a delicate choice for the local sections $\Upsilon_v$ at all archimedean and non-archimedean places, and then proving certain identities (``local pullback formulas") involving local zeta integrals. See Sect.\ \ref{pullbacknonarchsectionssec} for the definition of local sections  in the non-archimedean case(s) and Sect.\ \ref{pullbackarchsectionssec} for the definition in the archimedean case. The local zeta integrals are calculated in Sect.\ \ref{pullbacknonarchsec} (non-archimedean case) and Sect.\ \ref{pullbackarchsec} (archimedean case).

\vspace{3ex}
{\bf The Siegel-Weil formula, entireness, functoriality}

\vspace{2ex}
We have so far proved that $L(s,\pi\times\tau)$ has only one possible pole in ${\rm Re}(s)\geq1/2$, namely at the point $s=1$. In order to prove the holomorphy at this point, it suffices to show (because of the second integral representation) that the residue of the $\U(3,3)$ Eisenstein series $E((g_1, g_2), s; \Upsilon)$ at a relevant point $s_0$, when integrated against the adelization $\phi$ of our Siegel cusp form $F$, vanishes.

\vspace{3ex}
To do this, we employ the regularized Siegel-Weil formula for $\U(n,n)$ due to Ichino \cite{ich}, which asserts that this residue of $E((g_1, g_2), s; \Upsilon)$ at $s_0$ is equal to a  regularized theta integral. Consequently, if $L(s,\pi\times\tau)$ has a pole, then the integral of the (adelized) Siegel modular form $F$ against a regularized theta integral is non-zero (Proposition \ref{propintegralnonzero}). An argument using the \emph{seesaw diagram}

$$
 \xymatrix{\U(2,2)\ar@{-}[dr]\ar@{-}[d]&\O(2,2)\ar@{-}[d]\\
    \Sp(4)\ar@{-}[ur]&\U(1,1)}
$$

then shows that $\pi_1$, the cuspidal, automorphic representation of $\SSp_4(\A)$ generated by $F$, participates in the theta correspondence with a split orthogonal group ${\rm O}(2,2)$. But this is impossible by explicit knowledge of the archimedean local theta correspondence \cite{tomasz}. This proves the holomorphy of $L(s,\pi\times\tau)$ at the point $s=1$.

\vspace{3ex}
Thus, $L(s,\pi\times\tau)$ has no poles in the region ${\rm Re}(s)\geq1/2$. By the functional equation, it follows that it has no poles in the region ${\rm Re}(s)\leq1/2$. We thus obtain Theorem \ref{entirenesstheorem}, which states that $L(s,\pi\times\tau)$ is an entire function. As observed earlier, the theorem of Gelbart and Lapid~\cite{GL} now implies boundedness in vertical strips. We have finally achieved our goal of proving the ``niceness" --- analytic continuation to an entire function that satisfies the functional equation and is bounded in vertical strips --- of $L(s,\pi\times\tau)$. By the Converse Theorem and Kim's exterior square lifting, our main lifting theorem (Theorem A) now follows.

\vspace{3ex}
{\bf Applications}

\vspace{2ex}
Having established the liftings, we now turn to applications.  Applying a backwards lifting from $\GL_4$ to $\SO_5(\A)\cong\PGSp_4(\A)$, we prove in Theorem \ref{genericswitchtheorem} the existence of a globally generic representation on $\GSp_4(\A)$ in the same $L$-packet as $\pi$. Also thanks to our liftings, the machinery of Rankin-Selberg $L$-functions on $\GL_n\times\GL_m$ is available for the study of $L$-functions related to Siegel modular forms.

\vspace{3ex}
All this is exploited in Sect.\ \ref{analyticpropertiesapplicationssec}. We obtain the niceness of a host of $L$-functions associated to Siegel cusp forms, including $L$-functions for $\GSp_4\times\GL_n$ for any $n$, and for $\GSp_4\times\GSp_4$; here, on the $\GSp_4$-factors, we can have the $4$-dimensional or the $5$-dimensional representation of the dual group. We also obtain niceness for the degree 10 (adjoint) $L$-function of Siegel modular forms, as well as some analytic properties for the degree $14$ and the degree $16$ $L$-functions. For the precise results, see Theorems~\ref{Lrhonanalyticpropertiestheorem},~\ref{Lrhonsigmaranalyticpropertiestheorem}
and~\ref{Lrhonrhoranalyticpropertiestheorem}.

\vspace{3ex}
To give a flavor of the results obtained, we restate below part of the $\GSp_4 \times \GSp_4$ result in classical language.

 \vspace{3ex}
{\bf Theorem C:} \emph{Let $F$ and $G$ be Siegel cusp forms of full level and weights $k$, $l$ respectively, and suppose that neither of them is a Saito-Kurokawa lift. Assume further that $F$ and $G$ are eigenfunctions for all the Hecke operators $T(n)$, with eigenvalues $\lambda_F(n)$ and $\lambda_G(n)$ respectively. Let $L(s,F,{\rm spin})$ and $L(s,G,{\rm spin})$ denote their spin $L$-functions, normalized so that the functional equation takes $s$ to $1-s$. Concretely, these $L$-functions are defined by Dirichlet series
$$
 L(s,F,{\rm spin}) = \zeta(2s+1) \sum_{n=1}^\infty \frac{\lambda_F(n)}{n^{s+k-3/2}}, \quad L(s,G,{\rm spin}) = \zeta(2s+1) \sum_{n=1}^\infty \frac{\lambda_G(n)}{n^{s+l-3/2}},
$$
that analytically continue to entire functions, and possess Euler products,
$$
 L(s,F,{\rm spin}) = \prod_p \prod_{i=1}^4 \left(1-\beta_{F,p}^{(i)}\,p^{-s}\right)^{-1}, \quad L(s,G,{\rm spin}) = \prod_p \prod_{i=1}^4 \left(1-\beta_{F,p}^{(i)}\,p^{-s}\right)^{-1}.
$$
Define the degree 16 convolution $L$-function $L(s, F\times G)$ by the following Euler product:
$$
 L(s,F\times G) = \prod_p \prod_{i=1}^4 \prod_{j=1}^4 \left(1-\beta_{F,p}^{(i)}\beta_{G,p}^{(j)}\,p^{-s}\right)^{-1}.
$$
Then $L(s, F\times G)$ is absolutely convergent for ${\rm Re}(s)>1$, has meromorphic continuation to the entire complex plane, and is non-vanishing on ${\rm Re}(s) = 1$. Moreover, $L(s, F\times G)$ is entire, except in the special case $k=l$ and $\lambda_F(n) = \lambda_G(n)$ for all $n$, when it has a simple pole at $s=1$.}

\vspace{3ex}
By combining our lifting results with the results of \cite{La2003}, we also prove that $L(1/2,F,{\rm spin}) \ge 0$. We prove similar  non-negativity results for the ``spin$\,\times\,$standard'' $L$-function as well as for suitable $L$-functions on $\GSp_4 \times \GL_2$ and $\GSp_4 \times \GL_3$; see Theorem~\ref{Lnonnegativitytheorem} for the precise statement.

\vspace{3ex}
We also obtain critical value results in the spirit of Deligne's conjecture for $\GSp_4 \times \GL_1$ (Theorem~\ref{criticalgsp4gl1}) and for $\GSp_4 \times \GL_2$ (Theorem~\ref{specialgsp4gl2}). Theorem~\ref{criticalgsp4gl1} follows by combining our lifting theorem with a critical value result for $\GL_4 \times \GL_1$ proved by Grobner and Raghuram~\cite{raghugrob}. Theorem~\ref{specialgsp4gl2}, on the other hand, follows directly from the second global integral representation (Theorem \ref{theoremsecondintegralrep}) using the methods of~\cite{pullback}.

\vspace{3ex}
{\bf Further remarks}

\vspace{2ex}
As for related works, the transfer from $\GSp_4$ to $\GL_4$ for \emph{all} cuspidal, automorphic representations should eventually follow from the trace formula. At the time of this writing, we do not know whether all the necessary elements for this far reaching program of Arthur's have been completed. The existence of a globally generic representation of $\GSp_4(\A)$ in the same $L$-packet as $\pi$ (Theorem \ref{genericswitchtheorem}) is also proved in \cite{We2008} using theta liftings and the topological trace formula. We hope, however, that our present work is of independent interest, both because it provides a ``proof of concept" that certain cases of non-generic transfer can be established without resorting to trace formula arguments, and because the explicit nature of our integral representation makes it a useful tool to attack other problems related to Siegel cusp forms. As an example of the latter, we would like to mention Gross-Prasad type questions for $\GSp_4 \times \GL_2$ as a potential future application. Also, the above mentioned special value result for $\GSp_4 \times \GL_2$, which is an application of our integral representation, does not immediately follow from the transfer obtained via the trace formula.

\vspace{3ex}
{\bf Acknowledgements}

\vspace{2ex}
We would like to thank Paul-Olivier Dehaye, Mark McKee and Paul Nelson for their help with various parts of this paper. We would also like to thank the FIM and Emmanuel Kowalski at ETH for providing an excellent working environment for us during the final phase in which this paper was written.

\vspace{3ex}

\section*{Notation}
\addcontentsline{toc}{section}{Notation}
\subsubsection*{Basic objects}
\begin{enumerate}
\item The symbols $\Z$, $\Z_{\ge0}$, $\Q$, $\R$, $\C$, $\Z_p$ and $\Q_p$ have the usual meanings. The symbol $\A_F$ denotes the ring of adeles of an algebraic number field $F$, and $\A_F^\times$ denotes its group of ideles. The symbols $\A$ and $\A^\times$ will always denote $\A_\Q$ and $\A_\Q^\times$ respectively.
\item For any commutative ring $R$ and positive integer $n$, let $\text{Mat}_{n,n}(R)$
denote the ring of $n\times n$ matrices with entries in $R$, and let $\GL_n(R)$ denote the group of invertible elements in $\text{Mat}_{n,n}(R)$. We use $R^{\times}$ to denote $\GL_1(R)$. If $A\in \text{Mat}_{n,n}(R)$, we let $^t\!A$ denote its transpose.
 \item Define $J_n\in\text{Mat}_{n,n}(\Z)$  by
  $$
   J_n =\begin{bmatrix}0 & I_n\\-I_n & 0\\\end{bmatrix}.
  $$
 \item In this paper, all non-archimedean local fields will be understood to be of characteristic zero. If $F$ is such a field, let $\OF$ be its ring of integers and $\p$ be the maximal ideal of $\OF$. Let $\varpi$ be a generator of $\p$, and let $q$ be the  cardinality of the residue class field $\OF/\p$.
 \item Let $F$ be as above. If $L$ is a quadratic field extension of $F$, or $L=F\oplus F$, let $\big(\frac L{\p}\big)$ be the Legendre symbol. By definition, $\big(\frac L{\p}\big)=-1$ if $L/F$ is an unramified field extension (the \emph{inert case}), $\big(\frac L{\p}\big)=0$ if $L/F$ is a ramified field extension (the \emph{ramified case}), and $\big(\frac L{\p}\big)=1$ if $L=F\oplus F$ (the \emph{split case}). In the field case, let $\bar x$ denote the Galois conjugate of $x\in L$. In the split case, let $\overline{(x,y)}=(y,x)$. In all cases, the norm is defined by $N(x)=x\bar x$. If $L$ is a field, then let $\OF_L$ be its ring of integers. If $L = F \oplus F$, then let $\OF_L = \OF \oplus \OF$. Let $\varpi_L$ be a generator of $\p_L$ if $L$ is a field, and set $\varpi_L = (\varpi,1)$ if $L$ is not a field. We fix the following ideal in $\OF_L$,
 \begin{equation}\label{ideal defn}\renewcommand{\arraystretch}{1.3}
  \P := \p\OF_L = \left\{
                  \begin{array}{l@{\qquad\text{if }}l}
                    \p_L & \big(\frac L{\p}\big) = -1,\\
                    \p_L^2 & \big(\frac L{\p}\big) = 0,\\
                    \p \oplus \p & \big(\frac L{\p}\big) = 1.
                  \end{array}
                \right.
 \end{equation}
 Here, $\p_L$ is the maximal ideal of $\OF_L$ when $L$ is a field extension. Note that $\P$ is prime only if $\big(\frac L\p\big)=-1$. We have $\P^n\cap\OF=\p^n$ for all $n\geq0$.
 \item We fix additive characters once and for all, as follows. If $F$ is a non-archimedean local field, $\psi$ is required to have conductor $\OF$. If $F=\R$, then $\psi(x)=e^{-2\pi ix}$. For any $a\in F$, let $\psi^a(x)=\psi(ax)$.
\end{enumerate}
\subsubsection*{The quadratic extension}
Let $F$ be a non-archimedean local field of characteristic zero, or $F=\R$. The unitary groups we shall be working with are defined with respect to a quadratic extension $L/F$. We shall now explain the conventions for this quadratic extension. We fix three elements $\mathbf{a}, \mathbf{b}, \mathbf{c} \in F$ such that $\mathbf{d}:=\mathbf{b}^2-4\mathbf{a}\mathbf{c}\neq0$. Then let
\begin{equation}\label{Ldefeq}\renewcommand{\arraystretch}{1.2}
 L = \left\{
      \begin{array}{l@{\qquad\mbox{if }}l}
        F(\sqrt{\mathbf{d}})&\mathbf{d} \notin F^{\times2},\\
        F \oplus F&\mathbf{d} \in F^{\times2}.
      \end{array}
    \right.
\end{equation}
We shall make the following {\bf assumptions}.
\begin{itemize}
 \item If $F$ is non-archimedean, assume that $\mathbf{a},\mathbf{b}\in\OF$ and $\mathbf{c}\in\OF^\times$. Assume moreover that if $\mathbf{d} \notin F^{\times2}$, then $\mathbf{d}$ is the generator of the discriminant of $L/F$, and if $\mathbf{d} \in F^{\times2}$, then $\mathbf{d} \in \OF^{\times}$.
 \item If $F=\R$, assume that $S=\mat{\mathbf{a}}{\mathbf{b}/2}{\mathbf{b}/2}{\mathbf{c}}\in {\rm Mat}_{2,2}(\R)$ is a positive definite matrix. Equivalently, $\mathbf{c}>0$ and $\mathbf{d}<0$.
\end{itemize}
Hence, if $F=\R$, we always assume that $L=\C$. In all cases let
\begin{equation}\label{alphadefeq}
 \alpha=\left\{\begin{array}{l@{\qquad\text{if }L}l}
 \displaystyle\frac{\mathbf{b}+\sqrt{\mathbf{d}}}{2\mathbf{c}}&\text{ is a field},\\[2ex]
 \displaystyle\Big(\frac{\mathbf{b}+\sqrt{\mathbf{d}}}{2\mathbf{c}},\frac{\mathbf{b}-\sqrt{\mathbf{d}}}{2\mathbf{c}}\Big)&=F\oplus F.
 \end{array}\right.
\end{equation}
An important role will be played by the matrix
\begin{equation}\label{etadefeq}
 \eta = \begin{bmatrix}1&0&&\\\alpha&1&&\\&&1&-\bar{\alpha}\\&&0&1\end{bmatrix}.
\end{equation}
We further define
\begin{equation}\label{eta0defeq2}
 \eta_0=\left\{\begin{array}{l@{\qquad\text{if }}l}
   \eta&F\text{ is $p$-adic},\\
   \frac1{\sqrt{2}}\begin{bmatrix}1&i\\i&1\\&&1&i\\&&i&1\end{bmatrix}&F=\R.
 \end{array}\right.
\end{equation}
\subsubsection*{Algebraic groups}
For simplicity we will make all definitions over the local field $F$, but it is clear
how to define the corresponding global objects.
\begin{enumerate}
 \item Let $H=\GSp_4$ and $G_j=\GU(j,j;L)$ be the algebraic $F$-groups whose $F$-points are given by
 \begin{align}
  H(F)&=\{g \in \GL_4(F)\;|\; ^tgJ_2g = \mu_2(g)J_2,\:\mu_2(g)\in F^{\times} \},  \label{HFdefinition}\\
  G_j(F)&=\{g \in \GL_{2j}(L)\;|\; ^t\bar{g}J_jg = \mu_j(g)J_j,\:\mu_j(g)\in F^{\times}\}. \label{Gjdefinition}
 \end{align}
 \item We define, for $\zeta\in L^\times$ and $\mat{a}{b}{c}{d}\in G_1(F)$,
  \begin{equation}\label{m1m2defeq}
   m_1(\zeta)=\begin{bmatrix}\zeta\\&1\\&&\bar\zeta^{-1}\\&&&1\end{bmatrix},\qquad
   m_2(\mat{a}{b}{c}{d})=\begin{bmatrix}1\\&a&&b\\&&\bar ad-b\bar c\\&c&&d\end{bmatrix}.
  \end{equation}
 \item Let $P$ be the standard maximal parabolic subgroup of $G_2(F)$ with a non-abelian unipotent radical. Let $P = MN$ be the Levi decomposition of $P$. We have $M = M^{(1)}M^{(2)}$, where
 \begin{align}
  M^{(1)}(F)&=\{m_1(\zeta)\;|\;\zeta\in L^\times\} \label{M1defn},\\
  M^{(2)}(F)&=\{m_2(g)\;|\;g\in G_1(F)\} \label{M2defn},\\
  N(F) &=\{\begin{bmatrix}
                 1 & z &  &  \\
                  & 1 &  &  \\
                  &  & 1 &  \\
                  &  & -\overline{z} & 1 \\
               \end{bmatrix}
               \begin{bmatrix}
                 1 &  & x & y \\
                  & 1 & \overline{y} &  \\
                  &  & 1 &  \\
                  &  &  & 1 \\
               \end{bmatrix}
              \;|\;x\in F,\;y,z \in L \}\label{Ndefn}.
 \end{align}
 The modular factor of the parabolic $P$ is given by
 \begin{equation}\label{deltaPformulaeq}
  \delta_P(m_1(\zeta)m_2(g))
  =|N(\zeta)\mu_1^{-1}(g)|^3,
 \end{equation}
 where $|\cdot|$ is the normalized absolute value on $F$.

\item Let $P_{12}$ be the maximal parabolic subgroup of $G_3$, defined by
\begin{equation}\label{P12defeq}
 P_{12}=G_3\cap\begin{bmatrix} *&*&*&*&*&*\\ *&*&*&*&*&*\\ *&*&*&*&*&*\\
 &&&*&*&*\\&&&*&*&*\\&&&*&*&*\end{bmatrix}.
\end{equation}

Let $P_{12} = M_{12}N_{12}$ be the Levi decomposition, with
$$
 M_{12}(F):=\left\{ m(A,v) = \begin{bmatrix}A & 0\\0&
  v\; ^t\!\bar{A}^{-1}\end{bmatrix} \big|\;A \in \GL_3(L),\;v \in F^\times \right\},
$$
$$
 N_{12}(F):= \left\{\begin{bmatrix}1 & b\\0&1 \end{bmatrix}\big|\;b\in {\rm Mat}_{3,3}(L),\;^t\bar b =b \right\}.
$$

The modular function of $P_{12}$ is given by
\begin{equation}\label{P12modulareq}
 \delta_{12}(\mat{A}{}{}{v\,^t\!\bar A^{-1}})=|v^{-3}N(\det(A))|^3,
 \qquad v\in F^\times,\:A\in\GL_3(L).
\end{equation}

 \item Let $\iota$ be the embedding of $\{(g_1,g_2)\in G_1(F)\times G_2(F):\:\mu_1(g_1)=\mu_2(g_2)\}$ into $G_3(F)$ defined by
  \begin{equation}\label{GUembeddingeq}
   \iota \big(\mat{a}{b}{c}{d},\mat{A}{B}{C}{D}\big) = \begin{bmatrix}A&&B\\&a&&-b\\C&&D\\&-c&&d\end{bmatrix}.
  \end{equation}
\end{enumerate}
\subsubsection*{Congruence subgroups}
Assuming that $F$ is $p$-adic, we will use the following notation for congruence subgroups,
\begin{align}
 K^{(0)}(\P^n)&=G_1(\OF)\cap\mat{\OF_L}{\OF_L}{\P^n}{\OF_L},\label{congruencesubgroupeq1}\\
 K^{(1)}(\P^n)&=G_1(\OF)\cap\mat{1+\P^n}{\OF_L}{\P^n}{\OF_L},\label{congruencesubgroupeq2}\\
  K^{(1)}_1(\P^n)&=U(1,1;L)(\OF)\cap K^{(1)}(\P^n)\nonumber\\
   &=U(1,1;L)(\OF)\cap\mat{1+\P^n}{\OF_L}{\P^n}{\OF_L}
   =U(1,1;L)(\OF)\cap\mat{1+\P^n}{\OF_L}{\P^n}{1+\P^n},\label{congruencesubgroupeq3}\\
 K^{(1)}(\p^n)&=\GL_2(\OF)\cap\mat{1+\p^n}{\OF}{\p^n}{\OF}.\label{congruencesubgroupeq4}
\end{align}
If $\tau$ is an irreducible, admissible representation of $\GL_n(F)$, we let $a(\tau)$ be the non-negative integer such that $\p^{a(\tau)}$ is the conductor of $\tau$; see Theorem \ref{GL2newformtheorem} for a characterization in the $\GL_2$ case. If $\chi$ is a character of $F^\times$, then $a(\chi)$ is the smallest non-negative integer such that $\chi$ is trivial on $\OF^\times\cap(1+\p^{a(\chi)})$.
\subsubsection*{Representations of $\GL_2(\R)$}
If $p$ is a positive integer and $\mu\in\C$, we let $\mathcal{D}_{p,\mu}$ be the irreducible representation of $\GL_2(\R)$ with minimal weight $p+1$ and central character satisfying $a\mapsto a^{2\mu}$ for $a>0$. Every other irreducible, admissible representation of $\GL_2(\R)$ is of the form $\beta_1\times\beta_2$ with characters $\beta_1,\beta_2$ of $\R^\times$; see (\ref{gl2inducedeq}). Note that, if $\mu\in i\R$, then $\mathcal{D}_{p,\mu}$ is a discrete series representation.

\section{Distinguished vectors in local representations}
In this section we will develop some local theory, both archimedean and non-archimedean, which will be utilized in subsequent sections on global integral representations.  Recall the definitions of the groups $G_j=\GU(j,j;L)$ from~\eqref{Gjdefinition}. The local theory will exhibit distinguished vectors in certain parabolically induced representations of $G_2(F)$, where $F=\R$ or $F$ is $p$-adic. We will also study the behavior of these vectors under local intertwining operators. Since the distinguished vectors are characterized by right transformation properties, the intertwining operators map distinguished vectors to distinguished vectors. This fact will later be applied to obtain the functional equation of global $L$-functions.

\vspace{3ex}

Unless otherwise noted, $F$ is a non-archimedean local field of characteristic zero, or $F=\R$. We let $L,\alpha,\eta$ be as in (\ref{Ldefeq}), (\ref{alphadefeq}), (\ref{etadefeq}), respectively.

\subsection{Parabolic induction to $\GU(2,2)$}\label{parabolicinductionsec}
Let $(\tau,V_\tau)$ be an irreducible, admissible, infinite-dimensional representation
of $\GL_2(F)$, and let $\chi_0$ be a character of $L^\times$ such
that $\chi_0\big|_{F^\times}$ coincides with $\omega_{\tau}$, the
central character of $\tau$. Then the pair $(\chi_0,\tau)$ defines a representation
of $G_1(F)\cong M^{(2)}(F)$ on the same space $V_\tau$ via
\begin{equation}\label{M2representationseq}
 \tau(\lambda g)=\chi_0(\lambda)\tau(g),\qquad\lambda\in L^\times,\;g\in\GL_2(F).
\end{equation}
We denote this representation by $\chi_0\times\tau$. Every irreducible, admissible representation of $G_1(F)$ is of the form (\ref{M2representationseq}). If $V_\tau$ is a space of functions on $\GL_2(F)$ on which
$\GL_2(F)$ acts by right translation, then $\chi_0\times\tau$ can
be realized as a space of functions on $M^{(2)}(F)$ on which
$M^{(2)}(F)$ acts by right translation. This is accomplished by
extending every $W\in V_\tau$ to a function on $M^{(2)}(F)$ via
\begin{equation}\label{extendedWformulaeq}
 W(\lambda g)=\chi_0(\lambda)W(g),\qquad\lambda\in L^\times,\:g\in\GL_2(F).
\end{equation}
If $s$ is a complex parameter, $\chi$ is any character of $L^\times$ and $\chi_0\times\tau$ is a representation of $M^{(2)}(F)$ as above, we denote by $I(s,\chi,\chi_0,\tau)$ the induced representation of $G_2(F)$ consisting of functions $f:\:G_2(F)\rightarrow V_\tau$ with the transformation property
\begin{equation}\label{Isfctnspropeq}
 f(m_1(\zeta)m_2(b)ng)=\big|N(\zeta)\mu_1^{-1}(b)\big|^{3(s+\frac12)}
 \chi(\zeta)(\chi_0\times\tau)(b)f(g)
\end{equation}
for $\zeta\in L^\times$ and $b\in G_1(F)$.

Now taking $V_\tau$ to be the Whittaker model of $\tau$ with respect to the character $\psi$, if we associate to each $f$ as above the function on $G_2(F)$ given by $W_f(g)=f(g)(1)$, then we obtain another model $I_W(s,\chi,\chi_0,\tau)$ of $I(s,\chi,\chi_0,\tau)$ consisting of functions $W:\:G_2(F)\rightarrow\C$. These functions satisfy
\begin{equation}\label{Wsharpproperty1eq}
 W(m_1(\zeta)m_2(\mat{\lambda}{}{}{\lambda})g)
 =|N(\zeta\lambda^{-1})|^{3(s+\frac12)}
 \chi(\zeta)\chi_0(\lambda)W(g),\qquad \zeta,\lambda\in L^\times,
\end{equation}
and
\begin{equation}\label{Wsharpproperty2eq}
 W(\begin{bmatrix}1 & z\\& 1\\&  & 1\\&  & -\overline{z} & 1 \\\end{bmatrix}
  \begin{bmatrix}1&&x&y\\& 1 & \overline{y} &w\\&  & 1\\&  &  & 1 \\\end{bmatrix}g)
 =\psi(w)W(g),\qquad w,x\in F,\;y,z\in L.
\end{equation}
Assume on the other hand that $\tau$ is a parabolically induced representation $\beta_1\times\beta_2$, not necessarily irreducible, with characters $\beta_1,\beta_2:\:F^\times\rightarrow\C^\times$. The standard model of $\beta_1\times\beta_2$ consists of functions $\varphi:\:\GL_2(F)\rightarrow\C$ with the transformation property
\begin{equation}\label{gl2inducedeq}
 \varphi(\mat{a}{b}{}{d}g)=|ad^{-1}|^{1/2}\beta_1(a)\beta_2(d)\varphi(g)\qquad
 \text{for all }a,d\in F^\times,\:b\in F,\:g\in\GL_2(F).
\end{equation}
If we associate to $f$ as in (\ref{Isfctnspropeq}), now taking values in the standard model of $\beta_1\times\beta_2$, the function $\Phi_f$ on $G_2(F)$ given by $\Phi_f(g)=f(g)(1)$, then we obtain another model of $I(s,\chi,\chi_0,\tau)$, which we denote by $I_\Phi(s,\chi,\chi_0,\tau)$. It consists of functions $\Phi:\:G_2(F)\rightarrow\C$ with the transformation property
\begin{align}\label{Wsharpproperty3eq}
 &\Phi(\begin{bmatrix}\zeta&*&*&*\\&\lambda&*&*\\&&\bar\zeta^{-1}N(\lambda)\\
   &&*&\lambda\end{bmatrix}
 \begin{bmatrix}1\\&a\\&&ad\\&&&d\end{bmatrix}g)\nonumber\\
 &\hspace{20ex}=|N(\zeta\lambda^{-1})|^{3(s+\frac12)}|a|^{-3s-1}|d|^{-3s-2}
 \chi(\zeta)\chi_0(\lambda)\beta_1(a)\beta_2(d)\Phi(g)
\end{align}
for all $\zeta,\lambda\in L^\times$, $a,d\in F^\times$.

\subsubsection*{Intertwining operators}
Assume that $\tau$ is generic, and let $\chi,\chi_0$ be as above. For $f\in I(s,\chi,\chi_0,\tau)$ with ${\rm Re}(s)$ large enough, the local intertwining
operator is defined by
\begin{equation}\label{locintdefeq1}
 (M(s)f)(g)=\int\limits_{N(F)}f(w_1ng)\,dn,\qquad
  w_1=\begin{bmatrix}&&1&\\&1&&\\-1&&&\\&&&1\end{bmatrix}.
\end{equation}
Calculations show that $M(s)$ defines an intertwining map
\begin{equation}\label{Mstargeteq}
 M(s):\:I(s,\chi,\chi_0,\tau)\longrightarrow
  I(-s,\bar\chi^{-1},\chi\bar\chi\chi_0,\chi\tau),
\end{equation}
where by $\chi\tau$ we mean the twist $(\chi\big|_{F^\times})\otimes\tau$.  It is easily checked that the above formula (\ref{locintdefeq1}) also defines intertwining operators $M(s)$ from $I_\Phi(s,\chi,\chi_0,\tau)$ to $I_\Phi(-s,\bar\chi^{-1},\chi\bar\chi\chi_0,\chi\tau)$ and from $I_W(s,\chi,\chi_0,\tau)$ to $I_W(-s,\bar\chi^{-1},\chi\bar\chi\chi_0,\chi\tau)$. In Corollary \ref{distinguishedvectornonarchtheorem} (non-archimedean case) and Corollary \ref{distinguishedvectorarchtheoremcor} (archimedean case) we will identify a distinguished element
$$
 W^\#=W^\#(\,\cdot\,,s,\chi,\chi_0,\tau)
$$
in $I_W(s,\chi,\chi_0,\tau)$. This distinguished function will have the property
\begin{equation}\label{Ksdefeq}
 M(s)W^\#(\,\cdot\,,s,\chi,\chi_0,\tau)
  =K(s)W^\#(\,\cdot\,,-s,\bar\chi^{-1},\chi\bar\chi\chi_0,\chi\tau).
\end{equation}
with a ``constant'' $K(s)$ (independent of $g\in G_2(F)$, but dependent on $s$, as well as $\chi$, $\chi_0$ and $\tau$). In most cases $K(s)$ exists because $W^\#$ is characterized, up to scalars, by \emph{right} transformation properties. An exception is the archimedean ``different parity'' Case C, defined in Sect.\ \ref{distvecarchsec}, in which case said right transformation properties characterize a \emph{two}-dimensional space. In this case the existence of the function $K(s)$ such that (\ref{Ksdefeq}) holds will follow from explicit calculations. Note that if $\eta_0\in G_2(F)$ is such that $W^\#(\eta_0)=1$, then we obtain the formula
\begin{equation}\label{Ksformulaeq}
 K(s)=\int\limits_{N(F)}W^\#(w_1n\eta_0,s,\chi,\chi_0,\tau)\,dn
\end{equation}
by evaluating at $\eta_0$. Explicitly,
\begin{equation}\label{Ksformulaexpliciteq}
 K(s)=\int\limits_L \int\limits_L \int\limits_F
  W^\#(w_1\begin{bmatrix}1&z&&\\&1&&\\&&1&\\&&-\bar{z}&1\end{bmatrix}
  \begin{bmatrix}1&&x&y\\&1&\bar{y}&\\&&1&\\&&&1\end{bmatrix}\eta_0)\,dx\,dy\,dz.
\end{equation}
Our goal in Sects.\ \ref{inter-non-arch} and \ref{inter-arch} will be to calculate the function $K(s)$. We will then also be more precise about the measures on $F$ and $L$ used in (\ref{Ksformulaexpliciteq}).

\subsection{Distinguished vectors: non-archimedean case}\label{distvecnonarchsec}
In this section let $F$ be a non-archimedean local field of characteristic zero. Let $\tau$ be any irreducible, admissible representation of $\GL_2(F)$, and let $\chi_0$ be a character of $L^\times$ such that $\chi_0 |_{F^\times} = \omega_\tau$, the central character of $\tau$. Let $\Lambda$ be an unramified character of $L^\times$, and let $\chi$ be the character of $L^\times$ defined by
\begin{equation}\label{chi-lambda-char-condition}
 \chi(\zeta) = \Lambda(\bar{\zeta})^{-1} \chi_0(\bar{\zeta})^{-1}.
\end{equation}
For a complex parameter $s$, let $I(s, \chi, \chi_0, \tau)$ be as in Sect.\ \ref{parabolicinductionsec}.
Let $K^{G_2}=G_2(F)\cap\GL_4(\OF_L)$, a maximal compact
subgroup. We define the principal congruence subgroups $ \Gamma(\P^r):= \{ g \in G_2(F)\;|\;g \equiv 1 \pmod{\P^r} \} $ with $\P$ as in (\ref{ideal defn}). For $r=0$ we understand that $\Gamma(\P^r)=K^{G_2}$. For any $m \geq 0$, we let
\begin{equation}\label{etamdef2eq}
 \eta_m = \begin{bmatrix}1&0&&\\\alpha \varpi^m&1&&\\
  &&1&-\bar{\alpha}\varpi^m\\&&0&1\end{bmatrix}.
\end{equation}
For systematic reasons, we let $\eta_\infty$ be the identity matrix. Note that $\eta_0=\eta$; see (\ref{etadefeq}).
\begin{proposition}\label{disj-doub-coset-decomp-general-n-prop}
 For any $r\geq0$ the following disjoint double coset decompositions hold,
 \begin{equation}\label{double-coset-decomp-eqn}
    G_2(F) = \bigsqcup\limits_{0 \leq m \leq \infty}P(F)\eta_m K^H = \bigsqcup\limits_{0 \leq m \leq r}P(F)\eta_m K^H \Gamma(\P^r).
 \end{equation}
 Moreover, for any $0\leq m<r$, we have
 \begin{equation}\label{double-coset-decomp-eqn2}
  P(F)\eta_m K^H\Gamma(\P^r)=P(F)\eta_m K^H.
 \end{equation}
\end{proposition}
\begin{proof} Using the Iwasawa decomposition, (\ref{double-coset-decomp-eqn}) follows from
\begin{equation}\label{K-double-coset-decomp}
 K^{G_2}=\bigsqcup_{0\leq m\leq\infty} P(\OF)\eta_mK^H =\bigsqcup\limits_{0 \leq m \leq r}P(\OF)\eta_m K^H \Gamma(\P^r).
\end{equation}
One can show that the double cosets on the right hand side of (\ref{K-double-coset-decomp}) are disjoint by observing that the function
$$
 K^{G_2} \ni g \mapsto {\rm min}\big(v((gJ\,^t\!g)_{3,2}),v((gJ\,^t\!g)_{3,4})\big)
$$
takes different values on the double cosets. We take the above function modulo $\P^r$ for the disjointness of the double cosets involving $\Gamma(\P^r)$. Knowing disjointness, one obtains the second equality in (\ref{K-double-coset-decomp}) by multiplying the double cosets by $\Gamma(\P^r)$ on the right. We only sketch a proof of the first equality in (\ref{K-double-coset-decomp}). The first step consists in showing that $K^{G_2} = P(\OF)K^H\Gamma(\P)\,\sqcup\,P(\OF)\eta K^H\Gamma(\P)$, which can be done explicitly by considering the three cases -- inert, ramified and split -- separately. Then, for $g\in P(\OF) \gamma_0 K^H$ with $\gamma_0\in\Gamma(\P)$ or $\gamma_0\in\eta \Gamma(\P)$, observe that
$$
 \gamma_0\in G_2(F)\cap\begin{bmatrix}\OF_L^\times&\P&\OF_L&\OF_L\\
  \OF_L&\OF_L^\times&\OF_L&\OF_L\\\P&\P&\OF_L^\times&\OF_L\\\P&\P&\P&\OF_L^\times
  \end{bmatrix}.
$$
Using appropriate matrix identities one can show, for any $\gamma_0$ in this set, that there exist $p \in P(\OF)$, $h \in K^H$ and a unique $m \in \{0, 1, 2, \cdots, \infty \}$ such that $\gamma_0 = p \eta_m h$; this completes the proof of (\ref{double-coset-decomp-eqn}).

\vspace{3ex}
To prove (\ref{double-coset-decomp-eqn2}), we rewrite (\ref{double-coset-decomp-eqn}) as
$$
 G_2(F) = \bigsqcup\limits_{0 \leq m<r}P(F)\eta_m K^H\sqcup X,\qquad
 X=\bigsqcup\limits_{r \leq m\leq\infty}P(F)\eta_m K^H,
$$
and also
$$
 G_2(F)=\bigsqcup_{0\leq m<r}P(F)\eta_mK^H\Gamma(\P^r)\;\sqcup Y,
 \qquad Y=P(F)\;\eta_rK^H\Gamma(\P^r).
$$
For $m\geq r$, we have $\eta_m\in P(F)K^H\Gamma(\P^r)=P(F)\eta_rK^H\Gamma(\P^r)$.
Hence $X\subset Y$.
Evidently, for $m<r$, we have $P(F)\eta_m K^H\subset P(F)\eta_mK^H\Gamma(\P^r)$.
It follows that $P(F)\eta_m K^H= P(F)\eta_mK^H\Gamma(\P^r)$.
\end{proof}

We recall the standard newform theory for $\GL_2$. Let the congruence subgroup $K^{(1)}(\p^n)$ be as in (\ref{congruencesubgroupeq4}). The following result is well known (see \cite{Cas}, \cite{De}).
\begin{theorem}\label{GL2newformtheorem}
 Let $(\tau,V_\tau)$ be a generic, irreducible, admissible representation of $\GL_2(F)$.
 Then the spaces
 $$
  V_\tau(n)=\{v\in V_\tau\;|\;\tau(g)v=v\text{ for all }g\in K^{(1)}(\p^n)\}
 $$
 are non-zero for $n$ large enough. If $n$ is minimal with $V_\tau(n)\neq0$, then
 $\dim(V_\tau(n))=1$. For $r \geq n$, we have $\dim(V_\tau(r)) = r-n+1$.
\end{theorem}
If $n$ is minimal such that $V_\tau(n)\neq0$, then $\p^n$ is called the
\emph{conductor} of $\tau$, and we write $n=a(\tau)$. Any non-zero vector in $V_\tau(a(\tau))$ is called a \emph{local newform}. The following theorem is a local newforms result for the induced representations $I(s,\chi,\chi_0,\tau)$ with respect to the congruence subgroups $K^H\Gamma(\P^r)$.

\begin{theorem}\label{unique-W-theorem}
 Let $(\tau,V_\tau)$ be a generic, irreducible, admissible representation of $\GL_2(F)$ with central character $\omega_{\tau}$ and conductor $\p^n$. Let $\chi_0$ be a character of $L^\times$ such that $\chi_0\big|_{F^\times}=\omega_\tau$ and $\chi_0((1+\P^n)\cap\OF_L^\times)=1$, and let $\chi$ be the character of $L^\times$ defined by (\ref{chi-lambda-char-condition}), where $\Lambda$ is unramified. Let
 $$
  V(r) := \{\phi\in I(s,\chi,\chi_0,\tau)\;|\;\phi(g \gamma, s) = \phi(g, s)
   \mbox{ for all } g \in G(F),\:\gamma \in K^H\Gamma(\P^r) \}
 $$
 for a non-negative integer $r$. Then
 $$
  \dim(V(r)) = \left\{\begin{array}{ll}
   \displaystyle\frac{(r-n+1)(r-n+2)}2& \hbox{ if } r \geq n,\\
   0&\hbox{ if } r < n.\end{array}\right.
 $$
\end{theorem}
\begin{proof} Let $\phi\in V(r)$. By Proposition \ref{disj-doub-coset-decomp-general-n-prop},
$\phi$ is completely determined by its values on $\eta_m$, $0 \leq m \leq r$. For such $m$, and any $g=\mat{a}{b}{c}{d} \in K^{(1)}(\p^{r-m})$, we have
$A:=m_2(g)\in M(F) N(F) \cap \eta_m K^H\Gamma(\P^r) \eta_m^{-1}$. It follows that $\phi(\eta_m) = \phi(A \eta_m)= \tau(g)\phi(\eta_m)$. Hence, for $0 \leq m \leq r$, the vector $v_m:=\phi(\eta_m)$ lies in $V_\tau(r-m)$. Since the conductor of $\tau$ is $\p^n$, we conclude that $v_m = 0$ if $r-m < n$. Therefore $\dim(V(r)) = 0$ for all $r < n$.

\vspace{2ex}
Now suppose that $r \geq n$. We will show that, for any $m$ such that $r-m \geq n$,
if $v_m$ is chosen to be any vector in $V_\tau(r-m)$, then we obtain a well-defined
function $\phi$ in $V(r)$. For $m=r$ this is easy to check, since in this case $n=0$
and all the data is unramified. Assume therefore that $r>m$.
We have to show that for $m_1n_1\eta_m k_1\gamma_1=m_2n_2\eta_m k_2\gamma_2$,
with $m_i\in M(F)$, $n_i\in N(F)$, $k_i\in K^H$ and $\gamma_i\in\Gamma(\P^{r})$,
\begin{equation}\label{Wsharpwelldeflemmanaeq1}
 |N(\zeta_1)\cdot\mu_1^{-1}|^{3(s+1/2)}\chi(\zeta_1)\,
   (\chi_0 \times \tau) (\mat{a_1}{b_1}{c_1}{d_1}) v_m
 =|N(\zeta_2)\cdot\mu_2^{-1}|^{3(s+1/2)}\chi(\zeta_2)\,
   (\chi_0 \times \tau)(\mat{a_2}{b_2}{c_2}{d_2}) v_m.
\end{equation}
We have $\eta_m^{-1}m_2^{-1}m_1 n^\ast \eta_m \in K^H \Gamma(\P^{r})$, where
$n^\ast \in N(F)$ depends on $m_1, m_2, n_1, n_2$. Write
$$
 m_2^{-1}m_1=\begin{bmatrix}\zeta\\&\tilde a&&\tilde b\\
  &&\mu\bar{\zeta}^{-1}\\&\tilde c&&\tilde d\end{bmatrix}.
$$
By definition, $\zeta_1=\zeta_2\zeta$ and $\mu_1=\mu_2\mu$. We have $\zeta\in\OF_L^\times$ and $\mu\in\OF^\times$.
 Hence
(\ref{Wsharpwelldeflemmanaeq1}) is equivalent to
\begin{equation}\label{Wsharpwelldeflemmanaeq2}
 \chi(\zeta)\,(\chi_0 \times \tau)(\mat{a_1}{b_1}{c_1}{d_1}) v_m
 = (\chi_0 \times \tau)(\mat{a_2}{b_2}{c_2}{d_2}) v_m.
\end{equation}
One can check that $\tilde a\bar\zeta^{-1}\in 1+\P^{r-m}$ and
$\tilde c\in \P^{r-m}$. Hence, using the definition of $\chi, \chi_0$
(with unramified $\Lambda$) and the fact that $v_m \in V_\tau(r-m)$,
\begin{align*}
 \chi(\zeta)\,(\chi_0 \times \tau)(\mat{a_1}{b_1}{c_1}{d_1})v_m
  &=\chi(\zeta)\,(\chi_0 \times \tau) (\mat{a_2}{b_2}{c_2}{d_2}
   \mat{\tilde a}{\tilde b}{\tilde c}{\tilde d}) v_m\\
  &=\chi(\zeta)\chi_0(\tilde a)\,(\chi_0 \times \tau)(\mat{a_2}{b_2}{c_2}{d_2}
   \mat{1}{\tilde b/\tilde a}{\tilde c/\tilde a}{\tilde d/\tilde a})v_m\\
  &=\chi_0(\bar\zeta^{-1})\chi_0(\tilde a)\,(\chi_0 \times \tau)(\mat{a_2}{b_2}{c_2}{d_2})v_m\\
  &=(\chi_0 \times \tau)(\mat{a_2}{b_2}{c_2}{d_2})v_m,
\end{align*}
as claimed. Now, using the formula for $\dim(V_\tau(r-m))$ from Theorem \ref{GL2newformtheorem}
completes the proof of the theorem.
\end{proof}

Assume that $W^{(0)}$ is the newform in the Whittaker model of $\tau$ with respect to an additive character of conductor $\OF$. Then it is known that $W^{(0)}(1)\neq0$, so that this function can be normalized by $W^{(0)}(1)=1$. The following is an immediate consequence of the above theorem (and its proof).

\begin{corollary}\label{distinguishedvectornonarchtheorem}
There exists a unique element $W^\#(\,\cdot\,,s)$ in $I_W(s,\chi,\chi_0,\tau)$ with the properties \begin{equation}\label{distinguishedvectornonarchtheoremeq}
   W^\#(gk,s)=W^\#(g,s)\qquad\text{for }g\in G_2(F),\;k\in K^H\Gamma(\P^n),
 \end{equation}
 and
 \begin{equation}\label{distinguishedvectornonarchtheoremeq2}
   W^\#(\eta_0,s)=1,
 \end{equation}
 where $\eta_0=\eta$ as in (\ref{eta0defeq2}). The function $W^\#(\,\cdot\,,s)$ is supported on $P(F)\eta_0 K^H\Gamma(\P^n)$. On this double coset,
 \begin{equation}\label{Wsharpformulaeq}
     W^\#(m_1(\zeta)m_2(g)\eta_0,s)
    =|N(\zeta)\cdot\mu_1^{-1}(g)|^{3(s+1/2)}\chi(\zeta)\,
    W^{(0)}(g),
 \end{equation}
 where $\zeta\in L^\times$, $g\in G_1(F)$, and $W^{(0)}$ is the newform in $V_\tau$, normalized by $W^{(0)}(1)=1$, and extended to a function on $G_1(F)$ via the character $\chi_0$ (see (\ref{extendedWformulaeq})).
\end{corollary}
\subsection{Distinguished vectors: archimedean case}\label{distvecarchsec}
Let $F=\R$ and $L=\C$, and $G_2=\GU(2,2;\C)$ as in the notations. Consider the symmetric domains $\HH_2 := \{Z \in {\rm Mat}_{2,2}(\C)\;|\; i(\,^{t}\!\bar{Z}-Z) \mbox{ is positive definite}\}$ and $\SH_2 :=\{Z \in \HH_2\;|\;^{t}Z = Z \}$. The group $G_2^{+}(\R) := \{g \in G_2(\R)\;|\; \mu_2(g) > 0 \}$ acts on $\HH_2$ via $(g,Z) \mapsto g\langle Z \rangle$, where
$$
 g\langle Z \rangle = (AZ+B)(CZ+D)^{-1}, \mbox{ for } g =
 \mat{A}{B}{C}{D} \in G_2^{+}(\R), Z \in \HH_2.
$$
Under this action, $\SH_2$ is stable by $H^+(\R) = \GSp_4^+(\R)$.
The group $K^{G_2}_\infty=\{g\in G_2^{+}(\R): \mu_2(g) = 1,\, g\langle
i_2\rangle = i_2 \}$ is a maximal compact subgroup of $G_2^{+}(\R)$.
Here, $i_2=\mat{i}{}{}{i} \in\HH_2$. By the Iwasawa decomposition
\begin{equation}\label{realiwasawaeq}
 G_2(\R)=M^{(1)}(\R)M^{(2)}(\R)N(\R)K^{G_2}_\infty,
\end{equation}
where $M^{(1)}(\R)$, $M^{(2)}(\R)$ and $N(\R)$ are as defined in
(\ref{M1defn}), (\ref{M2defn}) and (\ref{Ndefn}).  For $g\in G_2^+(\R)$ and $Z\in\mathbb{H}_2$, let $J(g,Z)=CZ+D$ be
the automorphy factor. Then, for any integer $l$, the map
\begin{equation}\label{realcompactcharactereq}
 k\longmapsto \det(J(k,i_2))^l
\end{equation}
defines a character $K^{G_2}_\infty\rightarrow\C^\times$. Let $K^H_\infty =K_{\infty} \cap H^+(\R)$. Then $K^H_\infty$ is a maximal compact subgroup of $H^+(\R)$.

\vspace{3ex}
Let $(\tau,V_\tau)$ be a  generic, irreducible, admissible representation of $\GL_2(\R)$ with central character $\omega_{\tau}$. Let $l_2$ be an integer of the same parity as the weights of $\tau$ (the precise value of $l_2$ is largely irrelevant, and we will later make a specific choice). Let $\chi_0$ be the character of $\C^\times$ such that $\chi_0\big|_{\R^\times}=\omega_\tau$ and $\chi_0(\zeta)=\zeta^{l_2}$ for $\zeta\in\C^\times,\:|\zeta|=1$. Let $\chi$ be the character of $\C^\times$ given by
\begin{equation}\label{chi-lambda-char-condition-arch}
 \chi(\zeta)=\chi_0(\bar{\zeta})^{-1}.
\end{equation}
We interpret $\chi$ as a character of $M^{(1)}(\R)\cong\C^\times$. We extend $\tau$ to a representation of $G_1(\R)$ as in (\ref{M2representationseq}). \emph{In the archimedean case, we can always realize $\tau$ as a subrepresentation of a parabolically induced representation $\beta_1\times\beta_2$}, with characters $\beta_1,\beta_2:\:\R^\times\rightarrow\C^\times$ (see (\ref{gl2inducedeq})). We define the complex numbers $t_1,t_2,p,q$ by
\begin{equation}\label{t1t2pqeq}
 \beta_1(a)=a^{t_1},\qquad\beta_2(a)=a^{t_2},\qquad p=t_1-t_2,\qquad q=t_1+t_2
\end{equation}
for $a>0$.

\vspace{3ex}
{\bf Remark:} Evidently, $q$ is related to the central character of $\tau$ via $\omega_\tau(a)=a^q$ for $a>0$. The number $p$ could also be more intrinsically defined via the eigenvalue of the Laplace operator. Note that if $\tau$ belongs to the principal series and $\beta_1$ and $\beta_2$ are interchanged, then $p$ changes sign; this ambiguity will be irrelevant. We also note that if $\tau$ is a discrete series representation of lowest weight $l_1$, then $p=l_1-1$.

\vspace{3ex}
The induced representation $I_\Phi(s,\chi,\chi_0,\tau)$ is now a subrepresentation of $I_\Phi(s,\chi,\chi_0,\beta_1\times\beta_2)$. Any $\Phi \in I_\Phi(s,\chi,\chi_0,\tau) $ satisfies the transformation property (\ref{Wsharpproperty3eq}); in view of the Iwasawa decomposition, $\Phi$ is determined by its restriction to $K^{G_2}_\infty$. Conversely, given a function $\Phi:\:K^{G_2}_\infty\rightarrow\C$, it can be extented to a function on $G_2(\R)$ with the property (\ref{Wsharpproperty3eq}) if and only if
\begin{equation}\label{Wsharpwelldeflemmaeq2}
    \Phi(\hat\zeta_1k)=\zeta^{l_2}\,\Phi(k), \qquad
    \Phi(\hat\zeta_2k)=\zeta^{l_2}\,\Phi(k),
\end{equation}
for all $\zeta\in S^1$ and $k\in K^{G_2}_\infty$. Here, we used the notation
$$
 \hat\zeta_1 = \begin{bmatrix}\zeta\\&1\\&&\zeta\\&&&1\end{bmatrix},\qquad
 \hat\zeta_2 = \begin{bmatrix}1\\&\zeta\\&&1\\&&&\zeta\end{bmatrix}.
$$
We will therefore study certain spaces of functions on $K^{G_2}_\infty$ with the property (\ref{Wsharpwelldeflemmaeq2}).

\subsubsection*{The spaces $W^\Delta_{m,l,l_2}$}
We begin by describing the Lie algebra and the finite-dimensional representations of $K^{G_2}_\infty$. Let $\mathfrak{g}$ be the Lie algebra of $U(2,2)$. Let $X \mapsto -\,^t\!\bar X$ be the Cartan involution on $\mathfrak{g}$.
Let $\mathfrak{k}$ be the $+1$ eigenspace and let $\mathfrak{p}$ be the
$-1$ eigenspace of the Cartan involution. We denote by $\mathfrak{k}_\C$ and $\mathfrak{p}_\C$ the complexifications
of $\mathfrak{k}$ and $\mathfrak{p}$, respectively. Then
$$
 \mathfrak{k}_\C=\{\mat{A}{B}{-B}{A}\;|\;A,B\in{\rm Mat}_{2,2}(\C)\}, \qquad
 \mathfrak{p}_\C=\{\mat{A}{B}{B}{-A}\;|\;A,B\in{\rm Mat}_{2,2}(\C)\}.
$$
Hence $\mathfrak{g}_\C=\mathfrak{k}_\C\oplus\mathfrak{p}_\C=\mathfrak{gl}(4,\C)$.
The following eight elements constitute a convenient basis for $\mathfrak{k}_\C$.
\begin{align*}
 &U_1={\textstyle\frac12}\left[\begin{smallmatrix}1&&-i\\&0\\i&&1\\&&&0
  \end{smallmatrix}\right],\qquad
 U_2={\textstyle\frac12}\left[\begin{smallmatrix}0\\&1&&-i\\&&0\\&i&&1
  \end{smallmatrix}\right],\qquad
 V_1={\textstyle\frac12}\left[\begin{smallmatrix}-1&&-i\\&0\\i&&-1\\&&&0
  \end{smallmatrix}\right],\qquad
 V_2={\textstyle\frac12}\left[\begin{smallmatrix}0\\&-1&&-i\\&&0\\&i&&-1
  \end{smallmatrix}\right],\\
 &P_+={\textstyle\frac12}\left[\begin{smallmatrix}0&1&&-i\\&0\\&i&0&1\\&&&0
   \end{smallmatrix}\right],\qquad
 P_-={\textstyle\frac12}\left[\begin{smallmatrix}0\\1&0&-i\\&&0\\i&&1&0
   \end{smallmatrix}\right],\qquad
 Q_+={\textstyle\frac12}\left[\begin{smallmatrix}0\\-1&0&-i\\&&0\\i&&-1&0
   \end{smallmatrix}\right],\qquad
 Q_-={\textstyle\frac12}\left[\begin{smallmatrix}0&-1&&-i\\&0\\&i&0&-1\\&&&0
   \end{smallmatrix}\right].
\end{align*}
We have
\begin{equation}\label{kCgl2gl2eq}
 \mathfrak{k}_\C=\langle U_1,U_2,P_+,P_-\rangle\oplus\langle V_1,V_2,Q_+,Q_-\rangle
 \cong\mathfrak{gl}(2,\C)\oplus\mathfrak{gl}(2,\C).
\end{equation}
The center of $\mathfrak{k}$ is $2$-dimensional, spanned by
\begin{equation}\label{kcentereq}
 i(U_1+U_2+V_1+V_2)=\left[\begin{smallmatrix}&&1\\&&&1\\-1\\&-1\end{smallmatrix}\right]
 \qquad\text{and}\qquad
 i(U_1+U_2-V_1-V_2)=\left[\begin{smallmatrix}i\\&i\\&&i\\&&&i\end{smallmatrix}\right].
\end{equation}
The Casimir operators for the two copies of $\mathfrak{sl}(2,\C)$ are given by
$$
 \Delta_1=(U_1-U_2)^2+2(P_+P_-+P_-P_+),\qquad
 \Delta_2=(V_1-V_2)^2+2(Q_+Q_-+Q_-Q_+).
$$
The irreducible representations of $\mathfrak{k}_\C$ which lift to representations of $K^{G_2}_\infty$ are parametrized by four integers,
\begin{align*}
 \text{$m_1$:}\;\;\text{highest weight of }\langle U_1-U_2,P_+,P_-\rangle,& \quad
 \text{$n_1$:}\;\;\text{eigenvalue of }U_1+U_2,\\
 \text{$m_2$:}\;\;\text{highest weight of }\langle V_1-V_2,Q_+,Q_-\rangle,& \quad
 \text{$n_2$:}\;\;\text{eigenvalue of }V_1+V_2,
\end{align*}
subject to the condition that they all have the same parity and that
$m_1,m_2\geq0$. The parity condition is a consequence of overlapping one-parameter subgroups generated by $U_1\pm U_2$ and $V_1\pm V_2$. Let $\rho_{m_1,n_1,m_2,n_2}$ be the irreducible representation of $K^{G_2}_\infty$ corresponding to $(m_1,n_1,m_2,n_2)$. Then $\dim \rho_{m_1,n_1,m_2,n_2}=(m_1+1)(m_2+1)$, and the contragredient representation is given by $\tilde\rho_{m_1,n_1,m_2,n_2}=\rho_{m_1,-n_1,m_2,-n_2}$.
\begin{lemma}\label{rhom1m2n1n2characterizationlemma}
 Let $m_1,n_1,m_2,n_2$ be integers of the same parity with $m_1,m_2\geq0$.
 \begin{enumerate}
 \item Any vector $v$ in $\rho_{m_1,n_1,m_2,n_2}$ satisfies
  \begin{equation}\label{rhom1m2n1n2characterizationlemmaeq1}
   \Delta_1v=m_1(m_1+2)v,\quad\Delta_2v=m_2(m_2+2)v,\quad (U_1+U_2)v=n_1v,\quad(V_1+V_2)v=n_2v.
  \end{equation}
 \item The representation $\rho_{m_1,n_1,m_2,n_2}$ of $K^{G_2}_\infty$ contains the trivial
 representation of $K^H_\infty$ if and only if $m_1=m_2$ and $n_1=-n_2$.
 If these conditions are satisfied, then the trivial representation of $K^H_\infty$
 appears in $\rho_{m_1,n_1,m_2,n_2}$ exactly once.
 \end{enumerate}
\end{lemma}
\begin{proof} Equation (\ref{rhom1m2n1n2characterizationlemmaeq1}) holds by definition of $\rho_{m_1,n_1,m_2,n_2}$ and because the Casimir
operator acts as the scalar $m(m+2)$ on the irreducible representation of
$\mathfrak{sl}(2,\C)$ of dimension $m+1$. The complexification of the Lie algebra of $K^H_\infty$ is given by $
 \mathfrak{k}^H_\C=\langle U_1+V_1,\,U_2+V_2,\,P_++Q_+,\,P_-+Q_-\rangle$.
This Lie algebra is isomorphic to $\mathfrak{gl}(2,\C)$ and sits
diagonally in the product of the two copies of $\mathfrak{gl}(2,\C)$ in
(\ref{kCgl2gl2eq}). It follows that the restriction of the representation
$\rho_{m_1,n_1,m_2,n_2}$ to $\mathfrak{k}^H_\C$ is given by
$\rho_{m_1,n_2}\otimes\rho_{m_2,n_2}$, with each factor being a representation
of $\mathfrak{k}^H_\C\cong\mathfrak{gl}(2,\C)$. Such a tensor product contains
the trivial representation (and then with multiplicity one) if and only if
the second factor is the contragredient of the first, i.e., if and only if
$(m_2,n_2)=(m_1,-n_1)$.\
\end{proof}

Let $m$ be a non-negative integer, and $l$ and $l_2$ be any integers. Recall that $l_2$ determines the extension of the central character of $\tau$ to $\C^\times$. In our later applications $l$ will indicate the scalar minimal $K$-type of a lowest weight representation of $\GSp_4(\R)$, but for now $l$ is just an integer. Let $W^\Delta_{m,l,l_2}$ be the space of smooth, $K^{G_2}_\infty$-finite functions $\Phi:\:K^{G_2}_\infty\rightarrow\C$ with the properties
\begin{align}
 \Phi(\hat\zeta_1g)&=\Phi(\hat\zeta_2g)=\zeta^{l_2}\Phi(g)\qquad
  \text{for }g\in K^{G_2}_\infty,\;\zeta\in S^1,\label{WDeltacondeq1}\\
 \Phi(gk)&=\det(J(k,i_2))^{-l}\Phi(g)\qquad
  \text{for }g\in K^{G_2}_\infty,\;k\in K_\infty^H,\label{WDeltacondeq2}\\
 \Delta_1\Phi&=\Delta_2\Phi=m(m+2)\Phi\label{WDeltacondeq3}.
\end{align}
In (\ref{WDeltacondeq3}), the Casimir elements $\Delta_i$ are understood to act by right translation. As noted above, property (\ref{WDeltacondeq1}) is required to extend $\Phi$ to an element of $I_\Phi(s,\chi,\chi_0,\tau)$. Property (\ref{WDeltacondeq2}) will become important when we evaluate local zeta integrals in Sect.\ \ref{localzetasec}. Imposing the additional condition (\ref{WDeltacondeq3}) will result in a certain uniqueness which is useful for calculating intertwining operators; see Sect.\ \ref{inter-arch}. Evidently, the group consisting of all elements
$$
 \hat{r}(\theta) := m_2(\mat{\cos(\theta)}{\sin(\theta)}{-\sin(\theta)}{\cos(\theta)}),\qquad\theta\in\R,
$$
acts on $W^\Delta_{m,l,l_2}$ by left translation. Let $W^\Delta_{m,l,l_2,l_1}$ be the subspace of $W^\Delta_{m,l,l_2}$ consisting of $\Phi$ with the additional property
\begin{equation}\label{WDeltacondeq4}
 \Phi(\hat r(\theta)g)=e^{il_1\theta}\Phi(g)\qquad
  \text{for }g\in G_2(\R),\;\theta\in\R.
\end{equation}
Then
\begin{equation}\label{WDeltaWDeltal1eq}
 W^\Delta_{m,l,l_2}=\bigoplus_{l_1\in\Z}W^\Delta_{m,l,l_2,l_1}.
\end{equation}
Let $D$ be the function on $K^{G_2}_\infty$ given by $D(g)=\det(J(g,i_2))$. It is easily verified that
$$
 (U_1-U_2)D=(V_1-V_2)D=P_\pm D=Q_\pm D=0.
$$
Hence $\Delta_iD=0$ for $i=1,2$, and consequently $\Delta_i(fD^l)=(\Delta_if)D^l$ for any smooth function $f$ on $K^{G_2}_\infty$. It is then easy to see that the map $\Phi\mapsto\Phi D^l$ provides isomorphisms
\begin{equation}\label{Wll1l2isoeq}
 W^\Delta_{m,l,l_2}\stackrel{\sim}{\longrightarrow}W^\Delta_{m,0,l_2+l}
 \qquad\text{and}\qquad
 W^\Delta_{m,l,l_2,l_1}\stackrel{\sim}{\longrightarrow}W^\Delta_{m,0,l_2+l,l_1-l}.
\end{equation}
Let $L^2(K^{G_2}_\infty)_{\rm fin}$ be the space of smooth, $K^{G_2}_\infty$-finite functions
$K^{G_2}_\infty\rightarrow\C$. It is a module for $K^{G_2}_\infty\times K^{G_2}_\infty$ via $
 ((g_1,g_2).f)(h)=f(g_1^{-1}hg_2)$. By the Peter-Weyl theorem, as $K^{G_2}_\infty\times K^{G_2}_\infty$-modules,
$$
 L^2(K^{G_2}_\infty)_{\rm fin}\cong\bigoplus_\rho(\tilde\rho\otimes\rho)\qquad
 \text{(algebraic direct sum)},
$$
where $\rho$ runs through all equivalence classes of irreducible representations
of $K^{G_2}_\infty$, and where $\tilde\rho$ denotes the contragredient. Evidently,
\begin{equation}\label{WDeltarhodecompeq}
 W^\Delta_{m,0,l_2+l,l_1-l}=\bigoplus_\rho\big(W^\Delta_{m,0,l_2+l,l_1-l}\cap(\tilde\rho\otimes\rho)\big),
\end{equation}
and analogously for $W^\Delta_{m,0,l_2+l}$.
\begin{lemma}\label{Wrhocontributionlemma}
 Let $m$ be a non-negative integer, and $l$ and $l_2$ be any integers.
 \begin{enumerate}
  \item Let $\rho=\rho_{m_1,n_1,m_2,n_2}$. Then, for $l_1\in\Z$,
   $$
    \dim\big(W^\Delta_{m,0,l_2+l,l_1-l}\cap(\tilde\rho\otimes\rho)\big)
    =\left\{\begin{array}{l@{\qquad}l}
    1&\text{if }m_1=m_2=m,\;n_1=l_2+l,\;n_2=-(l_2+l),\\
    &|l_1-l|\leq m,\;l_1-l\equiv l_2+l\equiv m\;{\rm mod}\;2,\\
    0&\text{otherwise}.\end{array}\right.
   $$
  \item For $l_1\in\Z$,
   $$
    \dim\big(W^\Delta_{m,0,l_2+l,l_1-l}\big)
    =\left\{\begin{array}{l@{\qquad}l}
    1&\text{if }|l_1-l|\leq m,\;l_1-l\equiv l_2+l\equiv m\;{\rm mod}\;2,\\
    0&\text{otherwise}.\end{array}\right.
   $$
  \item
   $$
    \dim\big(W^\Delta_{m,l,l_2}\big)
    =\left\{\begin{array}{l@{\qquad}l}
    m+1&\text{if }l_2+l\equiv m\;{\rm mod}\;2,\\
    0&\text{otherwise}.\end{array}\right.
   $$
 \end{enumerate}
\end{lemma}
\begin{proof} i) By the right $K_\infty^H$-invariance of functions in $W^\Delta_{m,0,l_1-l,l_2+l}$  and Lemma \ref{rhom1m2n1n2characterizationlemma} ii), if $\tilde\rho\otimes\rho$ contributes to $W^\Delta_{m,0,l_2+l,l_1-l}$, then necessarily $m_1=m_2$ and $n_1=-n_2$. Condition (\ref{WDeltacondeq3}) forces $m_1=m_2=m$. Assume all of this is satisfied, say $\rho=\rho_{m,n,m,-n}$. Then, by Lemma \ref{rhom1m2n1n2characterizationlemma} ii), there exists a non-zero vector $v_0\in\rho$, unique up to multiples, such that $v_0$ is fixed by $K^H$. Hence, any element $w\in W^\Delta_{m,0,l_2+l,l_1-l}\cap(\tilde\rho\otimes\rho)$ is of the form $w=v\otimes v_0$ for some $v\in\tilde\rho=\rho_{m,-n,m,n}$. Taking into account that the first element of the center of $\mathfrak{k}_\C$ in (\ref{kcentereq}) acts trivially on $W^\Delta_{m,0,l_2+l,l_1-l}$, any element $\Phi$ of this space has the following transformation properties under left translation $L$,
\begin{alignat}{2}
 L(U_1-U_2)\Phi&=(l_1-l)\Phi,\qquad&L(V_1-V_2)\Phi
  &=(l_1-l)\Phi,\label{Wrhocontributionlemmaeq1}\\
 L(U_1+U_2)\Phi&=-(l_2+l)\Phi,\qquad
  &L(V_1+V_2)\Phi&=(l_2+l)\Phi.\label{Wrhocontributionlemmaeq2}
\end{alignat}
Since $U_1+U_2$ and $V_1+V_2$ are in the center, (\ref{Wrhocontributionlemmaeq2}) implies that $R(U_1+U_2)\Phi=(l_2+l)\Phi$ and $R(V_1+V_2)\Phi=-(l_2+l)\Phi$, where $R$ is right translation. It follows that $n=l_2+l$. This number must have the same parity as $m$. From (\ref{Wrhocontributionlemmaeq1}) we conclude that $v$ is a vector of weight $(l_1-l,l_1-l)$ in $\tilde\rho$. There exists such a vector $v$ in $\tilde\rho$ if and only if $-m\leq l_1-l\leq m$ and $l_1-l\equiv m$ mod $2$, and in this case $v$ is unique up to multiples.

\vspace{2ex}
ii) follows from i) and (\ref{WDeltarhodecompeq}).

\vspace{2ex}
iii) For $l=0$ the statement follows from ii) and (\ref{WDeltaWDeltal1eq}). For other values of $l$, it follows from the $l=0$ case and (\ref{Wll1l2isoeq}).
\end{proof}

Our next task will be to find an explicit formula for the function spanning the one-dimensional space $W^\Delta_{m,l,l_2,l_1}$. We define, for $g\in K^{G_2}_\infty$,
\begin{align*}
& \hat a(g)=(1,1)\text{--coefficient of }J(g\,^t\!g,i_2), \qquad
 \hat b(g)=(1,2)\text{--coefficient of }J(g\,^t\!g,i_2),\\
& \hat c(g)=(2,1)\text{--coefficient of }J(g\,^t\!g,i_2), \qquad
 \hat d(g)=(2,2)\text{--coefficient of }J(g\,^t\!g,i_2).
\end{align*}
Since they can be written in terms of matrix coefficients, these are $K^{G_2}_\infty$-finite functions. It is not difficult to calculate the action of $P_\pm,Q_\pm$, the torus elements and the Casimir elements on the functions $\hat a,\hat b,\hat c,\hat d$ under left and  right translation explicitly. The following lemma summarizes the results.
\begin{lemma}\label{Deltaabcdlemma}
 Let $m$ be a non-negative integer.
 \begin{enumerate}
  \item If $f=\hat a^{i_1}\,\hat b^{i_2}\,\hat c^{i_3}\,\hat d^{i_4}$ with non-negative integers $i_1,\ldots,i_4$ such that $i_1+i_2+i_3+i_4=m$, then, under  right translation,
   $$
    \Delta_i f^m = m(m+2) f^m\qquad\text{for }i=1,2.
   $$
  \item The functions $f$ as in i) are contained in $\tilde\rho\otimes\rho$ with $\rho=\rho_{m,m,m,-m}$ and are right invariant under $K^H_\infty$.
  \item Let $f=\hat b^{m-j}\,\hat c^{j}$ with $0\leq j\leq m$. Then, with $L$ being left translation,
   \begin{alignat}{2}
    L(U_1-U_2)f&=(m-2j)f,&\qquad L(V_1-V_2)f&=(m-2j)f,\label{Deltaabcdlemmaeq1}\\
    L(U_1+U_2)f&=-mf,&\qquad L(V_1+V_2)f&=mf.\label{Deltaabcdlemmaeq2}
   \end{alignat}
 \end{enumerate}
\end{lemma}
Using this lemma, it is easy to verify that the function
$$
 \hat b^{\frac{m+l_1-l}2}\,\hat c^{\frac{m-l_1+l}2}
 (\hat a\hat d-\hat b\hat c)^{\frac{l_2+l-m}2}
$$
lies in $W^\Delta_{m,0,l_2+l,l_1-l}$, provided all exponents are integers and the first two are non-negative. In view of (\ref{Wll1l2isoeq}), we obtain the following result.
\begin{proposition}\label{W0l1l2distvecprop}
 Let $m$ be a non-negative integer, and $l,l_1,l_2$ be any integers. We assume that
 $|l_1-l|\leq m$ and $l_1-l\equiv l_2+l\equiv m\;{\rm mod}\;2$, so that the space $W^\Delta_{m,l,l_2,l_1}$ is one-dimensional. Then this space is spanned by the function
 \begin{equation}\label{W0l1l2distvecpropeq1}
  \Phi^\#_{m,l,l_2,l_1}:=(-i)^m\,\hat b^{\frac{m+l_1-l}2}\,\hat c^{\frac{m-l_1+l}2}
  (\hat a\hat d-\hat b\hat c)^{\frac{l_2+l-m}2}D^{-l},
 \end{equation}
 where $D(g)=\det(J(g,i_2))$. This function has the property that
 \begin{equation}\label{W0l1l2distvecpropeq2}
  \Phi^\#_{m,l,l_2,l_1}(\eta_0)=1,
 \end{equation}
 with $\eta_0$ as in (\ref{eta0defeq2}).
\end{proposition}
\subsubsection*{Special vectors in $I(s,\chi,\chi_0,\tau)$}
We return to the induced representation $I_\Phi(s,\chi,\chi_0,\tau)$, considered a subspace of the Borel induced representation $I_\Phi(s,\chi,\chi_0,\beta_1\times\beta_2)$. Since the functions $\Phi^\#_{m,l,l_2,l_1}$ defined in Proposition \ref{W0l1l2distvecprop} satisfy condition (\ref{Wsharpwelldeflemmaeq2}), they extend to elements of $I_\Phi(s,\chi,\chi_0,\beta_1\times\beta_2)$. We use the same notation for the extended functions.
\begin{lemma}\label{Phisharptaulemma}
 The function $\Phi^\#_{m,l,l_2,l_1}$ belongs to $I_\Phi(s,\chi,\chi_0,\tau)$ if and only if the weight $l_1$ occurs in $\tau$.
\end{lemma}
\begin{proof} As a subspace of $I_\Phi(s,\chi,\chi_0,\beta_1\times\beta_2)$, the representation $I_\Phi(s,\chi,\chi_0,\tau)$ consists of all functions $\Phi:\:G_2\rightarrow\C$ of the form
$$
 \Phi(m_1m_2nk)=\delta_P(m_1m_2)^{s+1/2}\chi(m_1)\varphi(m_2)J(k),
 \qquad m_i\in M_i(\R),\:n\in N(\R),\:k\in K^{G_2}_\infty,
$$
where $\varphi$ lies in $\chi_0\times\tau$, and where $J$ is an appropriate function on $K^{G_2}_\infty$. It follows that $\Phi\in I_\Phi(s,\chi,\chi_0,\beta_1\times\beta_2)$ lies in $I_\Phi(s,\chi,\chi_0,\tau)$ if and only if the function
$$
 M_2(\R)\ni m_2\longmapsto\Phi(m_2)\delta_P(m_2)^{-s-1/2}
$$
belongs to $\chi_0\times\tau$. Since $\Phi^\#_{m,l,l_2,l_1}$ satisfies (\ref{WDeltacondeq4}), the function $m_2\mapsto\Phi^\#_{m,l,l_2,l_1}(m_2)$ has weight $l_1$. The assertion follows.
\end{proof}

For simplicity, we will from now on let $\mathbf{c}=1$ for the rest of this section; this is all we need for the global application. Let the classical Whittaker function $W_{k,m}$ be the same as in \cite[p.\ 244]{Bu} or \cite[7.1.1]{MOS}. We fix a point $t^+\in\R_{>0}$, depending on $p$, such that
\begin{equation}\label{tplusdefeq}
 W_{\pm\frac{l_1}2, \frac p2}(t^+)\neq0\qquad\text{for all }l_1\in\Z.
\end{equation}
Note that, if $p$ is a positive integer (corresponding to $\tau$ being a discrete series representation), we can choose $t^+=1$, since $W_{\pm\frac{l_1}2, \frac p2}$ is essentially an exponential function. Let $W_{l_1}$ be the vector of weight $l_1$ in the $\psi^{-1}$ Whittaker model of $\tau$. Using differential operators and solving differential equations, one can show that there exist constants $a^+,a^-\in\C$ such that
\begin{equation}\label{W1Whittakerformulaeq}
 W_{l_1}(\mat{t}{0}{0}{1}) = \renewcommand{\arraystretch}{1.3}
  \left\{\begin{array}{ll}
 a^+\omega_\tau((4\pi t)^{1/2})W_{\frac{l_1}2, \frac p2}(4\pi t) & \hbox{ if } t > 0, \\
 a^-\omega_\tau((-4\pi t)^{1/2})
  W_{-\frac{l_1}2, \frac p2}(-4\pi t)& \hbox{ if } t <0 .\end{array}\right.
\end{equation}
Our choice of additive character implies that $a^+$ is non-zero as long as $l_1>0$. We will normalize the constant $a^+=a^+_{l_1,p,q}$ such that
$$
 W_{l_1}(\mat{t^+}{}{}{1})=1.
$$
i.e.,
\begin{equation}\label{aplusdefeq}
 a^+=a^+_{l_1,p,q}=\big(\omega_\tau((4\pi t^+)^{1/2})W_{\frac{l_1}2, \frac p2}(4\pi t^+)\big)^{-1}
 =(4\pi t^+)^{-q/2}\,W_{\frac{l_1}2, \frac p2}(4\pi t^+)^{-1}.
\end{equation}
Consider the Whittaker realization $I_W(s,\chi,\chi_0,\tau)$ of $I(s,\chi,\chi_0,\tau)$, with $\tau$ given in its $\psi^{-1}$ Whittaker model (see (\ref{Wsharpproperty1eq}), (\ref{Wsharpproperty2eq})).
We extent $W_{l_1}$ to a function on $G_1(\R)$ via the character $\chi_0$; see (\ref{extendedWformulaeq}). Using the Iwasawa decomposition, we define a function $W^\#_{m,l,l_2,l_1}$ in $I_W(s,\chi,\chi_0,\tau)$ by
\begin{equation}\label{Wsharpwelldeflemmaeq1}
  W^\#_{m,l,l_2,l_1}(m_1m_2nk,s)=(t^+)^{q/2}\delta_P^{s+1/2}(m_1m_2)\,\chi(m_1)\,W_{l_1}(m_2)\,\Phi^\#_{m,l,l_2,l_1}(k),
\end{equation}
where $m_1\in M^{(1)}(\R)$, $m_2\in M^{(2)}(\R)$, $n\in N(\R)$ and
$k\in K^{G_2}_\infty$, and where $\Phi^\#_{m,l,l_2,l_1}$ is the same function as in (\ref{W0l1l2distvecpropeq1}); this is well-defined by the transformation properties of $W_{l_1}$ and $\Phi^\#_{m,l,l_2,l_1}$. Note that

$$
 W^\#_{m,l,l_2,l_1}({\rm diag}(\sqrt{t^+},t^+,\sqrt{t^+},1)\eta_0,s)=1.
$$
There is an intertwining operator $\Phi\mapsto W_\Phi$ from $I_\Phi(s,\chi,\chi_0,\tau)$ to $I_W(s,\chi,\chi_0,\tau)$, which, in the region of convergence, is given by
\begin{equation}\label{WPhieq}
 W_\Phi(g)=\int\limits_\R e^{-2\pi ix}\,\Phi(\begin{bmatrix}1\\&&&1\\&&1\\&-1\end{bmatrix}\begin{bmatrix}1\\&1&&x\\&&1\\&&&1\end{bmatrix}g)\,dx.
\end{equation}
Outside the region of convergence the intertwining operator is given by the analytic continuation of this integral. This operator is simply an extension of a standard intertwining operator for the underlying $\GL_2(\R)$ representation $\beta_1\times\beta_2$. It is easy to see that under this intertwining operator the function $\Phi^\#_{m,l,l_2,l_1}$ maps to a multiple of $W^\#_{m,l,l_2,l_1}$. Let $\kappa_{l_1,p,q}$ be the constant such that
\begin{equation}\label{PhiWkappaeq}
 W_{\Phi^\#_{m,l,l_2,l_1}}=\kappa_{l_1,p,q}\,W^\#_{m,l,l_2,l_1}
\end{equation}
We will distinguish three disjoint cases A, B, C according to the type of $\tau$ and the constellation of its weights relative to the integer $l$ (which later will be a minimal $\GSp_4$ weight).
\begin{align}\label{ABCdefeq}
 &\bullet\quad\text{{\bf Case A}: Neither the weight $l$ nor the weight $l-1$ occur in $\tau$.}\nonumber\\
 &\bullet\quad\text{{\bf Case B}: The weight $l$ occurs in $\tau$.}\\
 &\bullet\quad\text{{\bf Case C}: The weight $l-1$ occurs in $\tau$.}\nonumber
\end{align}
Note that in Case A necessarily $\tau=\mathcal{D}_{p,q/2}$, a discrete series representation with Harish-Chandra parameter $p\geq l$ (and central character satisfying $a\mapsto a^q$ for $a>0$). In this case let us set $l_1=p+1$, which is the minimal weight. It satisfies $l_1\geq2$. In each of the three cases we will define a non-negative integer $m$ and a distinguished function $\Phi^\#$ as a linear combination of certain $\Phi^\#_{m,l,l_2,l_1}$ as in (\ref{W0l1l2distvecpropeq1}). The definition is as in the following table. The last column of the table shows $W^\#$, by definition the image of $\Phi^\#$ under the intertwining operator $\Phi\mapsto W_\Phi$.
\begin{equation}\label{ABCtableeq}\renewcommand{\arraystretch}{1.5}
 \begin{array}{cccc}
  \text{Case}&m&\Phi^\#&W^\#\\\hline
  \text{A}&l_1-l&\Phi^\#_{m,l,l_2,l_1}&\kappa_{l_1,p,q}W^\#_{m,l,l_2,l_1}\\
  \text{B}&0&\Phi^\#_{m,l,l_2,l}&\kappa_{l,p,q}W^\#_{m,l,l_2,l}\\
  \text{C}&1&\Phi^\#_{m,l,l_2,l+1}+\big(3s-\frac{p+q}2\big)\Phi^\#_{m,l,l_2,l-1}
   &\kappa_{l+1,p,q}W^\#_{m,l,l_2,l+1}\\
   &&&+\big(3s-\frac{p+q}2\big)\kappa_{l-1,p,q}W^\#_{m,l,l_2,l-1}
 \end{array}
\end{equation}
In all cases, by Lemma \ref{Phisharptaulemma}, the function $\Phi^\#$ lies in $I_\Phi(s,\chi,\chi_0,\tau)$.

\begin{theorem}\label{distinguishedvectorarchtheorem}
 Let $(\tau,V_\tau)$ be a generic, irreducible, admissible representation of $\GL_2(\R)$ with central character $\omega_{\tau}$. We realize $\tau$ as a subrepresentation of an induced representation $\beta_1\times\beta_2$, and define $p,q\in\C$ by (\ref{t1t2pqeq}).
 Let $l_2$ be an integer of the same parity as the weights of $\tau$. Let $\chi_0$ be the character of $\C^\times$ such that $\chi_0\big|_{\R^\times}=\omega_\tau$ and $\chi_0(\zeta)=\zeta^{l_2}$ for $\zeta\in S^1$, and let $\chi$ be the character of $\C^\times$ defined by (\ref{chi-lambda-char-condition-arch}).
 Assume that $l$ is a positive integer. Let $m$ and $\Phi^\#$ be chosen according to table (\ref{ABCtableeq}).
 \begin{enumerate}
  \item The function $\Phi^\#$ satisfies
   \begin{equation}\label{distinguishedvectorarchtheoremeq5}
    \Phi^\#(gk)=\det(J(k,i_2))^{-l}\Phi^\#(g)
     \qquad\text{for }g\in G_2(\R),\;k\in K_\infty^H
   \end{equation}
   and
   \begin{equation}\label{distinguishedvectorarchtheoremeq6}
    \Delta_1\Phi^\#=\Delta_2\Phi^\#=m(m+2)\Phi^\#.
   \end{equation}
  \item Assume we are in Case A or B. Then, up to scalars, $\Phi^\#$ is the unique $K^{G_2}_\infty$-finite element of $I_\Phi(s,\chi,\chi_0,\tau)$ with the properties (\ref{distinguishedvectorarchtheoremeq5}) and (\ref{distinguishedvectorarchtheoremeq6}).
  \item Assume we are in Case C. Then the space of $K^{G_2}_\infty$-finite functions in $I_\Phi(s,\chi,\chi_0,\tau)$ with the properties (\ref{distinguishedvectorarchtheoremeq5}) and (\ref{distinguishedvectorarchtheoremeq6}) is two-dimensional, spanned by $\Phi^\#_{m,l,l_2,l-1}$ and $\Phi^\#_{m,l,l_2,l+1}$.
 \end{enumerate}
\end{theorem}
\begin{proof} i) is obvious, since $\Phi^\#$ lies in $W^\Delta_{m,l,l_2}$.

\vspace{2ex}
ii) Assume first we are in Case A. By our hypotheses, $0<l<l_1$. Assume that $\Phi\in I_\Phi(s,\chi,\chi_0,\tau)$ is $K^{G_2}_\infty$-finite and satisfies (\ref{distinguishedvectorarchtheoremeq5}) and (\ref{distinguishedvectorarchtheoremeq6}). Then, evidently, the restriction of $\Phi$ to $K^{G_2}_\infty$ lies in $W^\Delta_{m,l,l_2}$. By (\ref{WDeltaWDeltal1eq}) and Proposition \ref{W0l1l2distvecprop},
$$
 W^\Delta_{m,l,l_2}=\bigoplus_{\substack{j\in\Z\\|j-l|\leq m\\j-l\equiv m\;{\rm mod}\;2}}\C\Phi^\#_{m,l,l_2,j}.
$$
If a $\Phi^\#_{m,l,l_2,j}$ occurring in this direct sum is an element of $I_\Phi(s,\chi,\chi_0,\tau)$, then, by Lemma \ref{Phisharptaulemma}, the weight $j$ occurs in $\tau$. Since $\tau$ has minimal weight $l_1$, this implies $j\leq-l_1$ or $j\geq l_1$. The first inequality leads to a contradiction, and the second inequality implies $j=l_1$. This proves the uniqueness in Case A. In Case B, as before, the restriction of any $K^{G_2}_\infty$-finite $\Phi\in I_\Phi(s,\chi,\chi_0,\tau)$ satisfying (\ref{distinguishedvectorarchtheoremeq5}) and (\ref{distinguishedvectorarchtheoremeq6}) to $K^{G_2}_\infty$ lies in $W^\Delta_{m,l,l_2}$. By Lemma \ref{Wrhocontributionlemma}, this space is one-dimensional.

\vspace{2ex}
iii) Again, the restriction of any $K^{G_2}_\infty$-finite $\Phi\in I_\Phi(s,\chi,\chi_0,\tau)$ satisfying (\ref{distinguishedvectorarchtheoremeq5}) and (\ref{distinguishedvectorarchtheoremeq6}) to $K^{G_2}_\infty$ lies in $W^\Delta_{m,l,l_2}$. By Lemma \ref{Wrhocontributionlemma}, this space is two-dimensional.
\end{proof}

Since the functions $\Phi^\#$ and $W^\#$ have the same right transformation properties, the following is an immediate consequence of Theorem \ref{distinguishedvectorarchtheorem}.
\begin{corollary}\label{distinguishedvectorarchtheoremcor}
 Let the non-negative integer $m$ and the function $W^\#$ in $I_W(s,\chi,\chi_0,\tau)$ be chosen according to table (\ref{ABCtableeq}).
 \begin{enumerate}
  \item The function $W^\#$ satisfies
   \begin{equation}\label{distinguishedvectorarchtheoremcoreq5}
    W^\#(gk)=\det(J(k,i_2))^{-l}W^\#(g)
     \qquad\text{for }g\in G_2(\R),\;k\in K_\infty^H
   \end{equation}
   and
   \begin{equation}\label{distinguishedvectorarchtheoremcoreq6}
    \Delta_1W^\#=\Delta_2W^\#=m(m+2)W^\#.
   \end{equation}
  \item Assume we are in Case A or B. Then, up to scalars, $W^\#$ is the unique $K^{G_2}_\infty$-finite element of $I_W(s,\chi,\chi_0,\tau)$ with the properties (\ref{distinguishedvectorarchtheoremcoreq5}) and (\ref{distinguishedvectorarchtheoremcoreq6}).
  \item Assume we are in Case C. Then the space of $K^{G_2}_\infty$-finite functions in $I_W(s,\chi,\chi_0,\tau)$ with the properties (\ref{distinguishedvectorarchtheoremcoreq5}) and (\ref{distinguishedvectorarchtheoremcoreq6}) is two-dimensional, spanned by $W^\#_{m,l,l_2,l-1}$ and $W^\#_{m,l,l_2,l+1}$.
 \end{enumerate}
\end{corollary}

\subsubsection*{A relation between unknown constants}
In this section we defined the constants $\kappa_{l_1,p,q}$ and $a^+_{l_1,p,q}$; see (\ref{PhiWkappaeq}) and (\ref{aplusdefeq}). Note that these constants also depend on the choice of the point $t^+$, which is not reflected in the notation. We do not know the explicit value of any of these constants. However, the following lemma describes a relation between these constants which will become important in the proof of Lemma \ref{Xarchlemma}.
\begin{lemma}\label{kappaalemma}
 Let $\beta_1$ and $\beta_2$ be characters of $\R^\times$ such that the induced representation $\beta_1\times\beta_2$ is irreducible. Let $p,q\in\C$ be as in (\ref{t1t2pqeq}). Then, for any integer $l$ whose parity is different from the parity of the weights of $\beta_1\times\beta_2$,
 \begin{equation}\label{kappaalemmaeq}
  \frac{\kappa_{l-1,p,q}\:a^+_{l-1,p,q}}{\kappa_{l+1,p,q}\:a_{l+1,p,q}^+}=-\frac12(p+l).
 \end{equation}
\end{lemma}
\begin{proof}
We consider the intertwining operator $\varphi\mapsto W_\varphi$ from $\tau$ to the Whittaker model $\mathcal{W}(\beta_1\times\beta_2,\psi^{-\mathbf{c}})$ which, in the region of convergence, is given by
\begin{equation}\label{archgl2whittakereq1}
 W_\varphi(g)=\int\limits_\R e^{-2\pi ix}\varphi(\mat{}{1}{-1}{}\mat{1}{x}{}{1}g)\,dx.
\end{equation}
For a weight $k$ occurring in $\beta_1\times\beta_2$ let $\varphi_k$ be the element of $\beta_1\times\beta_2$ of weight $k$ satisfying $\varphi_k(1)=1$, and let $W_k$ be the element of $\mathcal{W}(\beta_1\times\beta_2,\psi^{-1})$ of weight $k$ satisfying $W_k(\mat{t^+}{}{}{1})=1$. Then $W_{\varphi_k}=\kappa_{k,p,q}(t^+)^{q/2}\,W_k$ with the same $\kappa_{k,p,q}$ as in (\ref{PhiWkappaeq}).

\vspace{3ex}
Recall that the constants $a^+_{k,p,q}$ defined in (\ref{aplusdefeq}) were designed so that
\begin{equation}\label{wkdefeq}
 w_k(t):=W_k(\mat{t}{}{}{1})=a_{k,p,q}^+\,(4\pi t)^{q/2}\,W_{\frac k2,\frac p2}(4\pi t)
\end{equation}
satisfies $w_k(t^+)=1$. If $L$ denotes the Lie algebra element $\frac12\mat{1}{-i}{-i}{-1}$, then straightforward calculations show that
\begin{equation}\label{kappaalemmaeq2}
  \tau(L)\varphi_k=\frac{p+1-k}2\varphi_{k-2},\qquad  (\tau(L)W_k)(\mat{t}{}{}{1})=\Big(-\frac q2-\frac k2+2\pi t\Big)w_k(t)+tw_k'(t)
\end{equation}
(where $\tau$ stands for the right translation action on both $\beta_1\times\beta_2$ and its Whittaker model). For $k$ one of the weights appearing in $\beta_1\times\beta_2$, define constants $\lambda_{k,p,q}$ and $\mu_{k,p,q}$ by
$$
 \tau(L)\varphi_k=\lambda_{k,p,q}\,\varphi_{k-2}\qquad\text{and}\qquad
 \tau(L)W_k=\mu_{k,p,q}W_{k-2}.
$$
By our normalizations, $\lambda_{k,p,q}=(\tau(L)\varphi_k)(1)$ and $\mu_{k,p,q}=(\tau(L)W_k)(\mat{t^+}{}{}{1})$.
Hence, by (\ref{kappaalemmaeq2}),
\begin{equation}\label{muvalueeq}
 \lambda_{k,p,q}=\frac{p+1-k}2,\qquad\mu_{k,p,q}=-\frac q2-\frac k2+2\pi t^++t^+w_k'(t^+).
\end{equation}
Following the function $\varphi_k$ through the commutative diagram
$$
 \begin{CD}
  \beta_1\times\beta_2@>\tau(L)>>\beta_1\times\beta_2\\
  @VVV @VVV\\
  \mathcal{W}(\tau,\psi^{-1})@>>\tau(L)>\mathcal{W}(\tau,\psi^{-1})
 \end{CD}
$$
we get the identity
\begin{equation}\label{kappamulambdarelationeq}
 \kappa_{k,p,q}\,\mu_{k,p,q}=\lambda_{k,p,q}\,\kappa_{k-2,p,q}.
\end{equation}
To further calculate the constant $\mu_{k,p,q}$, we will take the derivative of the function $w_k$ defined in (\ref{wkdefeq}). We will make use of the following identity for Whittaker functions,
\begin{equation}\label{whittakerderivativeeq}
 z\,W'_{k,b}(z)=\Big(k-\frac z2\Big)W_{k,b}(z)
 -\Big(b^2-\Big(k-\frac12\Big)^2\Big)W_{k-1,b}(z)
\end{equation}
(see \cite[7.2.1]{MOS}). Using this, one obtains from~\eqref{aplusdefeq},~\eqref{wkdefeq} and~\eqref{muvalueeq} that
$$
 \mu_{k,p,q}=-\frac{p^2-(k-1)^2}4\frac{a_{k,p,q}^+}{a^+_{k-2,p,q}}.
$$
Substituting the values of $\lambda_{k,p,q}$ and $\mu_{k,p,q}$ into (\ref{kappamulambdarelationeq}) and setting $k=l+1$ proves the asserted identity.
\end{proof}

\subsection{Intertwining operator: non-archimedean case}\label{inter-non-arch}
In this section let $F$ be $p$-adic. We use the notation from Theorem \ref{unique-W-theorem}. In addition, we will assume that $\Lambda\big|_{F^\times}=1$; this will be sufficient for our global applications. In this section we will calculate the function $K(s)$ given in (\ref{Ksformulaeq}).

\vspace{3ex}
Let us be precise about the measure on $N(F)$. Recall that $N(F)$ consists of one copy of $F$ and two copies of $L$. The measure on $F$ is the one that is self-dual with respect to the character $\psi$, and the measure on $L$ is the one that is self-dual with respect to the character $\psi\circ{\rm tr}_{L/F}$. Since we are assuming that $\psi$ has conductor $\OF$, it follows (see Sect.\ 2.2 of \cite{TT}) that
\begin{equation}\label{OFOLvolumeseq}
 {\rm vol}(\OF)=1\qquad\text{and}\qquad{\rm vol}(\OF_L)=N(\mathbf{d})^{-1/2}.
\end{equation}
Recall here that the norm of the different is the discriminant, and that $\mathbf{d}=\mathbf{b}^2-4\mathbf{a}\mathbf{c}$ generates the discriminant of $L/F$ by our conventions. If we let $\mathbf{d}\OF=\varpi^\delta\OF$ (where $\delta=0$ unless $L/F$ is a ramified field extension), then ${\rm vol}(\OF_L)=q^{-\delta/2}$. This explains the factor $q^{-\delta}$ in the following result.
\begin{proposition}[Gindikin-Karpelevich Formula]\label{GKprop}
 Let $\delta$ be the valuation of the discriminant of $L/F$ if $L/F$ is a ramified field extension, and $\delta=0$ otherwise. If $\tau$ is unramified, then
   $$
    K(s)=q^{-\delta}\frac{L(6s,\chi\big|_{F^\times})
    L(3s,\tau\times\mathcal{AI}(\Lambda)\times\chi\big|_{F^\times})}
    {L(6s+1,\chi\big|_{F^\times})
    L(3s+1,\tau\times\mathcal{AI}(\Lambda)\times\chi\big|_{F^\times})}.
   $$
\end{proposition}
This formula can be obtained by a straightforward integral calculation; we omit the details. For non-spherical $\tau$ it will be necessary to distinguish the inert, split and ramified cases. For our global applications it turns out that explicit knowledge of $K(s)$ at finitely many finite places is not necessary. Thus, we will only calculate $K(s)$ in the inert and split cases.

\vspace{3ex}
We will first assume that $L/F$ is an unramified field extension. We write the explicit formula (\ref{Ksformulaexpliciteq}) as $I_1+I_2$, where in $I_1$ the $z$-integration is restricted to the set $\OF_L$, and in $I_2$ the $z$-integration is restricted to $L\setminus\OF_L$. After some changes of variables, we get
\begin{equation}\label{intertwining-int-I1}
 I_1=\int\limits_{\OF_L}\int\limits_L\int\limits_F
 W^\#(\begin{bmatrix}1&&&\\y&1&&\\x&&1&-\bar{y}\\&&&1\end{bmatrix} w_1
 \begin{bmatrix}1&z&&\\&1&&\\&&1&\\&&-\bar{z}&1\end{bmatrix} \eta, s)\,dz\,dy\,dx
\end{equation}
and
\begin{align}\label{intertwining-int-I2}
 I_2&=\int\limits_{L\setminus\OF_L}\int\limits_L\int\limits_F|z\bar z|\,
  W^\#(\begin{bmatrix}\bar{z}^{-1}&&&1\\&z^{-1}&1&\\&&z&\\&&&\bar{z}\end{bmatrix}
  \begin{bmatrix}1&&&\\&1&&xz\bar{z}-\bar{y}z-y\bar{z}\\&&1&\\&&&1\end{bmatrix}\nonumber\\
 &\qquad \times \begin{bmatrix}1&&&\\y&1&&\\x&&1&-\bar{y}\\&&&1\end{bmatrix}
  \begin{bmatrix}&&&1\\-1&&&\\&-1&&\\&&-1&\end{bmatrix}
  \begin{bmatrix}1\\z^{-1}&1\\&&1&-\bar{z}^{-1}\\&&&1\end{bmatrix}\eta,s)\,dx\,dy\,dz.
\end{align}
The argument of $W^\#$ needs to be written as $pk$, where $p \in P(F)$ and $k \in K^{G_2}$.   For both $I_1$ and $I_2$ the key is decomposing the matrix $g = \begin{bmatrix}1&&&\\y&1&&\\x&&1&-\bar{y}\\&&&1\end{bmatrix}$ in this way. There are five cases depending on the values of $x$ and $y$. For instance, if $x \in \OF, y \in \OF_L$ then $g$ already lies in $K^{G_2}$. On the other hand if $x \in \OF, y \not\in \OF_L$ then
  \begin{equation}\label{xinO-ynotinO}
   g = \begin{bmatrix}-y^{-1}&-1&&\\&-y&&\\&&-\bar{y}&\\&&1&-\bar{y}^{-1}\end{bmatrix} \begin{bmatrix}&1&&\\-1&-y^{-1}&&\\-x\bar{y}^{-1}&&-\bar{y}^{-1}&1\\-x&&-1&\end{bmatrix}.
  \end{equation}
Similar matrix identities (which we omit for reasons of brevity) exist in the three remaining cases
$$
 x \notin \OF,\;y \in \OF_L,\qquad
 x \notin \OF,\;y \not\in \OF_L,\;y x^{-1} \in \OF_L,\qquad
 x \notin \OF,\;y \not\in \OF_L,\;y x^{-1} \not\in \OF_L.
$$
We now have ten cases, five for $z\in\OF_L$ and five for $z\in L\setminus\OF_L$. In each case let $k$ denote the $K^{G_2}$ component of the argument of $W^\#$. Using the fact that $W^\#$ is supported on $P(F)\eta_0 K^H\Gamma(\P^n)$ gives the following conditions on $k$. The notation is such that $y = y_1 + \alpha y_2$ and $z = z_1 + \alpha z_2$ with $y_1, y_2, z_1, z_2 \in \OF$.
$$
 \begin{array}{cccccc}
  \text{case}&x&y&z&yx^{-1}&\text{condition for $k$ to be in the support of $W^\#$}\\\hline
   i)&\in\OF&\in\OF_L&\in\OF_L&&y_2 + (x+yz+\bar{y}\bar{z})
   \in \OF^\times\quad\mbox{or}\quad z_2 - z\bar{z} \in \OF^\times\\
   ii)&\in\OF&\notin\OF_L&\in\OF_L&&z_2 - z\bar{z}\in \OF^\times\\
   iii)&\notin\OF&\in\OF_L&\in\OF_L&&\text{always}\\
   iv)&\notin\OF&\notin\OF_L&\in\OF_L&\in\OF_L&\frac{y_2}x-\frac{y\bar{y}}{x^2}z_2
    +\big(\frac{x+yz}{x}\big)\big(\frac{x+\bar{y}\bar{z}}{x}\big) \in \OF^\times\\
   v)&\notin\OF&\notin\OF_L&\in\OF_L&\notin\OF_L
    &\frac 1{\alpha-\bar{\alpha}}\big(\frac x{\bar{y}}-\frac xy\big) - z_2
     + \big(\frac xy + z\big)\big(\frac x{\bar{y}} + \bar{z}\big)\in \OF^\times\\
   vi)&\in\OF&\in\OF_L&\notin\OF_L&&\text{always}\\
   vii)&\in\OF&\notin\OF_L&\notin\OF_L&&\text{always}\\
   viii)&\notin\OF&\in\OF_L&\notin\OF_L&&\text{never}\\
   ix)&\notin\OF&\notin\OF_L&\notin\OF_L&\in\OF_L
    &\frac{y\bar{y}}{x^2}+\frac{y_2}x\in\OF^\times\\
   x)&\notin\OF&\notin\OF_L&\notin\OF_L&\notin\OF_L&\text{always}
 \end{array}
$$
According to these cases, $K(s)$ is the sum of ten integrals $I_{i)},\ldots,I_{x)}$. By the support conditions, $I_{viii)}=0$. We split the first case up into $i)a$, the case where $z_2-z\bar z\in\OF^\times$, and $i)b$, the case where $z_2-z\bar z\in\p$ and $y_2 + (x+yz+\bar{y}\bar{z})\in\OF^\times$.
To evaluate the function $W^\#$ in $I_1$ and $I_2$, we will write the argument of $W^\#$ as $p\eta\kappa$ with $p\in P(F)$ and $\kappa\in K^H$. Only the $p$ part is important for the evaluation. Once the argument of $W^\#$ is written as $p\eta\kappa$, it is straightforward to perform an initial evaluation of the integrals. We list only the results.
\begin{align*}
   I_{i)a}&=\Big(\int\limits_{\substack{\OF_L\\z_2-z\bar z\in\OF^\times}}
    \chi_0(z_2-z\bar z)\,dz\Big)\:W^{(0)}(\mat{}{1}{1}{})\\
   I_{i)b}&=\int\limits_{\substack{\OF_L\\z_2-z\bar z\in\p}}
    \int\limits_{\OF^\times}W^{(0)}(\mat{x}{}{}{1}\mat{1}{}{z_2-z\bar z}{1})\,dx\,dz\\
   I_{ii)}&=\int\limits_{\substack{\OF_L\\z_2-z\bar z\in\OF^\times}}
    \int\limits_{L\setminus\OF_L}
    |y|_L^{-3(s+\frac12)}\chi\Big(\frac1{z_2-z\bar z}\Big)\Lambda(\bar y)
    W^{(0)}(\mat{y\bar y\:}{y_2+yz+\bar y\bar z}{}{1}\mat{}{-1}{1}{})\,dy\,dz\\
   I_{iii)}&=\int\limits_{\OF_L}\int\limits_{\OF_L}\int\limits_{F\setminus\OF}
    |x|^{-6(s+\frac12)}\chi(x^{-1})
    W^{(0)}(\mat{\frac{y_2+x+yz+\bar y\bar z}x}{}{}{1}
    \mat{1}{}{\frac{z_2-z\bar z}x}{1})\,dx\,dy\,dz\\
   I_{iv)}&=\int\limits_{\OF_L}\int\limits_{\substack{\OF_L\\u\in \OF^\times}}
    \int\limits_{\substack{F\setminus\OF\\yx\notin\OF_L}}
    |x|^{-6s-1}\chi(x^{-1})\\
    &\hspace{20ex}W^{(0)}(\mat{1}{y\bar yx}{}{1}
     \mat{y_2-y\bar{y}z_2+(1+yz)(1+\bar y\bar z)}{}{}{1}
     \mat{1}{}{\frac{z_2-z\bar z}x}{1})\,dx\,dy\,dz\\
   I_{v)}&=\int\limits_{\substack{\OF_L\\z_2-z\bar z\in\OF^\times}}
    \int\limits_{L\setminus\OF_L}\int\limits_{F\setminus\OF}
    |x|^{-6s-1}\chi(x^{-1})\,|y|_L^{-3(s+\frac12)}\Lambda(\bar y)\\
    &\hspace{20ex} W^{(0)}(\mat{1}{y\bar yx}{}{1}
    \mat{y_2-y\bar yz_2+(1+yz)(1+\bar y\bar z)}{}{}{1}
    \mat{1}{}{\frac{z_2-z\bar z}x}{1})\,dx\,dy\,dz\\
   I_{vi)}&=\int\limits_{L\setminus\OF_L}|z|_L^{-3s-\frac12}\Lambda(z)
    W^{(0)}(\mat{1}{}{}{\;z\bar z(1-\frac{z_2}{z\bar z})}\mat{}{1}{-1}{})\,dz\\
   I_{vii)}&=\int\limits_{L\setminus\OF_L}\int\limits_{L\setminus\OF_L}
    |y|_L^{-3(s+\frac12)}|z|_L^{-3s-\frac12}\Lambda(\bar yz)\\
    &\hspace{20ex}
    \psi^{-\mathbf{c}}\Big(-\frac yz-\frac{\bar y}{\bar z}+\frac{y_2}{z_2-z\bar z}\Big)
    W^{(0)}(\mat{y\bar y}{}{}{z\bar z}
    \mat{1}{}{}{\frac{z_2-z\bar z}{z\bar z}}\mat{}{1}{-1}{})\,dy\,dz\\
   I_{ix)}&=\int\limits_{L\setminus\OF_L}
    \int\limits_{\substack{\OF_L^\times\\y_2-y\bar y\in\OF^\times}}
    \int\limits_{F\setminus\OF}
    |z|_L^{-3s-\frac12}\Lambda(z)|x|^{-6s-1}\chi(-x^{-1})
    \psi^{-\mathbf{c}}\Big(\frac{x(y\bar y+z\bar z+\bar yz+y\bar z)}{z\bar z}\Big)\\
    &\hspace{20ex}
     W^{(0)}(\mat{1}{}{}{z\bar z}\mat{y_2-y\bar y+\frac{y\bar yz_2}{z\bar z}}{}{}{1}
     \mat{1}{}{-\frac{z_2-z\bar z}{xz\bar z}}{1})\,dx\,dy\,dz\\
   I_{x)}&=\int\limits_{L\setminus\OF_L}\int\limits_{L\setminus\OF_L}\int\limits_{F\setminus\OF}
    |y|_L^{-3(s+\frac12)}|z|_L^{-3s-\frac12}\Lambda(\bar yz)|x|^{-6s-1}\chi_0(x)\\
    &\hspace{5ex}W^{(0)}(\mat{1}{}{}{z\bar z}
    \mat{1}{x(y\bar y+z\bar z-\bar yz-y\bar z)}{}{1}
    \mat{y\bar y(1+\frac{y_2}{y\bar y}-\frac{z_2}{z\bar z})}{}{}{1}
    \mat{1}{}{\frac{z_2-z\bar z}{xz\bar z}}{1})\,dx\,dy\,dz.
\end{align*}

These integrals can be calculated further, using standard $p$-adic techniques and known properties of the $\GL_2$ Whittaker function $W^{(0)}$. We will omit the details of the calculation for reasons of brevity.

\vspace{2ex}

The calculations for the split case (when $L = F \oplus F$) are similar. In this case the explicit formula~\eqref{Ksformulaexpliciteq} gives us an integral over five $F$-variables (coming from the two $L$-variables and one $F$-variable). Also, note that in the split case, we have the isomorphism
\begin{align*}
\GU(2,2; F \oplus F)   &\cong \GL_4(F) \times \GL_1(F) \\
g := (g_1, g_2) & \rightarrow (g_1, \mu(g)).
\end{align*}
Using this, we can break up the integral~\eqref{Ksformulaexpliciteq} into several smaller integrals, which we evaluate in a manner similar to the inert case.
After all the integrals are computed and combined, one obtains the following result, which is true in the inert as well as the split case.

\begin{theorem}\label{inertintertwiningtheorem}
 Let $(\tau,V_\tau)$ be an irreducible, admissible, generic representation of $\GL_2(F)$. Assume that $L/F$ is either an unramified field extension or $L = F \oplus F$. Assume also that the conductor $\p^n$ of $\tau$ satisfies $n \ge 1$. Let the character $\chi_0$ of $L^\times$ be such that $\chi_0\big|_{F^\times}=\omega_\tau$ and $\chi_0((1+\P^n)\cap\OF_L^\times)=1$. Let $\Lambda$ be an unramified character of $L^\times$ such that $\Lambda\big|_{F^\times}=1$. Let the character $\chi$ of $L^\times$ be defined by (\ref{chi-lambda-char-condition}). Let $W^\#(\,\cdot\,,s)$ be the distinguished function in $I(s,\chi,\chi_0,\tau)$ from Corollary \ref{distinguishedvectornonarchtheorem}, normalized such that $W^\#(\eta_0,s)=1$. Then the function $K(s)$ defined by (\ref{Ksdefeq}) is given by
 $$
  K(s) = \chi_{L/F}(\varpi)^n  \omega_\tau(\mathbf{c}^2/\mathbf{d}) \frac{\varepsilon(3s+1,\tilde\tau,\psi^{-\mathbf{c}})^2}
 {\varepsilon(6s,\omega_\tau^{-1},\psi^{-\mathbf{c}})} \frac{L(6s,\chi\big|_{F^\times})
    L(3s,\tau\times\mathcal{AI}(\Lambda)\times\chi\big|_{F^\times})}
    {L(1-6s,\chi^{-1}\big|_{F^\times})
    L(3s+1,\tau\times\mathcal{AI}(\Lambda)\times\chi\big|_{F^\times})}.
 $$
 \end{theorem}

\subsection{Intertwining operator: archimedean case}\label{inter-arch}
In this section let $F=\R$. We use the notation and setup from Sect.\ \ref{distvecarchsec}. Hence, $(\tau,V_\tau)$ is an irreducible, admissible, generic representation of $\GL_2(\R)$, and $l_2\in\Z$ has the same parity as the weights of $\tau$. The character $\chi_0$ of $\C^\times$ is such that $\chi_0\big|_{\R^\times}=\omega_\tau$ and $\chi_0(\zeta)=\zeta^{l_2}$ for $\zeta\in\C^\times,\:|\zeta|=1$, and $\chi(\zeta)=\chi_0(\bar\zeta)^{-1}$. We realize $\tau$ as a subrepresentation of some $\beta_1\times\beta_2$, and the quantities $p,q$ are defined by (\ref{t1t2pqeq}). Let $W^\#$ be the distinguished function in $I_W(s,\chi,\chi_0,\tau)$ defined in table (\ref{ABCtableeq}). In this section we calculate the function $K(s)$ defined by (\ref{Ksdefeq}). It is easily checked that the operator $M(s)$, defined by the same integral formula (\ref{locintdefeq1}), defines an intertwining map from $I_\Phi(s,\chi,\chi_0,\tau)$ to $I_\Phi(-s,\bar\chi^{-1},\chi\bar\chi\chi_0,\chi\tau)$. In fact, there is a commutative diagram
$$
 \begin{CD}
  I_\Phi(s,\chi,\chi_0,\tau)@>{M(s)}>>
   I_\Phi(-s,\bar\chi^{-1},\chi\bar\chi\chi_0,\chi\tau)\\
   @VVV @VVV\\
  I_W(s,\chi,\chi_0,\tau)@>>{M(s)}>I_W(-s,\bar\chi^{-1},\chi\bar\chi\chi_0,\chi\tau)
 \end{CD}
$$
in which the vertical maps are the intertwining operators $\Phi\mapsto W_\Phi$ given, in the region of convergence, by formula (\ref{WPhieq}). The commutativity follows from a straightforward calculation in the region of convergence, and by analytic continuation outside this region. It follows that the function $K(s)$, instead of (\ref{Ksdefeq}), can also be determined from the equation
\begin{equation}\label{KsPhieq}
 M(s)\Phi^\#(\,\cdot\,,s,\chi,\chi_0,\tau)
  =K(s)\,\Phi^\#(\,\cdot\,,-s,\bar\chi^{-1},\chi\bar\chi\chi_0,\chi\tau).
\end{equation}
Here, $\Phi^\#\in I_\Phi(s,\chi,\chi_0,\tau)$ is defined in table (\ref{ABCtableeq}). At this point, we do not yet know in all cases that a function $K(s)$ with the property (\ref{KsPhieq}) actually exists. We \emph{do} know that it exists in the Cases A and B defined in (\ref{ABCdefeq}); since $M(s)$ preserves right transformation properties, this follows from the uniqueness statement in Theorem \ref{distinguishedvectorarchtheorem} ii) and iii). In view of the normalization (\ref{W0l1l2distvecpropeq2}), we have the formula
\begin{equation}\label{Ksformula3eq}
 K(s)=\int\limits_{N(\R)}\Phi^\#(w_1n\eta_0,s,\chi,\chi_0,\tau)\,dn
\end{equation}
in Cases A and B. In Case C, part iv) of Theorem \ref{distinguishedvectorarchtheorem} assures that the left side of (\ref{KsPhieq}) is a linear combination of $\Phi^\#_{m,l,l_2,l+1}$ and $\Phi^\#_{m,l,l_2,l-1}$. It would be more precise to write these functions as
$$
 \Phi^\#_{m,l,l_2,l\pm1}(\,\cdot\,,-s,\bar\chi^{-1},\chi\bar\chi\chi_0,\chi\tau)
 \qquad\text{or}\qquad
 \Phi^\#_{m,l,l_2,l\pm1}(\,\cdot\,,-s,\beta_2^{-1}\times\beta_1^{-1})
$$
since they are defined with respect to the data $(-s,\bar\chi^{-1},\chi\bar\chi\chi_0,\chi\tau)$, and $\chi\tau$ is a subrepresentation of $\beta_2^{-1}\times\beta_1^{-1}$. The calculation will show that this linear combination is precisely a function $K(s)$ times the distinguished vector $\Phi^\#(\,\cdot\,,-s,\bar\chi^{-1},\chi\bar\chi\chi_0,\chi\tau)$ for the data $(-s,\bar\chi^{-1},\chi\bar\chi\chi_0,\chi\tau)$. This will establish the existence of $K(s)$ with the property (\ref{KsPhieq}) in all cases.

\vspace{3ex}
Concerning the measure on $N(\R)$, similar remarks as in the $p$-adic case apply. As a measure space, $N(\R)\cong\R\times\C\times\C$. The measure on $\R$ is the usual Lebesgue measure, but the measure on $\C$ is \emph{twice} the usual Lebesgue measure; see Sect.\ 2.2 of \cite{TT}.

\vspace{3ex}
{\bf Remark:} The reason we are calculating $K(s)$ from equation (\ref{KsPhieq}) and not from equation (\ref{Ksdefeq}) is that the relevant archimedean integrals are much easier to handle in the induced model than in the Whittaker model. The price one has to pay for this procedure are the non-explicit constants $\kappa_{l_1,p,q}$ defined in (\ref{PhiWkappaeq}). They will not appear any further in this section, but later in Sect.\ \ref{localzetasec} when we calculate local zeta integrals; see Corollary \ref{archlocalzetatheoremcor}. In our application to the functional equation in Sect.\ \ref{functleqsec}, the unknown constants $\kappa_{l_1,p,q}$ will cancel out with the constants $a_{l_1,p,q}$ defined in (\ref{aplusdefeq}), via the identity given in Lemma \ref{kappaalemma}.

\begin{theorem}\label{archintertwiningtheorem}
 Let $(\tau,V_\tau)$ be a  generic, irreducible, admissible representation of $\GL_2(\R)$ with central character $\omega_{\tau}$. We assume that $\tau$ is isomorphic to a subrepresentation of $\beta_1\times\beta_2$ with characters $\beta_1,\beta_2$ of $\R^\times$. Let the complex numbers $p$ and $q$ be as defined in (\ref{t1t2pqeq}).
 Let $l$ be a fixed positive integer. Let $l_2=-l_1$ in Case A, $l_2=-l$ in Case B, and $l_2=1-l$ in Case C. Let $\chi_0$ be the character of $\C^\times$ such that $ \chi_0\big|_{\R^\times}=\omega_\tau$ and $\chi_0(\zeta)=\zeta^{l_2}$ for $\zeta\in\C^\times,\:|\zeta|=1$. Let $\chi$ be the character of $\C^\times$ given by (\ref{chi-lambda-char-condition-arch}). Let $W^\#\in I_W(s,\chi,\chi_0,\tau)$ be the distinguished function defined in table (\ref{ABCtableeq}). Then the identity (\ref{Ksdefeq}) holds with the function $K(s)$ given as follows.
 \begin{enumerate}
  \item In Case A,
\begin{equation}\label{archinterresult1}
    K(s)=4\pi^{5/2}\,i^{2l-l_1}\frac{\Gamma(3s-\frac q2+\frac{1}{2})\Gamma(3s-\frac{q}{2})}
     {(3s-\frac{q}2+\frac{l_1}2-\frac{1}{2})^2\Gamma(3s-\frac q2+\frac{l_1}2+\frac32-l)
      \Gamma(3s-\frac q2+l-\frac{l_1}2-\frac12)}.
\end{equation}
  \item In Case B,
\begin{equation}\label{archinterresult2}
    K(s)=4\pi^{5/2}\,i^l\frac{\Gamma(3s-\frac{q}{2} + \frac{1}{2})\Gamma(3s-\frac{q}{2})}
    {(3s-\frac q2+\frac p2)(3s-\frac q2-\frac p2)\Gamma(3s-\frac q2-\frac l2+\frac12)
    \Gamma(3s-\frac q2+\frac l2+\frac12)}.
\end{equation}
  \item In Case C,
   \begin{align}\label{archinterresult3}
    K(s)&=-4\pi^{5/2}i^{l+1}\frac{(3s-\frac q2-\frac l2)(3s-\frac q2-1-\frac{p}2)}
     {(3s-\frac q2+1+\frac{p}{2})(3s-\frac q2+\frac{p}2)(3s-\frac q2-\frac{p}2)}\nonumber\\
    &\hspace{20ex}\times \frac{\Gamma(3s-\frac{q}{2}+\frac{1}{2})\Gamma(3s-\frac{q}{2})}
     {\Gamma(3s-\frac{q}{2}+1-\frac{l}{2})\Gamma(3s-\frac{q}{2}+1+\frac{l}{2})}.
   \end{align}
 \end{enumerate}
\end{theorem}
\begin{proof} i) In Case A, by (\ref{ABCtableeq}), we have $\Phi^\#=\Phi^\#_{m,l,l_2,l_1}$, where $m=l_1-l$ and $l_1$ is the lowest weight of the representation $\tau$. The function $\Phi^\#_{m,l,l_2,l_1}$ is given in Proposition \ref{W0l1l2distvecprop}. By (\ref{Ksformula3eq}), we have to calculate
\begin{equation}\label{Ksformula3Aeq}
 K(s)=\int\limits_{N(\R)}\Phi^\#_{m,l,-l_1,l_1}(w_1n\eta_0,s,\chi,\chi_0,\tau)\,dn.
\end{equation}
We abbreviate
$$
 u=\sqrt{1+x^2},\qquad v=\sqrt{1+y\bar y},\qquad w=\sqrt{1+z\bar z},
$$
and
$$
 r_1=\begin{bmatrix}w^{-1}&zw^{-1}\\-\bar zw^{-1}&w^{-1}\\&&w^{-1}&zw^{-1}\\
  &&-\bar zw^{-1}&w^{-1}\end{bmatrix},\qquad
 r_2=\begin{bmatrix}u^{-1}&&xu^{-1}\\&1\\-xu^{-1}&&u^{-1}\\&&&1\end{bmatrix},
$$
$$
 r_3=\begin{bmatrix}v^{-1}&&&yv^{-1}\\&v^{-1}&\bar yv^{-1}\\&-yv^{-1}&v^{-1}\\
 -\bar yv^{-1}&&&v^{-1}\end{bmatrix}.
$$
The elements $r_1,r_2,r_3$ lie in $K^{G_2}_\infty$. Starting from (\ref{Ksformula3Aeq}), it is not difficult to show that
\begin{equation}\label{Ksformula4eq}
 K(s)=\int\limits_\C\int\limits_\C\int\limits_\R u
  (uvw)^{q-6s-2}(vw^{-1})^p\,\Phi^\#_{m,l,-l_1,l_1}
  (w_1r_3r_2r_1\eta_0)\,dx\,dy\,dz.
\end{equation}
A calculation verifies that, with $k=w_1r_3r_2r_1\eta_0$,


$$
 \det(J(k,i_2))=-i\frac{1-ix}u,
 \qquad
 \det(J(\,^tk,i_2))=i\frac{1+ix}u,
$$
and
$$
 \hat b(k)=(1-z\bar z)\Big(\frac{v^2 - ix(1-y \bar y)}{uv^2w^2}\Big)+2i\frac{(y\bar z+\bar yz)}{v^2w^2}.
$$
Hence
\begin{align*}
 \Phi^\#_{m,l,-l_1,l_1}(k)
   &=(-i)^m\,\hat b(k)^{l_1-l}\det(J(\,^tk,i_2))^{l-l_1}\det(J(k,i_2))^{-l_1}\\
 &=i^{2l-l_1}\bigg((1-z\bar z)\Big(\frac{v^2 - ix(1-y \bar y)}{uv^2w^2}\Big)
   +2i\frac{(y\bar z+\bar yz)}{v^2w^2}\bigg)^{l_1-l}\Big(\frac{1+ix}u\Big)^l,
\end{align*}
so that
\begin{align*}
 K(s)&=i^{2l-l_1}\int\limits_\C\int\limits_\C\int\limits_\R u
  (uvw)^{q-6s-2}(vw^{-1})^p\\
 &\hspace{5ex}\bigg((1-z\bar z)\Big(\frac{v^2 - ix(1-y \bar y)}{uv^2w^2}\Big)
   +2i\frac{(y\bar z+\bar yz)}{v^2w^2}\bigg)^{l_1-l}\Big(\frac{1+ix}u\Big)^l\,dx\,dy\,dz.
\end{align*}
We now introduce polar coordinates for $y$ and $z$. More precisely, put $y = \sqrt{f}e^{ i \theta_1}$, $z=\sqrt{g}e^{ i \theta_2}$, and let $\theta = \theta_1 - \theta_2$.
Also, put $s_0 = 3s + \frac{1}{2} + \frac{l_1-l}{2} - \frac{q}{2}$. With these substitutions, and using the fact that $p=l_1-1$, the intertwining integral becomes
\begin{align*}
 K(s)&=i^{2l-l_1}2\pi\int\limits_0^{2\pi}\int\limits_\R
 \int\limits_0^\infty \int\limits_0^\infty(1+x^2)^{-s_0}
   (1 + f)^{-s_0 -1 + \frac{l}{2}}(1 + g)^{-s_0+\frac l2-l_1}
   \Big(\frac{1+ix}u\Big)^l\\
 &\hspace{5ex}\Big( (1-g) (2 - (1 - f)(1+ix)) + 4iu\sqrt{fg}\,\cos\theta \Big)^{l_1-l}
   \,df\,dg\,dx\,d\theta.
\end{align*}
Note here that the measure on $\C$ is twice the usual Lebesgue measure. By Lemma \ref{t:archintcalcs} further below, we get the result.

\vspace{2ex}
ii) Next we evaluate the intertwining integral in the case where $\tau$ contains the weight $l$. As in the previous case, $K(s)$ is given by formula (\ref{Ksformula3eq}). The same calculation that led to
(\ref{Ksformula4eq}) now shows that
\begin{equation}\label{Ksformula5eq}
 K(s)=\int\limits_\C\int\limits_\C\int\limits_\R u
  (uvw)^{q-6s-2}(vw^{-1})^p\,\Phi^\#_{m,l,-l,l}(w_1r_3r_2r_1\eta_0)\,dx\,dy\,dz.
\end{equation}
This time $\Phi^\#_{m,l,-l,l}=i^l\big(\frac{1+ix}u\big)^l$, so that
$$
 K(s)=i^l\int\limits_\C\int\limits_\C\int\limits_\R u
  (uvw)^{q-6s-2}(vw^{-1})^p\Big(\frac{1+ix}u\Big)^l\,dx\,dy\,dz.
$$
This integral can be calculated as before by using polar coordinates. The result follows.

\vspace{2ex}
iii) This case is the most complicated one, since we do not yet know that a function $K(s)$ with the property (\ref{KsPhieq}) exists. We do know, however, by part iv) of Theorem \ref{distinguishedvectorarchtheorem}, that there exist functions $K_1(s)$ and $K_2(s)$ such that
\begin{align}\label{KsPhiCeq}
 M(s)\Phi^\#(\,\cdot\,,s,\chi,\chi_0,\tau)
  &=K_1(s)\,\Phi^\#_{1,l,l_2,l+1}
   (\,\cdot\,,-s,\bar\chi^{-1},\chi\bar\chi\chi_0,\chi\tau)\nonumber\\
  &\qquad+K_2(s)\,\Phi^\#_{1,l,l_2,l-1}
   (\,\cdot\,,-s,\bar\chi^{-1},\chi\bar\chi\chi_0,\chi\tau).
\end{align}
The calculation of $K_1(s)$ and $K_2(s)$ is in the same spirit as in Cases A and B, and we omit the details. Eventually it turns out that (\ref{KsPhieq}) holds with
$K(s)$ as in (\ref{archinterresult3}).
\end{proof}

We would like to thank Paul-Olivier Dehaye for his help with the proof of the following lemma, which was used in the above calculations.
\begin{lemma}\label{t:archintcalcs}
 For non-negative integers $l$ and $t$, and for all $s\in\C$ with $\Re(s)$ large enough,
 \begin{align*}
  &\int\limits_0^{2\pi}\int\limits_\R\int\limits_0^\infty \int\limits_0^\infty
   (1+x^2)^{-s} (1 + f)^{-s-1 + \frac{l}{2}}(1 + g)^{-s-\frac{l}{2} - t}
   \Big(\frac{1+ix}u\Big)^l\\
  &\hspace{15ex}\Big( (1-g) (2 - (1 - f)(1+ix)) + 4iu\sqrt{fg}\,\cos\theta \Big)^{t}
   \,df\,dg\,dx \,d\theta\\
  &\qquad=\frac{2 \pi^{3/2} \Gamma(s-\frac{t}{2})\Gamma(s-\frac{t}{2}-\frac{1}{2})}
   {(s+\frac{l}{2}-1)^2 \Gamma(s-\frac{l}{2}+1) \Gamma(s+\frac{l}{2}-t-1)}.
 \end{align*}
\end{lemma}
\begin{proof}
Let $LHS$ denote the quantity on the left hand side of the asserted formula.
We start off by completely expanding $\big( (1-g) (2 - (1 - f)(1+ix)) + 4iu\sqrt{fg}\,\cos\theta \big)^{t}$ using the bimomial theorem. Then, using (6.16), (6.17) of \cite{Bu} and the following well-known formulas,
$$
 \int\limits_0^{2\pi}\cos(\theta)^k\,d\theta= \begin{cases}
  0 & \text{ if } k \text{ is odd, }\\
  \displaystyle2 \sqrt{\pi}\,\frac{\Gamma(\frac{k+1}{2})}{\Gamma(\frac{k+2}{2})}
   & \text{ if } k \text{ is even,} \end{cases}
$$
$$
 \int\limits_0^\infty r^{t_1} (1+r)^{-t_2} \,dr = \frac{\Gamma(1+t_1)\Gamma(-1 + t_2 -t_1)}{\Gamma(t_2)},
$$
we arrive at (using the multinomial symbol)
\begin{align*}
 LHS &= \sum_{k,j,r,v}\bigg((-1)^{k+r+v+j}2^{2k +t -j +1} \binom{t}{2k, j, t-2k-j}   \binom{t-2k}{v} \binom{j}{v} \frac{\Gamma(k + \frac12)\Gamma(1+k+r) }{\Gamma(k+1)\Gamma(s+\frac{l}{2} + t)}\pi \\
  & \quad\frac{\Gamma(-1 +s +t-k-r+\frac{l}{2})  \Gamma(s-v-k-\frac{l}{2}) \Gamma(1+k+v) \Gamma(s-k-\frac{j}{2})\Gamma(s-k - \frac{j}{2} -\frac12)}{\Gamma(s - \frac{l}{2} - k -j) \Gamma(s + \frac{l}{2} -k)\Gamma(s-\frac{l}{2} + 1)}\bigg),
\end{align*}
where the sum is taken over non-negative integers $k$, $j$, $r$, $v$ satisfying $2k + j \le t$, $r+2k \le t$ and $v \le j$. Next, using well-known summation formulas for the gamma functions and algebraic manipulations (we omit the details of this step, which were performed with the aid of Mathematica), it turns out that the above expression simplifies significantly. As a result, the lemma reduces to proving a certain algebraic identity. Let $x^{(n)} = x(x+1)\ldots (x+n-1)$ denote the Pochhammer symbol. Then the identity we are reduced to proving is
$$
 \sum_{k=0}^T \sum_{\substack{v\ge 0,\:n \ge 0 \\ v+n \le 4T+1}}4^k \frac{(-1)^{k+v}}{n!v!(2T-2k)!} \cdot \frac{(x+k-n)^{(n)} (x-k+1)^{(v)}}{(x-2T+v)^{(k+1)} (x+2T-k-n)^{(k+1)}} = - \frac{1}{x^2 (2T)!},
$$
where $T$ is any non-negative integer, and $x$ is an indeterminate. To prove this identity, observe that each summand above can be written using the partial fraction decomposition as a sum of rational functions, where each numerator is a rational number and the denominators are  terms of the form $(x-a)^b$ with $b$ equal to 1 or 2, and $-3T-1 \le a \le 3T+1$. So to prove the identity, it is enough to show that the sum of the numerators coincide on both sides for each such denominator. This is straightforward combinatorics, and we omit the details.
\end{proof}

\section{Global $L$-functions for $\GSp_4\times\GL_2$}
In this section, we will use the integral defined by Furusawa in \cite{Fu} to obtain an integral representation of the $L$-function $L(s,\pi\times\tau)$, where $\pi$ is a cuspidal, automorphic representation of $\GSp_4(\A)$ of the type corresponding to full level Siegel cusp forms, and where $\tau$ is an arbitrary cuspidal, automorphic representation of $\GL_2(\A)$. We will use this to obtain the functional equation of the $L$-function, with some restriction on the $\GL_2$ representation. We will first do the non-archimedean calculation, followed by the archimedean calculation and put it all together to get the global result.
\subsection{Bessel models for $\GSp_4$}\label{besselsec}
Let $F$ be an algebraic number field and $\A_F$ its ring of adeles. We fix three elements $\mathbf{a},\mathbf{b},\mathbf{c}\in F$ such that $\mathbf{d}=\mathbf{b}^2-4\mathbf{a}\mathbf{c}$ is a non-square in $F^\times$. Then $L=F(\sqrt{\mathbf{d}})$ is a quadratic field extension of $F$. Let
$$
 S=\mat{\mathbf{a}}{\frac{\mathbf{b}}2}{\frac{\mathbf{b}}2}{\mathbf{c}},\qquad
 \xi=\mat{\frac{\mathbf{b}}2}{\mathbf{c}}{-\mathbf{a}}{\frac{-\mathbf{b}}2}.
$$
Then $F(\xi)=F+F\xi$ is a two-dimensional $F$-algebra isomorphic to $L$, an isomorphism being given by $x+y\xi\mapsto x+y\frac{\sqrt{\mathbf{d}}}2$. The determinant map on $F(\xi)$ corresponds to the norm map on $L$. Let
\begin{equation}\label{Tdefeq}
 T=\{g\in\GL_2\;|\;^tgSg=\det(g)S\}.
\end{equation}
This is an algebraic $F$-group with $T(F)=F(\xi)^\times\cong L^\times$ and $T(\A_F)\cong\A_L^\times$. We consider $T$ a subgroup of $H=\GSp_4$ via
$$
 T\ni g\longmapsto\mat{g}{}{}{\det(g)\,^tg^{-1}}\in H.
$$
Let
$$
 U=\{\mat{1_2}{X}{}{1_2}\in H\;|\;^tX=X\}
$$
and $R=TU$. We call $R$ the \emph{Bessel subgroup} of $H$ (with respect to the given data $\mathbf{a},\mathbf{b},\mathbf{c}$). Let $\psi$ be a non-trivial character $F\backslash\A_F\rightarrow\C^\times$, chosen once and for all. Let $\theta:\:U(\A_F)\rightarrow\C^\times$ be the character given by
\begin{equation}\label{thetadefeq}
 \theta(\mat{1}{X}{}{1})=\psi(\tr(SX)).
\end{equation}
We have $\theta(t^{-1}ut)=\theta(u)$ for all $u\in U(\A_F)$ and $t\in T(\A_F)$. Hence, if $\Lambda$ is any character of $T(\A_F)\cong\A_L^\times$, then the map $tu\mapsto\Lambda(t)\theta(u)$ defines a character of $R(\A_F)$. We denote this character by $\Lambda\otimes\theta$.

\vspace{3ex}
Analogous definitions can be made over any local field $F$.
In this case, let $\pi$ be an irreducible, admissible representation of $H(F)$. Let $\Lambda$ be a character of $T(F)\cong L^\times$ such that the restriction of $\Lambda$ to $F^\times$ coincides with the central character of $\pi$. Let $\Lambda\otimes\theta$ be the character of $R(F)$ defined above. We say that $\pi$ has a \emph{Bessel model} of type $(S,\Lambda,\psi)$ if $\pi$ is isomorphic to a space of functions $B:\:H(F)\rightarrow\C$ with the transformation property
$$
 B(tuh)=\Lambda(t)\theta(u)B(h)\qquad\text{for all }t\in T(F),\;u\in U(F),\;h\in H(F),
$$
with the action of $H(F)$ on this space given by right translation. Such a model, if it exists, is known to be unique; we denote it by $\mathcal{B}_{S,\Lambda,\psi}(\pi)$.

\vspace{3ex}
Now let $F$ be global, and let $\pi=\otimes\pi_v$ be a cuspidal, automorphic representation of $H(\A_F)$. Let $V_\pi$ be the space of automorphic forms realizing $\pi$. Assume that a Hecke character $\Lambda$ as above is chosen such that the restriction of $\Lambda$ to $\A_F^\times$ coincides with $\omega_\pi$, the central character of $\pi$. For each $\phi\in V_\pi$ consider the corresponding Bessel function
\begin{equation}\label{Bphieq}
 B_\phi(g)=\int\limits_{Z_H(\A_F)R(F)\backslash R(\A_F)}
 (\Lambda\otimes\theta)(r)^{-1}\phi(rg)\,dr,
\end{equation}
where $Z_H$ is the center of $H$. If one of these integrals is non-zero, then all
are non-zero, and we obtain a model $\mathcal{B}_{S,\Lambda,\psi}(\pi)$
of $\pi$ consisting of functions on $H(\A_F)$ with the obvious transformation property on the left with respect to $R(\A_F)$. In this case, we say that $\pi$ has a \emph{global Bessel model} of type $(S,\Lambda,\psi)$. It implies that the \emph{local} Bessel model $\mathcal{B}_{S,\Lambda_v,\psi_v}(\pi_v)$ exists for every place $v$ of $F$. In fact, there is a canonical isomorphism
$$
 \bigotimes\limits_v\mathcal{B}_{S,\Lambda_v,\psi_v}(\pi_v)\cong
 \mathcal{B}_{S,\Lambda,\psi}(\pi).
$$
If $(B_v)_v$ is a collection of local Bessel functions $B_v\in
\mathcal{B}_{S,\Lambda_v,\psi_v}(\pi_v)$ such that $B_v\big|_{H(\OF_v)}=1$
for almost all $v$, then this isomorphism is such that $\otimes_vB_v$ corresponds
to the global function
\begin{equation}\label{locglobBreleq}
 B(g)=\prod_vB_v(g_v),\qquad g=(g_v)_v\in H(\A_F).
\end{equation}
\subsubsection*{Explicit formulas: the spherical Bessel function}
Explicit formulas for local Bessel functions are only known in a few cases. One of these is the $p$-adic unramified case, which we review next. Hence, let $F$ be a non-archimedean local field of characteristic zero. Let the character $\psi$ of $F$ have conductor $\OF$, the ring of integers. Let $(\pi,V_\pi)$ be an unramified, irreducible, admissible representation of $H(F)$. Let $\Lambda$ be an unramified character of $T(F)\cong L^\times$. We assume that $V_\pi=\mathcal{B}_{S,\Lambda,\psi}(\pi)$ is the Bessel model with respect to the character $\Lambda\otimes\theta$ of $R(F)$. Let $B\in V_\pi$ be a spherical vector. By \cite{Su}, Proposition 2-5, we have $B(1)\neq0$. For $l,m\in\Z$ let
\begin{equation}\label{hlmdefeq}
 h(l,m)=\begin{bmatrix}\varpi^{2m+l}\\&\varpi^{m+l}\\&&1\\&&&\varpi^m\end{bmatrix}.
\end{equation}
Then, as in (3.4.2) of \cite{Fu},
\begin{equation}\label{RFKHrepresentativeseq}
 H(F)=\bigsqcup_{l\in\Z}\bigsqcup_{m\geq0}R(F)h(l,m)K^H,\qquad
 K^H=H(\OF).
\end{equation}
By Lemma (3.4.4) of \cite{Fu} we have $B(h(l,m))=0$ for $l<0$, so that $B$ is
determined by the values $B(h(l,m))$ for $l,m\geq0$. In \cite{Su}, 2-4, Sugano has given a formula for $B(h(l,m))$ in terms of a generating function. The full formula is required only in the case where the $\GL_2$ representation $\tau$ is unramified; this case has been treated in \cite{Fu}. For other cases
we only require the values $B(h(l,0))$, which are given by
\begin{equation}\label{suganox0eq}
 \sum_{l\geq0}B(h(l,0))y^l= \frac{H(y)}{Q(y)},
\end{equation}
where
\begin{align}\label{suganoQeq}
 Q(y)&=\prod_{i=1}^4\big(1-\gamma^{(i)}(\varpi)q^{-3/2}y\big)
\end{align}
and
\begin{equation}\label{suganoHeq}
 H(y) = \left\{\begin{array}{ll}
    1-q^{-4}\Lambda(\varpi)y^2& \hbox{ if }
     \big(\frac L\p\big)=-1,\\[1ex]
    1-q^{-2}\Lambda(\varpi_L)y& \hbox{ if }
     \big(\frac L\p\big)= 0,\\[1ex]
    1-q^{-2}\big(\Lambda(\varpi_L)+\Lambda(\varpi\varpi_L^{-1})\big)y
     +q^{-4}\Lambda(\varpi)y^2& \hbox{ if } \big(\frac L\p\big)= 1.
 \end{array}\right.
\end{equation}
The $\gamma^{(i)}$ are the Satake parameters of $\pi$, as in Sect.\ (3.6) of \cite{Fu}.
\subsubsection*{Explicit formulas: the highest weight case}
Another situation where an explicit formula for a Bessel function is known is the archimedean lowest weight case. Hence, let $F=\R$. Let $l$ be an integer such that $l\geq2$. Let $\pi$ be the discrete series representation (or limit of such if $l=2$) of $\PGSp_4(\R)$ with minimal $K$-type $(l,l)$; here, we write elements of the weight lattice as pairs of integers, precisely as in \cite{PS}, Sect.\ 2.1. Such representations $\pi$ appear as the archimedean components of the automorphic representations of $H(\A)$ attached to (scalar valued) Siegel modular forms of weight $l$. Recall that $S$ is a positive definite matrix. Let the function $B:\:H(\R)\rightarrow\C$ be defined by
\begin{equation}\label{archBesselformula2eq}\renewcommand{\arraystretch}{1.2}
 B(h) := \left\{\begin{array}{ll}
 \mu_2(h)^l\:\overline{\det(J(h, i_2))^{-l}}\,e^{-2\pi i\,
 {\rm tr}(S\overline{h\langle i_2\rangle})}& \hbox{ if } h \in H^+(\R),\\
 0& \hbox{ if } h \notin H^+(\R),\\
\end{array}\right.
\end{equation}
where $i_2=\mat{i}{}{}{i}$. One can check that $B$ satisfies the
Bessel transformation property with the character $\Lambda \otimes
\theta$ of $R(\R)$, where $\Lambda$ is trivial. Also
\begin{equation}\label{realBproperty2eq}
 B(hk)=\det(J(k,i_2))^lB(h)\qquad\text{for }
 h\in H(\R),\;k\in K^H_\infty.
\end{equation}
In fact, by the considerations in \cite{Su} 1-3, or by \cite{PS} Theorem 3.4, $B$
is the highest weight vector (weight $(-l,-l)$) in $\mathcal{B}_{S,\Lambda,\psi}(\pi)$. Note that the function $B$ is determined by its values on a set of representatives for $R(\R)\backslash H(\R)/K^H_\infty$. Such a set can be obtained as follows. Let $T^1(\R)=T(\R)\cap\SL(2,\R)$. Then $T(\R)=T^1(\R)\cdot\{\mat{\zeta}{}{}{\zeta}\;|\;\zeta>0\}$. As in \cite{Fu}, p.\ 211, let $t_0\in\GL_2(\R)^+$ be such that
$T^1(\R)=t_0\SO(2)t_0^{-1}$. (We will make a specific choice of $t_0$ when we choose the matrix $S$ below.) It is not hard to see that
\begin{equation}\label{HRBesseldecompeq}
 H(\R)=R(\R)\cdot\big\{\begin{bmatrix}\lambda t_0\mat{\zeta}{}{}{\zeta^{-1}}&\\
  &^tt_0^{-1}\mat{\zeta^{-1}}{}{}{\zeta}\end{bmatrix}\;|\;\lambda\in\R^\times,\,\zeta\geq1\big\}
  \cdot K^H_\infty.
\end{equation}
One can check that the double cosets in (\ref{HRBesseldecompeq}) are pairwise disjoint.
\subsection{Local zeta integrals}\label{localzetasec}
Let $F$ be a non-archimedean local field of characteristic zero, or $F=\R$. Let $\mathbf{a},\mathbf{b},\mathbf{c}\in F$ and $L,\alpha,\eta$ be according to our conventions; see (\ref{Ldefeq}), (\ref{alphadefeq}), (\ref{etadefeq}). Let $\tau,\chi_0,\chi$ be as in Corollary \ref{distinguishedvectornonarchtheorem} (non-archimedean case) resp.\ Corollary \ref{distinguishedvectorarchtheorem} (archimedean case). Let $W^\#(\,\cdot\,,s)$ be the unique vector in $I(s,\chi,\chi_0,\tau)$ exhibited in these corollaries. The calculation in the proof of Theorem (2.4) of \cite{Fu} shows that
\begin{equation}\label{Wetatransformationpropertyeq}
 W^\#(\eta tuh,s)=\Lambda(t)^{-1}\theta(u)^{-1}W^\#(\eta h,s)\qquad
 \text{for all }t\in T(F),\:u\in U(F),\;h\in H(F).
\end{equation}
Here, $\Lambda$ is an unramified character of $L^\times$ in the non-archimedean case, and $\Lambda=1$ in the archimedean case; we always have $\chi(\zeta) = \Lambda(\bar{\zeta})^{-1} \chi_0(\bar{\zeta})^{-1}$. Let $\pi$ be an irreducible, admissible representation of $H(F)$ which has a Bessel model of type $(S,\Lambda,\psi)$. Then, for any $B\in\mathcal{B}_{S,\Lambda,\psi}(\pi)$, equation (\ref{Wetatransformationpropertyeq}) shows that the integral
\begin{equation}\label{localZseq}
 Z(s,W^\#,B)=\int\limits_{R(F)\backslash H(F)}W^\#(\eta h,s)B(h)\,dh
\end{equation}
makes sense. We shall now explicitly calculate these integrals in the case of $B$ being the spherical vector in an unramified $p$-adic representation $\pi$, and $B$ being the highest weight vector in an archimedean (limit of) discrete series representation with scalar minimal $K$-type.
\subsubsection*{The non-archimedean case}
Assume that $F$ is non-archimedean. Recall the explicit formula for the distinguished function $W^\#(\,\cdot\,,s)$ given in Corollary \ref{distinguishedvectornonarchtheorem}.
It involves $W^{(0)}$, the normalized local newform in the Whittaker model of $\tau$ with respect to the character $\psi^{-\mathbf{c}}(x)=\psi(-\mathbf{c}x)$. Since this character has conductor $\OF$, the values $W^{(0)}(\mat{\varpi^l}{}{}{1})$ are zero for negative $l$. For non-negative $l$, one can use formulas for the local newform with respect to the congruence subgroup $\GL_2(\OF)\cap\mat{\OF}{\OF}{\p^n}{1+\p^n}$ (given, amongst other places, in \cite{Sc1}), together with the local functional equation, to obtain the following.
\begin{equation}\label{nonarchnewformtableeq}\renewcommand{\arraystretch}{1.8}
 \begin{array}{cc}
 \tau &\rule[-3ex]{0ex}{7.4ex}
   W^{(0)}({\renewcommand{\arraystretch}{1.0}\mat{\varpi^l}{}{}{1}})\;\;(l\geq0)\\ \hline
 \beta_1\times \beta_2 \mbox{ with } \beta_1, \beta_2 \mbox{ unramified, } \beta_1 \beta_2^{-1}
  \neq|\,|^{\pm 1} &q^{-l/2}
  \sum_{k=0}^l\beta_1(\varpi)^k\beta_2(\varpi)^{l-k}\\
 \beta_1\times \beta_2 \mbox{ with } \beta_1 \mbox{ unramified, } \beta_2 \mbox{ ramified}
  & \beta_2(\varpi^l)q^{-\frac l2} \\
 \Omega\,{\rm St}_{\GL_2} \mbox{ with } \Omega \mbox{ unramified }
  & \Omega(\varpi^l)q^{-l} \\
 \mbox{ supercuspidal OR ramified twist of Steinberg}&1\quad\mbox{ if }l=0\\[-2ex]
  \mbox{ OR } \beta_1\times \beta_2 \mbox{ with } \beta_1, \beta_2 \mbox{ ramified, } \beta_1\beta_2^{-1} \neq|\,|^{\pm 1}  & 0 \quad\mbox{ if } l > 0
\end{array}
\end{equation}
Assume that $\pi$ is an unramified representation and that $B\in\mathcal{B}_{S,\Lambda,\psi}(\pi)$ is the spherical Bessel function as in (\ref{suganox0eq}). In the following we shall assume that the conductor $\p^n$ of $\tau$ satisfies $n>0$, since for unramified $\tau$ the local integral has been computed by Furusawa; see Theorem (3.7) in \cite{Fu}. Since both functions $B$ and $W^\#$ are right $K^H$-invariant, it follows from (\ref{RFKHrepresentativeseq}) that the integral (\ref{localZseq}) is given by
\begin{equation}\label{zeta-integral-sum}
 Z(s,W^\#,B)=\sum\limits_{l,m \geq 0} B(h(l,m))W^\#(\eta h(l,m),s)V_mq^{3m+3l}.
\end{equation}
Here, as in Sect.\ 3.5 of \cite{Fu}, $V_m = {\rm vol}\big(T(F)\backslash T(F)\mat{\varpi^m}{}{}{1}\GL_2(\OF)\big)$. Calculations confirm that $\eta h(l,m)$ lies in the support of $W^\#(\,\cdot\,,s)$ if and only if $m=0$. It follows that the sum (\ref{zeta-integral-sum}) reduces to
\begin{equation}\label{zeta-integral-sum-m=0}
 Z(s,W^\#,B)= \sum\limits_{l \geq 0} B(h(l,0))W^\#(\eta h(l,0),s)q^{3l}.
\end{equation}
By  (\ref{Wsharpformulaeq}),
\begin{equation}\label{W-value-h(l,0)}
 W^\#(\eta h(l,0),s)=q^{-3(s+1/2)l} \omega_\pi(\varpi^{-l}) \omega_\tau(\varpi^{-l})
  W^{(0)}(\mat{\varpi^l}{}{}{1}).
\end{equation}
Substituting the values of $W^{(0)}(\mat{\varpi^l}{}{}{1})$ from the table above and the values of $B(h(l,0))$ from (\ref{suganox0eq}), we get the following result.
\begin{theorem}\label{nonarchlocalzetatheorem}
 Let $\tau,\chi,\chi_0,\Lambda$ and $W^\#(\,\cdot\,,s)$ be as in Corollary \ref{distinguishedvectornonarchtheorem}. Let $\pi$ be an irreducible, admissible, unramified representation of $H(F)$, and let $B$ be the unramified Bessel function given by formula (\ref{suganox0eq}). Then the local zeta integral $Z(s,W^\#,B)$ defined in (\ref{localZseq}) is given by
 \begin{equation}\label{final-integral-l-fn-formula}
  Z(s,W^\#,B)=\frac{L(3s+\frac 12, \tilde\pi \times \tilde\tau)}
  {L(6s+1,\chi|_{F^\times})L(3s+1,\tau \times \AI(\Lambda) \times \chi|_{F^\times})}Y(s),
 \end{equation}
 where
 \begin{align*}
  Y(s)=\left\{\begin{array}{l@{\hspace{-5ex}}l}
   1&\text{if }\tau=\beta_1\times\beta_2,\;\beta_1,\beta_2\text{ unramified},\\
   L(6s+1,\chi|_{F^\times})
    &\text{if }\tau=\beta_1\times\beta_2,\;\beta_1\text{ unram.},\:\beta_2\text{ ram.},\:
     \big(\frac L\p\big)=\pm1,\\
&\qquad\text{OR }\tau=\beta_1\times\beta_2,\;\beta_1\text{ unram.},\:\beta_2\text{ ram.},\\
      &\qquad\qquad\big(\frac L\p\big)=0 \text{ and }\beta_2\chi_{L/F}\text{ ramified},\\
&\qquad\text{OR } \tau=\Omega\St_{\GL(2)},\;\Omega\text{ unramified},\\
     \displaystyle\frac{L(6s+1,\chi|_{F^\times})}
      {1-\Lambda(\varpi_L)(\omega_\pi \beta_2)^{-1}(\varpi)q^{-3s-1}}
    &\text{if }\tau=\beta_1\times\beta_2,\;\beta_1\text{ unram.},\:\beta_2\text{ ram.},\:
     \big(\frac L\p\big)=0,\\
    &\qquad\text{and }\beta_2\chi_{L/F}\text{ unramified},\\[1ex]
   L(6s+1,\chi|_{F^\times})L(3s+1,\tau \times \AI(\Lambda) \times \chi|_{F^\times})
    &\hspace{5ex}\text{ if } \tau=\beta_1\times\beta_2,\;\beta_1,\beta_2\text{ ramified},\\
    &\qquad\text{OR } \tau=\Omega\St_{\GL(2)},\:\Omega\text{ ramified},\\
    &\qquad\text{OR } \tau \text{ supercuspidal}.
  \end{array}\right.
 \end{align*}
 In (\ref{final-integral-l-fn-formula}), $\tilde\pi$ and $\tilde\tau$ denote the
 contragredient of $\pi$ and $\tau$, respectively. The symbol $\AI(\Lambda)$
 stands for the $\GL_2(F)$ representation attached to the character $\Lambda$
 of $L^\times$ via automorphic induction, and $\chi_{L/F}$ stands for the quadratic character of $F^\times$ associated with the extension $L/F$.  The function $L(3s+1,\tau \times \AI(\Lambda) \times \chi|_{F^\times})$ is a standard $L$-factor for $\GL_2\times\GL_2\times\GL_1$.
\end{theorem}
\begin{proof}
If $\tau=\beta_1\times\beta_2$ with unramified $\beta_1$ and $\beta_2$, then
this is Theorem (3.7) in Furusawa's paper \cite{Fu}. If $\tau=\beta_1\times\beta_2$
with unramified $\beta_1$ and ramified $\beta_2$, then, from
the local Langlands correspondence, we have the following
$L$-functions attached to the representations $\tilde\pi \times
\tilde\tau$ of $\GSp_4(F) \times \GL_2(F)$ and $\tau \times
\AI(\Lambda) \times \chi|_{F^\times}$ of $\GL_2(F) \times \GL_2(F)\times \GL_1(F)$,
\begin{equation}\label{gsp4-gl2-l-fn}
 L(s, \tilde\pi \times \tilde\tau)=\prod_{i=1}^4
 \big(1-(\gamma^{(i)}\beta_1)^{-1}(\varpi)q^{-s}\big)^{-1},
\end{equation}
where $\gamma^{(i)}$ are the Satake parameters of $\pi$, as in Sect.\ (3.6) of \cite{Fu}, and
$$\renewcommand{\arraystretch}{1.4}
 \frac1{L(s, \tau \times \AI(\Lambda) \times \chi|_{F^\times})}=
  \left\{\begin{array}{l@{\hspace{-1ex}}l}
    1-\big(\Lambda (\omega_\pi \beta_1)^{-2}\big)(\varpi)q^{-2s}
     &\text{if } \big(\frac L\p\big)=-1,\\
    1-\Lambda(\varpi_L) (\omega_\pi\beta_1)^{-1}(\varpi)q^{-s}
     &\text{if } \big(\frac L\p\big)= 0\text{ and }\beta_2\chi_{L/F}\text{ ram.},\\
    (1-\Lambda(\varpi_L) (\omega_\pi\beta_1)^{-1}(\varpi)q^{-s})\\
    \qquad(1-\Lambda(\varpi_L) (\omega_\pi \beta_2)^{-1}(\varpi)q^{-s})
     &\text{if } \big(\frac L\p\big)=0\text{ and }\beta_2\chi_{L/F}\text{ unram.},\\
    (1-\Lambda(\varpi_L) (\omega_\pi\beta_1)^{-1}(\varpi)q^{-s})\\
     \qquad(1-\Lambda(\varpi\varpi_L^{-1})(\omega_\pi\beta_1)^{-1}(\varpi)q^{-s})
     &\qquad\text{if } \big(\frac L\p\big)= 1.\end{array}\right.
$$
The desired result therefore follows from (\ref{suganoQeq}) and (\ref{suganoHeq}).
If $\tau$ is an unramified twist of the Steinberg representation, then the result was proved in Theorem 3.8.1 of \cite{PS1}. In all remaining cases we have $L(s,\tilde\pi\times\tilde\tau)=1$ and $Z(s,W^\#,B)=1$, so that the asserted formula holds.
\end{proof}
\subsubsection*{The archimedean case}
Now let $F=\R$. We will calculate the zeta integral (\ref{localZseq}) for the distinguished function $W^\#$ given in Theorem \ref{distinguishedvectorarchtheorem}. It is enough to calculate these integrals for the functions $W^\#_{m,l,l_2,l_1}$, where $l_1$ is one of the weights occurring in $\tau$, and where $l_2\in\Z$ has the same parity as $l_1$. Recall the explicit formula (\ref{Wsharpwelldeflemmaeq1}) for these functions.


\vspace{3ex}
As for the Bessel function ingredient in (\ref{localZseq}), let $\pi$ be a (limit of) discrete series representation of $\PGSp_4(\R)$ with scalar minimal $K$-type $(l,l)$, where $l\geq2$. Let $B:\:H(\R)\rightarrow\C$ be the function defined in (\ref{archBesselformula2eq}). Then $B$ is a vector of weight $(-l,-l)$ in $\mathcal{B}_{S,\Lambda,\psi}(\pi)$, where $\Lambda=1$ and $\psi(x)=e^{-2\pi ix}$. By (\ref{distinguishedvectorarchtheoremcoreq5}) and (\ref{realBproperty2eq}), the function $W^\#(\eta h, s) B(h)$ is right invariant under $K^H_\infty$. Using this
fact and the disjoint double coset decomposition (\ref{HRBesseldecompeq}), we obtain
\begin{align}\label{archintegral1eq}
 Z(s,W^\#_{m,l,l_2,l_1},B)&= \pi\int\limits_{\R^{\times}}\int\limits_1^{\infty}
  W^\#_{m,l,l_2,l_1}\Big(\eta \begin{bmatrix}\lambda t_0\mat{\zeta}{}{}{\zeta^{-1}}&\\
  &^tt_0^{-1}\mat{\zeta^{-1}}{}{}{\zeta}\end{bmatrix},s\Big) \nonumber \\
 &\hspace{10ex}B\Big(\begin{bmatrix}\lambda t_0\mat{\zeta}{}{}{\zeta^{-1}}&\\
  &^tt_0^{-1}\mat{\zeta^{-1}}{}{}{\zeta}\,\end{bmatrix}\Big)(\zeta-\zeta^{-3})\lambda^{-4}\,d\zeta\,d\lambda;
\end{align}
see (4.6) of \cite{Fu} for the relevant integration formulas. The above calculations are valid for any choice of $\mathbf{a},\mathbf{b},\mathbf{c}$ as long as $S=\mat{\mathbf{a}}{\mathbf{b}/2}{\mathbf{b}/2}{\mathbf{c}}$ is positive definite. We will compute (\ref{archintegral1eq}), in two special cases, namely when $S$ is of the form $S=\mat{D/4}{}{}{1}$ or $S=\mat{(1+D)/4}{1/2}{1/2}{1}$ with a positive number $D$. By the argument in Sect.\ 4.4 of \cite{PS1}, we may assume that $S$ is of the first kind. Then $\eta =\begin{bmatrix}1&&&\\\frac{\sqrt{-D}}{2}&1&&\\&&1&\frac{\sqrt{-D}}{2}\\ &&&1\end{bmatrix}$, and we can choose $t_0
=\mat{2^{1/2}D^{-1/4}}{}{}{2^{-1/2}D^{1/4}}$. From formula
(\ref{archBesselformula2eq}),
\begin{equation}\label{archbesselformula3eq}
 B\Big(\begin{bmatrix}\lambda t_0\mat{\zeta}{}{}{\zeta^{-1}}&\\
  &^tt_0^{-1}\mat{\zeta^{-1}}{}{}{\zeta}\,\end{bmatrix}\Big) =
  \renewcommand{\arraystretch}{1.3}\left\{\begin{array}{ll}\lambda^le^{-2\pi \lambda
  D^{1/2}\frac{\zeta^2+\zeta^{-2}}2}& \hbox{ if } \lambda > 0,\\
    0& \hbox{ if } \lambda < 0.
 \end{array}\right.
\end{equation}
Next, the argument of $W^\#_{m,l,l_2,l_1}$ can be rewritten as an element of $MNK_{\infty}^{G_2}$ as
$$
 \begin{bmatrix}\lambda D^{-\frac14}\Big(\frac{\zeta^2+\zeta^{-2}}{2}\Big)^{-\frac 12}\\
  &\lambda D^{\frac 14}\Big(\frac{\zeta^2+\zeta^{-2}}{2}\Big)^{\frac 12}\\
  &&D^{\frac 14}\Big(\frac{\zeta^2+\zeta^{-2}}{2}\Big)^{\frac 12}\\
  &&&D^{-\frac14}\Big(\frac{\zeta^2+\zeta^{-2}}{2}\Big)^{-\frac12}
 \end{bmatrix}  \begin{bmatrix}1&-i\zeta^2&&\\0&1&&\\&&1&0\\&&-i\zeta^2&1\end{bmatrix} k,
$$
where $k =  \mat{k_0}{}{}{k_0}$ with $k_0 = (\zeta^2+\zeta^{-2})^{-1/2}\mat{\zeta^{-1}}{i \zeta}{i\zeta}{\zeta^{-1}} \in \SU(2)$. From now on assume that $m=|l-l_1|$. We have
$$
 \Phi^\#_{m,l,l_2,l_1}(\mat{k_0}{}{}{k_0})=
  \Big(\frac{\zeta^2+\zeta^{-2}}2\Big)^{-m},
$$
and hence
\begin{align}\label{whittakerfirstformulaeq}
 &W^\#_{m,l,l_2,l_1}\Big(\eta \begin{bmatrix}\lambda t_0\mat{\zeta}{}{}{\zeta^{-1}}&\\
  &t_0^{-1}\mat{\zeta^{-1}}{}{}{\zeta}\end{bmatrix},s\Big) \nonumber \\
 &\qquad=\Big(\frac{\zeta^2+\zeta^{-2}}2\Big)^{-|l-l_1|}
  \Big|\lambda D^{-\frac12}
  \big(\frac{\zeta^2+\zeta^{-2}}{2}\big)^{-1}\Big|^{3(s+\frac12)}
  \omega_\tau(\lambda)^{-1}
  W_{l_1}(\mat{\lambda D^{\frac12}\big(\frac{\zeta^2+\zeta^{-2}}{2}\big)}{}{}{1}).
\end{align}
If $q\in\C$ is as in (\ref{t1t2pqeq}), then $\omega_\tau(y)=y^q$ for $y>0$. It
follows from (\ref{W1Whittakerformulaeq}),
(\ref{archbesselformula3eq}) and (\ref{whittakerfirstformulaeq}) that
\begin{align}\label{realI1eq}
 Z(s,W^\#_{m,l,l_2,l_1},B)&=a^+\pi D^{-\frac{3s}2-\frac 34+\frac q4}(4\pi)^{\frac q2}
   \int\limits_0^{\infty}\int\limits_1^{\infty}
   \lambda^{3s+\frac 32+l-\frac q2}
   \Big(\frac{\zeta^2+\zeta^{-2}}2\Big)^{-3s-\frac 32+\frac q2-|l-l_1|}\nonumber\\
  &\hspace{10ex}W_{\frac{l_1}2, \frac{p}2}\big(4 \pi \lambda
   D^{1/2}\frac{\zeta^2+\zeta^{-2}}2\big)e^{-2\pi \lambda
   D^{1/2}\frac{\zeta^2+\zeta^{-2}}2}(\zeta-\zeta^{-3})\lambda^{-4}\,d\zeta\,d\lambda.
\end{align}
Using the substitutions $u = (\zeta^2+\zeta^{-2})/2$ and  $x = 4 \pi \lambda D^{1/2}u$, together with the integral formula for the Whittaker function from
\cite[p.\ 316]{MOS}, we get
$$
 Z(s,W^\#_{m,l,l_2,l_1},B)=a^+\pi
  \frac{D^{-3s-\frac l2+\frac q2}\,(4\pi)^{-3s+\frac 32 -l+q}}{6s+l+|l-l_1|-q-1}\,
  \frac{\Gamma(3s+l-1+\frac p2-\frac q2)
  \Gamma(3s+l-1-\frac p2-\frac q2)}{\Gamma(3s+l-\frac{l_1}2-\frac12-\frac q2)}.
$$
Here, for the calculation of the $u$-integral, we have assumed that ${\rm Re}(6s+l+|l-l_1|-q-1)>0$. We summarize our result in the following theorem. We will use the notation \begin{equation}\label{GammaRCdefeq}
 \Gamma_\R(s)=\pi^{-s/2}\,\Gamma\Big(\frac s2\Big),\qquad
 \Gamma_\C(s)=2(2\pi)^{-s}\,\Gamma(s).
\end{equation}
The proof of (\ref{archmaintheoremeq2}) below follows from the tables at the end of this section.
\begin{theorem}\label{archlocalzetatheorem}
 Assume that the matrix $S$ is of the form
 \begin{equation}\label{archmaintheoremeq0}
  S=\mat{D/4}{}{}{1}\qquad\text{or}\qquad S=\mat{(1+D)/4}{1/2}{1/2}{1}
 \end{equation}
 with a positive number $D$. Let $l\geq2$ be an integer, and let $\pi$ be the (limit of) discrete series representation of $\PGSp_4(\R)$ with scalar minimal $K$-type $(l,l)$. Let $l_2\in\Z$ and $\tau$, $\chi_0$, $\chi$ be as in Corollary \ref{distinguishedvectorarchtheoremcor}. Let $l_1$ be one of the weights occurring in $\tau$, and let $W^\#_{|l-l_1|,l,l_2,l_1}$ be the function defined in (\ref{Wsharpwelldeflemmaeq1}). Let $B:\:H(\R)\rightarrow\C$ be the function defined in (\ref{archBesselformula2eq}). Then, for ${\rm Re}(6s+l+|l-l_1|-q-1)>0$, with the local archimedean integral as in (\ref{localZseq}),
 \begin{align}\label{archmaintheoremeq1}
  Z(s,W^\#_{|l-l_1|,l,l_2,l_1},B)
  &=a^+_{l_1,p,q}D^{-3s-\frac l2+\frac q2}\,2^{-3s+\frac12 -l+\frac{3q}2+\frac{l_1}2}\,\pi^{1+\frac q2+\frac{l_1}2}\nonumber\\
  &\quad\times\frac{1}{6s+l+|l-l_1|-q-1}\,
  \frac{\Gamma_\C(3s+l-1+\frac p2-\frac q2)
  \Gamma_\C(3s+l-1-\frac p2-\frac q2)}{\Gamma_\C(3s+l-\frac{l_1}2-\frac12-\frac q2)}.
 \end{align}
 Here, $a^+_{l_1,p,q}$ is as defined in (\ref{aplusdefeq}). The numbers $p,q\in\C$ are defined in (\ref{t1t2pqeq}). With $\Lambda$ being the trivial character, we can rewrite formula (\ref{archmaintheoremeq1}) as
 \begin{equation}\label{archmaintheoremeq2}
  Z(s,W^\#_{|l-l_1|,l,l_2,l_1},B)=\frac{L(3s+\frac 12, \tilde\pi \times \tilde\tau)}
  {L(6s+1,\chi|_{\R^\times})L(3s+1,\tau \times \AI(\Lambda) \times \chi|_{\R^\times})}
   Y_{l,l_1,p,q}(s),
 \end{equation}
 where, with $u=0$ if $l_1$ is even and $u=1/2$ if $l_1$ is odd,
 \begin{align*}
  Y_{l,l_1,p,q}(s)=\left\{\!\!\!\begin{array}{l@{\;\;}l}
   \displaystyle\frac{a^+_{l_1,p,q}D^{-3s-\frac l2+\frac q2}2^{q-l+\frac{l_1}2+u}
     \pi^{1+\frac q2+\frac{l_1}2}(3s-\frac q2+\frac p2)}
     {3s+\frac{l+|l-l_1|}2-\frac12-\frac q2}\\[3ex]
    \displaystyle\times\frac{\Gamma_\C\big(3s+l-1-\frac p2-\frac q2\big)
       \Gamma_\C\big(3s+\frac12-\frac q2+u\big)}
      {\Gamma_\C\big(3s+\frac12-\frac q2+|l-\frac32-\frac p2|\big)
       \Gamma_\C\big(3s+l-\frac{l_1}2-\frac12-\frac q2\big)}
     &\text{if }\tau=\mathcal{D}_{p,\frac q2},\,p \geq 1,\\[5ex]
   \displaystyle\frac{a^+_{l_1,p,q}D^{-3s-\frac l2+\frac q2}2^{q-1-l+\frac{l_1}2+u}\pi^{1+\frac q2+\frac{l_1}2}
     \Gamma_\C\big(3s+\frac12-\frac q2+u\big)}
    {\big(3s+\frac{l+|l-l_1|}2-\frac12-\frac q2\big)
     \Gamma_\C\big(3s+l-\frac{l_1}2-\frac12-\frac q2\big)}
     &\text{if }\tau=\beta_1\times\beta_2.
  \end{array}\right.
 \end{align*}
\end{theorem}

{\bf Remarks:} a) The factor $Y_{l,l_1,p,q}(s)$ is of the form $D^{-3s}$ times a rational function in $s$.

\vspace{2ex}
b) For $l=l_1$ we recover Theorem 4.4.1 of \cite{PS1}. We point out that
in our present approach the number $l_1$ (the $\GL_2$ weight) can be chosen independently of $l$ (the $\GSp_4$ weight), including the case of different parity.

\vspace{2ex}
c) In one of our later applications, the number $D$ will be a fundamental discriminant satisfying $D\equiv0$ mod $4$ or $D\equiv3$ mod $4$. Having the above theorem available for the two cases of $S$ in (\ref{archmaintheoremeq0}) assures that $S$ can be chosen to be a half-integral matrix.

\begin{corollary}\label{archlocalzetatheoremcor}
 Let all hypotheses be as in Theorem \ref{archlocalzetatheorem}. Let $W^\#\in I_W(s,\chi,\chi_0,\tau)$ be the distinguished function defined in table (\ref{ABCtableeq}). Then
 \begin{equation}\label{archmaintheoremcoreq2}
  Z(s,W^\#,B)=\frac{L(3s+\frac 12, \tilde\pi \times \tilde\tau)}
  {L(6s+1,\chi|_{\R^\times})L(3s+1,\tau \times \AI(\Lambda) \times \chi|_{\R^\times})}
   Y(s),
 \end{equation}
 with
 \begin{equation}\label{archmaintheoremcoreq3}\renewcommand{\arraystretch}{1.5}
  Y(s)=\left\{\begin{array}{l@{\qquad\text{in Case }}l}
   \kappa_{p+1,p}\,Y_{l,p+1,p,q}(s)&A,\\
   \kappa_{l,p}\,Y_{l,l,p,q}(s)&B,\\
   \displaystyle\kappa_{l+1,p}\,Y_{l,l+1,p,q}(s)
    +\Big(3s-\frac{p+q}2\Big)\kappa_{l-1,p}\,Y_{l,l-1,p,q}(s)&C.
   \end{array}\right.
 \end{equation}
 Here, the constants $\kappa_{*,p}$ are defined in (\ref{PhiWkappaeq}), and the factors $Y_{l,*,p,q}(s)$ are defined in Theorem \ref{archlocalzetatheorem}.
\end{corollary}
\subsubsection*{Tables for archimedean factors}
The archimedean Euler factors appearing in (\ref{archmaintheoremeq2}) can be easily calculated via the archimedean local Langlands correspondence. We omit the details and simply show the results in the following tables. For the principal series case $\beta_1\times\beta_2$, the numbers $p,q\in\C$ are such that
$\beta_1(a)=a^{\frac{q+p}2}$ and $\beta_2(a)=a^{\frac{q-p}2}$.

$$\renewcommand{\arraystretch}{1.2}
 \begin{array}{ccc}
  \tau&L(s,\pi\times\tilde\tau)&\varepsilon(s,\pi\times\tilde\tau,\psi^{-1})\\\hline
  \rule{0ex}{3ex}\mathcal{D}_{p,\mu},\;p\geq1,\:\mu\in\C&\Gamma_\C\big(s-\mu+\frac{p}{2}+\frac{1}{2}\big)
     \Gamma_\C\big(s-\mu+\frac{p}{2}-\frac{1}{2}\big)&i^{2l+3p-3+|2l-3-p|}\\
  &\Gamma_\C\big(s-\mu+l-\frac{3}{2}+\frac{p}{2}\big)
     \Gamma_\C\big(s-\mu+\big|l-\frac{3}{2}-\frac{p}{2}\big|\big)\\[2ex]
  \beta_1\times\beta_2
   &\Gamma_\C\big(s+\frac{1-q-p}2\big)\Gamma_\C\big(s+\frac{1-q+p}2\big)&1\\
   &\Gamma_\C\big(s+l-\frac{q+p+3}2\big)\Gamma_\C\big(s+l-\frac{q-p+3}2\big)
 \end{array}
$$
\vspace{2ex}

The next table shows $L$- and $\varepsilon$-factors for $\tau\times\AI(\Lambda)\times\chi|_{\R^\times}$.

$$\renewcommand{\arraystretch}{1.2}
 \begin{array}{ccc}
  \tau&L(s,\tau\times\AI(\Lambda)\times\chi|_{\R^\times})&\varepsilon(s,\tau\times\AI(\Lambda)\times\chi|_{\R^\times},\psi^{-1})\\\hline
  \rule{0ex}{3ex}\mathcal{D}_{p,\mu},\;p\geq1,\:\mu\in\C
   &\Gamma_\C\big(s-\mu+\frac{p}{2}\big)^2&(-1)^{p+1}\\[2ex]
  \beta_1\times\beta_2
   &\Gamma_\C\big(s-\frac{q+p}2\big)\Gamma_\C\big(s-\frac{q-p}2\big)&-1
 \end{array}
$$

\subsection{The global integral representation}\label{globalintrepsec}
Let $F$ be an algebraic number field and $\A_F$ its ring of adeles. Let $L$ be a quadratic field extension of $F$; the extension $L/F$ defines the unitary group $G_2$. The Eisenstein series $E(h,s;f)$ entering into the global integral (\ref{globalintegraleq}) below will be defined from a section $f$ in a global induced representation of $G_2(\A_F)$. We therefore start by discussing various models of such induced representations.
\subsubsection*{Global induced representations}
Let $(\tau,V_\tau)$ be a cuspidal, automorphic representation of $\GL_2(\A_F)$.
Let $\chi_0$ be a character of $L^\times\backslash\A_L^\times$ such that
the restriction of $\chi_0$ to $\A_F^\times$ concides with $\omega_\tau$, the
central character of $\tau$. Then, as in (\ref{M2representationseq}) in the local case,
$\chi_0$ can be used to extend $\tau$ to a representation of $M^{(2)}(\A_F)$,
denoted by $\chi_0 \times \tau$. Let $\chi$ be another character of
$L^\times\backslash\A_L^\times$, considered as a character of $M^{(1)}(\A_F)$.
This data defines a family of induced representations
$I(s,\chi,\chi_0,\tau)$ of $G_2(\A_F)$ depending on a complex parameter $s$.
The space of $I(s,\chi,\chi_0,\tau)$ consists of functions
$\varphi:\:G_2(\A_F)\rightarrow V_\tau$ with the transformation property
$$
 \varphi(m_1m_2ng)=\delta_P(m_1m_2)^{s+1/2}\chi(m_1)(\chi_0 \times \tau)(m_2)\varphi(g)
$$
for all $m_1\in M^{(1)}(\A_F)$, $m_2\in M^{(2)}(\A_F)$ and $n\in N(\A_F)$.
Since the representation $\tau$ is given as a space of automorphic forms, we may realize $I(s,\chi,\chi_0,\tau)$ as a space of $\C$-valued functions on $G_2(\A_F)$. More precisely, to each $\varphi$ as above we may attach the function $f_\varphi$ on $G_2(\A_F)$ given by $f_\varphi(g)=(\varphi(g))(1)$. Each function $f_\varphi$ has the property that $\GL_2(\A_F)\ni h\mapsto f_\varphi(hg)$ is an element of $V_\tau$, for each $g\in G_2(\A_F)$. Let $I_\C(s,\chi,\chi_0,\tau)$ be the model of $I(s,\chi,\chi_0,\tau)$ thus obtained. A third model of the same representation is obtained by attaching to $f\in I_\C(s,\chi,\chi_0,\tau)$ the function
\begin{equation}\label{Wffreleq}
 W_f(g) = \int\limits_{F\backslash \A_F}
 f\Big(\begin{bmatrix}1&&&\\&1&&x\\&&1&\\&&&1\end{bmatrix}g\Big)\psi(\mathbf{c}x)dx,
 \qquad g\in G_2(\A_F).
\end{equation}
Here, $\mathbf{c}\in F^\times$ is a fixed element. The map $f\mapsto W_f$ is injective since $\tau$ is cuspidal.
Let $I_W(s,\chi,\chi_0,\tau)$ be the space of all functions $W_f$. Now write $\tau\cong\otimes\tau_v$ with local representations $\tau_v$ of $\GL_2(F_v)$. We also factor $\chi=\otimes\chi_v$ and $\chi_0=\otimes\chi_{0,v}$, where $\chi_v$ and $\chi_{0,v}$ are characters of $\prod_{w|v}L_w^\times$.
Then there are isomorphisms
\begin{equation}\label{locglobindrepdiagrameq}
 \begin{CD}
  I(s,\chi,\chi_0,\tau)@>\sim>>\otimes_v I(s,\chi_v,\chi_{0,v},\tau_v)\\
  @V\sim VV @VV=V\\
  I_\C(s,\chi,\chi_0,\tau)@>\sim>>\otimes_v I(s,\chi_v,\chi_{0,v},\tau_v)\\
  @V\sim VV @VV\sim V\\
  I_W(s,\chi,\chi_0,\tau)@>\sim>>\otimes_v I_W(s,\chi_v,\chi_{0,v},\tau_v)
 \end{CD}
\end{equation}
Here, the local induced representation $I(s,\chi_v,\chi_{0,v},\tau_v)$ consists of functions taking values in a model $V_{\tau_v}$ of $\tau_v$; see Sect.\ \ref{parabolicinductionsec} for the precise definition. Assume that $V_{\tau_v}=\mathcal{W}(\tau_v,\psi_v^{-\mathbf{c}})$ is the Whittaker model of $V_{\tau_v}$ with respect to the additive character $\psi_v^{-\mathbf{c}}$. If we attach to each $f_v\in I(s,\chi_v,\chi_{0,v},\tau_v)$ the function $W_{f_v}(g)=f_v(g)(1)$, then we obtain the model $I_W(s,\chi_v,\chi_{0,v},\tau_v)$ of the same induced representation. The bottom isomorphism in diagram (\ref{locglobindrepdiagrameq}) is such that if $W_v\in I_W(s,\chi_v,\chi_{0,v},\tau_v)$ are given, with the property that $W_v\big|_{G_2(\OF_v)}=1$ for almost all $v$, then the corresponding element of $I_W(s,\chi,\chi_0,\tau)$ is the function
\begin{equation}\label{Wfactorizationeq}
 W(g)=\prod_{v\leq\infty}W_v(g_v),\qquad g=(g_v)_v\in G_2(\A_F).
\end{equation}
\subsubsection*{The global integral and the basic identity}
Now let $\mathbf{a},\mathbf{b},\mathbf{c},\mathbf{d}, S,L,\Lambda$ be as in Sect.\ \ref{besselsec}. Let $(\pi,V_\pi)$ be a cuspidal, automorphic representation of $H(\A_F)$ which has a global Bessel model of type $(S,\Lambda,\psi)$. Let further $(\tau,V_\tau)$ be a cuspidal, automorphic representation of $\GL_2(\A_F)$,
extended to a representation of $M^{(2)}(\A_F)$ via a character $\chi_0$ of
$L^\times\backslash\A_L^\times$. Define the character $\chi$ of $L^\times\backslash\A_L^\times$ by
\begin{equation}\label{chilambdachi0globaleq}
 \chi(y)=\Lambda(\bar y)^{-1}\chi_0(\bar y)^{-1},\qquad y\in\A_L^\times.
\end{equation}
Let $f(g,s)$ be an analytic family in $I_\C(s,\chi,\chi_0,\tau)$. For ${\rm Re}(s)$ large enough we can form the Eisenstein series
\begin{equation}\label{Edefeq}
 E(g,s;f)=\sum_{\gamma\in P(F)\backslash G_2(F)}f(\gamma g,s).
\end{equation}
In fact, $E(g,s;f)$ has a meromorphic continuation to the entire complex plane. In \cite{Fu} Furusawa studied integrals of the form
\begin{equation}\label{globalintegraleq}
 Z(s,f,\phi)=\int\limits_{H(F)Z_H(\A_F)\backslash H(\A_F)}E(h,s;f)\phi(h)\,dh,
\end{equation}
where $\phi\in V_\pi$. Theorem (2.4) of \cite{Fu}, the ``Basic Identity'', states that
\begin{equation}\label{basicidentityeq}
 Z(s,f,\phi)=Z(s,W_f,B_\phi):=\int\limits_{R(\A_F)\backslash H(\A_F)}W_f(\eta h,s)B_\phi(h)\,dh,
 \qquad\eta\text{ as in }(\ref{etadefeq}),
\end{equation}
where $R(\A_F)$ is the Bessel subgroup determined by $(S,\Lambda,\psi)$, and $B_\phi$ is the Bessel function corresponding to $\phi$; see (\ref{Bphieq}). The function $W_f(\,\cdot\,,s)$ appearing in (\ref{basicidentityeq}) is the element of $I_W(s,\chi,\chi_0,\tau)$ corresponding to $f(\,\cdot\,,s)\in I_\C(s,\chi,\chi_0,\tau)$; see (\ref{Wffreleq}) for the formula relating $f$ and $W_f$.

\vspace{3ex}
The importance of the basic identity lies in the fact that the integral on the right side of (\ref{basicidentityeq}) is Eulerian. Namely, assume that $f(\,\cdot\,,s)$ corresponds to a pure tensor $\otimes f_v$ via the middle isomorphism in (\ref{locglobindrepdiagrameq}). Assume that $W_v\in I_W(s,\chi_v,\chi_{0,v},\tau_v)$ corresponds to $f_v\in I(s,\chi_v,\chi_{0,v},\tau_v)$. Then
$$
 W_f(g,s)=\prod_{v\leq\infty}W_v(g_v,s),\qquad g=(g_v)_v\in G_2(\A_F),
$$
see (\ref{Wfactorizationeq}). Assume further that the global Bessel function $B_\phi$
factorizes as in (\ref{locglobBreleq}). Then it follows from (\ref{basicidentityeq}) that
\begin{equation}\label{Zfactorizationeq}
 Z(s,f,\phi)=\prod_{v\leq\infty}Z_v(s,W_v,B_v),
\end{equation}
with the local zeta integrals
\begin{equation}\label{localzetaeq}
 Z_v(s,W_v,B_v)=\int\limits_{R(F_v)\backslash H(F_v)}W_v(\eta h,s)B_v(h)\,dh.
\end{equation}
Furusawa has calculated the local integrals (\ref{localzetaeq}) in the case where all
the data is unramified. In our non-archimedean Theorem \ref{nonarchlocalzetatheorem} we calculated these integrals in the case where the $\GSp_4$ data is still unramified, but the $\GL_2$ data is arbitrary. Here, we took for $W_v$ the distinguished vector $W^\#$ from Corollary \ref{distinguishedvectornonarchtheorem}. In our archimedean Corollary \ref{archlocalzetatheoremcor} we calculated these integrals in the case where the $\GSp_4$ data is a scalar minimal $K$-type lowest weight representation, and the $\GL_2$ data is arbitrary. Here, we took for $W_v$ the distinguished vector $W^\#$ defined in table (\ref{ABCtableeq}).
\subsubsection*{The global integral representation over $\Q$}
The important fact in the theory outlined above is that the local functions $W_v$ can be chosen such that the integrals (\ref{localzetaeq}) are all non-zero. We have to make sure, however, that the data entering the local theorems, in particular the characters $\chi$, $\chi_0$ and $\Lambda$, fit into a global situation. For simplicity, we assume from now on that the number field is $F=\Q$ (this, however, is not essential).

\begin{lemma}\label{globalchi0lemma}
 Let $L$ be an imaginary quadratic field extension of $\Q$.
 Let $\omega=\otimes\omega_p$ be a character of $\Q^\times\backslash\A^\times$.
 Let $l_2$ be an integer such that $(-1)^{l_2}=\omega_\infty(-1)$. Then there exists
 a character $\chi_0=\otimes\chi_{0,v}$ of $L^\times\backslash\A_L^\times$ such that
 \begin{enumerate}
  \item the restriction of $\chi_0$ to $\A^\times$ coincides with $\omega$, and
  \item $\chi_{0,\infty}(\zeta)=\zeta^{l_2}$ for all $\zeta\in S^1$.
 \end{enumerate}
\end{lemma}
\begin{proof}
Since $\omega$ is trivial on $L^\times\cap\A^\times=\Q^\times$, we
can extend $\omega$ to a character of $L^\times\A^\times$ in such a way that
$\omega\big|_{L^\times}=1$. Since $S^1\cap(L^\times\A^\times)=\{\pm1\}$, we
can further extend $\omega$ to a character of $S^1L^\times\A^\times$ in such
a way that $\omega(\zeta)=\zeta^{l_2}$ for all $\zeta\in S^1$.
For each finite place $v$ of $L$ we will choose a compact subgroup $U_v$ of $\OF_{L,v}^\times$
such that $\omega$ can be extended to $S^1L^\times\A^\times\big(\prod_{v<\infty}U_v\big)$, with $\omega$ trivial on $\prod_{v<\infty}U_v$ and $U_v=\OF_{L,v}^\times$ for almost all $v$. Hence, the $U_v$ should be chosen such that $\omega$ is trivial on $\big(\prod_{v<\infty}U_v\big)\cap S^1L^\times\A^\times$. We consider the intersection
\begin{equation}\label{globalchi0lemmaeq1}
 \big(\prod_{v<\infty}U_v\big)\cap S^1L^\times\A^\times
 =\big(\prod_{v<\infty}U_v\big)\cap\C^\times L^\times\big(\prod_{p<\infty}\Z_p^\times\big).
\end{equation}
Let $zax$ be an element of this intersection, where $z\in\C^\times$,
$a\in L^\times$ and $x\in\prod_{p<\infty}\Z_p^\times$. We have
$a\in L^\times\cap\prod_{v<\infty}\OF_{L,v}^\times=\OF_L^\times$, which is a finite set, say $\{a_1,\ldots,a_m\}$. For $i$ such that $a_i\notin\Q$, choose a prime $p$ such that
$a_i\notin\Z_p^\times$. Then choose a place $v$ lying above $p$, and choose $U_v$ so
small that $a_i\notin U_v\Z_p^\times$. Then the intersection (\ref{globalchi0lemmaeq1})
equals
\begin{equation}\label{globalchi0lemmaeq2}
 \big(\prod_{v<\infty}U_v\big)\cap\C^\times\Q^\times\big(\prod_{p<\infty}\Z_p^\times\big).
\end{equation}
We can choose $U_v$ even smaller, so that $\omega$ is trivial
on this intersection. We can therefore extend $\omega$ to a character of
\begin{equation}\label{globalchi0lemmaeq3}
 S^1L^\times\A^\times\big(\prod_{v<\infty}U_v\big)
 =\C^\times L^\times\big(\prod_{v<\infty}U_v\big)\big(\prod_{p<\infty}\Z_p^\times\big).
\end{equation}
in such a way that $\omega$ is trivial on $\prod_{v<\infty}U_v$. The group (\ref{globalchi0lemmaeq3}) is of finite index in $\C^\times L^\times\big(\prod_{v<\infty}\OF_{L,v}^\times\big)$, and therefore of finite index in $\A_L^\times$ (using the finiteness of the class number). By Pontrjagin duality, we can now extend $\omega$ to a character $\chi_0$ of $\A_L^\times$ with the desired properties.
\end{proof}
%

We now explain the setup for the global integral representation. For simplicity we will work over the rational numbers. We require the following ingredients.
\begin{itemize}
 \item $\psi=\prod_v\psi_v$ is a character of $\Q\backslash\A$ such that $\psi_\infty(x)=e^{-2\pi ix}$. Also, we require that $\psi_p$ has conductor $\Z_p$ for all finite $p$. There is exactly one such character $\psi$.
 \item Let $D>0$ be such that $-D$ is a fundamental discriminant, and define $\mathbf{a},\mathbf{b},\mathbf{c}\in \Q$ and the matrix $S$ by
  \begin{equation}\label{SminusDdefeq}
   S=S(-D):=\mat{\mathbf{a}}{\mathbf{b}/2}{\mathbf{b}/2}{\mathbf{c}}=\left\{\begin{array}{l@{\qquad\text{if }}l}
    \mat{D/4}{}{}{1}&D\equiv0\mod4,\\[2ex]
    \mat{(1+D)/4}{1/2}{1/2}{1}&D\equiv3\mod4.
   \end{array}\right.
  \end{equation}
 \item Let $L$ be the imaginary quadratic field $\Q(\sqrt{-D})$. The unitary groups $G_i$ are defined with respect to the extension $L/\Q$.
 \item Let $\pi=\otimes\pi_v$ be a cuspidal, automorphic representation of $H(\A)$ with the following properties. The archimedean component $\pi_\infty$ is a (limit of) discrete series representation with minimal $K$-type $(l,l)$, where $l\geq2$, and trivial central character. If $v$ is a non-archimedean place, then $\pi_v$ is unramified and has trivial central character.
 \item Let $\tau=\otimes\tau_v$ be a cuspidal, automorphic representation of $\GL_2(\A)$ with central character $\omega_\tau$.

 \item Let $\chi_0$ be a character of $L^\times\backslash\A_L^\times$ such that $\chi_0\big|_{\A^\times}=\omega_\tau$ and $\chi_{0,\infty}(\zeta)=\zeta^{l_2}$ for $\zeta\in S^1$. Here, $l_2$ is any integer of the same parity as the weights of $\tau$. Such a character exists by Lemma \ref{globalchi0lemma}.
 \item Let $\Lambda=\otimes\Lambda_v$ be a character of $L^\times\backslash\A_L^\times$ such that $\Lambda_\infty=1$ and such that $\Lambda_v$ is unramified for all finite $v$. Hence, $\Lambda$ is a character of the ideal class group
 \begin{equation}\label{idealclassgroupeq}
   \Big(L^\times\C^\times
   \big(\prod_{v\nmid\infty}\OF_v^\times\big)\Big)\Big\backslash\A_L^\times.
 \end{equation}
 \item Let $\chi$ be the character of $\A_L^\times$ defined by (\ref{chilambdachi0globaleq}).

     \end{itemize}
Let $l_1$ be any weight occurring in $\tau_\infty$. Let $\Psi$ be the unique cusp form in the space of $\tau$ that is a newform at all non-archimedean places and corresponds to a vector of weight $l_1$ at the archimedean place. We normalize $\Psi$ such that the corresponding Whittaker function
\begin{equation}\label{WPsieq}
 W_\Psi(g) = \int\limits_{\Q \bs \A} \Psi( \begin{bmatrix}1&x\\&1\end{bmatrix}g)\, \psi(x)\,dx
\end{equation}
satisfies $W_\Psi(\mat{t^+}{}{}{1})=1$, where $t^+$ is the positive real number chosen in (\ref{tplusdefeq}), considered as an idele with trivial non-archimedean components. Let $\Psi$ be extended to a function on $G_1(\A)$ via $\Psi(ag) = \chi_0(a)\Psi(g)$ for $a \in \A_L^\times$, $g\in \GL_2(\A)$. Let us explicitly describe a section $f_{|l_1-l|,l, l_2, l_1}(g, s) \in I_\C(s, \chi, \chi_0, \tau)$. For a non-archimedean place $v$, let $\tau_v$ have conductor $\p^n$ and let $J_v$ be the function on $K_v^{G_2}=G_2(\OF_v)$ defined by
\begin{equation}J_v(k) = \begin{cases} 1 & \text{ if } k \in P(\OF)\eta_0 K^H \Gamma(\P^n), \\  0& \text{ otherwise} \end{cases}
\end{equation}
(see (\ref{etamdef2eq}) for the definition of $\eta_0$). For $n=0$ this is the characteristic function of $K_v^{G_2}$. Define
$$
 J_{|l_1-l|,l, l_2, l_1}(k, s) =\Phi^\#_ {|l_1-l|,l, l_2, l_1}(k_\infty, s) \cdot \prod_{v < \infty} J_v(k_v),\qquad\text{where }k=(k_v)_v\in \prod_vK_v^{G_2};
$$
see (\ref{W0l1l2distvecpropeq1}). Finally, let
\begin{equation}\label{globalfdefeq}
 f_{|l_1-l|,l, l_2, l_1}(g, s)= \delta_P(m_1m_2)^{s+\frac12}\chi(m_1)\Psi(m_2) J_{|l_1-l|,l, l_2, l_1}(k, s)
\end{equation}
for $g =m_1m_2nk$ with $m_1\in M^{(1)}(\A)$, $m_2\in M^{(2)}(\A)$, $n \in N(\A)$, $k \in \prod_v K_v^{G_2}$.
It is easy to see that $f = f_{|l_1-l|,l, l_2, l_1}$ belongs to $I_\C(s, \chi, \chi_0, \tau)$ and that $(W_f)_v$ corresponds to the vector in Corollary~\ref{distinguishedvectornonarchtheorem} if $v$ is non-archimedean, and to the vector $W^\#_{|l_1 - l|,l,l_2,l_1}$ given by (\ref{Wsharpwelldeflemmaeq1}) if $v=\infty$.
In view of Theorem \ref{nonarchlocalzetatheorem} and Corollary \ref{archlocalzetatheoremcor}, the following important result is now immediate.

\begin{theorem}[Global Integral Representation]\label{globalintegralrepresentationtheoremversion1}
 Let $\psi,D,S$ and $\pi,\tau,\chi_0,\chi,\Lambda$ be as above. Let $f = f_{|l_1-l|,l, l_2, l_1}$ be the section in $I_\C(s, \chi, \chi_0, \tau)$ defined above, and $\phi=\otimes\phi_v$ be a vector in the space of $\pi$ such that $\phi_v$ is unramified for all finite $v$ and such that $\phi_\infty$ is a vector of weight $(-l,-l)$ in $\pi_\infty$. Then the global zeta integral $Z(s,f, \phi)$ defined in (\ref{globalintegraleq}) is given by
 \begin{equation}
 \label{globalintegralrepresentationtheoremversion1eq1}
  Z(s,f, \phi)=\frac{L(3s+\frac 12, \tilde\pi \times \tilde\tau)}
  {L(6s+1,\chi|_{\A^\times})L(3s+1,\tau \times \AI(\Lambda) \times \chi|_{\A^\times})}\cdot B_{\phi}(1) \cdot Y_{l,l_1,p,q}(s) \cdot \prod_{v<\infty}Y_v(s),
 \end{equation}
 with $B_\phi$ as in (\ref{Bphieq}), with the factors $Y_v(s)$ for non-archimedean $v$ given by Theorem \ref{nonarchlocalzetatheorem}, and with the archimedean factor  $Y_{l,l_1,p,q}(s)$ given by Theorem~\ref{archlocalzetatheorem}.
 In (\ref{globalintegralrepresentationtheoremversion1eq1}), $\tilde\pi$ and $\tilde\tau$ denote the
 contragredient of $\pi$ and $\tau$, respectively. The symbol $\AI(\Lambda)$
 stands for the $\GL_2(\A)$ representation attached to the character $\Lambda$
 of $\A_L^\times$ via automorphic induction, and
 $L(3s+1,\tau \times \AI(\Lambda) \times \chi|_{\A^\times})$ is
 a standard $L$-factor for $\GL_2\times\GL_2\times\GL_1$.
\end{theorem}
Next, we state a second version of the above theorem where we choose the distinguished vector at all places, including the archimedean ones. Recall the Cases A,B,C defined in (\ref{ABCdefeq}). Let $l_2$ be as in Theorem \ref{archintertwiningtheorem}.
%
The following result is also an immediate consequence of Theorem \ref{nonarchlocalzetatheorem} and Corollary \ref{archlocalzetatheoremcor} and will be key for the functional equation.

\begin{theorem}\label{globalintegralrepresentationtheorem}
 Let $\psi,D,S$ and $\pi,\tau,\chi_0,\chi,\Lambda$ be as above. Let $B_v$ be the unramified Bessel function given by formula (\ref{suganox0eq}) if $v$ is non-archimedean, and let $B_v$ be the function defined in (\ref{archBesselformula2eq}) if $v$ is archimedean. Let $W_v^\#(\,\cdot\,,s)$ be as in Corollary \ref{distinguishedvectornonarchtheorem} if $v$ is non-archimedean, and as in table (\ref{ABCtableeq}) if $v$ is archimedean. Let
 $$
  W^\#(g,s)=\prod\limits_vW_v^\#(g_v,s),\qquad B(h)=\prod\limits_vB_v(h_v),
 $$
 for $g=(g_v)_v\in G_2(\A)$ and $h=(h_v)_v\in H(\A)$.  Then the global zeta integral $Z(s,W^\#,B)$ defined in (\ref{basicidentityeq}) is given by
 \begin{equation}\label{globalintegralrepresentationtheoremeq1}
  Z(s,W^\#,B)=\frac{L(3s+\frac 12, \tilde\pi \times \tilde\tau)}
  {L(6s+1,\chi|_{\A^\times})L(3s+1,\tau \times \AI(\Lambda) \times \chi|_{\A^\times})}Y(s),
 \end{equation}
 where $Y(s)=\prod_v Y_v(s)$, a finite product, with the local factors given in Theorem \ref{nonarchlocalzetatheorem} (non-archimedean case) and Corollary \ref{archlocalzetatheoremcor} (archimedean case).
\end{theorem}

\subsection{The functional equation}\label{functleqsec}
In this section we prove that, in the setting of Theorem \ref{globalintegralrepresentationtheorem}, the global $L$-function $L(s,\pi\times\tau)$ satisfies the expected functional equation. We begin with some local preparations.
\subsubsection*{The $X$ factor}
Assume that $F$ is a non-archimedean local field of characteristic zero, or $F=\R$. Let $\tau,\chi,\chi_0,\Lambda$ and $\pi$ be as in
Theorem \ref{nonarchlocalzetatheorem} (non-archimedean case) and Theorem \ref{archlocalzetatheorem} (archimedean case). We will calculate the function
\begin{align}\label{Xlocaldefeq}
 X(s)&=K(s)\frac{L(6s+1,\chi|_{F^\times})
   L(3s+1,\tau\times \AI(\Lambda) \times \chi|_{F^\times})}
   {L(6s,\chi|_{F^\times})
   L(3s,\tau\times \AI(\Lambda) \times \chi|_{F^\times})}\nonumber\\
  &\qquad\times\varepsilon(6s,\chi|_{F^\times},\psi^{-1})\,
   \varepsilon(3s,\tau\times \AI(\Lambda) \times \chi|_{F^\times},\psi^{-1})
   \frac{\hat Y(-s)}{Y(s)},
\end{align}
which will be relevant for the functional equation. Here, $K(s)$ is the factor resulting from the local intertwining operator, defined in (\ref{Ksdefeq}) and explicitly given in Proposition~\ref{GKprop} (non-archimedean case with $n=0$), Theorem \ref{inertintertwiningtheorem} (non-archimedean case with $n>0$) and Theorem \ref{archintertwiningtheorem} (archime\-dean case). The factor $Y(s)$ results from the local zeta integral calculation and is given in Theorem \ref{nonarchlocalzetatheorem} (non-archimedean case) and Corollary \ref{archlocalzetatheoremcor} (archimedean case). The factor $\hat Y(s)$ is similar to $Y(s)$, but with the data $(\chi,\chi_0,\tau)$ replaced by $(\bar\chi^{-1},\chi\bar\chi\chi_0,\chi\tau)$.

\begin{lemma}\label{Xnonarchlemma}
 Assume that $F$ is $p$-adic. Let $\delta$ be the valuation of the discriminant of $L/F$ if $L/F$ is a ramified field extension, and $\delta=0$ otherwise. Let $X(s)$ be as in (\ref{Xlocaldefeq}). Let $\p^n$ be the conductor of $\tau$. Assume that the restriction of $\Lambda$ to $F^\times$ is trivial\footnote{If the $\GSp_4(F)$ representation $\pi$ has a $(S,\Lambda,\psi)$ Bessel model, this means that the central character of $\pi$ is trivial.}, so that $\chi\big|_{F^\times}=\omega_\tau^{-1}$.
 \begin{enumerate}
  \item If $\tau=\beta_1\times\beta_2$ with unramified characters $\beta_1$ and $\beta_2$ of $F^\times$, then
   \begin{equation}\label{Xvcalclemmaeq1}
    X(s)=\chi(\varpi)^\delta\chi_{L/F}(-1)q^{-6\delta s}.
   \end{equation}
  \item If $L/F$ is an unramified field extension or $L=F\oplus F$, then
   \begin{equation}\label{Xvcalclemmaeq2}
    X(s)=\omega_\tau(c^2/d)\varepsilon(1/2,\tilde\tau,\psi^{-1})^4q^{-12ns}.
   \end{equation}

 \end{enumerate}
\end{lemma}
\begin{proof}
i) By Proposition \ref{GKprop},
$$
 X(s)=q^{-\delta}\,\varepsilon(6s,\chi|_{F^\times},\psi^{-1})\,
   \varepsilon(3s,\tau\times \AI(\Lambda) \times \chi|_{F^\times},\psi^{-1})
   \frac{\hat Y(-s)}{Y(s)}.
$$
For unramified $\tau$ we have $Y(s)=1$, and the character $\chi|_{F^\times}$ is unramified. Hence

$$
 X(s)=q^{-\delta}\,\varepsilon(3s,\chi\tau\times\mathcal{AI}(\Lambda),\psi^{-1})
  =\chi(\varpi)^\delta\chi_{L/F}(-1)q^{-6\delta s}.
$$
ii)  In the case of $\tau$ being a spherical representation, (\ref{Xvcalclemmaeq2}) follows from (\ref{Xvcalclemmaeq1}). We may therefore assume that $n>0$. Using standard properties of the $\varepsilon$-factors, we can check that
$$
 \varepsilon(6s,\chi|_{F^\times},\psi^{-1})\,
   \varepsilon(3s,\tau\times \AI(\Lambda) \times \chi|_{F^\times},\psi^{-1})
$$
equals
$$
 \chi_{L/F}(\varpi)^n q^{-(6s-1)(n + a(\omega_\tau)) - \frac{a(\omega_\tau)}{2}} \varepsilon(\frac12,\chi|_{F^\times},\psi^{-1})\,\varepsilon(\frac12,\tilde\tau,\psi^{-1}).
$$
Now the lemma follows directly from Theorem \ref{inertintertwiningtheorem} and Theorem \ref{nonarchlocalzetatheorem}. We note here that, in the case under consideration, we have $n > 0$ and $\big(\frac L\p\big)=\pm1$, so that $Y(s) =  L(6s+1,\chi|_{F^\times})$ .
\end{proof}

\begin{lemma}\label{Xarchlemma}
 Assume that $F=\R$. Let $X(s)$ be as in (\ref{Xlocaldefeq}). Assume that $\pi$ is the lowest weight representation of $\PGSp_4(\R)$ with scalar minimal $K$-type $(l,l)$, where $l\geq2$. We assume that $\Lambda=1$, so that $\chi\big|_{\R^\times}=\omega_\tau^{-1}$. Then
 $$
  X(s)=-\omega_\tau(-D)^{-1}\,\varepsilon(s,\tilde\pi\times\tilde\tau,\psi^{-1})D^{6s}.
 $$
\end{lemma}
\begin{proof}
The ingredients in the definition (\ref{Xlocaldefeq}) of the $X$-factor are all known; see Theorem \ref{archintertwiningtheorem} for the factor $K(s)$, Corollary \ref{archlocalzetatheoremcor} for the factor $Y(s)$, and the tables in Sect.\ \ref{localzetasec} for the $L$- and $\varepsilon$-factors of $\tau\times\AI(\Lambda)\times\chi|_{\R^\times}$ and $\tilde\pi\times\tilde\tau$. The asserted formula is then obtained by going through the various possibilities for the type of representation $\tau$ and the parity of $l$, substituting the ingredients and simplifying. This is where Lemma \ref{kappaalemma} is used. We omit the details.
\end{proof}
\subsubsection*{The global functional equation}
We can now prove the global functional equation for many of the $L$-functions $L(s,\pi\times\tau)$, provided that the $\GSp_4$ representation $\pi$ is of the type considered before and has an appropriate global Bessel model. Once we complete the transfer to $\GL_4$, we will be able to remove all restrictions on the $\GL_2$ representation $\tau$; see Theorem \ref{Lrhonsigmaranalyticpropertiestheorem}.

\begin{theorem}[Functional Equation]\label{functionalequationtheorem}
 Assume that the positive integer $D$ is such that $-D$ is the discriminant of the number field $L:=\Q(\sqrt{-D})$. Let $S(-D)$ be as in (\ref{SminusDdefeq}). Let $\Lambda=\otimes\Lambda_w$ be a character of $L^\times\backslash\A_L^\times$ such that $\Lambda_\infty=1$ and such that $\Lambda_v$ is unramified for all finite places $v$. Let $\pi=\otimes\pi_v$ be a cuspidal, automorphic representation of $\GSp_4(\A)$ with the following properties.
 \begin{enumerate}
  \item $\pi$ has trivial central character;
  \item There exists an integer $l\geq2$ such that $\pi_\infty$ is the (limit of) discrete series representation of $\PGSp_4(\R)$ with scalar minimal $K$-type $(l,l)$;
  \item $\pi_p$ is unramified for all primes $p$;
  \item $\pi$ has a global Bessel model of type $(S(-D),\Lambda,\psi)$ (see Sect.\ \ref{besselsec}).
 \end{enumerate}
 Let $\tau=\otimes\tau_v$ be a cuspidal, automorphic representation of $\GL_2(\A)$ such that $\tau_p$ is unramified for the primes $p$ dividing $D$. Then $L(s,\pi\times\tau)$ has meromorphic continuation to all of $\C$ and satisfies the functional equation
 \begin{equation}\label{functionalequationtheoremeq2}
  L(s,\pi\times\tau)=\varepsilon(s,\pi\times\tau)L(1-s,\tilde\pi\times\tilde\tau).
 \end{equation}
 Here, $\varepsilon(s,\pi\times\tau)=\prod_v\varepsilon(s,\pi_v\times\tau_v,\psi_v^{-1})$, and the local $\varepsilon$-factors are the ones attached to $\pi_v\times\tau_v$ via the local Langlands correspondence.
\end{theorem}

{\bf Remark:} The hypothesis on $\tau$ will be removed later; see Theorem \ref{Lrhonsigmaranalyticpropertiestheorem} for a statement where $\tau$ is any cuspidal representation on any $\GL_n$.

\begin{proof}
Note that $D=-\mathbf{d}$. Let the characters $\chi_0$, $\chi$ and $\Lambda$ of $L^\times\backslash\A_L^\times$ be as in Theorem \ref{globalintegralrepresentationtheorem}. Let $f=\otimes f_v\in I_\C(s,\chi,\chi_0,\tau)$ be the function corresponding to the distinguished vector $W^\#=\otimes W^\#_v$; see the diagram (\ref{locglobindrepdiagrameq}). Let $E(g,s;f)$ be the Eisenstein series defined in (\ref{Edefeq}). By the general theory of Eisenstein series,
\begin{equation}\label{Efctleqeq}
 E(g,s;f)=E(g,-s;M(s)f),
\end{equation}
where $M(s)$ is the global intertwining operator given by a formula similar to (\ref{locintdefeq1}) in the local case. Note that the Eisenstein series on the right hand side of (\ref{Efctleqeq}) is defined with respect to the data $(\bar\chi^{-1},\chi\bar\chi\chi_0,\chi\tau)$ instead of $(\chi,\chi_0,\tau)$; see (\ref{Mstargeteq}). By our uniqueness results Corollary \ref{distinguishedvectornonarchtheorem} and Corollary \ref{distinguishedvectorarchtheoremcor}, and the explicit archimedean calculations resulting in Theorem \ref{archintertwiningtheorem},
\begin{equation}\label{Ksfeq}
 M(s)f(\,\cdot\,,s,\chi,\chi_0,\tau)
  =K(s)f(\,\cdot\,,-s,\bar\chi^{-1},\chi\bar\chi\chi_0,\chi\tau),
\end{equation}
where $K(s)=\prod_vK_v(s)$, and the local functions $K_v(s)$ are the same as in (\ref{Ksdefeq}). Hence
\begin{equation}\label{Efctleqeq2}
 E(g,s;f)=K(s)E(g,-s;\hat f),
\end{equation}
where $\hat f$ abbreviates $f(\,\cdot\,,-s,\bar\chi^{-1},\chi\bar\chi\chi_0,\chi\tau)$. For the global zeta integrals defined in (\ref{globalintegraleq}) it follows that
\begin{equation}\label{Zfctleqeq}
 Z(s,f,\phi)=K(s)Z(-s,\hat f,\phi).
\end{equation}
By the basic identity (\ref{basicidentityeq}),
\begin{equation}\label{Zfctleqeq2}
 Z(s,W^\#,B_\phi)=K(s)Z(-s,\hat W^\#,B_\phi),
\end{equation}
where $\hat W^\#$ abbreviates $W^\#(\,\cdot\,,-s,\bar\chi^{-1},\chi\bar\chi\chi_0,\chi\tau)$. Now we let $B_\phi$ be the distinguished Bessel vector as in Theorem \ref{globalintegralrepresentationtheorem}, and apply this theorem to both sides of (\ref{Zfctleqeq2}). The result is
\begin{align}\label{Lfctleqeq1}
 &\frac{L(3s+\frac 12, \tilde\pi \times \tilde\tau)}{L(6s+1,\chi|_{\A^\times})
   L(3s+1,\tau \times \AI(\Lambda) \times \chi|_{\A^\times})}Y(s)\nonumber\\[2ex]
 &\hspace{10ex}= K(s)\frac{L(-3s+\frac 12, \tilde\pi \times \widetilde{\chi\tau})}
   {L(-6s+1,\chi^{-1}|_{\A^\times})
   L(-3s+1,\chi\tau \times \AI(\Lambda) \times \chi^{-1}|_{\A^\times})}\hat Y(-s).
\end{align}
Note that $\Lambda(\zeta)=\chi_0(\zeta)^{-1}\chi(\bar\zeta)^{-1}$, and this character does not change under $(\chi,\chi_0)\mapsto(\bar\chi^{-1},\chi\bar\chi\chi_0)$. However, since $\Lambda^{-1}=\bar\Lambda$, we have $\mathcal{AI}(\Lambda)=\mathcal{AI}(\bar\Lambda)=\mathcal{AI}(\Lambda^{-1})$.
Using $\chi\tau\cong\tilde\tau$ and the global functional equations for characters and for representations of $\GL_2\times\GL_2$ (see \cite{ja}), we can rewrite (\ref{Lfctleqeq1}) as
\begin{align}\label{Lfctleqeq3}
 \frac{L(3s+\frac 12, \tilde\pi \times \tilde\tau)}{L(-3s+\frac 12, \tilde\pi \times \tau)}
   &=K(s)\frac{L(6s+1,\chi|_{\A^\times})
   L(3s+1,\tau \times \AI(\Lambda) \times \chi|_{\A^\times})}
   {L(6s,\chi|_{\A^\times})
   L(3s,\tau \times \AI(\Lambda) \times \chi|_{\A^\times})}\nonumber\\
  &\qquad\times\varepsilon(6s,\chi|_{\A^\times})\,
   \varepsilon(3s,\tau \times \AI(\Lambda) \times \chi|_{\A^\times})
   \frac{\hat Y(-s)}{Y(s)}\nonumber\\
  &=\prod_vX_v(s),
\end{align}
with local quantities $X_v(s)$ as in (\ref{Xlocaldefeq}). These quantities were calculated in Lemmas \ref{Xnonarchlemma} and \ref{Xarchlemma}. For a prime $p$ let $\delta_p$ be the $p$-valuation of $D$, so that $D=\prod_pp^{\delta_p}$. Let $n_p$ be the conductor of $\tau_p$. Let $S$ be the finite set of primes $p$ such that $\tau_p$ is not unramified. By hypothesis, if $p\in S$, then $p\nmid D$, i.e., $L_p/\Q_p$ is not a ramified field extension. Using Lemmas \ref{Xnonarchlemma} and \ref{Xarchlemma}, and the fact that $\varepsilon(s,\tilde\pi_p\times\tilde\tau_p,\psi_p^{-1})=\varepsilon(s,\tilde\tau_p,\psi_p^{-1})^4$ for all finite places $p$, a straightforward calculation shows that
$$
 \prod_vX_v(s)=\varepsilon(3s+1/2,\tilde\pi\times\tilde\tau).
$$
Replacing $s$ by $\frac13s-\frac16$ and $\tau$ by $\tilde\tau$, and observing that $\pi$ is self-contragredient, we obtain the claim of the theorem.
\end{proof}

%
\section{The pullback formula}

In this section, we prove a second integral representation for our $L$-function. This is achieved via the ``pullback formula", which expresses the (relatively complicated) Eisenstein series $E(g,s;f)$, defined in (\ref{Edefeq}), as the inner product of an automorphic form in the space of $\tau$ with the pullback of a simple Siegel-type Eisenstein series on $G_3$.

\vspace{3ex}
We will first prove a local version of the pullback formula. This is the key technical ingredient behind the (global) pullback formula, which, when coupled with the results of the previous sections, will lead to the second integral representation. This will be crucial for proving the entireness of the $\GSp_4\times\GL_2$ $L$-function $L(s,\pi \times \tau)$.
\subsection{Local sections: non-archimedean case}\label{pullbacknonarchsectionssec}
Let $F$ be $p$-adic. We use the notation from Theorem \ref{unique-W-theorem}. In addition, we will assume that $\Lambda\big|_{F^\times}=1$. We define the principal congruence subgroup
\begin{equation}\label{principalcongruencesubgroupeq}
 \Gamma^{(3)}(\P^n)=\{g\in G_3(\OF)\;|\;g\equiv1\;{\rm mod}\;\P^n\},
\end{equation}
and consider the subgroup
\begin{equation}\label{N1defeq}
 N_1(\OF)=\iota(\mat{1}{\OF}{}{1},1)
\end{equation}
(see (\ref{GUembeddingeq}) for the definition of the embedding $\iota$). The group $N_1(\OF)$ is normalized by the group
$$
 \tR(\OF) = \Big\{\i(\begin{bmatrix}1&\\& \lambda\end{bmatrix}, h)\;|\; h \in H(\OF), \lambda = \mu_2(h)\Big\}.
$$
As before, let $n$ be such that $\p^n$ is the conductor of $\tau$. Define the congruence subgroup $C(\P^n)$ of $G_3(\OF)$ by
\begin{equation}\label{candidategroupdefeq}
 C(\P^n):=\tR(\OF)N_1(\OF)\Gamma^{(3)}(\P^n).
\end{equation}
Note that this is really a group, since $\Gamma^{(3)}(\P^n)$ is normal in
the maximal compact subgroup $G_3(\OF)$.

We note here an alternate description of $C(\P^n)$ that will be useful: It consists of precisely the matrices $g \in G_3(\OF)$ that satisfy
\begin{equation}\label{C-looks-like-this}
 g \equiv \begin{bmatrix}\OF &\OF& &\OF& \OF&\\
\OF &\OF &&\OF &\OF &\\&& 1&&& \OF\\
\OF &\OF &&\OF &\OF &\\\OF &\OF &&\OF &\OF &\\&&&&&\OF^\times \end{bmatrix} \pmod{\P^n}.
\end{equation}
We define $\widetilde{\chi}$ to be the character on $P_{12}$ (see (\ref{P12defeq})) given by
\begin{equation}\label{localtildechidefeq}
 \widetilde{\chi}(m(A,v)n) = \chi(v^{-1}\det(A)).
\end{equation}
For $s \in \C,$ we form the induced representation
\begin{equation}\label{Ichisdefeq}
 I(\widetilde{\chi},s) = \text{Ind}_{P_{12}(F)}^{G_3(F)} \big(\widetilde{\chi}\,\delta_{12}^s\big)
\end{equation}
(see \eqref{P12modulareq}),
consisting of smooth functions $\Xi$ on $G_3(F)$ such that
\begin{equation}\label{e:upsilondeflocalformula}
 \Xi(n_0m(A,v)g,s) = |v|^{-9(s +\frac{1}{2})}|N (\det A) |^{3(s + \frac{1}{2})}
 \chi(v^{-1}\det A)\,\Xi(g,s)
\end{equation}
for $n_0\in N_{12}(F)$, $m(A,v)\in M_{12}(F)$, $g\in G_3(F)$. For any $t \in L$, set
$$
 \Omega(t):= \begin{bmatrix}1\\&1\\&&1\\&\alpha&t&1\\\bar\alpha&&&&1\\\bar{t}&&&&&1\end{bmatrix},
$$
where $\alpha$ is the element defined in (\ref{alphadefeq}).
We define
$$
 I_L=\begin{cases}
    \{\varpi^r\;|\;0\leq r\leq n\}&\text{if }\big(\frac L\p\big)=-1,\\[1ex]
    \{(\varpi^{r_1},\varpi^{r_2})\;|\;0\leq r_1,r_2\leq n\}&\text{if }\big(\frac L\p\big)=1,\\[1ex]
    \{\varpi_L^r\;|\;0\leq r\leq2n\}&\text{if }\big(\frac L\p\big)=0.
   \end{cases}
$$
From  Lemma \ref{crucialproperty} below it follows that there exists, for each $t\in I_L$, a unique well-defined section $\Upsilon_t\in I(\widetilde{\chi},s)$ satisfying all of the following,
\begin{enumerate}
\item $\Upsilon_t(\Omega(t), s) =1$,
\item $\Upsilon_t(gk,s)=\Upsilon_t(g,s)$ for all $g \in G_3(F)$, $k\in C(\P^n)$,
\item $\Upsilon_t(g,s) = 0$ if $g \notin P_{12}(F)\Omega(t) C(\P^n)$.
\end{enumerate}
We define $\Upsilon\in I(\widetilde{\chi},s)$ by \begin{equation}\label{upsilondefinert}
 \Upsilon= \sum_{t\in I_L} \Upsilon_t.
\end{equation}
\begin{lemma}\label{crucialproperty}
 Let $A \in \GL_3(F)$, $v\in F^\times$, $n_0\in N_{12}(F)$ and $t \in \OF_L$ be such that $$
  \Omega(t)^{-1}n_0m(A,v)\Omega(t) \in C(\P^n).
 $$
 Then
 $$
  v^{-1}\det(A) \in (1+\P^n)\cap\OF_L^\times.
 $$
\end{lemma}
\begin{proof} Since the statement is trivial for $n=0$, we will assume $n>0$. Let $P=n_0 m(A,v) = \begin{bmatrix}A & B\\0&v\; ^t\!\bar{A}^{-1}\end{bmatrix}$ with
$A=\begin{bmatrix}a_1&a_2&a_3\\a_4&a_5&a_6\\a_7&a_8&a_9\end{bmatrix} \mbox{ and } B = \begin{bmatrix}b_1&b_2&b_3\\b_4&b_5&b_6\\b_7&b_8&b_9\end{bmatrix}$. Note that $A^{-1}B$  is self-adjoint; however we won't use this. Suppose $M := \Omega(t)^{-1}n_0m(A,v)\Omega(t) \in C(\P^n)$. This implies that $A \in \GL_3(\OF_L)$, $v\in \OF^\times$ and $n_0\in N_{12}(\OF)$. Let us set $d := \det(A) \in \OF_L^\times$. We will use the description given in (\ref{C-looks-like-this}) for a matrix in $C(\P^n)$. Since the $(1,6),(2,6),(3,4),(3,5)$ entries of $M$ are in $\P^n$, we obtain $b_3,b_6,b_7,b_8\in \P^n$. Looking at the $(3,2), (3,3)$ entries of $M$, we obtain $a_8 \in \P^n$ and $a_9 \in 1+\P^n$. Looking at the $(5,6)$ entry of $M$ and using the fact that $v,d \in \OF_L^\times$ we deduce that $a_2a_7 \in \P^n$. Calculating the determinant of $A$ along the third row, we obtain
$$
 d = a_7(a_2a_6-a_3a_5) - a_8(a_1a_6-a_3a_4) + a_9(a_1a_5-a_2a_4) \equiv a_1a_5-a_2a_4 - a_7a_3a_5 \pmod{\P^n}.
$$
Since $d \in \OF_L^\times$, it follows that either $a_2$ or $a_5$ is a unit. Set
$$
 g_2 := \mat{a_2+\alpha b_1}{b_1}{-\bar\alpha(a_2+\alpha b_1)-\alpha \bar{a}_2 \frac vd}{\,\,\,\,-\bar\alpha b_1 - \bar{a}_2 \frac vd}, \quad g_5 := \mat{a_5+\alpha b_4}{b_4}{\alpha(-a_5-\alpha b_4 + \bar{a}_5 \frac vd)}{\,\,\,\,-\alpha b_4 + \bar{a}_5 \frac vd}.
$$
Since $g_2$ and $g_5$ are submatrices of $M$ mod $\P^n$, they have entries in $\OF + \P^n$. The following simple fact,
\begin{equation}\label{well-defined-sub-lemmaeq1}
 \text{If }x \in \OF + \P^n,\text{ then }x \equiv \bar{x} \pmod{(\alpha-\bar\alpha)\P^n},
\end{equation}
applied to the entries of $g_2$ resp.\ $g_5$, leads to the desired conclusion.
\end{proof}
\subsection{The local pullback formula: non-archimedean case}\label{pullbacknonarchsec}
In this subsection, we will prove the local pullback formula in the non-archimedean case. Recall the congruence subgroups defined in (\ref{congruencesubgroupeq1}) -- (\ref{congruencesubgroupeq4}). We note that \begin{equation}\label{upsilrightinvar}\Upsilon(g \cdot \i(k_1,k_2),s) = \Upsilon(g,s)\end{equation} for any pair of elements  $k_1 \in K^{(1)}(\P^n)$, $k_2\in K^H\Gamma(\P^n)$, satisfying $\mu_1(k_1) = \mu_2(k_2)$. This follows from the right-invariance of $\Upsilon$ by $C(\P^n)$. Let $Q$ be the element
\begin{equation} \label{Q-formula}
 Q=\begin{bmatrix}0&1&&&&\\1&0&&&&\\&&0&&&-1\\
  &&&0&1&-1\\&&&1&0&\\&1&1&&&0\end{bmatrix}
    \in G_3(F).
\end{equation}
For $g=\mat{a}{b}{c}{d}$ and $m_2(g)$ as in (\ref{m1m2defeq}),
\begin{equation}\label{QQinversepropertyeq}
 Q\cdot \i(g,m_2(g)) \cdot Q^{-1} =  \begin{bmatrix}a&&-b&b&&\\&1&&& &\\-c&&d&&&c\\
  &&&d&&c\\&&&&\mu_1(g)&\\&&&b&&a\end{bmatrix},
\end{equation}
where the matrix on the right side lies in $P_{12}$. It follows that if $g \in G_1(\OF)$, then for any $h \in G_3(F)$,
\begin{equation}\label{upsilleftinvar}
 \Upsilon(Q \cdot \i(g,m_2(g))h,s) = \chi(\mu_1(g)^{-1}\det(g))\Upsilon(Qh,s).
\end{equation}
Let $W^{(0)}$ be the local newform for $\tau$, as in Corollary \ref{distinguishedvectornonarchtheorem}. For each $0 \le m \le n$, let the elements $\eta_m$ be as in \eqref{etamdef2eq}. The main object of study for the local pullback formula is the following local zeta integral,
\begin{equation}\label{e:deflocalzeta}
 Z(g,s;g_2) = q(n) \int\limits_{\U(1,1)(F)}\Upsilon(Q\cdot \i(h,g_2),s)
 W^{(0)}(gh)\,\chi^{-1}(\det(h))\,dh,
\end{equation}
where $g \in G_1(F)$, $g_2 \in \U(2,2)(F)$ and $q(n)$ is a normalizing factor equal to $[G_1(\OF) : K^{(0)}(\P^n)]^{-1}$. The above integral converges absolutely for $\Re(s)$ sufficiently large.

\begin{theorem}[Non-archimedean Local Pullback Formula]\label{theorem-local-pullback-nonarch} Let $0 \le m \le n$. Then, for $\Re(s)$ sufficiently large, $$Z(g, s; \eta_m) = \begin{cases}0 & \text{ if }0 <m \le n,\\ T(s)W^{(0)}(g) & \text{ if }m=0,\end{cases}$$ where the factor $T(s)$ satisfies
$$
 T(s) Z(s,W^\#,B)  = \begin{cases}
 \displaystyle\frac{L(3s+\frac 12, \tilde\pi \times \tilde\tau)}{L(6s+1,\chi|_{F^\times})L(6s+2,\chi_{L/F}\chi|_{F^\times})L(6s+3,  \chi|_{F^\times})} & \text{if } n = 0, \\[4ex]
 L(3s+\frac 12, \tilde\pi \times \tilde\tau) & \text{if } n > 0.
\end{cases}
$$
Here, $Z(s,W^\#,B)$ is the local integral computed in Theorem \ref{nonarchlocalzetatheorem}.
\end{theorem}
The proof of Theorem \ref{theorem-local-pullback-nonarch} will require the following lemmas.
\begin{lemma}\label{upsilondoublecoset}
 As a function of $h$, the quantity $\Upsilon(Q\cdot \i(h,\eta_m),s)$ depends only on the double coset
 $K_1^{(1)}(\P^n)hK_1^{(1)}(\P^n)$.
\end{lemma}
\begin{proof}The right invariance by $K_1^{(1)}(\P^n)$ follows easily from the right invariance of $\Upsilon$ by $C(\P^n)$. On the other hand, given $k\in K_1^{(1)}(\P^n)$, we have
\begin{align*}\Upsilon(Q\cdot \i(kh,\eta_m),s)&=\Upsilon(Q\cdot \i(kh,m_2(k)m_2(k)^{-1}\eta_m),s)\\&=\Upsilon(Q\cdot \i(h,m_2(k)^{-1}\eta_m),s)\\&=\Upsilon(Q\cdot \i(h,\eta_m \eta_m^{-1}m_2(k)^{-1}\eta_m),s)\\&=\Upsilon(Q\cdot \i(h,\eta_m),s)
\end{align*}
Note that we have used~\eqref{upsilrightinvar}, \eqref{upsilleftinvar} and the fact that $\eta_m^{-1}m_2(k)\eta_m\in K^H\Gamma(\P^n)$.
\end{proof}

Next, we note down the Cartan decompositions for $\U(1,1)(F)$. These follow directly from the Cartan decomposition for $\GL_2(F)$.
Suppose $\big(\frac{L}{\p}\big) = -1 $. Then
\begin{equation}\label{inert-Cartan-decomp}
 \U(1,1;L)(F) = \bigsqcup_{t \ge 0}K_1^{(1)}(1)A_tK_1^{(1)}(1),
 \qquad \mbox{ where }A_t =\begin{bmatrix}\varpi^t&0\\0&\varpi^{-t}\end{bmatrix}.
\end{equation}
Suppose $\big(\frac{L}{\p}\big) = 1 $. Then
\begin{equation}\label{split-Cartan-decomp}
 \U(1,1;L)(F) = \bigsqcup\limits_{t_1 \geq t_2} K_1^{(1)}(1) A_{t_1,t_2} K_1^{(1)}(1),
 \qquad \mbox{ where } A_{t_1, t_2} = \mat{\varpi_L^{t_1} \bar{\varpi}_L^{-t_2}}{}{}{\varpi_L^{t_2} \bar{\varpi}_L^{-t_1}}.
\end{equation}
Suppose $\big(\frac{L}{\p}\big) = 0 $. Then
\begin{equation}\label{ramified-Cartan-decomp}
 \U(1,1;L)(F) = \bigsqcup\limits_{t\geq0} K_1^{(1)}(1) A_tK_1^{(1)}(1),
 \qquad \mbox{ where } A_t= \mat{\varpi_L^t}{}{}{\bar\varpi_L^{-t}}.
\end{equation}

\begin{lemma}\label{zlmprop}
 For each $0 \le m \le n$, there exists a function $L_{m}(s)$, depending on the local data ($F$, $L$, $\chi_0$, $\chi$, $\tau$) but independent of $g$, such that, for $\Re(s)$ sufficiently large,
 $$
  Z(g,s;\eta_m)= L_{m}(s) W^{(0)}(g)
 $$
 for all $g \in G_1(F)$.
\end{lemma}
\begin{proof}
We will only give the proof for the cases $\big(\frac{L}{\p}\big) = -1$ or $0$; the proof for the split case $\big(\frac{L}{\p}\big) = 1$ is obtained by replacing $A_t$ by $A_{t_1, t_2}$ everywhere below. Recall that $V_\tau$ is the space of Whittaker functions on $\GL_2(F)$ realizing the representation $\tau$ with respect to the character $\psi^{-\mathbf{c}}$. $W^{(0)}(g)$ is (up to a constant) the unique function in $V_\tau$ that is right-invariant by $K^{(1)}(\p^n)$. Observe that, by \eqref{inert-Cartan-decomp} resp.\ \eqref{ramified-Cartan-decomp}, we can write
\begin{equation}\label{localzetasplitt}
 q(n)^{-1}Z(g,s;\eta_m)=\sum_{t\ge 0}\;\int\limits_{K_1^{(1)}(1)
 A_tK_1^{(1)}(1)}\Upsilon(Q\cdot \i(h,\eta_m),s)
 W^{(0)}(gh)\,\chi^{-1}(\det h)\,dh.
\end{equation}
For $g \in G_1(F)$, denote
\begin{equation}\label{defit}
 I_{t}(g;s) = \int\limits_{K_1^{(1)}(1)
 A_tK_1^{(1)}(1)}\Upsilon(Q\cdot \i(h,\eta_m),s)
 W^{(0)}(gh)\,\chi^{-1}(\det h)\,dh.
\end{equation}
By writing $K_1^{(1)}(1)A_tK_1^{(1)}(1)$ as a finite disjoint union $\bigsqcup_\gamma \gamma K_1^{(1)}(\P^n)$ and using Lemma~\ref{upsilondoublecoset}, we see that $I_{t}$ is a finite sum of right translates of $W^{(0)}$. Thus, $I_{t}$ lies in $V_\tau$ for each $t$. In fact, we will show that $I_{t}$ is a multiple of $W^{(0)}$. Let $k \in K_1^{(1)}(\P^n)$. By a change of variables, and using Lemma~\ref{upsilondoublecoset}, we see that \begin{equation}\label{irfirstinvar}
 I_{t}(gk,s) = I_{t}(g,s).
\end{equation}
Next, let $l \in \OF^\times$ and put $k_l = \begin{bmatrix}1&\\&l\end{bmatrix}$.
Then
\begin{align}\label{irtorusinvar2}
 I_{t}(gk_l, s)&=\int\limits_{K_1^{(1)}(1)A_tK_1^{(1)}(1)}\Upsilon(Q\cdot \i(h,\eta_m),s)
  W^{(0)}(gk_lh)\chi^{-1}(\det h)\,dh\nonumber\\
  &=\int\limits_{K_1^{(1)}(1)A_tK_1^{(1)}(1)}\Upsilon(Q\cdot \i(k_l^{-1}hk_l,m_2(k_l)^{-1}
  (m_2(k_l)\eta_mm_2(k_l)^{-1})m_2(k_l)),s)\nonumber\\
 &\qquad \qquad \times W^{(0)}(gh)\chi^{-1}(\det h)\,dh\nonumber\\
 &=\int\limits_{K_1^{(1)}(1)A_tK_1^{(1)}(1)}\Upsilon(Q\cdot \i(h,\eta_m),s)
  W^{(0)}(gh)\chi^{-1}(\det h)\,dh.
\end{align}
In the last step above we used (\ref{upsilleftinvar}) and the fact that $m_2(k_l)\eta_mm_2(k_l)^{-1} = \eta_m$. The above calculations show that \begin{equation}\label{irtorusinvar}I_{t}(gk_l, s)= I_{t}(g,s)\end{equation} for all $l \in \OF^\times$. From this and~\eqref{irfirstinvar}, we conclude that $I_{t}(gk,s) = I_{t}(g,s)$ for all $k\in K^{(1)}(\P^n)$. The fact that the conductor of $\tau$ equals $\p^n$  implies that, for each $s$, the function $I_{t}(\cdot, s)$ is a multiple of $W^{(0)}$. Now the assertion follows immediately from~\eqref{localzetasplitt} and~\eqref{defit}.
\end{proof}

\vspace{1ex}\noindent{\bf Proof of Theorem~\ref{theorem-local-pullback-nonarch}.}\hspace{0.5em}	 
Let us first prove that $Z(g, s; \eta_m) = 0$ for $0 <m \le n$. We assume $n >0$ as otherwise the assertion is vacuous. Recall from Lemma~\ref{zlmprop} that, for each $s$, the function $Z(g,s; \eta_m)$ restricted to $\GL_2(F)$ lies in $V_\tau$.
Using $\eta_m^{-1}m_2(k)\eta_m\in K^H\Gamma(\P^n)$ for $k \in K^{(1)}(\p^{n-m})$,
and a similar calculation as in (\ref{irtorusinvar2}), we get
$$
 Z(gk,s; \eta_m)=Z(g,s; \eta_m)
$$
for any $k \in K^{(1)}(\p^{n-m}) \cap \SL_2(\OF)$. Together with \eqref{irtorusinvar} it follows that $Z(g, s; \eta_m)$ is right invariant under $K^{(1)}(\p^{n-m})$. However, because the conductor of $\tau$ is $n$, $V_\tau$ does not contain any non-zero function that is right invariant under $K^{(1)}(\p^{n-m})$ for $m>0$. This proves that $Z(g, s; \eta_m) = 0$ whenever $m>0$.

\vspace{3ex}
For the rest of this proof, we assume that $m=0$, so $\eta_m = \eta$. Our task is to evaluate $Z(g,s; \eta)$. We first consider the case $\big(\frac{L}{\p}\big) = -1 $. For $l \in L$, we use $\widetilde{l}$ to denote the element $\begin{bmatrix}l & \\&\bar{l}^{-1}\end{bmatrix}$. It is not hard to prove that the following decomposition holds,
\begin{equation}\label{e:brucartsteinbsteinb}
 \begin{split}\U(1,1)(F)&= \bigsqcup_{l \in \OF_L^\times/(1+\P)}K_1^{(1)}(\P)\widetilde{l}K_1^{(1)}(\P)\quad \sqcup \quad \bigsqcup_{l \in \OF_L^\times/(1+\P)}K_1^{(1)}(\P)w\widetilde{l} K_1^{(1)}(\P)\\ &\sqcup \quad \bigsqcup_{\substack{t>0\\l \in \OF_L^\times/(1+\P)}} K_1^{(1)}(\P)A_t\widetilde{l} K_1^{(1)}(\P) \quad \sqcup \quad \bigsqcup_{\substack{t>0\\l \in \OF_L^\times/(1+\P)}} K_1^{(1)}(\P)A_t w\widetilde{l}K_1^{(1)}(\P)  \\ &\sqcup \quad \bigsqcup_{\substack{t>0\\l \in \OF_L^\times/(1+\P)}} K_1^{(1)}(\P)wA_t\widetilde{l} K_1^{(1)}(\P) \quad \sqcup\quad \bigsqcup_{\substack{t>0\\l \in \OF_L^\times/(1+\P)}} K_1^{(1)}(\P)wA_t w\widetilde{l}K_1^{(1)}(\P),\end{split}
\end{equation}
where  $w=\begin{bmatrix}&1\\-1&\end{bmatrix}$. We have the following facts about the support of $\Upsilon$,
\begin{align}\label{support1}
  Q\cdot \i(W,\eta) \notin P_{12}(F)\Omega(u)C(\P),
  & \qquad  \text{ for } W \in \{ A_tw\widetilde{l},\,wA_t\widetilde{l},\, wA_tw\widetilde{l}\;|\;t\in\Z_{>0}\}, \ u\in I_L,\\
  \label{support4}Q\cdot \i(w\widetilde{l},\eta)\notin P_{12}(F)\Omega(u)C(\P),
  & \qquad  \text{ for } \ u\in I_L.
\end{align}
The statements \eqref{support1} and \eqref{support4} are proved by direct computations involving the relevant $6 \times 6$ matrices; we omit the details. From the above statements, Lemma~\ref{zlmprop}, and~\eqref{e:brucartsteinbsteinb}, we see that
$$
 Z(g,s; \eta) = q(n) \frac{W^{(0)}(g)}{W^{(0)}(1)}\sum_{t \ge 0}
  \sum_{l \in \OF_L^\times/(1+\P)}\;\int\limits_{K_1^{(1)}(\P)A_t\tl K_1^{(1)}(\P)}
 \Upsilon(Q\cdot \i(A_t\tl,\eta),s)
 W^{(0)}(h)\chi^{-1}(\det h)\,dh.
$$
Now, we have the non-disjoint double coset decomposition
\begin{equation}
 \bigsqcup\limits_{l \in \OF_L^\times/(1+\P)} K_1^{(1)}(\P) A_t \tilde{l}
 K_1^{(1)}(\P) = \bigsqcup\limits_{k=1}^n \bigsqcup_{l \in \OF_L^\times/(1+\P^n)}
 \bigcup\limits_{\substack{y\in \p/\p^n\\v(y)=k}}
 K_1^{(1)}(\P^n) A_{t}\begin{bmatrix}1\\y&1\end{bmatrix}\tilde{l}  K_1^{(1)}(\P^n).
\end{equation}
Again, by explicit calculation, one verifies that $Q\cdot\i(A_t\begin{bmatrix}1\\y&1\end{bmatrix}\widetilde{l},\eta)$ does not belong to any of the sets $P_{12}(F)\Omega_r C(\P^n)$ if $v(y)<n$. It follows that
\begin{align*}
 Z(g,s; \eta) &= q(n) \frac{W^{(0)}(g)}{W^{(0)}(1)}\;\sum_{t \ge 0}\;
  \sum_{l \in \OF_L^\times/(1+\P^n)}\Upsilon(Q\cdot \i(A_t\tl,\eta),s)
  \int\limits_{K_1^{(1)}(\P^n)
  A_t\tl K_1^{(1)}(\P^n)}\!W^{(0)}(h)\chi^{-1}(\det h)\,dh\\
 &= q(n) \frac{W^{(0)}(g)}{W^{(0)}(1)}\;\sum_{t \ge 0}\;
  \sum_{l \in \OF_L^\times/(1+\P^n)}\Upsilon(Q\cdot \i(A_t\tl,\eta),s)
  \chi(l^{-1})\int\limits_{K_1^{(1)}(\P^n)A_t K_1^{(1)}(\P^n)}W^{(0)}(h)\,dh\\
 & = T(s)W^{(0)}(g),
\end{align*}
where
\begin{equation}\label{zetafinalexpr}
 T(s) =  \frac{q(n)}{W^{(0)}(1)}\;\sum_{t \ge 0}\;\sum_{l \in \OF_L^\times/(1+\P^n)}
 \Upsilon(Q\cdot \i(A_t\tl,\eta),s)\chi(l^{-1})\int\limits_{K_1^{(1)}(\P^n)
 A_t K_1^{(1)}(\P^n)}W^{(0)}(h)\,dh.
\end{equation}
To evaluate $T(s)$, we note from the theory of Hecke operators on $\GL_2(F)$ that \begin{equation}\label{GL2U11heckelemmaeq2}
    \int\limits_{K^{(1)}_1(\P^n)A_{t}K^{(1)}_1(\P^n)}\tau(h)W^{(0)}\,dh=
    {\rm vol}(K^{(1)}_1(\P^n))\lambda_{t}W^{(0)},
\end{equation}
where $\lambda_t$ depends on $t$ and $\tau$. Using familiar double coset decompositions, the eigenvalues $\lambda_t$ can easily be calculated. The result is as follows.
\begin{itemize}
 \item If $\tau=\beta_1\times\beta_2$ with unramified characters $\beta_1$, $\beta_2$, then $\lambda_t = \gamma_t - \gamma_{t-1}$ where
   $$
    \gamma_t=q^t\omega_{\tau}(\varpi)^{-t}\:\frac{\beta_1(\varpi)^{2t+1}-\beta_2(\varpi)^{2t+1}}{\beta_1(\varpi) - \beta_2(\varpi)}
   $$
   for $t\geq0$, and $\gamma_t=0$ for $t<0$ (for $\beta_1=\beta_2$, the fraction is to be interpreted as $(2t+1)\beta_1(\varpi)^{2t}$).
  \item If $\tau$ is an unramified twist of the Steinberg representation, then $\lambda_t=1$ for all $t\geq0$.
  \item If $\tau=\beta_1\times\beta_2$ is a principal series representation with
     an unramified character $\beta_1$ and a ramified character $\beta_2$, then
     $\lambda_{t}=q^{t}\beta_1(\varpi)^{-t}\beta_2(\varpi)^{t}$ for all $t\geq0$.
  \item If $\tau$ is supercuspidal, or a ramified twist of the Steinberg representation, or an irreducible principal series representation induced from two ramified characters $\beta_1, \beta_2$, then $\lambda_0 = 1$ and $\lambda_t=0$ for $t>0$.
\end{itemize}
We substitute the above formulas for $\lambda_t$ in the integral inside~\eqref{zetafinalexpr}. Then, we use the definition of $\Upsilon$ to compute the term $\Upsilon(Q\cdot \i(A_t\tl,\eta),s)$; it turns out that
\begin{equation}\label{upsilonevaluationeq}
 \Upsilon(Q\cdot \i(A_t\tl,\eta),s)= q^{-6t(s+ \frac{1}{2})}\chi(l) \chi(\varpi^t).
\end{equation}
After making these substitutions, it is easy to evaluate $T(s)$ for the possible types of $\tau$ listed above simply by summing the geometric series. This proves Theorem~\ref{theorem-local-pullback-nonarch} in the inert case. The proofs for the cases $\big(\frac{L}{\p}\big) = 0 $ or 1 are very similar to the above. The details are left to the reader.
\hfill\qed\vspace{1ex}
\subsection{Local sections: archimedean case}\label{pullbackarchsectionssec}
In this subsection, $F = \R$ and $L = \C$. Let $\tau$ be as in Sect.\ \ref{distvecarchsec}, and let $l_1$ be any of the weights occurring in $\tau$. Let $\chi_0$ be the character of $\C^\times$ such that $\chi_0\big|_{\R^\times}=\omega_\tau$ and $\chi_0(\zeta)=\zeta^{-l_1}$ for $\zeta\in\C^\times,\:|\zeta|=1$. Let $\chi$ be the character of $\C^\times$ given by $ \chi(\zeta)=\chi_0(\overline{\zeta})^{-1}$.

\vspace{3ex}
We define $I(\widetilde{\chi},s)$ in the present (archimedean) case in exactly the same manner as it was defined in the non-archimedean case (see \eqref{Ichisdefeq}, \eqref{e:upsilondeflocalformula}). In this subsection, we will construct a special element of $I(\widetilde{\chi},s)$. Let $\eta_0$ be the matrix defined in (\ref{W0l1l2distvecpropeq2}). For $\theta \in \R$, let
$$
 r(\theta)=\mat{\cos(\theta)}{\sin(\theta)}{-\sin(\theta)}{\cos(\theta)}\in\SO(2),
$$
and
\begin{equation}\label{rtimesdefeq}
 r_\times(\theta)=\begin{bmatrix}\cos(\theta)&&&&&\sin(\theta)\\&1&&&0\\
 &&\cos(\theta)&\sin(\theta)\\&&-\sin(\theta)&\cos(\theta)\\&0&&&1\\
 -\sin(\theta)&&&&&\cos(\theta)\end{bmatrix} \in K^{G_3}_\infty,
\end{equation}
where $K^{G_3}_\infty$ is the maximal compact subgroup of $G_3^+(\R)=\{g\in G_3(\R)\:|\:\mu_3(g)>0\}$. Explicitly,
$$
 K^{G_3}_\infty=\Big\{\mat{A}{B}{-B}{A}\;|\;A,B\in \Mat_{3,3}(\C),\;
 ^t\!\bar AB=\,^t\bar BA,\:^t\!\bar AA+\,^t\bar BB=1\Big\}.
$$
Also, we let
\begin{equation}w_Q=\begin{bmatrix}
&1\\1\\&&&&&-1\\&&&&1\\&&&1\\&&1\end{bmatrix},\qquad s_1=
\begin{bmatrix}&1\\1\\&&1\\&&&&1\\&&&1\\&&&&&1\end{bmatrix},
\end{equation}
and
\begin{equation}\label{tinftyppdefeq}
 t_\infty =w_Q\cdot\i(1,\eta_0)=\i(r(\pi/2),s_1\eta_0).
\end{equation}
Let $l$ be a positive integer (in our application we will consider a discrete series representation of $\PGSp_4(\R)$ with scalar minimal $K$-type $(l,l)$). To ease notation, we will denote the function $\Phi^\#_{|l-l_1|,l,-l_1,l_1}$ defined in (\ref{W0l1l2distvecpropeq1}) by $J_\infty$. Explicitly,
\begin{equation}\label{Jinftydefeq2}\renewcommand{\arraystretch}{1.2}
  J_\infty=\left\{\begin{array}{l@{\qquad\text{if }}l}
  i^{l-l_1}\hat b^{l_1-l}(\hat a\hat d-\hat b\hat c)^{l-l_1}D^{-l}&l\leq l_1,\\
  i^{l_1-l}\hat c^{l-l_1}D^{-l}&l\geq l_1,
  \end{array}\right.
\end{equation}
where $D(g)=\det(J(g,i_2))$, and the functions $\hat a, \hat b,\hat c,\hat d$ are defined before Lemma \ref{Deltaabcdlemma}. Note that $J_\infty(\eta_0) = 1$.

\vspace{3ex}
By the Iwasawa decomposition, $G_3(\R) = P_{12}(\R)K^{G_3}_\infty$. The following lemma provides a criterion for when functions on $K^{G_3}_\infty$ can be extented to nice sections in  $I(\widetilde{\chi},s)$.
\begin{lemma}\label{UpsiloninftyKlemma}Suppose $\Upsilon_\infty$ is an analytic function on $K^{G_3}_\infty$ that satisfies the following conditions.
\begin{enumerate}
 \item For all $A\in \U(3)$ and all $g\in K^{G_3}_\infty$,
      \begin{equation}\label{UpsiloninftyKlemmaeq1}
       \Upsilon_\infty(\mat{A}{}{}{A}g)=\det(A)^{-l_1}\,\Upsilon_\infty(g).
      \end{equation}
    \item For all $\theta\in\R$ and all $k\in K^{G_2}_\infty$,
     \begin{equation}\label{UpsiloninftyKlemmaeq2}
      \Upsilon_\infty(r_\times(\theta)\,t_\infty \,\i(1,k))
      =\Upsilon_\infty(r_\times(\theta)\,t_\infty )
      J_\infty(\eta_0 k).
     \end{equation}
    \item For all $\varphi\in\R$ and all $g\in K^{G_3}_\infty$,
      \begin{equation}\label{UpsiloninftyKlemmaeq3}
       \Upsilon_\infty(g\,\i(r(\varphi),1))=e^{-il_1\varphi}\,\Upsilon_\infty(g).
      \end{equation}
 \end{enumerate}
 Then $\Upsilon_\infty$ can be extended in a unique way
   to an analytic function on $G_3(\R)$ satisfying the following conditions.
   \begin{enumerate}
   \item \begin{equation}\label{archUpsiloncond1b}\Upsilon_\infty\in I(\tilde\chi,s).\end{equation}
 \item For all $\zeta\in S^1$ and all $h\in\U(1,1)(\R)$   \begin{equation}\label{UpsiloninftyKlemmaeq4}
      \Upsilon_\infty(Q\cdot \i(\mat{\zeta}{}{}{\zeta}h,\eta_0),s)
      =\zeta^{-l_1}\,\Upsilon_\infty(Q\cdot \i(h,\eta_0),s).
     \end{equation}
 \item We have the following equation for any $k\in K_\infty^{G_2}$ and $h\in\U(1,1)(\R)$:
  \begin{equation}\label{archUpsiloncond2b}
   \Upsilon_\infty(Q\cdot \i(h,\eta_0 k),s)
   =\Upsilon_\infty(Q \cdot\i(h,\eta_0),s)J_\infty(\eta_0 k).
  \end{equation}
 \item For all $\varphi\in\R$ and all $g\in G_3(\R)$,
   \begin{equation}\label{archUpsiloncond3b}
    \Upsilon_\infty(g\,\i(r(\varphi),1),s)
    =e^{-il_1\varphi}\,\Upsilon_\infty(g,s).
   \end{equation}
 \item For all $\varphi\in\R$ and all $h\in\U(1,1)(\R)$,
   \begin{equation}\label{archUpsiloncond4b}
    \Upsilon_\infty(Q\cdot\i(r(\varphi)h,\eta_0),s)
    =e^{-il_1\varphi}\,\Upsilon_\infty(Q\cdot\i(h,\eta_0),s).
   \end{equation}
\end{enumerate}
\end{lemma}

\begin{proof} Using the Iwasawa decomposition, it is easy to see that $\Upsilon_\infty$ can be extended in a unique way to an analytic function on $G_3(\R)$ satisfying condition (\ref{archUpsiloncond1b}). Note that condition~\eqref{UpsiloninftyKlemmaeq1} is tailored so that the extension is well-defined. Next, another appeal to the Iwasawa decomposition and the fact that $\Upsilon_\infty\in I(\tilde\chi,s)$ shows that~\eqref{UpsiloninftyKlemmaeq3} implies~\eqref{archUpsiloncond3b}. We now prove~\eqref{archUpsiloncond2b}. We have the identity
\begin{equation}\label{Qiotasqrtaeq1}
 Q\cdot\iota(\mat{\sqrt{a}}{}{}{\sqrt{a}^{-1}},1)=p_ak_aw_Q
\end{equation}
with
\begin{equation}\label{Qiotasqrtaeq2}
 p_a=\begin{bmatrix}1&&&&&\frac1{1+a}\\&1\\&&1&\frac1{1+a}\\&&&1\\&&&&1\\&&&&&1\end{bmatrix}
 \begin{bmatrix}\sqrt{\frac a{1+a}}\\&1\\&&\sqrt{\frac1{1+a}}\\  &&&\sqrt{\frac{1+a}a}\\&&&&1\\&&&&&\sqrt{1+a}\end{bmatrix}\in P_{12}(\R)
\end{equation}
and
\begin{equation}\label{Qiotasqrtaeq3}
 k_a=\begin{bmatrix}\sqrt{\frac a{1+a}}&&&&&\frac{-1}{\sqrt{1+a}}\\
 &1&&&0\\&&\sqrt{\frac a{1+a}}&\frac{-1}{\sqrt{1+a}}\\
 &&\frac1{\sqrt{1+a}}&\sqrt{\frac a{1+a}}\\
 &0&&&1\\\frac1{\sqrt{1+a}}&&&&&\sqrt{\frac a{1+a}}\end{bmatrix}\in K^{G_3}_\infty.
\end{equation}
On the other hand, observe that $k_a=r_\times(\theta)$ for $\theta$
ranging in an open subset of $\R/2\pi\Z$; so condition~\eqref{UpsiloninftyKlemmaeq2} is equivalent to
\begin{equation}\label{UpsiloninftyKlemmaeq7}
   \Upsilon_\infty(k_aw_Q\cdot \i(1,\eta_0 k),s)
   =\Upsilon_\infty(k_aw_Q\cdot\i(1,\eta_0),s)
   J_\infty(\eta_0 k).
\end{equation}
Using~\eqref{QQinversepropertyeq} and \eqref{Qiotasqrtaeq1}, properties of $J_\infty$ imply that condition~\eqref{archUpsiloncond2b} holds for all $h$ of the form $r(\varphi)\mat{\sqrt{a}}{}{}{\sqrt{a}^{-1}}$. A similar calculation shows that (\ref{archUpsiloncond2b}) holds for all elements $h$ of the form $\mat{\zeta}{}{}{\zeta}r(\varphi)\mat{\sqrt{a}}{}{}{\sqrt{a}^{-1}}$. In
combination with (\ref{archUpsiloncond3b}) it follows that
(\ref{archUpsiloncond2b}) holds for all $h\in\U(1,1)(\R)$. Finally, (\ref{UpsiloninftyKlemmaeq4}) and (\ref{archUpsiloncond4b}) can be verified using (\ref{upsilleftinvar}), (\ref{archUpsiloncond2b}), and the properties of $J_\infty$.
\end{proof}

We define the functions $x_{ij}$ on $K^{G_3}_\infty$ by
\begin{equation}\label{hatxijdefeq2}
 x_{ij}(g)=\text{$ij$-coefficient of }J(\,^t\bar g\tilde g,I),
 \qquad g\in K^{G_3}_\infty,
\end{equation}
where
$$
 \tilde g=\mat{A}{-B}{B}{A}\text{ for }g=\mat{A}{B}{-B}{A}.
$$
Any polynomial expression in the functions $x_{ij}$ and their complex conjugates
is $K^{G_3}_\infty$-finite. We further define
\begin{align*}
 X_1:=\big((1-| x_{33}|^2) x_{11}+ x_{13} x_{31}\overline{ x_{33}}\big)
  \overline{ x_{13}}
  +\big((1-| x_{33}|^2) x_{12}+ x_{13} x_{32}\overline{ x_{33}}\big)
   \overline{ x_{23}},\\
 X_2:=\big((1-| x_{33}|^2) x_{21}+ x_{23} x_{31}\overline{ x_{33}}\big)
  \overline{ x_{13}}
  +\big((1-| x_{33}|^2) x_{22}+ x_{23} x_{32}\overline{ x_{33}}\big)
   \overline{ x_{23}},\\
 Y_1:=\big((1-| x_{33}|^2)\overline{ x_{11}}
  +\overline{ x_{13} x_{31}} x_{33}\big) x_{31}
  +\big((1-| x_{33}|^2)\overline{ x_{21}}
  +\overline{ x_{23} x_{31}} x_{33}\big) x_{32},\\
 Y_2:=\big((1-| x_{33}|^2)\overline{ x_{12}}
  +\overline{ x_{13} x_{32}} x_{33}\big) x_{31}
  +\big((1-| x_{33}|^2)\overline{ x_{22}}
  +\overline{ x_{23} x_{32}} x_{33}\big) x_{32}.
\end{align*}
Let $\Upsilon_0$ be the function on $K^{G_3}_\infty$ given by
\begin{equation}\label{Upsilon0defeq}
 \Upsilon_0=\left\{\begin{array}{l@{\qquad\text{if }}l}
 \big(\overline{ x_{31}}Y_2-\overline{ x_{32}}Y_1\big)^{l_1-l}&l\leq l_1,\\
 \big( x_{13}X_2- x_{23}X_1\big)^{l-l_1}&l\geq l_1.
 \end{array}\right.
\end{equation}
By explicit calculation, one verifies that
\begin{equation}\label{Upsilon0rtimesthetaeq2}
 \Upsilon_0(r_\times(\theta)\,\i(1,s_1\eta_0))
  =(-1)^{l_1-l}\,\sin(2\theta)^{4|l-l_1|}.
\end{equation}

\begin{lemma}\label{upsiloninftycandidatelemma}
 Let $\Upsilon_0$ be as in (\ref{Upsilon0defeq}). Then
 the function $\Upsilon_\infty(g):=\Upsilon_0(g)\det(J(g,i_2))^{-l_1}$ is
 $K^{G_3}_\infty$-finite and satisfies the conditions from Lemma
 \ref{UpsiloninftyKlemma}. Moreover,
 \begin{equation}\label{upsiloninftycandidatelemmaeq4}
    \Upsilon_\infty(r_\times(\theta)\,t_\infty )
    =(-i)^{l_1}\,(-1)^l\,\sin(2\theta)^{4|l-l_1|}
 \end{equation}
 for all $\theta\in\R$.
 If $\Upsilon_\infty(\,\cdot\,,s)$ denotes the extension of $\Upsilon_\infty$
 to a function on all of $G_3(\R)$, then
 \begin{equation}\label{upsiloninftycandidatelemmaeq5}
  \Upsilon_\infty(Q\,\i(\mat{\sqrt{a}}{}{}{\sqrt{a}^{-1}},\eta_0),s)
  =2^{4|l-l_1|}(-i)^{l_1}\,(-1)^l\,
   \big(\sqrt{a}+\sqrt{a}^{-1}\big)^{q-6(s+\frac12)-4|l-l_1|}
 \end{equation}
 for all $a>0$. Here, $q\in\C$ is such that $\omega_\tau(a)=a^q$ for $a>0$.
\end{lemma}

\begin{proof}
From the construction, it is a routine calculation to verify that $\Upsilon_\infty$
satisfies the conditions (\ref{UpsiloninftyKlemmaeq1}) and (\ref{UpsiloninftyKlemmaeq2}).
Property (\ref{UpsiloninftyKlemmaeq3}) follows from the right transformation
properties of the functions $x_{ij}$. Property (\ref{upsiloninftycandidatelemmaeq4}) follows easily from \eqref{Upsilon0rtimesthetaeq2}. To prove \eqref{upsiloninftycandidatelemmaeq5}, note that, by (\ref{Qiotasqrtaeq1}),
$$
 Q\,\i(\mat{\sqrt{a}}{}{}{\sqrt{a}^{-1}},1)=p_ar_\times(\theta)\,\i(\mat{}{1}{-1}{},s_1)
 =p_ar_\times(\theta)\,t_\infty \,\i(1,\eta_0^{-1})
$$
with $p_a$ as in (\ref{Qiotasqrtaeq2}) and $\theta\in\R$ such that
$\cos(\theta)=\sqrt{\frac a{1+a}}$ and $\sin(\theta)=\frac{-1}{\sqrt{1+a}}$. This leads to the claimed result in a straightforward manner.
\end{proof}
\subsection{The local pullback formula: archimedean case}\label{pullbackarchsec}
In this section, we will prove the local pullback formula in the archimedean case. Let $\Upsilon_\infty(\,\cdot\,,s)$ be the element of $I(\widetilde{\chi},s)$ constructed in Lemma \ref{upsiloninftycandidatelemma}. For any $g_2\in\U(2,2)(F)$, $g\in G_1(F)$ and $s \in \C$ let
\begin{equation}\label{W1Whittakerformulaeq2}
 Z_\infty(g,s;g_2)=\int\limits_{\U(1,1)(\R)}\Upsilon_\infty(Q\cdot\i(h,g_2),s)
  \,W_{l_1}(gh)\,\chi(\det(h))^{-1}\,dh,
\end{equation}
which converges absolutely for $\Re(s)$ sufficiently large. Here, $W_{l_1}$ is as in (\ref{W1Whittakerformulaeq}). For simplicity, we will assume that $\mathbf{c}=1$.
\begin{theorem}[Archimedean Local Pullback Formula]\label{theorem-arch-local-pullback}
 Let $l$ be a positive integer, and let $l_1$ be any of the weights occurring in $\tau$. Then, for $\Re(s)$ sufficiently large,
 \begin{equation}\label{theorem-arch-local-pullbackeq1}
  Z_\infty(g,s;\eta_0)= T_\infty(s)\,W_{l_1}(g),
 \end{equation}
 where, up to a non-zero constant (depending on $\tau_\infty$ and $l$ and $l_1$, but not on $s$),
 $$
  T_\infty(s)= 2^{-6s}
  \frac{\Gamma(3s+1-\frac q2+2|l_1-l| - \frac{p}{2}) \,
  \Gamma(3s+1-\frac q2+2|l_1-l| + \frac{p}{2})}{\Gamma(3s+\frac32-\frac q2+2|l_1-l|-\frac{l_1}{2})\,\Gamma(3s+\frac32-\frac q2+2|l_1-l|+\frac{l_1}{2})}.
 $$
 Here, $p$ and $q$ are as in (\ref{t1t2pqeq}).
\end{theorem}
\begin{proof} Recall that we have chosen $t^+\in\R_{>0}$ and normalized the function $W_{l_1}$ such that $W_{l_1}(\mat{t^+}{}{}{1})=1$; see (\ref{aplusdefeq}). By changing the value of $p$ slightly and using the holomorphy of both sides of
(\ref{theorem-arch-local-pullbackeq1}) in $p$, we may work under the additional assumption that $W_{l_1}(1)\neq0$. We have
\begin{align*}
 Z_\infty(g,s;\eta_0)&=\int\limits_{\U(1,1)(\R)}\Upsilon_\infty(Q\cdot\i(h,\eta_0),s)
  \,W_{l_1}(gh)\,\chi^{-1}(\det(h))\,dh\\
 &\stackrel{(\ref{UpsiloninftyKlemmaeq4})}{=}\int\limits_{\U(1,1)(\R)/Z}
  \Upsilon_\infty(Q\cdot\i(h,\eta_0),s)\,W_{l_1}(gh)\,\chi^{-1}(\det(h))\,dh\\
  &=\int\limits_{\SL(2,\R)}F_1(h)\,W_{l_1}(gh)\,dh,
\end{align*}
where the function $F_1$ on $\SL_2(\R)$ is defined by $F_1(h):=\Upsilon_\infty(Q\,\i(h,\eta_0),s)$. Hence $Z_\infty(g,s;\eta_0)$ is in the space of $\tau$. It follows from~\eqref{archUpsiloncond4b} that $Z_\infty(g,s;\eta_0)$ is a vector of weight $l_1$.
By irreducibility, there is (up to multiples) only one vector of weight $l_1$ in the
space of $\tau$, namely $W_{l_1}$. It follows that $Z_\infty(g,s;\eta_0)$ is a
multiple of $W_{l_1}(g)$. By an easy calculation, in terms of the Iwasawa decomposition,
$$
 F_1(\mat{1}{b}{}{1}\mat{\sqrt{a}}{}{}{\sqrt{a}^{-1}}r(\theta))
 =(-i)^{l_1}e^{-il_1\theta}\Big(\frac{(1+a)^2+b^2}a\Big)^{\frac q2-3(s+\frac12)-2|l_1-l|}
 \Big(\frac{b-i(a+1)}{|b-i(a+1)|}\Big)^{-l_1}.
$$
So we get that $Z_\infty(g,s;\eta_0)=T_\infty(s) W_{l_1}(g)$, where
\begin{align*}
 T_\infty(s)W_{l_1}(1)&=Z_\infty(1,s;\eta_0)\\
 &=(-i)^{l_1}\int\limits_{-\infty}^\infty\int\limits_0^\infty a^{-1}
  \Big(\frac{(1+a)^2+b^2}a\Big)^{\frac q2-3(s+\frac12)-2|l_1-l|}
  \Big(\frac{b-i(a+1)}{|b-i(a+1)|}\Big)^{-l_1}\\
  &\hspace{30ex}e^{2\pi ib}\,W_{l_1}(\mat{\sqrt{a}}{}{}{\sqrt{a}^{-1}})\,d^\times a\,db.
\end{align*}
For brevity, we make the following substitutions,
$$
 s'=\frac q2-3(s+\frac12)-2|l_1-l|,
$$
$$
 s_1 = -s' - \frac{l_1}{2}, \qquad s_2 = -s'- \frac{1}{2} - \frac{p}{2}, \qquad s_3=-s'- \frac{1}{2} + \frac{p}{2}, \qquad s_4 = -s' + \frac{l_1}{2}.
$$
By the integral formula from~\cite[(6.11)]{grosskudla1992}), the first formula in Sect.\ 7.5.2 of \cite{MOS}, and Lemma \ref{F21lemma} below,
\begin{align*}
 Z_\infty(1,s;\eta_0)&=
  \frac{(2\pi)^{-2s'}}{\Gamma(s_1)\Gamma(s_4)}
  \int\limits_0^\infty\int\limits_0^\infty a^{-s'-1}
   t^{s_1-1}(t+1)^{s_4-1}e^{-2\pi(a+1)(1+2t)}\,
   W_{l_1}(\mat{\sqrt{a}}{}{}{\sqrt{a}^{-1}})\,d^\times a\,dt\\
 &=a^+\,2^{q+2} \pi^{-s'+1 +\frac q2}\frac{\Gamma(s_2)
  \Gamma(s_3)}{\Gamma(s_1)^2\Gamma(s_4)}\int\limits_0^\infty
   t^{s_1-1}(t+1)^{s_4-1}e^{-2\pi(1+2t)}\ _2F_1\big(s_2,s_3,s_1;-t\big)\,dt\\
  &=a^+\,2^{q+ 1 -p - 2s_2} \pi^{1+\frac{q}{2}} \frac{\Gamma(s_2)
  \Gamma(s_3)}{\Gamma(s_1)\Gamma(s_4)}
  W_{\frac{l_1}2,\frac{p}2}(4\pi)\\&=4^{s'+1} \pi \frac{\Gamma(s_2)
  \Gamma(s_3)}{\Gamma(s_1)\Gamma(s_4)} W_{l_1}(1),
\end{align*}and so $$T_\infty(s) = 4^{s'+1} \pi \frac{\Gamma(s_2)
  \Gamma(s_3)}{\Gamma(s_1)\Gamma(s_4)}.
$$
This concludes the proof.
\end{proof}

\begin{lemma}\label{F21lemma}
 For complex numbers $s_1$, $s_2$, $s_3$ with $\Re(s_1)>0$,
 $$
  \int\limits_0^\infty t^{s_1-1}(t+1)^{s_2 + s_3-s_1}e^{-4\pi t}\
   _2F_1\big(s_2,s_3,s_1;-t\big)\,dt
   = \Gamma(s_1) (4 \pi)^{-\frac{s_2+s_3+1}{2}} e^{2 \pi}\,
    W_{\frac{s_2+s_3+1}{2} -s_1, \frac{s_3 - s_2}{2}}(4 \pi).
 $$
\end{lemma}
\begin{proof}
This follows by first applying the third equation of~\cite[(9.131, 1)]{grary}, followed by the integral formula~\cite[(7.522, 1)]{grary}.
\end{proof}
\subsection{The global pullback formula}
In the following we use the global set-up of Theorem \ref{globalintegralrepresentationtheoremversion1}. We set the number $l_2$ to be $-l_1$. We will hence work with the section $f = f_{|l_1-l|,l, -l_1, l_1}$ in $I_\C(s, \chi, \chi_0, \tau)$. It gives rise to the Eisenstein series $E(g,s;f)$ via~\eqref{Edefeq}. In this section we will prove the global pullback formula, which expresses the Eisenstein series $E(g,s;f)$ on $G_2(\A)$ in terms of the pullback of a simpler Eisenstein series on $G_3(\A)$.

\vspace{3ex}
Let $\widetilde{\chi}$ be the character on $P_{12}(\A)$ defined by
$\widetilde{\chi}(m(A,v)n) = \chi(v^{-1}\det(A))$; see (\ref{localtildechidefeq}) for the corresponding local definition. For $s \in \C,$ we form the global induced representation
\begin{equation}\label{Ichisglobaldefeq}
 I(\widetilde{\chi},s) = \text{Ind}_{P_{12}(\A)}^{G_3(\A)} (\widetilde{\chi}\delta_{12}^s)
\end{equation}
(see \eqref{P12modulareq}),
consisting of functions $\Upsilon$ on $G_3(\A)$ such that
\begin{equation}\label{e:upsilondefformula}
 \Upsilon(m(A,v)ng,s) = |v|^{-9(s +\frac{1}{2})}\,|N (\det A)|^{3(s + \frac12)}\,
 \chi(v^{-1}\det(A))\,\Upsilon(g,s)
\end{equation}
for $n\in N_{12}(\A)$, $m(A,v)\in M_{12}(\A)$, $g\in G_3(\A)$. \emph{Now, let $\Upsilon= \otimes_v \Upsilon_v \in I(\widetilde{\chi}, s)$, where $\Upsilon_v$ is defined by~\eqref{upsilondefinert} in the non-archimedean case and defined as in Lemma~\ref{upsiloninftycandidatelemma} in the archimedean case.} We define the Eisenstein series $E_{\Upsilon}(g,s)$ on $G_3(\A)$ by
\begin{equation}\label{E-upsilon-def}
 E_{\Upsilon}(g,s)=\sum_{\gamma\in P_{12}(\Q)\bs G_3(\Q)} \Upsilon(\gamma g,s)
\end{equation}
for $\Re(s)$ sufficiently large, and by analytic continuation elsewhere. Furthermore, let
\begin{equation}\label{definitionms}
 T(s)=\prod_v T_v(s),
\end{equation}
where the local functions $T_v(s)$ are defined by Theorem~\ref{theorem-local-pullback-nonarch} in the non-archimedean case and by Theorem~\ref{theorem-arch-local-pullback} in the archimedean case. Note that though~\eqref{definitionms} makes sense for $\Re(s)$ sufficiently large; it is clear from the definitions of $T_v(s)$ that $T(s)$ can be analytically continued to a meromorphic function on the entire complex plane (it is effectively just a ratio of global $L$-functions).

\begin{theorem}[Global Pullback Formula]\label{theorem-global-pullback}
 Let $\Psi$ be the cusp form in the space of $\tau$ corresponding to a local newform at all non-archimedean places, a vector of weight $l_1$ at the archimedean place, and with the same normalization for the corresponding Whittaker function $W_\Psi$ as in (\ref{WPsieq}). Let $\Psi$ be extended to a function on $G_1(\A)$ via $\Psi(ag) = \chi_0(a)\Psi(g)$ for $a \in \A_L^\times$, $g\in \GL_2(\A)$. For an element $g \in G_2(\A)$, let $\U[g](\A)$ denote the subset of $G_1(\A)$ consisting of all elements $h$ such that $\mu_2(g) =\mu_1(h)$. Then we have the following identity of meromorphic functions,
 \begin{equation}
  \chi(\mu_2(g))\int\limits_{  \U(1,1)(\Q) \backslash \U[g](\A)}E_{\Upsilon}(\i(h,g),s)\,\Psi(h)\,\chi(\det(h))^{-1}\,dh = T(s) E(g,s;f),
 \end{equation}
 where $f = f_{|l_1-l|,l, -l_1, l_1}$ as in Theorem~\ref{globalintegralrepresentationtheoremversion1}.
\end{theorem}
\begin{proof} Let $\mathcal{E}(g,s)$ denote the left hand side above. Note that $E_{\Upsilon}(\i(g,h),s)$ is slowly increasing away from its poles, while $\Psi(h)$ is rapidly decreasing. Thus $\mathcal{E}(g,s)$ converges uniformly and absolutely for $s \in \C$ away from the poles of the Eisenstein series $E_\Upsilon$. Hence, it is enough to prove the theorem for $\Re(s)$ sufficiently large. Since $E_\Upsilon$ is left invariant by $G_3(\Q)$,
\begin{equation}
 \mathcal{E}(g,s)=\chi(\mu_2(g))\int\limits_{  \U(1,1)(\Q) \backslash \U[g](\A)}
 E_{\Upsilon}(Q \cdot \i(h,g),s)\,\Psi(h)\,\chi(\det(h))^{-1}\,dh.
\end{equation}
Let $V(\Q)$ denote the subgroup of $G_3(\Q)$ defined by
$$
 V(\Q)=\{Q\:\i(g_1,g_2)\,Q^{-1}\;|\;g_i\in G_i(\Q),\:\mu_1(g_1)=\mu_2(g_2)\}.
$$
Recall from~\cite[Prop.\ 2.4]{shibook1} that $|P_{12}(\Q)\bs G_3(\Q) / V(\Q)|=2$.  We take the identity element as one of the double coset representatives, and denote the other one by $v$. Thus
$$
 G_3(\Q)= P_{12}(\Q)V(\Q) \sqcup P_{12}(\Q)vV(\Q).
$$
Let $R_1 \subset V(\Q)$ and $R_2 \subset vV(\Q)$ be corresponding sets of coset representatives, such that
$$
 P_{12}(\Q)V(\Q) = \bigsqcup_{s \in R_1}P_{12}(\Q)s, \quad P_{12}(\Q)vV(\Q) = \bigsqcup_{s \in R_2}P_{12}(\Q)s.
$$
For the Eisenstein series defined in (\ref{E-upsilon-def}), we can write $E_{\Upsilon}(h,s) =E_{\Upsilon}^1(h,s) +  E_{\Upsilon}^2(h,s)$, where
$$
 E_{\Upsilon}^1(h,s)= \sum_{\gamma \in R_1} \Upsilon(\gamma h,s), \qquad E_{\Upsilon}^2(h,s)= \sum_{\gamma \in R_2} \Upsilon(\gamma h,s).
$$
Now, by~\cite[22.9]{shibook1} the orbit of $v$ is `negligible' for our integral, that is,
$$
 \int\limits_{  \U(1,1)(\Q) \backslash \U[g](\A)}E_{\Upsilon}^2(Q\cdot\i(h,g),s)\,\Psi(h)\,\chi(\det(h))^{-1}\,dh =0
$$
for all $g\in G_2(\A)$. It follows that
\begin{equation}\label{negligibleeisenstein}
 \mathcal{E}(g,s)=\chi(\mu_2(g))\int\limits_{  \U(1,1)(\Q) \backslash \U[g](\A)}
 E_{\Upsilon}^1(Q\cdot \i(h,g),s)\,\Psi(h)\,\chi(\det(h))^{-1}\,dh.
\end{equation}
On the other hand, by~\cite[Prop. 2.7]{shibook1}, we can take $R_1$ to be the following set,
\begin{equation}
 R_1= \{Q\:\i(1,m_2(\xi)\beta)\,Q^{-1}\;|\;\xi \in \U(1,1)(\Q),\:\beta \in P(\Q) \bs G_2(\Q)\},
\end{equation}
where $m_2(\xi)$ is as in (\ref{m1m2defeq}), and where the $\beta$ are chosen to have $\mu_2(\beta)=1$. For $\Re(s)$ large, we therefore have
$$
 E_{\Upsilon}^1(Q \cdot\i(h,g),s) = \sum_{\substack{\xi \in \U(1,1)(\Q)\\ \beta \in P(\Q) \bs G_2(\Q)}}\Upsilon(Q \cdot\i(h,m_2(\xi)\beta g),s).
$$
Substituting into \eqref{negligibleeisenstein} and using that $Q\:\i(\xi,m_2(\xi))\,Q^{-1} \in P_{12}(\Q)$ by (\ref{QQinversepropertyeq}), we have
\begin{align*}
 \mathcal{E}(g,s)&=\chi(\mu_2(g))\int\limits_{\U(1,1)(\Q) \backslash \U[g](\A)}
 \sum_{\substack{\xi \in \U(1,1)(\Q)\\ \beta \in P(\Q) \bs G_2(\Q)}}
 \Upsilon(Q\cdot\i(h,m_2(\xi)\beta g),s)\,\Psi(h)\,\chi(\det(h))^{-1}\,dh\\
 &=\sum_{\beta \in P(\Q) \bs G_2(\Q)}\chi(\mu_2(g))\int\limits_{ \U[g](\A)}
  \Upsilon(Q \cdot\i(h,\beta g),s)\,\Psi(h)\,\chi(\det(h))^{-1}\,dh.
\end{align*}
Let
\begin{equation}
 \Upsilon_\Psi(g,s) = \chi(\mu_2(g))\int\limits_{ \U[g](\A)}
 \Upsilon(Q \cdot\i(h, g),s)\,\Psi(h)\,\chi(\det(h))^{-1}\,dh.
\end{equation}
If we can show that, for each $g\in G_2(\Q)$,
\begin{equation}\label{neededeqn}
 \Upsilon_\Psi(g,s)=T(s)f(g,s),
\end{equation}
the proof will be complete. By~\cite{shibook1}, we know that the integral above converges absolutely and uniformly on compact sets for $\Re(s)$ large. We are going to evaluate the above integral for such $s$. For a finite place $p$ such that $\tau_p$ has conductor $\p^n$, note that
$$
 G_2(\Q_p) = \bigsqcup_{m=0}^nP(\Q_p)\eta_m  K^H\Gamma(\P^n)
$$
by Proposition \ref{disj-doub-coset-decomp-general-n-prop}. For $k \in K^H\Gamma(\P^n)$, we may write $k =m_2(\mat{1}{0}{0}{\lambda}) k'$,
where $\lambda = \mu_2(k)$ and $\mu_2(k')=1$. Using the fact that both sides of~\eqref{neededeqn} are invariant under the right action on $g$ by elements $k_p \in K^H\Gamma(\P^n)$ satisfying $\mu_2(k_p)=1$, and the above observations, it follows that in order to prove our theorem, it is enough to prove~\eqref{neededeqn}  for $g \in G_2(\A)$ of the form
$$
 g=m_1(a)m_2(b)n\,\kappa\,k_\infty,
$$
where $m_i \in M^{(i)}(\A)$, $n\in N(\A)$, $k_\infty \in K_\infty^{G_2}$, and $\kappa = (\kappa_v)_v\in \prod_vK_v^{G_2}$ satisfies
\begin{itemize}
 \item $\kappa_v  \in \{\eta_0, \cdots \eta_n\}$ if $v=p$ and $\tau_p$ has conductor $p^n$, $n>0$,
 \item $\kappa_v = \eta_0$ if $v = \infty$,
 \item $\kappa_v=1$ otherwise.
\end{itemize}
For such $g$, we calculate
\begin{align*}
 \Upsilon_\Psi(g,s)
 &=\chi(\mu_1(b))
  \int\limits_{ \U[m_2(b)](\A)}\Upsilon(Q \cdot\i(h,m_1(a)m_2(b)n\kappa k_\infty),s)
  \,\Psi(h)\,\chi(\det(h))^{-1}\,dh\\
 &\stackrel{(\ref{QQinversepropertyeq})}{=}|\mu_1(b)|^{-3(s +\frac{1}{2})}\int\limits_{ \U[m_2(b)](\A)}\Upsilon(Q \cdot\i(b^{-1}h,m_1(a)n\kappa k_\infty),s)\,\Psi(bb^{-1}h)\,\chi(\det(b^{-1}h))^{-1}\,dh\\
 &\stackrel{(\ref{archUpsiloncond2b})}{=}|\mu_1(b)|^{-3(s +\frac{1}{2})}
  J_\infty(\eta_0 k_\infty)\int\limits_{\U(1,1)(\A)}\Upsilon(Q \cdot\i(h,m_1(a)n\kappa),s)
  \,\Psi(bh)\,\chi(\det(h))^{-1}\,dh\\
 &=\chi(a)|N (a)\mu_1(b)^{-1}|^{3(s+\frac12)}
  J_\infty(\eta_0 k_\infty)\int\limits_{\U(1,1)(\A)}\Upsilon(Q \cdot\i(h,\kappa),s)
  \,\Psi(bh)\,\chi(\det(h))^{-1}\,dh.
\end{align*}
Using the Whittaker expansion
\begin{equation}\label{defwhitt}
 \Psi(g) = \sum_{\lambda \in \Q^\times}W_{\Psi}(\mat{\lambda}{0}{0}{1}g),
\end{equation}
we have
\begin{equation}\label{defub}
 \int\limits_{\U(1,1)(\A)}\Upsilon(Q\cdot\i(h,\kappa),s)
 \,\Psi(bh)\,\chi(\det(h))^{-1}\,dh=\sum_{\lambda \in\Q^\times}
 Z(\mat{\lambda}{0}{0}{1}b,s;\kappa),
\end{equation}
where for $g\in G_1(\A)$, $g_2\in\U(2,2)(\A)$,
$$
 Z(g,s;g_2) = \int\limits_{\U(1,1)(\A)}\Upsilon(Q\cdot \i(h,g_2),s)
 \,W_{\Psi}(gh)\,\chi(\det(h))^{-1}\,dh.
$$
Note that the uniqueness of the Whittaker function implies $Z(g,s;\kappa) = \prod_v Z_v(g_v, s, \kappa_v)$, where the local zeta integral $Z_v(g_v, s, \kappa_v)$ is defined by
$$
 Z_v(g_v, s, \kappa_v) = \int\limits_{\U(1,1)(\Q_v)}\Upsilon(Q\cdot \i(h, \kappa_v),s)
 \,W^{(0)}(g_vh)\,\chi_v(\det(h))^{-1}\,dh;
$$
at the archimedean place we understand $W^{(0)}=W_{l_1}$. Hence, by Theorems~\ref{theorem-local-pullback-nonarch} and \ref{theorem-arch-local-pullback},
\begin{equation}
 \Upsilon_\Psi(g,s) = \chi(a)|N (a)\mu_1(b)^{-1}|^{3(s+1/2)}
  T(s)\Psi(b) J_\infty(\eta_0 k_\infty, s) \prod_{\substack{p < \infty\\ \tau_p \text{ ramified }}} J_p(\kappa_p),
\end{equation}
where for a finite place $p$ with $ \tau_p$ of conductor $p^n$, $n>0$,
\begin{equation}J_p(\kappa_p) = \begin{cases} 1 & \text{ if } \kappa_p = \eta_0 , \\  0& \text{ otherwise.} \end{cases}
\end{equation}
This proves~\eqref{neededeqn} and hence completes
the proof of the theorem.
\end{proof}

\begin{rem}\rm Pullback formulas in the spirit of Theorem~\ref{theorem-global-pullback} as a method to express complicated Eisenstein series on lower rank groups in terms of simpler Eisenstein series on higher rank groups have a long history. Garrett \cite{Ga1983} used pullback formulas for Eisenstein series on symplectic groups to study the triple product $L$-function, as well as to establish the algebraicity of certain
ratios of inner products of Siegel modular forms. Pullback formulas for Eisenstein series on unitary groups were first proved in a classical setting by Shimura~\cite{shibook1}. Unfortunately, Shimura only considers certain special types of Eisenstein series in his work, which do not include ours except in the very specific case when $\tau$ is unramified principal series at all finite places and holomorphic discrete series at infinity.
\end{rem}
\subsection{The second global integral representation}

In Theorem~\ref{globalintegralrepresentationtheoremversion1} we supplied a global integral representation for $L(s, \tilde\pi \times \tilde\tau)$. Using Theorem~\ref{theorem-global-pullback}, we can modify it into a second integral representation that is more suitable for certain purposes. Let
\begin{equation}\label{rsdef}
 R(s) = \frac{L(3s+1,\tau \times \AI(\Lambda) \times \chi|_{F^\times})}{L(6s+2,\chi_{L/F}\chi|_{F^\times})L(6s+3,  \chi|_{F^\times})\cdot T(s)\cdot Y_{l,l_1,p,q}(s)\cdot\prod_{v<\infty}Y_v(s)},
\end{equation}
where $T(s)$ is defined by~\eqref{definitionms}, the factors $Y_v(s)$ for non-archimedean $v$ are given by Theorem \ref{nonarchlocalzetatheorem}, and the archimedean factor  $Y_{l,l_1,p,q}(s)$ is given by Theorem~\ref{archlocalzetatheorem}. Note that $R(s)$ has an obvious Euler product $R(s) = \prod_v R_v(s)$, and that $R_v(s) = 1$ for all finite places $v$ where $\tau_v$ is unramified.

\vspace{3ex}
Recall the Eisenstein series $E_{\Upsilon}(g,s)$ defined in (\ref{E-upsilon-def}). We define the normalized Eisenstein series
\begin{equation}\label{normalizedeisensteineq}
 E_{\Upsilon}^\ast(g,s) = L(6s+1,  \chi|_{\A^\times})L(6s+2,\chi_{L/F}\chi|_{\A^\times})L(6s+3,  \chi|_{\A^\times}) E_{\Upsilon}(g,s).
\end{equation}
Let $Z_H$ and $Z_{G_1}$ denote respectively the centers of $H= \GSp_4$ and $G_1 = \GU(1,1)$. Given any $g \in G_1$ we define $H[g]$ to be the subgroup of $H$ consisting of elements $h\in H$ with $\mu_2(h)= \mu_1(g)$. From Theorem~\ref{globalintegralrepresentationtheoremversion1} and Theorem~\ref{theorem-global-pullback} we get the following result.
\begin{theorem}\label{theoremsecondintegralrep}
 Let $\phi=\otimes\phi_v$ be a vector in the space of $\pi$ such that $\phi_v$ is unramified for all finite $v$ and such that $\phi_\infty$ is a vector of weight $(-l,-l)$ in $\pi_\infty$. Let $\Psi$ be as in Theorem~\ref{theorem-global-pullback}.
 The following meromorphic functions are all equal,
 \begin{enumerate}
  \item $R(s)^{-1} B_\phi(1) L(3s+\frac 12, \tilde\pi \times \tilde\tau)$
  \item $\displaystyle\int\limits_{Z_H(\A) H(\Q) \backslash H(\A)}\phi(h)\chi(\mu_2 (h))\quad \int\limits_{\U(1,1)(\Q) \backslash \U[h](\A)}E_{\Upsilon}^\ast(\i(g,h),s)\Psi(g)\chi(\det (g))^{-1}\,dg\;dh$,
  \item $\displaystyle\int\limits_{Z_{G_1}(\A) G_1(\Q) \backslash G_1(\A)}\Psi(g)\chi\big(\frac{\mu_1(g)}{\det (g)}\big)\quad \int\limits_{\mathrm{Sp}_4(\Q) \backslash H[g](\A)}E_{\Upsilon}^\ast(\i(g,h),s))\phi(h)\,dh\;dg$.
 \end{enumerate}
\end{theorem}
For future reference, we record the following result about the poles of $E_{\Upsilon}^\ast(g,s)$.

\begin{proposition}\label{theorempoleseis}
 Assume that the number $q$ defined in (\ref{t1t2pqeq}) is zero.
 Then $E_{\Upsilon}^\ast(g,s)$ has no poles in the region $ 0 \le \Re(s) \le \frac{1}{4}$ except possibly a simple pole at the point $s = \frac{1}{6}$; this pole can exist only if $\omega_\tau = 1$.
\end{proposition}
\begin{proof}
First, note that the Eisenstein series $E_{\Upsilon}^\ast(g,s)$ on $\GU(3,3)$ has a pole at $s_0$ if and only if its restriction to $\U(3,3)$, which is an Eisenstein series on $\U(3,3)$, has a pole at $s=s_0$. Now the proof of the main Theorem of~\cite{Tan} shows exactly what we want. However, the statement there is a little ambiguous and seems to also allow for a possible simple pole at $s=0$, in addition to the one at $s=\frac16$. So we sketch the proof of holomorphy at $s=0$ here for completeness. Let $I_v(\chi, 0)$ be as defined in~\cite{Tan}; this space is completely reducible at each non-archimedean inert place $v$ (it is the direct sum of two irreducible representations). Now, we may choose any one of these irreducible components and work through the proof exactly as in~\cite{Tan}.
\end{proof}

\section{Holomorphy of global $L$-functions for $\GSp_4\times\GL_2$}
In this section we will prove that the global $L$-function $L(s,\pi\times\tau)$ appearing in Theorem~\ref{functionalequationtheorem} is entire. Our main tools are the global integral representation Theorem \ref{theoremsecondintegralrep} and Ichino's regularized Siegel-Weil formula for unitary groups, Theorem 4.1 of \cite{ich}.
\subsection{Preliminary considerations}
Our goal is to prove the following theorem.
\begin{theorem}[Holomorphy for $\GSp_4\times\GL_2$ ]\label{entirenesstheorem}
 Let $\pi=\otimes\pi_v$ be a cuspidal, automorphic representation of $\GSp_4(\A)$ with the properties enumerated in Theorem~\ref{functionalequationtheorem} and such that $\pi$ is not a Saito-Kurokawa lift. Let $\tau=\otimes\tau_v$ be a cuspidal, automorphic representation of $\GL_2(\A)$ such that $\tau_p$ is unramified for the primes $p$ dividing $D$. Then $L(s,\pi\times\tau)$ is an entire function.
\end{theorem}
The proof will be completed in Section \ref{siegelweilsec} below. To begin with, note that $\tau$ may be twisted by an unramified Hecke character of the form $|\cdot|^t$ to make sure that $\omega_\tau$ is of finite order. Such a twist will merely shift the argument of the $L$-function, because of the equation $L(s, \pi \times \tau \times | \cdot|^t) = L(s+t, \pi \times \tau)$. It is therefore sufficient to prove Theorem~\ref{entirenesstheorem} under the following assumption, which we will make throughout this section.
\begin{equation}
 \text{\emph{The central character $\omega_\tau$ of $\tau$ is of finite order.}}
\end{equation}
In particular, this means that the number $q$ defined in (\ref{t1t2pqeq}) is zero. Since $\varepsilon$-factors never have any zeros or poles, it follows from the functional equation Theorem \ref{functionalequationtheorem} that in order to prove  Theorem~\ref{entirenesstheorem}, \emph{it is enough to prove that $L(s,\pi\times\tau)$ has no poles in the region $\Re(s) \ge \frac{1}{2}$}.

\vspace{2ex}
{\bf Remark:} Recall that the hypothesis that $\tau_p$ is unramified for the primes $p$ dividing $D$ was necessary for Theorem \ref{functionalequationtheorem}. This is the only reason for this hypothesis in Theorem \ref{entirenesstheorem}; the following arguments work for general $\tau$. The restriction on $\tau$ will be removed in Theorem \ref{Lrhonsigmaranalyticpropertiestheorem}.

\vspace{2ex}Let $L_f(s,\pi\times\tau)$ be the finite part of $L(s,\pi\times\tau)$, i.e.,
$$
 L_f(s,\pi\times\tau) = \prod_{p < \infty} L_p(s,\pi_p\times\tau_p).
$$

\begin{lemma}\label{kimsarconv} The Dirichlet series defining $L_f(s,\pi\times\tau)$ converges absolutely for $\Re(s) > \frac{5}{4}.$
\end{lemma}
\begin{proof} In fact, the Dirichlet series converges absolutely for $\Re(s) > \frac{71}{64}.$ This follows directly from the global temperedness of $\pi$ due to Weissauer~\cite{weissram} and the best known bound towards the Ramanujan  conjecture for cusp forms on $\GL_2$ due to Kim-Sarnak~\cite{kimsar}.
\end{proof}

As a consequence, we get the following.
\begin{lemma}\label{poleregion}
 The completed $L$-function $L(s,\pi\times\tau)$ has no poles in the region  $\Re(s) > \frac{5}{4}$.
\end{lemma}
\begin{proof} In view of Lemma~\ref{kimsarconv}, we only have to show that $L_\infty(s,\pi\times\tau)$ has no poles in that region. In fact it turns out that $L_\infty(s,\pi\times\tau)$ has no poles in the region $\Re(s)>1$.  To see this, first note that $q, \mu$ are equal to zero by our assumption on $\omega_\tau$. Next, by the unitarizability of $\tau_\infty$, it follows that $p$ is a non-negative integer when $\tau_\infty$ is discrete series (or limit of discrete series) and $p \in i\R \cup (-1,1)$ if $\tau_\infty$ is principal series. Also, we have $l \ge 2$. Now the holomorphy of $L_\infty(s,\pi\times\tau)$ in the desired right-half plane follows from  the tables following Corollary~\ref{archlocalzetatheoremcor}.
\end{proof}

We will now use the second integral representation to reduce the possible set of poles to at most one point.
\begin{proposition}\label{atmostonepoleprop}
 $L(s,\pi\times\tau)$ has no poles in the region $\Re(s) \ge \frac{1}{2}$ except possibly a simple pole at the point $s=1$. This pole can exist only if $\omega_\tau =1$.
\end{proposition}
\begin{proof}
In Theorem~\ref{theoremsecondintegralrep}, the functions $\Psi$, $\Upsilon$, $\chi$, $R(s)$ all depend on a choice of an integer $l_1$ such that $\tau$ has a vector of weight $l_1$. We now make such a choice. If $\tau_\infty$ is a principal series representation, then put $l_1 =0$ or $1$ (exactly one of these weights occurs in $\tau$). If $\tau_\infty$ is a discrete series (or a limit of discrete series) representation, then put $l_1 = p +1$; hence, $l_1$ is the lowest weight of $\tau_\infty$.

\vspace{2ex}
With this choice, we can check by an explicit calculation that the function $R(s)$ defined in~\eqref{rsdef} has no poles in the region $0\leq\Re(s)\leq\frac{1}{4}$. Indeed, the only possible pole in that region can come from $R_\infty(s)$, and so it boils down to checking that the function $T_\infty(s)$ defined in Theorem~\ref{theorem-arch-local-pullback} and the function $Y_{l,l_1,p,q}(s)$ defined in Theorem~\ref{archlocalzetatheorem} are non-zero when $0\leq\Re(s)\leq\frac{1}{4}$. It is easy to verify that this is true with our choice of $l_1$.

\vspace{2ex}
On the other hand, by Theorem~\ref{theorempoleseis}, the only possible pole of $E_{\Upsilon}^\ast(g,s)$ in the region $0\leq\Re(s)\leq\frac{1}{4}$ is at $s = \frac16$; this pole can occur only if $\omega_\tau = 1$. The result now follows from Theorem~\ref{theoremsecondintegralrep} and Lemma~\ref{poleregion}.
\end{proof}
\subsection{Eisenstein series and Weil representations}
In view of Proposition \ref{atmostonepoleprop}, we will now assume that $\omega_\tau=1$, and that the integer $l_1$ used in the definition of $\Upsilon_\infty$ is equal to $p+1$ in the discrete series case, and $0$ or $1$ otherwise. By abuse of notation, we continue to use $E_{\Upsilon}(g,s)$ to denote its restriction to $\U(3,3)(\A)$. Indeed, this restricted function is an Eisenstein series on $\U(3,3)(\A)$.
For brevity, we will use $\G$ to denote $\U(3,3)$. Let $K^{\G}$ denote the standard maximal compact subgroup of $\G(\A)$. Let $I(\chi, s)$ be the set of holomorphic vectors in the global induced representation defined analogously to $ I(\widetilde{\chi},s)$ as in~\eqref{Ichisglobaldefeq}, except that we are now dealing with functions on $\U(3,3)$ rather than $\GU(3,3)$. In other words $I(\chi, s)$ consists of the sections $f^{(s)}$ on $\G(\A)$ such that \begin{equation}\label{e:upsilondefformula2}
 f^{(s)}(m(A,1)ng,s) = |N (\det A)|^{3(s + \frac12)}\,
 \chi(\det(A))\,f^{(s)}(g,s)
\end{equation}
for all $g\in \G(\A)$, and so that $f^{(s)}$ is holomorphic (in the sense of~\cite[p.\ 251]{ich}). In particular, any such section can be written as a finite linear combination of standard sections with holomorphic coefficients. A key example of a standard section is simply the restriction of the previously defined $\Upsilon(g,s)$ to $\U(3,3)$.

\vspace{2ex}
Recall that $\phi=\otimes\phi_v$ is a vector in the space of $\pi$ such that $\phi_v$ is unramified for all finite $v$ and such that $\phi_\infty$ is a vector of weight $(-l,-l)$ in $\pi_\infty$. We have the following lemma.

\begin{lemma}\label{lemmaeisresidue}
 Suppose that the Eisenstein series $E_{\Upsilon}( g,s)$ on $\U(3,3)(\A)$ has the property that for all $g_1 \in \U(1,1)(\A)$, we have
 \begin{equation}\label{lemmaeisresidueeq1}
  \int\limits_{\mathrm{Sp}_4(\Q) \backslash \mathrm{Sp}_4(\A)}\Res_{s=\frac16}E_{\Upsilon}(\i(g_1,h_1),s))\,\phi(h_1)\,dh_1 = 0.
 \end{equation}
 Then  $L(s,\pi\times\tau)$ is holomorphic at $s=1$.
\end{lemma}
\begin{proof}
By Theorem~\ref{theoremsecondintegralrep}, the fact that $R(s)$ has no pole at $s=\frac16$, and the fact that
$$
 L(6s+1,\chi|_{\A^\times})L(6s+2,\chi_{L/F}\,\chi|_{\A^\times})L(6s+3,  \chi|_{\A^\times})
$$
is finite and non-zero at $s = \frac16$, it follows that if
$$
 \int\limits_{\mathrm{Sp}_4(\Q) \backslash H[g](\A)}\Res_{s=\frac16}E_{\Upsilon}(\i(g,h),s))\,\phi(h)\,dh=0\qquad\text{for all }g \in G_1(\A),
$$
then $L(s,\pi\times\tau)$ is holomorphic at $s=1$. Suppose  $E_{\Upsilon}( g,s)$ has the property (\ref{lemmaeisresidueeq1}). If $g \in G_1(\A)$ with $\mu_1(g) = m$, we can write $m = \lambda z k$ with $\lambda \in \Q^\times$, $z \in \R^+$, $k \in \prod_{p<\infty} \Z_p^\times$. It follows that we can write
$$
 g=\mat{1}{}{}{\lambda}g_1\mat{z^{1/2}}{}{}{z^{1/2}}\mat{1}{}{}{k}
$$
with $g_1 \in \U(1,1)(\A)$. A similar decomposition holds for $h$ with $\mu_2(h) = m$. Thus
$$
 E_{\Upsilon}(\i(g,h),s))\phi(h) = E_{\Upsilon}(\i(g_1,h_1),s))\phi(h_1)
$$
with $g_1$, $h_1$ belonging to $\U(1,1)(\A)$, $\mathrm{Sp}_4(\A)$ respectively. The lemma follows.
\end{proof}

We will reinterpret the condition of the lemma in terms of Weil representations and theta liftings. Let $(V, \mathcal{Q})$ be a non-degenerate Hermitian space over $L$ of dimension $4$. We identify $\q$ with a Hermitian matrix of size $4$. Let $\U(V)$ be the unitary group of $V$; thus $$\U(V)(\Q)=\{g\in\GL_4(L)\;|\;^t\bar g\q g=\q\}.$$

Let $\chi$ be as above. Fix an additive character $\psi$ as before. As described in~\cite{ich}, there is a Weil representation $\omega_{\q}=\omega_{\q,\psi,\chi}$ of $\G(\A) \times \U(V)(\A)$ acting on the Schwartz space $\mathcal S(V^3(\A))$. The explicit formulas for the action can be found in~\cite[p.\ 246]{ich}.

\vspace{2ex}
Let $s_0 =\frac16$. Let $S(V^3(\A))$ denote the space of $K^{\G}$-finite vectors in $\S(V^3(\A))$. Write $\Pi(V)$ for the image of the $\G(\A)$ intertwining map from $S(V^3(\A))$ to $I(\chi,s_0)$ given by
$$
 \varphi \mapsto f_\varphi^{(s_0)},
$$
where $f_\varphi^{(s_0)}(g)=(\omega_{\q}(g)\varphi)(0)$. We can extend $f_\varphi^{(s_0)}$ to a standard section $f_\varphi^{(s)} \in I(\chi, s)$ via
$$
 f_\varphi^{(s)}(g,s) = |N (\det A)|^{3(s - s_0)}\,f_\varphi^{(s_0)}(g),
$$
where we use the Iwasawa decomposition to write $g=m(A,1)nk$ with $A\in\GL_3(\A_L)$, $n$ in the unipotent radical of the Siegel parabolic subgroup, and $k\in K^{G_3'}$.

\vspace{2ex}
Next we deal with the local picture. Suppose that $(V^{(v)}, \mathcal{Q}^{(v)})$ is a non-degenerate Hermitian space over $L_v$ of dimension $4$.
Then we have the local Weil representation $\omega_{\q_v} = \omega_{\q_v,\psi_v,\chi_v}$ of $\G(\Q_v) \times \U(V^{(v)})(\Q_v)$ acting on the Schwartz space $\mathcal S((V^{(v)})^3)$.
We define $R(V^{(v)})$  to be the image of the $\G(\Q_v)$ intertwining map from $S((V^{(v)})^3)$ to $I_v(\chi_v,s_0)$ given by
$$
 \varphi \mapsto f_\varphi^{(s_0)},
$$
where $f_\varphi^{(s_0)}(g)=(\omega_{\q_v}(g)\varphi)(0)$. The span of the various subspaces $R(V^{(v)})$ of $I_v(\chi_v,s_0)$ as $V^{(v)}$ ranges over the various inequivalent non-degenerate Hermitian spaces over $L_v$ of dimension $4$ is well understood. The non-archimedean case is treated in~\cite{kudswe} while the archimedean case is treated in~\cite{leezhu}. For instance, the following result \cite[Thm.\ 1.2]{kudswe} describes the case when $v$ is non-archimedean and $L_v$ is a field.

\begin{theorem}[Kudla--Sweet]
 Suppose that $v$ is non-archimedean and $L_v$ is a field. Let $V^{(v)}_1$ and $V^{(v)}_2$ be the two inequivalent non-degenerate Hermitian vector spaces over $L_v$ of dimension 4. Then $R(V^{(v)}_1)$ and $R(V^{(v)}_2)$  are distinct maximal submodules of $I_v(\chi_v,s_0)$, so that
 $$
  I_v(\chi_v,s_0) = R(V^{(v)}_1) + R(V^{(v)}_2).
 $$
\end{theorem}

In the case when $v$ is non-archimedean and $L_v = F_v \oplus F_v$, a similar result is provided by~\cite[Thm.\ 1.3]{kudswe}, while the case $v = \infty$ is dealt with in~\cite{leezhu}. Now, let $\mathcal{C} = \{V^{(v)}\}$ be a collection, over all places $v$ of $\Q$, of local  non-degenerate Hermitian spaces over $L_v$ of dimension 4. Whenever $v$ is non-archimedean and $L_v$ is a field, there are two inequivalent choices for $V^{(v)}$. Each of these spaces has an isotropic vector~\cite[Lemma 5.2]{kudswe}. If  $v$ is non-archimedean and $L_v = F_v \oplus F_v$, then the ``Galois" automorphism is given by $(x_1, x_2) \mapsto (x_2, x_1)$. In this case the resulting ``norm" map from $L_v$ to $F_v$ is surjective. So there is only one isometry class for $V^{(v)}$ and the ``unitary" group of $V^{(v)}$ is isomorphic to $\GL_4(F_v)$. Indeed, up to isometry, the space $V^{(v)}$ is explicitly given by $V^{(v)} = F_v^4 \oplus F_v^4$ with $(a,b) = [^t a_1\cdot b_2,\,^t a_2\cdot b_1]$ where $a=(a_1, a_2)$, $b=(b_1, b_2)$. Finally, if $v = \infty$, there are 5 such $V^{(v)}$, corresponding to spaces of signature $(p, q)$ with $p+q=4$. For any such collection $\mathcal{C}$ as above, let $\Pi(\mathcal{C})$ be the representation space defined by $$\Pi(\mathcal{C}) = \otimes_v R( V^{(v)} ).$$ The upshot of the local results from~\cite{kudswe} and~\cite{leezhu} is that the natural map from $\oplus _{\mathcal{C}} \Pi(\mathcal{C})$ to $I(\chi, s_0)$ is surjective; here the sum ranges over all inequivalent collections $\mathcal{C}$ as above. Let $\mathcal{A}(\G)$ denote the space of automorphic forms over $\G(\A)$. Define $A_{-1}$ to be the $\G(\A)$ intertwining map from $I(\chi,s_0)$ to $\mathcal{A}(\G)$ given by
$$
 f^{(s_0)} \longmapsto \Res_{s=s_0} E_{f^{(s)}}(g,s).
$$
We note here (see \cite[p.\ 252]{ich}) that the residue of the Eisenstein series at some point $s_0$ only depends on the section at $s_0$, so the above map is indeed well defined.

\vspace{2ex}
Next, for any local Hermitian space $V^{(v)}$ as above, with $v$ non-archimedean, let $V^{(v)}_0$ denote the complementary space, which is defined to be the space of dimension $2$ over $L_v$ in the same Witt class as $V^{(v)}$. Note that such a space exists because (by our comments above) $V^{(v)}$ always has an isotropic vector if $v$ is non-archimedean. The subspace  $R( V^{(v)}_0 )$ of $I_v(\chi_v, -s_0)$ is defined similarly as above. It turns out (see \cite{ich}, \cite{kudswe}) that for any non-archimedean place $v$, the restriction of the intertwining operator maps $R( V^{(v)})$ onto $R( V^{(v)}_0)$. This identifies $R( V^{(v)}_0)$ as a quotient of $R( V^{(v)})$; in fact it is the unique irreducible quotient of $R( V^{(v)})$. Moreover, if $f^{(s_0)}$ is a factorizable section, and the local section at a non-archimedean place $v$ lies in the kernel of the above map from $R(V^{(v)})$ to $R( V^{(v)}_0)$, then $A_{-1}(f^{(s_0)})=0$. This follows from~\cite[Lemma 6.1]{ich}; the lemma only states the result for the case that $L_v$ is a field, but the same proof also works for the split case using the local results from~\cite[Sect.\ 7]{kudswe}.

\vspace{2ex}
From the above discussion, we conclude that the map $A_{-1}$ factors through the quotient $$I_\infty(\chi_\infty, s_0) \otimes \left(\oplus \Pi(\mathcal{C'}) \right),$$ where  $\mathcal{C'}= \{V_0^{(v)} \}$ runs over all inequivalent collections of local Hermitian spaces $V_0^{(v)}$ of dimension 2 over $L_v$ with $v$ ranging over the \emph{non-archimedean} places. (Compare~\cite[Prop.\ 4.2]{kudral} for the analogous result in the symplectic case.) But we can say more. For any global Hermitian space $V_0$ of dimension $2$ over $L$, let $\Pi(V_0)$ be the  image of the $\G(\A)$ intertwining map from $S(V_0^3(\A))$ to $I(\chi,-s_0)$. Note that (at each place, and hence globally) $\Pi(V_0)$ is naturally a quotient (via the intertwining operator) of $\Pi(V)$, where $V$ is the complementary global Hermitian space of dimension 4 over $L$, obtained by adding a split space of dimension 2 to $V_0$.

\begin{proposition}\label{a1factor}The map $A_{-1}$ from $I(\chi,s_0)$ to $\mathcal{A}(\G)$ factors through $\oplus_{V_0} \Pi(V_0)$ where $V_0$ runs through all global Hermitian spaces of dimension $2$ over $L$.
\end{proposition}
\begin{proof} We have already seen that the map $A_{-1}$ factors through the quotient $I_\infty(\chi_\infty, s_0) \otimes \left(\oplus \Pi(\mathcal{C'}) \right),$ where  $\mathcal{C'}= \{V_0^{(v)} \}$ runs over all inequivalent collections of local Hermitian spaces $V_0^{(v)}$ of dimension 2 over $L_v$ with $v$ ranging over the non-archimedean places of $\Q$. The argument of~\cite[p.\ 363--364]{Tan2} takes care of the archimedean place, and we get that $A_{-1}$ factors through $\oplus \Pi(\mathcal{C'}),$ where  $\mathcal{C'} = \{V_0^{(v)} \}$ runs over all inequivalent collections of local Hermitian spaces $V_0^{(v)}$ of dimension 2 over $L_v$; here $v$ ranges over all the places of $\Q$ including $\infty$.

\vspace{2ex}
The question now is if there exists a global  Hermitian space  $V_0$ whose localizations are precisely the local spaces  $\{V^{(v)}_0\}$ in the collection $\mathcal{C'}.$ If such a global Hermitian space does not exist, then the collection $\mathcal{C'}$ is called incoherent, otherwise it is called coherent. From the local results quoted above, we know that each $\Pi(\mathcal{C'})$ is irreducible. Thus to complete the proof we only need to show that $ \Pi(\mathcal{C'})$ cannot be embedded in $\mathcal{A}(\G)$ if $\mathcal{C'}$ is incoherent.

\vspace{2ex}
The proof that such an embedding cannot exist is fairly standard. See, for instance~\cite[Thm. 3.1 (ii)]{kudral},~\cite[Prop. 2.6]{kudralsoud} or~\cite[Cor. 4.1.12]{Tan2}. Thus, we will be brief. For any (global) Hermitian matrix $\beta$ of size 3, let $W_\beta:\:\mathcal{A}(\G)\to \C$ denote the $\beta$-th Fourier coefficient, defined by
$$
 W_\beta(f) = \int\limits_{N_{12}(\Q) \bs N_{12}(\A)}f(nb)\psi(-\tr (b \beta))\,db.
$$
Let $\mathcal{D}$ be a non trivial embedding of $\Pi(\mathcal{C'})$ in $\mathcal{A}(\G)$ where $\mathcal{C'}$ is incoherent, and put $\mathcal{D}_\beta = W_\beta \circ D$. Then there must exist some $\beta$ such that $\mathcal{D}_\beta$ is non-zero. Moreover, if $\mathcal{D}_\beta = 0$ for all $\beta$ of rank $\ge2$, then the argument of~\cite[Lemma 2.5]{kudralsoud} shows that $D=0$. So there exists $\beta$ with rank$(\beta) \ge 2$ and  $\mathcal{D}_\beta \neq 0$. By well-known results on the twisted Jacquet functor (see~\cite[Lemmas 5.1 and 5.2]{ich}), this implies that $\beta$ is locally represented by $\{V^{(v)}_0\}$ at each place $v$, i.e., there exists $v_0^{(v)} \in (V^{(v)}_0)^3$ such that $(v_0^{(v)}, v_0^{(v)}) = \beta$. Since the dimension of $V^{(v)}_0$ is $2$, this implies that such a $\beta$ cannot be non-singular; thus rank$(\beta)= 2.$ Hence $\beta$ is (globally) equivalent to $\mat{\beta_0}{}{}{0}$ where $\beta_0$ is of size 2 and non-singular.
Let $\epsilon_v(V^{(v)}_0) = \pm 1$ denote the Hasse invariant of the local Hermitian space $V^{(v)}_0$. Since the collection $\mathcal{C'}$ is incoherent, we have $\prod_v \epsilon_v(V^{(v)}_0) =-1$. On the other hand, because $\beta$ is locally represented by  $\{V^{(v)}_0\}$, and rank$(\beta)$= dim$(V^{(v)}_0) = 2,$ it follows that the matrix for $ V^{(v)}_0$ equals $\beta_0$ for some suitable basis. But this means that $\epsilon_v(V^{(v)}_0) = \epsilon_v(\beta_0)$. So $\prod_v \epsilon_v(V^{(v)}_0) = \prod_v \epsilon_v(\beta_0) =1$, a contradiction.
\end{proof}
\subsection{The Siegel-Weil formula and the proof of entireness}\label{siegelweilsec}
In the previous subsection, we proved that the map $A_{-1}$ from $I(\chi,s_0)$ to $\mathcal{A}(\G)$ given by
$$
 f^{(s_0)} \longmapsto \Res_{s=s_0} E_{f^{(s)}}(g,s)
$$
factors through $\oplus_{V_0} \Pi(V_0)$, where $V_0$ runs through all global Hermitian spaces of dimension $2$ over $L$. It turns out that the same map is also given by a regularized theta integral. This is the content of the regularized Siegel-Weil formula, which we now recall. Let $(V_0, \q_0)$ be a global Hermitian space of dimension $2$ over $L$ and let $(V,\q)$ be the global Hermitian space of dimension $4$ over $L$ obtained by adding a split space of dimension 2 to $V_0$. Note that the Witt index of $V$ is at least 1, thus $V$ cannot be anisotropic. Given $\varphi_0 \in S(V_0^3(\A))$ we define the theta function
\begin{equation}
 \Theta(g,h;\varphi_0)=\sum_{x \in V_0^3(\Q)}\omega_{\q_0}(g,h)\varphi_0(x).
\end{equation}
This is a slowly increasing function on $(\G(\Q) \bs \G(\A)) \times (\U(V_0)(\Q) \bs \U(V_0)(\A))$. If $\q_0$ is anisotropic, we define
$$
 I_{\q_0}(g, \varphi_0) = \int\limits_{\U(V_0)(\Q) \bs \U(V_0)(\A)}\Theta(g,h;\varphi_0)\,dh.
$$
If $\q_0$ is isotropic, the above integral does not converge, so we define
$$
 I_{\q_0}(g, \varphi_0) = c_{\alpha}^{-1}\int\limits_{\U(V_0)(\Q) \bs \U(V_0)(\A)}\Theta(g,h;\omega_{\q_0}(\alpha)\varphi_0)\,dh,
$$
where $\alpha$, $c_\alpha$ are defined as in Sect.\ 2 of \cite{ich}. In fact, in the convergent case, the second definition automatically equals the first, so we might as well use it in both the cases. Next, one has a map of Schwartz functions $\pi_{\q}^{\q_0}\pi_K$ from $S(V^3(\A))$ to $K_0$-invariant functions in $S(V_0^3(\A))$; here $K_0$ is the standard maximal compact subgroup of $\U(V_0(\A))$. We refer the reader to~\cite{ich} for definitions and details. Let $\varphi \in  S(V^3(\A))$. Let $f_\varphi^{(s)} \in I(\chi, s)$ be the standard section attached to $\varphi$ via the Weil representation. Then the regularized Siegel-Weil formula~\cite[Thm.\ 4.1]{ich} in this setting says the following.

\begin{theorem}[Ichino]\label{theoremsiegelweil}  We have $$\Res_{s=s_0} E_{f_\varphi^{(s)}}(g,s) = c \ I_{\q_0}(g, \pi_{\q}^{\q_0}\pi_K \varphi)$$ for an explicit constant $c$ depending only on the normalization of Haar measures.
\end{theorem}

Theorem~\ref{theoremsiegelweil} and Proposition~\ref{a1factor} imply the following result.

\begin{proposition}\label{propintegralnonzero}
 Suppose that the Eisenstein series  $E_{\Upsilon}( g,s)$ does not satisfy the property (\ref{lemmaeisresidueeq1}). Then there exists a Hermitian space $(V_0, \q_0)$ of dimension $2$ over $L$ and a $K_0$-invariant Schwartz function $\varphi_0 \in S(V_0^3(\A))$ such that, for some $g\in\U(1,1)(\A)$,
 $$\int\limits_{\mathrm{Sp}_4(\Q) \backslash \mathrm{Sp}_4(\A)} I_{\q_0}(\i(g,h),\varphi_0) \phi(h)\,dh \ne 0.$$
\end{proposition}
We will now prove Theorem~\ref{entirenesstheorem}. In order to do so, it suffices to show that the conclusion of Proposition~\ref{propintegralnonzero} leads to a contradiction. First note that, given Schwartz functions $\varphi_1 \in S(V_0(\A))$, $\varphi_2 \in S(V_0^2(\A))$, we may form the Schwartz function $\varphi_0 = \varphi_2 \otimes \varphi_1 \in S(V_0^3(\A))$ defined by $\varphi(v_1, v_2, v_3) = \varphi_1(v_3)\varphi_2(v_1, v_2)$. The space generated by linear combinations of functions of this type is the full Schwartz space $S(V_0^3(\A))$. Suppose that the conclusion of Proposition~\ref{propintegralnonzero} holds. By the definition of $I_{\q_0}$ and the above discussion, it follows that we can find $\varphi_1\in S(V_0(\A))$, $\varphi_2\in S(V_0^2(\A))$  such that for some $g \in \U(1,1)(\A)$, we have
\begin{equation}\label{thetaintnonzero}
 \int\limits_{\mathrm{Sp}_4(\Q) \backslash\mathrm{Sp}_4(\A)}\:\int\limits_{U(V_0)(\Q) \bs U(V_0)(\A)}\Theta(\i(g,h),h';\varphi_2 \otimes \varphi_1)\,\phi(h)\,dh'\,dh \ne 0.
\end{equation}
For $g =\mat{a}{b}{c}{d}\in \U(1,1)$ let $\hat g = \mat{a}{-b}{-c}{d}$. It is easy to check that
$$
 \omega_{\q_0}(\i(g,h))(\varphi_2 \otimes \varphi_1) = \omega_{\q_0}(h)\varphi_2\otimes\omega_{\q_0}(\hat g)\varphi_1.
$$
Here, we are abusing notation and using $\omega_{\q_0}$ to denote the Weil representation of $G_i(\A)$ on $\S(V_0^i(\A))$ for various $i$. This gives the following factorization,
\begin{equation}\label{thetafactorise}
 \Theta(\i(g,h), h'; \varphi_2 \otimes \varphi_1) = \Theta(\hat g,h'; \varphi_1) \Theta(h,h'; \varphi_2).
\end{equation}
Define the automorphic form $\Theta(h';\phi, \varphi_2)$ on $\U(V_0)(\Q) \bs \U(V_0)(\A)$ by
$$
 \Theta(h';\phi, \varphi_2) = \int\limits_{\mathrm{Sp}_4(\Q) \backslash\mathrm{Sp}_4(\A)}\Theta(h,h'; \varphi_2)\phi(h)\,dh.
$$
Equations \eqref{thetaintnonzero} and~\eqref{thetafactorise} imply the following.
\begin{lemma}\label{propthetanonzero}
 Suppose that the conclusion of Proposition~\ref{propintegralnonzero} holds. Then there exists a Schwartz function $\varphi_2 \in S(V_0^2(\A))$ such that the automorphic form $\Theta(h';\phi, \varphi_2)$ on $\U(V_0)(\A)$ is non-zero.
\end{lemma}
We will now interpret the conclusion of this lemma in terms of theta liftings. Let $V_0^\ast$ denote the $4$-dimensional orthogonal space over $\Q$ obtained by considering $V_0$ as a space over $\Q$ and composing the hermitian form on $V_0$ with $\text{tr}_{L/\Q}$. We have the following seesaw diagram (see~\cite[p.\ 252]{kudnotes}) of dual reductive pairs.\footnote{We would like to thank Paul Nelson for pointing this out to us.}
$$
 \xymatrix{\U(2,2)\ar@{-}[dr]\ar@{-}[d]&\O(V_0^\ast)\ar@{-}[d]\\
    \Sp(4)\ar@{-}[ur]&\U(V_0)}
$$
Note that, at each place, $V_0^\ast$ is either the unique anisotropic space of dimension four, or the split quadratic space $V_{2,2}$. Let $\pi_1$ be the representation of $\Sp_4(\A)$ generated by the restriction of $\phi$ to $\SSp_4(\A)$. By~\cite{NPS}, we know that $\pi_1$ is an irreducible, automorphic, cuspidal representation. Moreover, $\pi_1$ is an anti-holomorphic discrete series representation at infinity. The above seesaw diagram and Lemma~\ref{propthetanonzero} imply that if the conclusion of Proposition~\ref{propintegralnonzero} holds, then $\pi_1$ has a non-zero theta lift to $\O(V_0^\ast)$.

\vspace{2ex}
However, if $\pi_1$ has a non-zero theta lift to $\O(V_0^\ast)$, then $V_0^\ast$ cannot be split at infinity. This is because there is no local archimedean theta lift of an anti-holomorphic discrete series representation from $\Sp_4(\R)$ to $\O(2,2)(\R)$, see~\cite{tomasz}. This means there must be a non-archimedean place $v$ where $V_0^\ast$ is ramified. But this implies that $\pi_1$ is also ramified at $v$; else the local theta lift would be zero. However, we know that $\pi_1$ is unramified at all finite places because $\phi$ is right-invariant under $\Sp_4(\Z_p)$ at all finite places $p$. This contradiction shows that the conclusion of Proposition~\ref{propintegralnonzero} cannot hold. Therefore the Eisenstein series $E_{\Upsilon}( g,s)$ on $\U(3,3)(\A)$ has the property that, for all $g \in \U(1,1)(\A)$,
$$
 \int\limits_{\mathrm{Sp}_4(\Q) \backslash \mathrm{Sp}_4(\A)}\Res_{s=\frac16}E_{\Upsilon}(\i(g,h),s))
\phi(h)\,dh = 0,
$$
and hence $L(s,\pi\times\tau)$ is holomorphic at $s=1$. This completes the proof of Theorem~\ref{entirenesstheorem}.
\section{Applications}
As a special case of Langlands functoriality, one expects that automorphic forms on $\GSp_4$ have a functorial transfer to automorphic forms on $\GL_4$, coming from the natural embedding of dual groups $\GSp_4(\C)\subset\GL_4(\C)$. For generic automorphic representations on $\GSp_4$ this transfer was established in \cite{AS}. There is also a conjectured functorial transfer from automorphic forms on $\PGSp_4$ to automorphic forms on $\GL_5$, coming from the morphism $\rho_5:\SSp_4(\C)\rightarrow\GL_5(\C)$ of dual groups, where $\rho_5$ is the irreducible $5$-dimensional representation of $\SSp_4(\C)$. Here, we are going to show the existence of both these transfers for full level holomorphic cuspidal Siegel eigenforms. Note that the automorphic representation generated by such a Siegel modular form is not globally generic, since its archimedean component, a holomorphic discrete series representation, is non-generic.

\vspace{2ex}
We will use the transfer results to prove analytic properties of several $L$-functions related to Siegel modular forms. In the last subsection we will derive some special value results for $\GSp_4\times\GL_1$ and $\GSp_4\times\GL_2$ $L$-functions.
\subsection{The transfer theorems}\label{transfersec}
In the following let $\A$ be the ring of adeles of $\Q$. As before we write $H$ for $\GSp_4$, considered as an algebraic group over $\Q$. Let $\pi=\otimes\pi_v$ be a cuspidal, automorphic representation of $H(\A)$ with the following properties.
\begin{itemize}
 \item $\pi$ has trivial central character.
 \item The archimedean component $\pi_\infty$ is a holomorphic discrete series representation with scalar minimal $K$-type $(l,l)$, where $l\geq3$.
 \item For each finite place $p$, the local representation $\pi_p$ is unramified.
\end{itemize}
It is well known that every such $\pi$ gives rise to a holomorphic cuspidal Siegel eigenform of degree $2$ and weight $l$ with respect to the full modular group $\SSp_4(\Z)$; see \cite{ASch}. Conversely, every such eigenform generates an automorphic representation $\pi$ as above (which is in fact irreducible; see \cite{NPS}). Well-known facts about classical full-level Siegel modular forms show that the cuspidality condition implies $l\geq10$. For the following lemma let $\psi$ be the standard global additive character that was used in Theorem \ref{globalintegralrepresentationtheoremversion1} and Theorem \ref{theorem-global-pullback}. Recall the definition of global Bessel models from Sect.\ \ref{besselsec}.
\begin{lemma}\label{piFBesselFourierlemma}
 Let $\pi$ be as above, and let $F$ be the corresponding Siegel cusp form. Assume that the Fourier expansion of $F$ is given by
 \begin{equation}\label{piFBesselFourierlemmaeq1}
  F(Z)=\sum_Sa(F,S)e^{2\pi i\,{\rm tr}(SZ)},
 \end{equation}
 where $Z$ lies in the Siegel upper half space of degree $2$, and $S$ runs through $2 \times 2$ positive definite, semi-integral, symmetric matrices. Then, given a positive integer $D$ such that $-D$ is a fundamental discriminant, the following are equivalent.
 \begin{enumerate}
  \item $a(F,S)\neq0$ for some $S$ with $D=4\det(S)$.
  \item $\pi$ has a Bessel model of type $(S(-D),\Lambda,\psi)$, where $S(-D)$ is the matrix defined in (\ref{SminusDdefeq}), and where $\Lambda$ is a character of the ideal class group (\ref{idealclassgroupeq}) of $L=\Q(\sqrt{-D})$.
 \end{enumerate}
\end{lemma}
{\bf Proof:} This follows from equation (4.3.4) in \cite{Fu} (which is based on (1-26) of \cite{Su}).\qed

\vspace{3ex}
The second author has recently shown that condition i) of the lemma is always satisfied for some $D$; see \cite{squarefree}. In fact, independently of whether $F$ is an eigenform or not, there exist infinitely many non-zero Fourier coefficients $a(F,S)$ such that $D=4\det(S)$ is odd and squarefree (in which case $-D$ is automatically a fundamental discriminant). The important fact for us to note is that \emph{there always exists a positive integer $D$ such that $-D$ is a fundamental discriminant and such that $\pi$ satisfies the hypotheses of Theorem \ref{functionalequationtheorem}}.

\vspace{3ex}
We shall write down the explicit form of the local parameters of the representations $\pi_v$. These are admissible homomorphisms from the local Weil groups to the dual group $\GSp_4(\C)$. Note that the trivial central character condition implies that the image of each local parameter lies in $\SSp_4(\C)$. As in \cite{Ta} (1.4.3), the real Weil group $W_\R$ is given by $W_\R=\C^\times\sqcup j\C^\times$ with the rules $j^2=-1$ and
$jzj^{-1}=\bar z$ for $z\in\C^\times$.
Then the parameter of $\pi_\infty$ is given by
\begin{equation}\label{localparameterarcheq}
 \C^\times\ni re^{i\theta}\longmapsto\begin{bmatrix}e^{i(2l-3)\theta}\\&e^{i\theta}\\
 &&e^{-i(2l-3)\theta}\\&&&e^{-i\theta}\end{bmatrix},\qquad
 j\longmapsto\begin{bmatrix}&&-1\\&&&-1\\1\\&1\end{bmatrix}.
\end{equation}
For a finite place $p$, there exist unramified characters $\chi_1$, $\chi_2$ and $\sigma$ of $\Q_p^\times$ such that $\pi_p$ is the spherical component of a parabolically induced representation $\chi_1\times\chi_2\rtimes\sigma$ (using the notation of \cite{ST}). If we identify characters of $\Q_p^\times$ with characters of the local Weil group $W_{\Q_p}$ via local class field theory, then the $L$-parameter of $\pi_p$ is given by
\begin{equation}\label{localparameternonarcheq}
 W_{\Q_p}\ni w\longmapsto\begin{bmatrix}\sigma(w)\chi_1(w)\\&\sigma(w)\chi_1(w)\chi_2(w)\\&&\sigma(w)\chi_2(w)\\&&&\sigma(w)\end{bmatrix}.
\end{equation}
The central character condition is $\chi_1\chi_2\sigma^2=1$, so that the image of this parameter lies in $\SSp_4(\C)$. Now, let $\Pi_\infty$ be the irreducible, admissible representation of $\GL_4(\R)$ with $L$-parameter (\ref{localparameterarcheq}). For a prime number $p$, let $\Pi_p$ be the irreducible, admissible representation of $\GL_4(\Q_p)$ with $L$-parameter (\ref{localparameternonarcheq}). Then the irreducible, admissible representation
\begin{equation}\label{candidaterepresentationeq}
 \Pi_4:=\otimes\Pi_v
\end{equation}
of $\GL_4(\A)$ is our candidate representation for the transfer of $\pi$ to $\GL_4$. Clearly, $\Pi_4$ is self-contragre\-dient.

\begin{theorem}\label{liftingtheorem}
 Let $\pi$ be a cuspidal automorphic representation of $\GSp_4(\A)$ as above, related to a cuspidal Siegel eigenform $F$. We assume that $F$ is not of Saito-Kurokawa type. Then the admissible representation $\Pi_4$ of $\GL_4(\A)$ defined above is cuspidal automorphic. Hence $\Pi_4$ is a strong functorial lifting of $\pi$. This representation is symplectic, i.e., the exterior square $L$-function $L(s,\Pi_4,\Lambda^2)$ has a pole at $s=1$.
\end{theorem}
\begin{proof} We will use the converse theorem for $\GL_4$ from \cite{CPS}, and therefore have to establish the ``niceness'' of the $L$-functions of twists of $\Pi$ by cusp forms on $\GL_1$ and $\GL_2$. As remarked above, there exists a positive integer $D$ such that $-D$ is a fundamental discriminant and such that $\pi$ satisfies the hypotheses of Theorem \ref{functionalequationtheorem}; we will fix such a $D$. Let $\tau=\otimes\tau_p$ be a cuspidal, automorphic representation of $\GL_2(\A)$ such that $\tau_p$ is unramified for $p|D$. By definition of the candidate representation $\Pi$, the $\GL_4\times\GL_2$ $L$-function $L(s,\Pi_4\times\tau)$ coincides with the $\GSp_4\times\GL_2$ $L$-function $L(s,\pi\times\tau)$. Therefore, by Theorem \ref{entirenesstheorem}, the $L$-function $L(s,\Pi_4\times\tau)$ has analytic continuation to an entire function. Moreover, by Theorem \ref{functionalequationtheorem}, it satisfies the functional equation
\begin{equation}\label{functionalequationtheoremeq2b}
  L(s,\Pi_4\times\tau)=\varepsilon(s,\Pi_4\times\tau)L(1-s,\tilde\Pi_4\times\tilde\tau).
\end{equation}
We will next prove that $L(s,\Pi_4\times\tau)$ is bounded in vertical strips.\footnote{We would like to thank Mark McKee for explaining this argument to us.} Consider the group $\GSp_8$ and its Levi subgroup $\GL_2\times\GSp_4$. One of the representations of the dual parabolic with Levi $\GL_2(\C)\times\GSp_4(\C)$ on the dual unipotent radical is the tensor product representation. This means that our
$L$-function $L(s,\pi\times\tau)$ is accessible via Langlands' method; see \cite{yale}. Now, Gelbart and Lapid proved that \emph{any} L-function that is accessible via Langlands' method is meromorphic of finite order; this is Theorem 2 in \cite{GL}. Here, a function $f:\C\rightarrow\C$ being of finite order means that there exist positive constants $r,c,C$ such that
$$
 |f(z)|\leq C\,e^{c|z|^r}\qquad\text{for all }z\in\C.
$$
By the Phragmen-Lindel\"of Theorem, if a holomorphic function of finite order is bounded on the left and right boundary of a vertical strip, then it is bounded on the entire vertical strip. For a large enough positive number $M$, our function $L(s,\pi\times\tau)$ is bounded on $\Re(s)=M$, since it is given as a product of archimedean Euler factors, which are bounded on vertical lines, times a convergent Dirichlet series. By the functional equation, $L(s,\pi\times\tau)$ is also bounded on $\Re(s)=-M$. It follows that $L(s,\pi\times\tau)$ is bounded on $-M\leq\Re(s)\leq M$. This proves that $L(s,\Pi_4\times\tau)$ is bounded in vertical strips.

\vspace{3ex}
A similar argument applies to twists of $\Pi_4$ by Hecke characters $\chi$ of $\A^\times$. The required functional equation of $L(s,\Pi_4\times\chi)=L(s,\pi\times\chi)$ is provided by \cite{KR}. The holomorphy follows from Theorem 2.2 of \cite{PSh}.

\vspace{3ex}
By Theorem 2 of \cite{CPS}, there exists an automorphic representation $\Pi'=\otimes\Pi'_v$ of $\GL_4(\A)$ such that $\Pi'_\infty\cong\Pi_\infty$ and $\Pi'_p\cong\Pi_p$ for all primes $p\nmid D$. We claim that in fact $\Pi'_p\cong\Pi_p$ for \emph{all} primes $p$; this will prove that the candidate representation $\Pi_4$ is automorphic (but not yet the cuspidality).
To prove our claim, observe that we have the functional equations
\begin{equation}\label{Pifunceq1}
 L(s,\Pi_4)=\varepsilon(s,\Pi_4)L(1-s,\tilde\Pi_4)
\end{equation}
and
\begin{equation}\label{Pifunceq2}
 L(s,\Pi')=\varepsilon(s,\Pi')L(1-s,\tilde\Pi').
\end{equation}
We have (\ref{Pifunceq1}) because $L(s,\Pi_4)=L(s,\pi)$ and $\varepsilon(s,\Pi_4)=\varepsilon(s,\pi)$ by definition of $\Pi_4$, so that we can use Andrianov's classical theory; see \cite{An1974}. We have (\ref{Pifunceq2}) because $\Pi'$ is an automorphic representation of $\GL_4(\A)$. Dividing (\ref{Pifunceq1}) by (\ref{Pifunceq2}) and observing that the local factors outside $D$ coincide, we obtain
\begin{equation}\label{Lquotienteq}
 \prod_{p|D}\frac{L(s,\Pi_p)L(1-s,\tilde\Pi'_p)\varepsilon(s,\Pi'_p)}
 {L(s,\Pi'_p)L(1-s,\tilde\Pi_p)\varepsilon(s,\Pi_p)}=1.
\end{equation}
It follows from unique prime factorization that if $p_1,\ldots,p_r$ are distinct primes, and if $R_1,\ldots,R_r\in\C(X)$ are such that
\begin{equation}\label{ratnlfctslemmaeq1}
  \prod_{i=1}^rR_i(p_i^s)=1\qquad\text{for all } s\in\C,
\end{equation}
then the rational functions $R_i$ are all constant. Hence, it follows from (\ref{Lquotienteq}) that
\begin{equation}\label{Lquotientpeq}
 \frac{L(s,\Pi_p)L(1-s,\tilde\Pi'_p)\varepsilon(s,\Pi'_p)}
 {L(s,\Pi'_p)L(1-s,\tilde\Pi_p)\varepsilon(s,\Pi_p)}\qquad\text{is constant for each }p|D.
\end{equation}
Fix a prime $p|D$, and write (\ref{Lquotientpeq}) as
\begin{equation}\label{Lquotientpeq2}
 \frac1{L(s,\Pi_p)}=c_pX^m\frac{L(1-s,\tilde\Pi'_p)}{L(s,\Pi'_p)L(1-s,\tilde\Pi_p)},
\end{equation}
where $c_p$ is a constant, $X=p^{-s}$, and $m$ is some exponent coming from the $\varepsilon$-factors.
Let $\alpha,\beta,\gamma,\delta$ be the Satake parameters of $\Pi_p$, so that
$$
 L(s,\Pi_p)=\frac1{(1-\alpha p^{-s})(1-\beta p^{-s})(1-\gamma p^{-s})(1-\delta p^{-s})}.
$$
Substituting into (\ref{Lquotientpeq2}), we obtain
\begin{align*}
 &(1-\alpha X)(1-\beta X)(1-\gamma X)(1-\delta X)\\
 &\;\;=(1-(\alpha p)^{-1}X^{-1})(1-(\beta p)^{-1}X^{-1})
 (1-(\gamma p)^{-1}X^{-1})(1-(\delta p)^{-1}X^{-1})
  c_pX^m\frac{L(1-s,\tilde\Pi'_p)}{L(s,\Pi'_p)}\\
 &\;\;=(X-(\alpha p)^{-1})(X-(\beta p)^{-1})
 (X-(\gamma p)^{-1})(X-(\delta p)^{-1})
  c_pX^{m-4}\frac{L(1-s,\tilde\Pi'_p)}{L(s,\Pi'_p)}.
\end{align*}
Consider the zeros of the functions on both sides of this equation.
On the left hand side, we have zeros exactly when $X = p^s$ is equal to
\begin{equation}\label{zeroslefteq}
 \alpha^{-1},\quad\beta^{-1},\quad\gamma^{-1},\quad\delta^{-1}
\end{equation}
(with repetitions allowed). On the right hand side, the factor $L(1-s,\tilde\Pi'_p)$
does not contribute any zeros, since local $L$-factors are never zero. The factor
$X^{m-4}$ might contribute the zero $0$, but this zero does certainly not
appear amongst the numbers (\ref{zeroslefteq}). Then there are the obvious
\emph{possible} zeros when $X$ equals
\begin{equation}\label{zerosrighteq}
 (\alpha p)^{-1},\quad(\beta p)^{-1},\quad(\gamma p)^{-1},\quad(\delta p)^{-1}.
\end{equation}
Recalling that $\alpha,\beta,\gamma,\delta$ originate from the Satake parameters of a holomorphic Siegel cusp form, the Ramanujan conjecture for such modular forms, proven in \cite{weissram}, implies that $|\alpha|=|\beta|=|\gamma|=|\delta|=1$. (Even without the full Ramanujan conjecture, known estimates as those in \cite{PS2} would lead to the same conclusion.) Hence there is no overlap between the numbers in (\ref{zeroslefteq}) and (\ref{zerosrighteq}). It follows that the factor $L(s,\Pi'_p)$ must contribute the zeros (\ref{zeroslefteq}) for the right hand side. In particular, $L(s,\Pi'_p)^{-1}$ is a polynomial in $p^{-s}$ of degree $4$, so that $\Pi'_p$ is a spherical representation. And then, evidently, its Satake parameters are precisely $\alpha$, $\beta$, $\gamma$ and $\delta$. This is equivalent to saying $\Pi'_p\cong\Pi_p$, proving our claim.

\vspace{3ex}
We now proved that the candidate representation $\Pi_4=\otimes\Pi_v$ is automorphic, and it remains to prove it is cuspidal. Assume that $\Pi_4$ is not cuspidal; we will obtain a contradiction. Being not cuspidal, $\Pi_4$ is a constituent of a globally induced representation from a proper parabolic subgroup of $\GL_4$. It follows that $L(s,\Pi_4)$ is, up to finitely many Euler factors, of one of the following forms.
\begin{enumerate}
 \item $L(s,\chi_1)L(s,\chi_2)L(s,\chi_3)L(s,\chi_4)$ with Hecke characters $\chi_i$ of $\A^\times$.
 \item $L(s,\chi_1)L(s,\chi_2)L(s,\tau)$ with Hecke characters $\chi_1,\chi_2$ of $\A^\times$ and a cuspidal, automorphic representation $\tau$ of $\GL_2(\A)$.
 \item $L(s,\chi_1)L(s,\tau)$ with a Hecke character $\chi_1$ of $\A^\times$ and a cuspidal, automorphic representation $\tau$ of $\GL_3(\A)$.
 \item $L(s,\tau_1)L(s,\tau_2)$ with cuspidal, automorphic representations $\tau_1,\tau_2$ of $\GL_2(\A)$.
\end{enumerate}
Note that all the characters and representations in this list must be unramified at every finite place, since the same is true for $\Pi_4$. If one of the cases i), ii) or iii) is true, then $L(s,\Pi_4\times\chi_1^{-1})$ has a pole. Since $L(s,\Pi_4\times\chi_1^{-1})=L(s,\pi\times\chi_1^{-1})$ and we are assuming that $F$ is not of Saito-Kurokawa type, this contradicts Theorem 2.2 of \cite{PSh}. Hence we are in case iv). But then $L(s,\Pi_4\times\tilde\tau_1)=L(s,\pi\times\tilde\tau_1)$ has a pole, contradicting Theorem \ref{entirenesstheorem}. This contradiction shows that $\Pi_4$ must be cuspidal.

\vspace{3ex}
It remains to prove the last statement. Since $\Pi_4$ is self-dual, it is well known that exactly one of the $L$-functions
$$
 L(s,\Pi_4,\Lambda^2)\qquad\text{or}\qquad L(s,\Pi_4,{\rm Sym}^2)
$$
has a pole at $s=1$. If $L(s,\Pi_4,{\rm Sym}^2)$ would have a pole at $s=1$, then $\Pi_4$ would be a (strong) lifting from the split orthogonal group $\SO_4$; see the Theorem on p.\ 680 of \cite{GJR2004} and the comments thereafter. By Lemma \ref{notfromSO4lemma} below, this is impossible. It follows that $L(s,\Pi_4,\Lambda^2)$ has a pole at $s=1$.
\end{proof}

\begin{lemma}\label{notfromSO4lemma}
 Let $F$ and $\pi=\otimes\pi_v$ be as in Theorem \ref{liftingtheorem}, and let $\Pi_4$ be the resulting lifting to $\GL_4(\A)$. Then there does not exist a cuspidal, automorphic representation $\sigma$ of $\SO_4(\A)$ such that $\Pi_4$ is a Langlands functorial lifting of $\sigma$.
\end{lemma}
\begin{proof}
The obstruction comes from the archimedean place. Recall that the dual group of $\SO_4$ is $\SO_4(\C)$, which we realize as
$$
 \SO_4(\C)=\{g\in\SL_4(\C)\;|\;^tg\mat{}{1_2}{1_2}{}g=\mat{}{1_2}{1_2}{}\}.
$$
Let $\varphi:\:W_\R\rightarrow\GL_4(\C)$ be the archimedean $L$-parameter given explicitly in (\ref{localparameterarcheq}). If $\Pi_4$ would come from $\SO_4$, there would exist a matrix $g\in\GL_4(\C)$ such that
$$
 g\varphi(w)g^{-1}\in\SO_4(\C)\qquad\text{for all }w\in W_\R.
$$
Then $^t(g\varphi(w)g^{-1})\mat{}{1_2}{1_2}{}(g\varphi(w)g^{-1})=\mat{}{1_2}{1_2}{}$ for all $w\in W_\R$, or equivalently
$$
 ^t\varphi(w)S\varphi(w)=S,\qquad\text{where }S=\,^tg\mat{}{1_2}{1_2}{}g.
$$
Letting $w$ run through non-zero complex numbers $re^{i\theta}$ shows that $S$ is of the form
$$
 S=\begin{bmatrix}&&a\\&&&b\\a\\&b\end{bmatrix}.
$$
But then letting $w=j$ yields the contradiction $-S=S$.
\end{proof}

We will next consider a backwards lifting of $\Pi_4$ in order to obtain a globally generic, cuspidal, automorphic representation on $\GSp_4(\A)$ in the same $L$-packet as $\pi$.
\begin{theorem}\label{genericswitchtheorem}
 Let $F$ and $\pi=\otimes\pi_v$ be as in Theorem \ref{liftingtheorem}. Then there exists a globally generic, cuspidal, automorphic representation $\pi^g=\otimes\pi^g_v$ of $\GSp_4(\A)$ such that $\pi^g_p\cong\pi_p$ for all primes $p$, and such that $\pi^g_\infty$ is the generic discrete series representation of $\PGSp_4(\R)$ lying in the same $L$-packet as $\pi_\infty$. Any globally generic, cuspidal automorphic representation $\sigma=\otimes\sigma_v$ of $\GSp_4(\A)$ such that $\sigma_p\cong\pi_p$ for almost all $p$ coincides with $\pi^g$.
\end{theorem}
\begin{proof} Let $\Pi_4$ be the lifting of $\pi$ to $\GL_4$ constructed in Theorem \ref{liftingtheorem}. Since $\Pi_4$ is symplectic, we can apply Theorem 4 of \cite{GRS}. The conclusion is that there exists a non-zero representation $\sigma=\sigma_1\oplus\ldots\oplus\sigma_m$ of $\PGSp_4(\A)$ such that each $\sigma_i$ is globally generic, cuspidal, automorphic and weakly lifts to $\Pi_4$. By Theorem 9 of \cite{GRS}, there can be only one $\sigma_i$, i.e., $\sigma$ is itself irreducible. Note that ``weak lift'' in \cite{GRS} includes the condition that the lift is functorial with respect to archimedean $L$-parameters (see \cite{GRS}, p.\ 733). In particular, the archimedean component of $\sigma$ is the generic discrete series representation of $\PGSp_4(\R)$ lying in the same $L$-packet as $\pi_\infty$. Evidently, the local components $\sigma_p$ and $\pi_p$ are isomorphic for almost all primes $p$. It remains to show that this is the case for \emph{all} primes $p$. This can be done by a similar argument as in the proof of Theorem \ref{liftingtheorem}. Dividing the functional equations for the degree $4$ $L$-functions $L(s,\pi)$ and $L(s,\sigma)$, and comparing the resulting zeros at a particular prime $p$, shows first that $L(s,\sigma_p)$ is a degree $4$ Euler factor. Hence $\sigma_p$ is an unramified representation. The same comparison of zeros then also implies that $\sigma_p$ and $\pi_p$ have the same Satake parameters. The last assertion follows from the strong multiplicity one result Theorem 9 of \cite{GRS}.
\end{proof}

With $F$ and $\pi$ as above, we constructed a strong functorial lifting of $\pi$ to $\GL_4$ with respect to the natural inclusion of dual groups $\SSp_4(\C)\subset\GL_4(\C)$. Similarly, we will now produce a strong functorial lifting of $\pi$ to $\GL_5$ with respect to the morphism $\rho_5:\:\SSp_4(\C)\rightarrow\GL_5(\C)$ of dual groups, where $\rho_5$ is the irreducible $5$-dimensional representation of $\SSp_4(\C)$. Let $L(s,\pi,\rho_5)$ be the degree $5$ (standard) $L$-function of $F$. If the $L$-parameter at a prime $p$ is given by (\ref{localparameternonarcheq}), then
\begin{equation}\label{deg5Eulerfactoreq}
 L(s,\pi_p,\rho_5)=\frac1{(1-p^{-s})(1-\chi_1(p)p^{-s})(1-\chi_1^{-1}(p)p^{-s})(1-\chi_2(p)p^{-s})(1-\chi_2^{-1}(p)p^{-s})}.
\end{equation}

\begin{theorem}\label{liftingGL5theorem}
 Let $F$ and $\pi=\otimes\pi_v$ be as in Theorem \ref{liftingtheorem}. Then there exists a cuspidal, automorphic representation $\Pi_5$ of $\GL_5(\A)$ such that
 \begin{equation}
  L(s,\pi,\rho_5)=L(s,\Pi_5)
 \end{equation}
 (equality of completed Euler products). The representation $\Pi_5$ is a strong functorial lifting of $\pi$ to $\GL_5$ with respect to the morphism $\rho_5:\:\SSp_4(\C)\rightarrow\GL_5(\C)$ of dual groups. Moreover, $\Pi_5$ is orthogonal, i.e., the symmetric square $L$-function $L(s,\Pi_5,{\rm Sym}^2)$ has a pole at $s=1$.
\end{theorem}
\begin{proof}
A straightforward calculation verifies that
\begin{equation}\label{exteriorsquaredeg5relationeq}
 L_f(s,\Pi_4,\Lambda^2)=L_f(s,\pi,\rho_5)\zeta(s).
\end{equation}
Here, the subscript $f$ indicates that the Euler product defining the $L$-functions is taken over finite places only, and $\zeta(s)$ denotes the Riemann zeta function. By Theorem \ref{liftingtheorem}, the function $L_f(s,\Pi_4,\Lambda^2)$ has a simple pole at $s=1$. It follows that $L_f(s,\pi,\rho_5)$ is holomorphic and non-zero at $s=1$. Together with \cite{Gr}, Theorem 2, we obtain that $L_f(s,\pi,\rho_5)$ has no poles on ${\rm Re}(s)=1$. Now by \cite{kimsar}, Theorem A, $L(s,\Pi_4,\Lambda^2)$ is the $L$-function of an automorphic representation of $\GSp_6(\A)$ of the form
\begin{equation}\label{liftingGL5theoremeq1}
 {\rm Ind}(\tau_1\otimes\ldots\otimes\tau_m)
\end{equation}
where $\tau_1,\ldots,\tau_m$ are unitary, cuspidal, automorphic representations of $\GL_{n_i}(\A)$, $n_1+\ldots+n_m=6$. Since $L_f(s,\Pi_4,\Lambda^2)$ has a simple pole at $s=1$, it follows that exactly one of the $\tau_i$, say $\tau_m$, is the trivial representation of $\GL_1(\A)$. Cancelling out one zeta factor, we see that
\begin{equation}\label{liftingGL5theoremeq2}
 L_f(s,\pi,\rho_5)=L_f(s,\tau_1)\ldots L_f(s,\tau_{m-1}).
\end{equation}
Observe that since $\pi$ is unramified at every finite place, the same must be true for the $\tau_i$. If we had $n_i=1$ for some $i$, then $L(s,\tau_i)$, and therefore the right hand side of (\ref{liftingGL5theoremeq1}), would have a pole on ${\rm Re}(s)=1$. This contradicts the observation from above that $L_f(s,\pi,\rho_5)$ has no poles on ${\rm Re}(s)=1$. Hence $n_i>1$ for all $i$, so that the only possibilities for the set $\{n_1,\ldots,n_{m-1}\}$ are $\{2,3\}$ and $\{5\}$. Assume the former is the case, so that, say, $\tau_1$ is a cuspidal representation of $\GL_2(\A)$ and $\tau_2$ is a cuspidal representation of $\GL_3(\A)$. Let $\Pi_5={\rm Ind}(\tau_1\otimes\tau_2)$. It is not hard to verify that
\begin{equation}\label{liftingGL5theoremeq3}
 L_f(s,\Pi_5,\Lambda^2)=L_f(s,\Pi_4,{\rm Sym}^2),
\end{equation}
which we know is an entire function. On the other hand,
\begin{equation}\label{liftingGL5theoremeq4}
 L_f(s,\Pi_5,\Lambda^2)=L_f(s,\omega_{\tau_1})L_f(s,\tau_1\times\tau_2)L_f(s,\omega_{\tau_2}\times\tilde\tau_2),
\end{equation}
where $\omega_{\tau_i}$ is the central character of $\tau_i$. Since the latter is everywhere unramified, the right hand side of (\ref{liftingGL5theoremeq4}) has a pole on ${\rm Re}(s)=1$. This contradiction shows that the assumption $\{n_1,\ldots,n_{m-1}\}=\{2,3\}$ must be wrong. Hence $\Pi_5:=\tau_1$ is a cuspidal representation of $\GL_5(\A)$ such that
\begin{equation}\label{liftingGL5theoremeq5}
 L_f(s,\pi,\rho_5)=L_f(s,\Pi_5).
\end{equation}
This implies that $\Pi_5$ is a lifting of $\pi$ (with respect to the morphism $\rho_5$ of dual groups) at every finite place. At the archimedean place, observe that the $L$-parameter of ${\rm Ind}(\tau_1\otimes\tau_2)$ equals the exterior square of the $L$-parameter of $\Pi_4$, since the lifting of \cite{kimsar} is strong. On the other hand, an explicit calculation shows that the exterior square of the $L$-parameter of $\Pi_4$ equals the $L$-parameter of $\pi$ composed with $\rho_5$, plus the trivial representation of $W_\R$ (in other words, the archimedean place behaves exactly as the finite places, so that (\ref{exteriorsquaredeg5relationeq}) holds in fact for the completed $L$-functions). Cancelling out the trivial representation on both sides, one obtains an equality of the $L$-parameter of $\tau_1$ with the $L$-parameter of $\pi$ composed with $\rho_5$. Hence $\Pi_5$ is a functorial lifting of $\pi$ also at the archimedean place.

\vspace{3ex}
Finally, $\Pi_5$ is orthogonal since the exterior square $L_f(s,\Pi_5,\Lambda^2)$ has no pole at $s=1$; see (\ref{liftingGL5theoremeq3}). This concludes the proof.
\end{proof}
\subsection{Analytic properties of $L$-functions}\label{analyticpropertiesapplicationssec}
For $n\in\{1,4,5,10,14,16\}$ let $\rho_n$ be the $n$-dimensional irreducible representation of $\SSp_4(\C)$. In the notation of \cite{FuHa1991}, Sect.\ 16.2, we have $\rho_4=\Gamma_{1,0}$, $\rho_5=\Gamma_{0,1}$, $\rho_{10}=\Gamma_{2,0}$, $\rho_{14}=\Gamma_{0,2}$ and $\rho_{16}=\Gamma_{1,1}$. Of course, $\rho_4$ is the natural representation of $\SSp_4(\C)$ on $\C^4$, which is also called the spin representation. An explicit formula for the representation $\rho_5$ as a map $\SSp_4(\C)\rightarrow\SO_5(\C)$ is given in Appendix A.7 of \cite{NF}. (Somewhat confusingly, in the theory of Siegel modular forms $\rho_5$ is often referred to as the standard representation, even though it is $\rho_4$ that is the non-trivial representation of lowest dimension.) The representation $\rho_{10}$ is the adjoint representation of $\SSp_4(\C)$ on its Lie algebra. We have the following relations,
\begin{align}
 \Lambda^2\rho_4&=\rho_1+\rho_5,\label{rhorelationseq1}\\
\Lambda^2\rho_5 = {\rm Sym}^2\rho_4&=\rho_{10},\label{rhorelationseq2}\\
 {\rm Sym}^2\rho_5&=\rho_1+\rho_{14},\label{rhorelationseq3}\\
 \rho_4\otimes\rho_5&=\rho_4+\rho_{16}.\label{rhorelationseq4}
\end{align}
Let $F$ and $\pi$ be as in Theorem \ref{liftingtheorem}.
To each $\rho_n$ we have an associated global $L$-function $L(s,\pi,\rho_n)$. We will list the archimedean $L$- and $\varepsilon$-factors (the latter with respect to the character $\psi^{-1}$, where $\psi(x)=e^{-2\pi ix}$). Let $\Gamma_\R$ and $\Gamma_\C$ be as in (\ref{GammaRCdefeq}). The archimedean factors depend only on the minimal $K$-type $(l,l)$ of $\pi_\infty$.
$$\renewcommand{\arraystretch}{1.3}
 \begin{array}{ccc}
  \rho&L(s,\pi_\infty,\rho)&\varepsilon(s,\pi_\infty,\rho,\psi^{-1})\\\hline
  \rho_1&\Gamma_\R(s)&1\\
  \rho_4&\Gamma_\C(s+\frac12)\Gamma_\C(s+l-\frac32)&(-1)^l\\
  \rho_5&\Gamma_\R(s)\Gamma_\C(s+l-1)\Gamma_\C(s+l-2)&1\\
  \rho_{10}&\Gamma_\R(s+1)^2\,\Gamma_\C(s+1)\Gamma_\C(s+l-1)\Gamma_\C(s+l-2)\Gamma_\C(s+2l-3)&1\\
  \rho_{14}&\Gamma_\R(s)^2\,\Gamma_\C(s+1)\Gamma_\C(s+l-1)\Gamma_\C(s+l-2)&1\\
    &\Gamma_\C(s+2l-2)\Gamma_\C(s+2l-3)\Gamma_\C(s+2l-4)&\\
  \rho_{16}&\Gamma_\C(s+\frac12)^2\,\Gamma_\C(s+l-\frac12)\Gamma_\C(s+l-\frac32)^2\,\Gamma_\C(s+l-\frac52)&-1\\
    &\Gamma_\C(2+2l-\frac52)\Gamma_\C(2+2l-\frac72)&
 \end{array}
$$
These factors are normalized so that they fit into a functional equation relating $s$ and $1-s$, and hence differ from the traditional factors used in the theory of Siegel modular forms. For example, the classical Andrianov spin $L$-function relates $s$ and $2l-2-s$; see \cite{An1974}, Theorem 3.1.1. To obtain the Andrianov $\Gamma$-factors, one has to replace $s$ by $s-l+\frac32$ in the above factor for $\rho_4$.

\begin{theorem}\label{Lrhonanalyticpropertiestheorem}
 Let $F$ and $\pi$ be as in Theorem \ref{liftingtheorem}. The Euler products defining the $L$-functions $L_f(s,\pi,\rho_n)$, for $n\in\{4,5,10,14,16\}$, are absolutely convergent for ${\rm Re}(s)>1$. They have meromorphic continuation to the entire complex plane, have no zeros or poles on ${\rm Re}(s)\geq1$, and the completed $L$-functions (using the above archimedean factors) satisfy the functional equation
 $$
  L(s,\pi,\rho_n)=\varepsilon(s,\pi,\rho_n)L(1-s,\pi,\rho_n).
 $$
Furthermore, for $n\in\{4,5,10\}$, the functions $L(s,\pi,\rho_n)$ are entire and bounded in vertical strips.
\end{theorem}
\begin{proof} By definition, $L(s,\pi,\rho_4)=L(s,\Pi_4)$ and $L(s,\pi,\rho_5)=L(s,\Pi_5)$. Hence, the analytic properties of $L(s,\pi,\rho_4)$ and $L(s,\pi,\rho_5)$ follow from the known analytic properties of $L$-functions of cuspidal representations on $\GL_n$. For the
absolute convergence of the Euler products in ${\rm Re}(s)>1$,  see \cite{JaSh1981}, Theorem 5.3. As for the adjoint $L$-function, it follows from (\ref{rhorelationseq2}) that
\begin{equation}\label{Lrhonanalyticpropertiestheoremeq1}
 L(s,\Pi_4,{\rm Sym}^2)=L(s,\pi,\rho_{10}).
\end{equation}
Since $\Pi_4$ is symplectic by Theorem \ref{liftingtheorem}, this is an entire function; see~\cite[Thm.\ 7.5]{BG}. The absolute convergence in ${\rm Re}(s)>1$ follows from \cite[Thm.\ 5.3]{JaSh1981}, together with the known automorphy, hence absolute convergence, of $L(s,\Pi_4,\Lambda^2)$. Since symmetric square $L$-functions are accessible via the Langlands-Shahidi method, the boundedness in vertical strips follows from \cite{GeSh2001}, and the functional equation follows from \cite[Cor.\ 6.7]{Sha1988}. The non-vanishing on ${\rm Re}(s)=1$ follows also from the Langlands-Shahidi method; see Sect.\ 5 of \cite{Sh1981}. From (\ref{rhorelationseq3}) we get
\begin{equation}\label{Lrhonanalyticpropertiestheoremeq2}
 L(s,\Pi_5,{\rm Sym}^2)=Z(s)L(s,\pi,\rho_{14}),
\end{equation}
where $Z(s)=\Gamma_\R(s)\zeta(s)$ is the completed Riemann zeta function. Observe that $L(s,\Pi_5,\Lambda^2)$ is absolutely convergent for ${\rm Re}(s)>1$ by (\ref{liftingGL5theoremeq3}). Together with \cite{JaSh1981}, Theorem 5.3, this implies the absolute convergence of $L(s,\Pi_5,{\rm Sym}^2)$, and hence of $L(s,\pi,\rho_{14})$, in ${\rm Re}(s)>1$. The meromorphic continuation of $L(s,\pi,\rho_{14})$ is obvious from (\ref{Lrhonanalyticpropertiestheoremeq2}). Since this is an identity of complete Euler products, and since our liftings are strongly functorial, it also implies the asserted functional equation. By Theorem \ref{liftingGL5theorem} the function $L(s,\Pi_5,{\rm Sym}^2)$ has a simple pole at $s=1$, while otherwise it is holomorphic and non-vanishing on ${\rm Re}(s)=1$. Since the same is true for $Z(s)$, it follows that $L(s,\pi,\rho_{14})$ is holomorphic and non-vanishing on ${\rm Re}(s)=1$. Since
\begin{equation}\label{Lrhonanalyticpropertiestheoremeq3}
 L(s,\Pi_4\times\Pi_5)=L(s,\pi)L(s,\pi,\rho_{16}).
\end{equation}
by (\ref{rhorelationseq4}), similar arguments apply to $L(s,\pi,\rho_{16})$.
\end{proof}

Let $r$ be a positive integer, and $\tau$ a cuspidal, automorphic representation of $\GL_r(\A)$. Let $\sigma_r$ be the standard representation of the dual group $\GL_r(\C)$.
Then we can consider the Rankin-Selberg Euler products $L(s,\pi\times\tau,\rho_n\otimes\sigma_r)$, where $\rho_n$ is one of the irreducible representations of $\SSp_4(\C)$ considered above. For $n=4$ or $n=5$, since $\Pi_4$ and $\Pi_5$ are functorial liftings of $\pi$, we have
\begin{equation}\label{GSp4GLnLfunctioneq}
 L(s,\pi\times\tau,\rho_n\times\sigma_r)=L(s,\Pi_n\times\tau),
\end{equation}
where the $L$-function on the right is a standard Rankin-Selberg $L$-function for $\GL_n\times\GL_r$. From the well-known properties of these $L$-functions, the following result is immediate. For $\varepsilon>0$ and a closed interval $I$ on the real line we use the notation $T_{\varepsilon,I}=\{s\in\C\;|\;{\rm Re}(s)\in I,\:|{\rm Im}(s)|\geq\varepsilon\}$, as in \cite{GeSh2001}.

\begin{theorem}\label{Lrhonsigmaranalyticpropertiestheorem}
 Let $F$ and $\pi$ be as in Theorem \ref{liftingtheorem}. Let $r$ be a positive integer, and $\tau$ a (unitary) cuspidal, automorphic representation of $\GL_r(\A)$. Let $n=4$ or $n=5$. Then the Euler products defining the $\GSp_4\times\GL_r$ $L$-functions $L(s,\pi\times\tau,\rho_n\otimes\sigma_r)$ are absolutely convergent for ${\rm Re}(s)>1$. They have meromorphic continuation to the entire complex plane, and the completed $L$-functions satisfy the functional equation
 \begin{equation}\label{functionalequationtheoremeq2c}
  L(s,\pi\times\tau,\rho_n\otimes\sigma_r)=\varepsilon(s,\pi\times\tau,\rho_n\otimes\sigma_r)L(1-s,\tilde\pi\times\tilde\tau,\rho_n\otimes\sigma_r).
 \end{equation}
 These $L$-functions are entire, bounded in vertical strips, and non-vanishing on ${\rm Re}(s)\geq1$, except in the cases
 \begin{itemize}
  \item $n=r=4$ and $\tau=|\det|^{it}\otimes\Pi_4$, where $t\in\R$ and $\Pi_4$ is the lifting of $\pi$ from Theorem \ref{liftingtheorem}, or
  \item $n=r=5$ and $\tau=|\det|^{it}\otimes\Pi_5$, where $t\in\R$ and $\Pi_5$ is the lifting of $\pi$ from Theorem \ref{liftingGL5theorem}.
 \end{itemize}
 In these cases the function $L(s,\pi\times\tau,\rho_n\otimes\sigma_r)$ is holomorphic except for simple poles at $s=-it$ and $s=1-it$, and is bounded on all sets of the form $T_{\varepsilon,I}$ with $\varepsilon>|t|$.
\end{theorem}
\begin{proof}
For the precise location of poles, see Theorem 2.4 of \cite{CPS2004}. For boundedness in vertical strips, see Corollary 2 on p.\ 80 of \cite{GeSh2001}.
\end{proof}

\begin{theorem}\label{Lrhonrhoranalyticpropertiestheorem}
 Let $F$ and $F'$ be Siegel cusp forms with respect to $\SSp_4(\Z)$. Assume that $F$ and $F'$ are Hecke eigenforms, that they are not Saito-Kurokawa lifts and that $\pi$ resp.\ $\pi'$ are the associated cuspidal, automorphic representations of $\GSp_4(\A)$. Let $n\in\{4,5\}$ and $n'\in\{4,5\}$. Then the Euler products defining the $\GSp_4\times\GSp_4$ $L$-functions $L(s,\pi\times\pi',\rho_n\otimes\rho_{n'})$ are absolutely convergent for ${\rm Re}(s)>1$. They have meromorphic continuation to the entire complex plane, and the completed $L$-functions satisfy the expected functional equation. These functions are entire, bounded in vertical strips, and non-vanishing on ${\rm Re}(s)\geq1$, except if $n=n'$ and $F$ and $F'$ have the same Hecke eigenvalues. In these cases the function $L(s,\pi\times\pi',\rho_n\otimes\rho_{n'})$ is holomorphic except for simple poles at $s=0$ and $s=1$, and is bounded on all sets of the form $T_{\varepsilon,I}$ with $\varepsilon>0$.
\end{theorem}
\begin{proof} By definition,
\begin{equation}\label{Lrhonrhoranalyticpropertiestheoremeq1}
 L(s,\pi\times\pi',\rho_n\otimes\rho_{n'})=L(s,\Pi_n\times\Pi'_{n'}),
\end{equation}
where $\Pi_n$ (resp.\ $\Pi'_{n'}$) is the lifting of $\pi$ (resp.\ $\pi'$) to $\GL_n$ (resp.\ $\GL_{n'}$). Evidently, $F$ and $F'$ have the same Hecke eigenvalues if and only if $\pi$ and $\pi'$ are nearly equivalent if and only if $\Pi_n=\Pi'_n$. Hence everything follows from the properties of $L$-functions for $\GL_n\times\GL_{n'}$.
\end{proof}

\begin{theorem}\label{Lnonnegativitytheorem}
 Let $F$ and $F'$ be Siegel cusp forms with respect to $\SSp_4(\Z)$. Assume that $F$ and $F'$ are Hecke eigenforms, that they are not Saito-Kurokawa lifts and that $\pi$ resp.\ $\pi'$ are the associated cuspidal, automorphic representations of $\GSp_4(\A)$. Let $\chi$ be a Hecke character of $\A^\times$ (possibly trivial) such that $\chi^2=1$, $\tau_2$ be a unitary, cuspidal, automorphic representation of $\GL_2(\A)$ with trivial central character, and $\tau_3$ be a unitary, self-dual, cuspidal, automorphic representation of $\GL_3(\A)$.  Then the central values
 $$L(1/2,\pi\otimes \chi,\rho_4), \quad L(1/2,\pi\otimes \tau_2,\rho_5 \otimes \sigma_2),\quad L(1/2,\pi\otimes \tau_3,\rho_4 \otimes \sigma_3),\quad L(1/2,\pi\times\pi',\rho_4\otimes\rho_5),$$ are all non-negative.

\end{theorem}
\begin{proof}
Recall that the lifting $\Pi_4$ is symplectic by Theorem \ref{liftingtheorem}, and the lifting $\Pi_5$ is orthogonal by Theorem \ref{liftingGL5theorem}. Furthermore, $\tau_2$ is symplectic and $\tau_3$ is orthogonal. All the assertions now follow from Theorem 1.1 of \cite{La2003}.
\end{proof}

\subsection{Critical values of $L$-functions}
If $L(s, \mathcal{M})$ is an arithmetically defined (or motivic) $L$-series associated to an arithmetic object $\mathcal{M}$, it is of interest to study its values at certain critical points $s=m$. For these critical points, conjectures due to Deligne predict that  $L(m,\mathcal{M})$ is the product of a suitable transcendental number $\Omega$ and an algebraic number $A(m,\mathcal{M})$ and furthermore, if $\sigma$ is an automorphism of $\C$, then $A(m,\mathcal{M})^\sigma = A(m, \mathcal{M}^\sigma)$.
In this subsection, we will prove critical value results in the spirit of the above conjecture for $L$-functions associated to a Siegel cusp form of full level.

\vspace{3ex}

For any subring $A \subset \C$, let $S_l\big(\SSp_4(\Z), A\big)$ be the $A$-module consisting of the holomorphic Siegel cusp forms $F(Z) = \sum_S a(F,S)e^{2\pi i\,{\rm tr}(SZ)} $ of weight $l$  for $\SSp_4(\Z)$ for which all the Fourier coefficients $a(F,S)$ lie in $A$. For $F \in S_l\big(\SSp_4(\Z), \C\big)$ and $\sigma \in \Aut(\C)$, define ${}^\sigma \!F$ by $$ {}^\sigma\! F(Z) = \sum_S \sigma (a(F,S)) e^{2\pi i\,{\rm tr}(SZ)}.$$ By work of Shimura~\cite{shimura75}, we know that ${}^\sigma \! F \in S_l\big(\SSp_4(\Z), \C\big)$  and $$S_l\big(\SSp_4(\Z), \Q\big) \otimes_{\Q} \C = S_l\big(\SSp_4(\Z), \C\big). $$ Also, if $F$ is a Hecke eigenform, so is $ {}^\sigma\! F$; see Kurokawa~\cite{kurokawa}.

\vspace{2ex}

Now, let $F \in S_l\big(\SSp_4(\Z), \C\big)$ be an eigenform for all the Hecke operators and let $\pi_F$ be the associated cuspidal, automorphic representation of $\GSp_4(\A)$. We assume that $F$ is not of Saito-Kurokawa type, so that the hypothesis of Theorems \ref{liftingtheorem} is satisfied. Let $\Pi_F$ be the resulting cuspidal, automorphic representation of $\GL_4(\A)$. The representation $\Pi_F$ is regular and algebraic in the sense of~\cite{clozel}. We define the $\sigma$-twist ${}^\sigma\! \Pi_F$ as in~\cite{clozel} or~\cite{walds85}. This can be described locally. If $\Pi_F = \otimes_p \Pi_{F,p} \otimes \Pi_{F,\infty} $, then ${}^\sigma\!\Pi_F= \otimes_p {}^\sigma\!\Pi_{F,p} \otimes \Pi_{F,\infty}$, where for any finite place $p$, \begin{equation}{}^\sigma\! \mathrm{Ind}_{B(\Q_p)}^{\GL_4(\Q_p)} (\chi_1 \otimes \ldots \otimes \chi_4) = \mathrm{Ind}_{B(\Q_p)}^{\GL_4(\Q_p)} ({}^\sigma \!\chi_1' \otimes \ldots \otimes {}^\sigma \!\chi_4') .\end{equation} Here $B$ is the standard Borel of $\GL_4$, $\chi_1, \ldots , \chi_4$ are characters of $\Q_p^\times$ and for any such $\chi$, $${}^\sigma \! \chi'(x) = \sigma(\chi(x)|x|^\frac12)|x|^{-\frac12}.$$
(See Waldspurger's example on~\cite[p.\ 125]{walds85}.) We have the following lemma.

\begin{lemma}\label{sigmatwistlemma}Let $F$ be a holomorphic Siegel cusp form for $\SSp_4(\Z)$ that is an eigenfunction for all the Hecke operators and $\sigma$ an automorphism of $\C$. Suppose that $F$ is not of Saito-Kurokawa type. Then $ {}^\sigma\!F$ is not of Saito-Kurokawa type. Furthermore, if $\Pi_{{}^\sigma\!F}$ is the cuspidal, automorphic representation of $\GL_4(\A)$ obtained from $ {}^\sigma\!F$ by Theorem \ref{liftingtheorem}, then $$\Pi_{ {}^\sigma\!F} = {}^\sigma\!\Pi_F.$$
\end{lemma}
\begin{proof} First of all, note that the condition of $F$ being of Saito-Kurokawa type is equivalent to simple relations among the Fourier coefficients of $F$ as in~\cite[p.\ 76]{eichzag}. These relations are preserved under the action of $\sigma$. This proves the first part of the lemma. For the second part, we need to show that  $\Pi_{ {}^\sigma\!F, p} = {}^\sigma\!\Pi_{F,p}$ for any prime $p$. Fix such a prime $p$. Suppose that $$\Pi_{ {}^\sigma\!F, p} = \mathrm{Ind}_{B(\Q_p)}^{\GL_4(\Q_p)} (\chi_1'' \otimes \ldots \otimes \chi_4'').$$ Let $\lambda_{F,m}$ be the eigenvalue for the Hecke operator $T(m)$ acting on $F$. For the exact definition of these Hecke operators, we refer the reader to Andrianov~\cite{An1974}.  By Kurokawa~\cite{kurokawa}, we know that $\sigma(\lambda_{F,m}) = \lambda_{{}^\sigma\!F,m}$. By writing the local degree-$4$ Euler factors in terms of the Hecke eigenvalues, we conclude that the multisets $\{\chi_1''(p),\ldots,\chi_4''(p)\}$ and
$\{{}^\sigma \!\chi_1'(p),\ldots,{}^\sigma \!\chi_4'(p)\}$ are identical. Hence
 $$ \mathrm{Ind}_{B(\Q_p)}^{\GL_4(\Q_p)} (\chi_1'' \otimes \ldots \otimes \chi_4'') =  \mathrm{Ind}_{B(\Q_p)}^{\GL_4(\Q_p)} ({}^\sigma \!\chi_1' \otimes \ldots \otimes {}^\sigma \!\chi_4'),$$ and therefore, $\Pi_{ {}^\sigma\!F} = {}^\sigma\!\Pi_F.$
\end{proof}

We now supply certain results on critical $L$-values for $\GSp_4 \times \GL_n$ where $n \in \{1,2 \}$.
\subsubsection*{Critical value result for $\GSp_4 \times \GL_1$}

In~\cite{raghugrob}, Grobner and Raghuram define certain periods of automorphic forms on $\GL_{2n}$ by comparing cohomologies in top degree. We refrain from giving the definition of these periods here in the interest of brevity and instead refer the reader to~\cite[Sec. 4]{raghugrob} for details. When the results of~\cite{raghugrob} are combined with our Theorem~\ref{liftingtheorem}, we obtain a special value result for twists of Siegel eigenforms by Dirichlet characters. We now briefly describe this result.

 \vspace{2ex}
Let $F$ be a holomorphic Siegel cusp form of weight $l$ for $\SSp_4(\Z)$ that is an eigenfunction for all the Hecke operators and is not of Saito-Kurokawa type. Let $\Pi_F = \Pi_{F, f} \otimes \Pi_{F, \infty} $ be the lift to $\GL_4(\A)$; here $\Pi_{F, f}$ denotes the finite part of the automorphic representation $\Pi_{F}$. Let $\Q(\Pi_F)$ denote the rationality field of $\Pi_F$ as defined in~\cite{clozel}. This is a totally real number field, and by the argument of Lemma~\ref{sigmatwistlemma}, we know that $\Q(\Pi_F)$ equals the field generated by all the Hecke eigenvalues of $F$. For convenience we will denote $\Q(\Pi_F)$ by $\Q(F)$.  For $\chi$ a Hecke character of $\A$ of finite order, let $\Q(\chi)$ denote the number field generated by the image of $\chi$ and let $\Q(F, \chi)$ denote the compositum of $\Q(F)$ and $\Q(\chi)$. Define $\Q({}^\sigma\!F)$ ,  $\Q({}^\sigma\!\chi)$ and  $\Q({}^\sigma\!F, {}^\sigma\!\chi)$ similarly.

\begin{rem} \rm By Mizumoto~\cite{miz}, it is known that for any integer $l$, there exists an orthogonal basis $\{F_1, F_2, \ldots , F_d \}$ comprising of Hecke eigenfunctions for $S_l\big(\SSp_4(\Z), \C\big)$ such that each $F_i \in S_l\big(\SSp_4(\Z),\Q(F_i)\big)$. \end{rem}

Let $\omega^+(\Pi_{F,f})$ and $\omega^-(\Pi_{F,f})$ be the periods as defined in~\cite[Sect.\ 4]{raghugrob}. For convenience, let us denote them by  $\omega^+(F)$ and $\omega^-(F)$ respectively. These are non-zero complex numbers  obtained from comparing cohomologies in top degree. Also let $c(\Pi_{F, \infty, 0})$ be as in~\cite[Sect.\ 6.6]{raghugrob} and denote $c(\Pi_{F, \infty}, 0)^{-1}$ by $\omega_\infty(l)$; this notation is justified because $c(\Pi_{F, \infty}, 0)$ depends only on the weight $l$. Then, applying the main theorem of~\cite{raghugrob} to the representation $\Pi_F$ leads to the following special value result.

\begin{theorem}[\cite{raghugrob}, Corollary 8.3.1]\label{criticalgsp4gl1} Let $F$ be a cuspidal Siegel eigenform of weight $l$ for $\SSp_4(\Z)$ that is not of Saito-Kurokawa type and let $\chi$ be a Hecke character of $\A$ of finite order. Let $\epsilon_\chi \in \{ + , - \}$ denote the sign of $\chi(-1)$, $\mathcal{G}(\chi_f)$ denote the Gauss sum for $\chi$ and $L_f(s,\pi_{F} \times \chi) = \prod_{p < \infty} L(s,\pi_{F,p} \times \chi_p) $ denote the finite part of the $L$-function. Define $$A(F, \chi) = \frac{L_f(\frac12, \pi_{F} \times \chi)}{\omega^{\epsilon_\chi}(F)\omega_\infty(l) \mathcal{G}(\chi_f)^2}.$$ Then we have \begin{enumerate}
\item $A(F, \chi) \in \Q(F, \chi),$
\item For any automorphism $\sigma$ of $\C$, we have $\sigma(A(F, \chi) ) = A({}^\sigma\!F, {}^\sigma\!\chi)$.
\end{enumerate}
\end{theorem}
\begin{rem} \rm In~\cite{harrisoccult}, Harris defined certain ``occult" periods for $\GSp_4$ by comparing rational structures on Bessel models and rational structures on coherent cohomology and used these to study the critical values of the degree $4$ $L$-function for $\GSp_4$.
\end{rem}
As a corollary to Theorem~\ref{criticalgsp4gl1}, we immediately obtain the following result.

\begin{corollary}\label{bochererevidence}Let $d_1$ and $d_2$ be two fundamental discriminants of the same sign, and let $\chi_{d_1}$, $\chi_{d_2}$ be the associated quadratic Dirichlet characters. Let $F$ be a cuspidal Siegel eigenform of weight $l$ for $\SSp_4(\Z)$ that is not of Saito-Kurokawa type. Then we have $$L_f(\frac12, \pi_{F} \times \chi_{d_1}) \sim_{\Q(F)} L_f(\frac12, \pi_{F} \times \chi_{d_2}),$$ where "$\sim_{\Q(F)}$" means up to multiplication by an element in the number field $\Q(F)$.
\end{corollary}

\begin{rem} \rm In~\cite{boch-conj}, B\"ocherer made a
remarkable conjecture that expresses the central values $L_f(\frac12, \pi_{F} \times \chi_{d})$, as $d$ varies over negative fundamental discriminants, in terms of the Fourier coefficients of $F$ of discriminant $d$. In particular,  B\"ocherer's conjecture implies Corollary~\ref{bochererevidence} above for the case that $d_1$, $d_2$ are both negative. Thus Corollary~\ref{bochererevidence} can be read as providing evidence towards B\"ocherer's conjecture.
\end{rem}

\subsubsection*{Critical value result for $\GSp_4 \times \GL_2$}

Next, we provide a critical value result for $\GSp_4 \times \GL_2$. This result will not use our lifting theorem, but instead will follow from the integral representation (Theorem~\ref{theoremsecondintegralrep}) using the methods of~\cite{pullback}.

\begin{theorem}\label{specialgsp4gl2} Let $F $ be a cuspidal Siegel eigenform of weight $l$ for $\SSp_4(\Z)$ such that $F\in S_l\big(\SSp_4(\Z),\Q(F)\big)$. Let $g \in S_l(N,\chi)$ be a primitive Hecke eigenform of level $N$ and nebentypus $\chi$; here $N$ is any positive integer, and $\chi$ a Dirichlet character mod $N$. Let $\pi_F$ and $\tau_g$ be the irreducible, cuspidal, automorphic representations of $\GSp_4(\A)$ and $\GL_2(\A)$ corresponding to $F$ and $g$. Let $\Q(F,g,\chi)$ be the field generated by the Hecke eigenvalues of $F$, the Hecke eigenvalues of $g$ and the values taken by $\chi$. For a positive integer $k$, $1\le k \le \frac{\ell}{2}-2$, define $$A(F,g;k) = \frac{L_f(\frac{\ell}{2}-k, \pi_F \times \pi_g)}{\pi^{5\ell-4k-4}\langle F, F\rangle \langle  g,  g \rangle}.$$ Then we have, \begin{enumerate} \item $A(F,g;k) \in \Q(F,g,\chi)$,  \item For an automorphism $\sigma$ of $\C$, $\sigma(A(F,g;k)) = A({}^\sigma\!F, {}^\sigma\!g;k).$ \end{enumerate}
\end{theorem}

\begin{rem} \rm Note that the first claim of the above theorem actually follows from the second.
\end{rem}

\begin{rem} \rm Partial results towards the above theorem have been previously obtained by  B\"ocherer--Heim~\cite{heimboch}, Furusawa~\cite{Fu}, and various combinations of the authors~\cite{ameya2011, PS1, saha1, pullback}.
\end{rem}

\begin{proof} The proof is essentially identical to that of Theorem~8.1 of~\cite{pullback} which proved the above result under certain restrictions on $N$, $\chi$ and $F$. More precisely, in~\cite{pullback}, $N$ was assumed to be squarefree and all its prime divisors inert in a certain quadratic field, $\chi$ was assumed to be trivial, and $F$ was assumed to satisfy a certain non-vanishing condition on the Fourier coefficients. These restrictions were necessary because the relevant integral representation~\cite[Thm. 6.4] {pullback} in that paper was proved only under these assumptions. The special value result in that paper followed from the integral representation by first rewriting the integral representation in classical language and then using results of Garrett and Harris and the theory of nearly holomorphic functions due to Shimura.

\vspace{1ex}

However, in the current paper, the second integral representation (Theorem~\ref{theoremsecondintegralrep}) works for general $N$ and $\chi$ and the non-vanishing assumption on $F$ is always satisfied, as shown in~\cite{squarefree}. Now, Theorem~\ref{specialgsp4gl2} follows in an identical manner as in~\cite{pullback}, because the remaining ingredients (the theory of nearly holomorphic functions and the results of Garrett and Harris) are true for general $N$ and $\chi$. It is worth noting, however, that we still need to assume that the weights of $F$ and $g$ are equal (even though the integral representation, Theorem~\ref{theoremsecondintegralrep}, works for arbitrary $g$) because otherwise the Eisenstein series $E_{\Upsilon}(g,s)$ at the right-most critical point (corresponding to $s = \frac{l}6 - \frac12$) is no longer holomorphic.

\end{proof}  

\addcontentsline{toc}{section}{Bibliography}
\bibliography{transferbib}{}

\bibliographystyle{plain}

\end{document}